\documentclass[10pt]{report}

\usepackage{indentfirst, times, graphicx, subfigure,
  setspace, amsmath, multirow, epstopdf, amsmath, amsthm,
  amssymb, algorithmic, mathrsfs}
  
\usepackage[chapter]{algorithm} 
\usepackage{color}
  
\usepackage{standalone}
\usepackage{lmodern}            

\usepackage{hyperref}

\def\blfootnote{\xdef\@thefnmark{}\@footnotetext}

\newcommand{\expect}[1]{\mathbb{E}{\l[#1\r]}}
\newcommand{\mf}[1]{\mathbf{#1}}

\newcommand{\wt}[1]{\widetilde{#1}}
\newcommand{\mc}[1]{\mathcal{#1}}

\newcommand{\dotp}[2]{\left\langle#1,#2\right\rangle}
\newcommand{\m}{\mathcal}
\newcommand{\mb}{\mathbb}
\newcommand\argmin{\mathop{\mbox{argmin}}}
\newcommand{\mx}{\mbox{\footnotesize{max}\,}}
\newcommand{\mn}{\mbox{\footnotesize{min}\,}}
\newcommand{\rank}{\mathrm{\,rank}}
\newcommand{\card}{\mathrm{Card}}
\newcommand{\sign}{\mathrm{sign}}
\def\r{\right}
\def\l{\left}
\newcommand{\eps}{\varepsilon}
\newcommand{\var}{\mbox{Var}}
\newcommand{\proj}{\mbox{{\rm Proj}}}

\newcommand{\pr}[1]{\mathbb{P}{\left(#1\right)}}

\usepackage{titlesec}
 \titleformat{\chapter}[display]
     {\normalfont\Large\bfseries}{\chaptertitlename\ \thechapter}{20pt}{\Large}

\newcommand{\beas}{\begin{eqnarray*}}
\newcommand{\enas}{\end{eqnarray*}}
\newcommand{\bea}{\begin{eqnarray}}
\newcommand{\ena}{\end{eqnarray}}
\newcommand{\bms}{\begin{multline*}}
\newcommand{\ems}{\end{multline*}}
\newcommand{\bels}{\begin{align*}}
\newcommand{\enls}{\end{align*}}
\newcommand{\bel}{\begin{align}}
\newcommand{\enl}{\end{align}}

\newcommand{\ignore}[1]{}
\newcommand{\tr}{\mbox{tr}}

\newtheorem{theorem}{Theorem}[section]
\newtheorem{corollary}{Corollary}[section]

\newtheorem{proposition}{Proposition}[section]
\newtheorem{remark}{Remark}[section]
\newtheorem{lemma}{Lemma}[section]
\newtheorem{definition}{Definition}[section]

\newtheorem{assumption}{Assumption}[section]

\begin{document}

\pagenumbering{roman}   

\begin{center}%
\doublespacing
\null
\vfill\vfill
\MakeUppercase{II. High Dimensional Estimation under Weak Moment Assumptions: Structured Recovery and Matrix Estimation}\par
 \vskip 0.16 true in
      by%
      \vskip 0.16 true in
      {\begin{tabular}[t]{c}Xiaohan Wei \\*[0.6 true in]%
       \rule{4in}{1.5pt}
       \end{tabular}%
       \par}

      \vskip 1.0 true in
      \singlespacing
      Presented to the     \\
      FACULTY OF THE USC GRADUATE SCHOOL      \\
      UNIVERSITY OF SOUTHERN CALIFORNIA   \\
      In Partial Fulfillment of the       \\
      Requirements for the Degree         \\
      DOCTOR OF PHILOSOPHY                \\
      \MakeUppercase{(Electrical Engineering)}           \\*[1.0 true cm]%
      \vfill
      {\small Dec \ \ 2019\par}
    \end{center}%
\par
\vfill
\begin{center}%
{\normalsize\ Copyright~ 2019 \ \hfill ~Xiaohan Wei}%
\end{center}%
           
\thispagestyle{empty} 

\newpage
\chapter*{}
\noindent   Approved by\\

\noindent Professor Stanislav Minsker,\\
Committee Chair,\\
Department of Mathematics,\\
\textit{University of Southern California}.\\

\noindent Professor Michael Neely,\\
Committee Chair,\\
Department of Electrical Engineering,\\
\textit{University of Southern California}.\\

\noindent Professor Larry Goldstein,\\
Department of Mathematics,\\
\textit{University of Southern California}.\\

\noindent Professor Mihailo Jovanovic,\\
Department of Electrical Engineering,\\
\textit{University of Southern California}.\\

\noindent Professor Ashutosh Nayyar,\\
Department of Electrical Engineering,\\
\textit{University of Southern California}.\\

\newpage


\chapter*{Dedication}
\addcontentsline{toc}{chapter}{Dedication}

To my parents and my wife, Yuhong, who supported me both mentally and financially over the years.

\newpage

\doublespacing
\chapter*{Acknowledgements}
\addcontentsline{toc}{chapter}{Acknowledgements}

First, I would like to thank my advisor professor Michael J. Neely for guiding me throughout the PhD journey since Summer 2013. He is a man of accuracy and rigorousness, always passionate about discussing concrete research problems, and willing to roll up the sleeves and grind through technical details with me. His way of treating research topics significantly impacts me. Rather than blindly following existing works and doing incremental works when trying to get into a new area, I learned to ask fundamental mathematical questions, making connections to the tools and theories we already familiar with and be not afraid of getting my hands dirty. His blazing new ideas are my morale boost when grasping in the dark.

Next, I would like to thank professor Stanislav Minsker, who is the advisor on my high-dimensional statistics research. I got to know him during the Math-547 statistical learning course Fall 2015. Though not much senior than me, he is already extremely knowledgable on the statistical learning area and has been widely recognized for his works on robust high-dimensional statistics. He is a quick thinker and can always point out meaningful new directions hiding rather deeply which eventually lead to high-quality publications. I would have published no paper on this area should I never met with him. Along the way, he also teaches me how to sell my works and helps me practicing my seminar talks, which lead to impressive presentations and Ming-Hsieh scholarships.    

Also, I would like to thank professor Larry Goldstein, whom I met during a small paper reading group Spring 2016. He is an expert on Stein's method and, as a senior professor, surprisingly accessible to PhD students and active on various research areas. Together with Prof. Minsker, we had quite a few fruitful discussions and made some nice progress on robust statistics.  

I would also like to thank professor Mihailo Jovanovic, Ashutosh Nayyar for discussing research problems with me and siting on my qualifying exam committee. I appreciate them for their valuable comments and suggestions. 

Moreover, I thank my senior lab mates Hao Yu and  Sucha Supittayapornpong who were always accessible to discussing problems with me and came up with new research ideas. Also, Ruda Zhang, Lang Wang, and Jie Ruan studied various math courses and interesting math problems with me and helped me clear up the hurdles on different stages, for which I really appreciate. Special thanks to professor Qing Ling, who was my undergraduate advisor, but continuously influences me on various aspects of my academic career. 

Last but not least, I would like to take the chance to express my gratitude for folks who made contribution on various stages of my research. In particular, I thank Zhuoran Yang, for lighting up new areas and expanding my research horizon, Dongsheng Ding, who brings idea from control perspective and is always passionate to try out research ideas with me, Sheng Chen for sharing with me his perspective on robust LASSO problems, professor Jason D. Lee for working on the geometric median problem with me, and Jianshu Chen from Tencent AI who introduced me to the area of reinforcement learning.

\singlespacing
\renewcommand{\contentsname}{Table of Contents}  
\tableofcontents  


\chapter*{Abstract}
\addcontentsline{toc}{chapter}{Abstract}
The purpose of this thesis is to develop new theories on high-dimensional structured signal recovery under a rather weak assumption on the measurements that only a finite number of moments exists. High-dimensional recovery has been one of the emerging topics in the last decade partly due to the celebrated work of Candes, Romberg and Tao (e.g. \cite{candes2006stable,candes2004robust}). The original analysis there (and the works thereafter) necessitates a strong concentration argument (namely, the restricted isometry property), which only holds for a rather restricted class of measurements with light-tailed distributions. It had long been conjectured that high-dimensional recovery is possible even if restricted isometry type conditions do not hold, but the general theory was beyond the grasp until very recently, when the works
\cite{mendelson2014learning,koltchinskii2015bounding} propose a new ``small-ball method''. 
In these two papers, the authors initiated a new analysis framework for general empirical risk minimization (ERM) problems with respect to the square loss, which is ``robust'' and can potentially allow heavy-tailed loss functions. The materials in this thesis are partly inspired by \cite{mendelson2014learning}, but are of a different mindset: rather than directly analyzing the existing ERMs for signal recovery for which it is difficult to avoid strong moment assumptions,
we show that, in many circumstances, by carefully re-designing the ERMs to start with, one can still achieve the minimax optimal statistical rate of signal recovery with very high probability under much weaker assumptions than existing works.



\clearpage
\pagenumbering{arabic}  

\doublespacing



\chapter{Introduction and a Heavy-tailed Framework}
The main focus of this thesis is to study robust recovery and estimation in the presence of heavy-tailed design or noises. In the analysis of regression models and matrix estimation procedures, it is common to assume that the data satisfy an certain model along with a set of assumptions such as i.i.d. observations from a Gaussian distribution. However, the data in practical world often violate such assumptions due to noise and outliers.  One of the viable ways to model noisy data and outliers is to assume that the observations are generated by a \textit{heavy-tailed distribution}\footnote{Throughout the thesis, a distribution is ``heavy-tailed'' if and only if finite number of moments exists.}. Therefore, 
the practical significance of this research is to relax the strong assumptions ubiquitous in previous high-dimensional recovery and estimation works, thereby reducing the gap between mathematical theories and the real world problems.


\section{Background}
\subsection{From least square to supremum of an empirical process}
Our main focus is the high-dimensional empirical risk minimization (ERM). We start by considering the classical least squares ERM, which is easy to understand and serves as a foundation for all subsequent development of this thesis.
Let $\Theta$ be a measurable subset of $\mathbb{R}^d$, let $\mathbf x\in\mathbb{R}^d$ be a random vector, and let $y\in\mathbb{R}$ be a target response variable. One would like to find some vector $\theta^*\in\Theta$ so that $\dotp{\mf x}{\theta^*}$ and $y$ are as close as possible. A classical way of measuring the distance is to consider the square loss function 
$
(\dotp{\mathbf x}{\theta} - y)^2,
$
and one hopes to select this $\theta^*\in\Theta$ so as to minimize the expected loss:
\[
\mathcal L(\theta) = \mathbb{E}(\dotp{\mathbf x}{\theta} - y)^2= \theta^T\expect{\mathbf{x}\mathbf{x}^T}\theta - 2 \expect{y\mathbf{x}^T}\theta
+ \expect{y^2}.
\]

The term $\expect{y^2}$ is irrelevant in terms of mimization. However,
it should be noted that in most cases, the expectations $\expect{\mathbf{x}\mathbf{x}^T}$ and $\expect{y\mathbf{x}^T}$ are \textit{not known}. Instead, we only have access to the i.i.d. samples $\{\mf x_i, y _i\}_{i=1}^N$ of $\{\mf x,y\}$. Thus, we instead aim to find 
$\widehat{\theta}_N\in \Theta$ minimizing the \textit{empirical loss}:
\begin{equation}\label{eq:l2-loss}
\mathcal L_N(\theta) = 
\frac1N\sum_{i=1}^N\theta^T\mathbf{x}_i\mathbf{x}_i^T\theta - \frac2N\sum_{i=1}^Ny_i\mathbf{x}_i^T\theta
\end{equation}

It should also be note that there are two aspects of this problem. One aspect is the estimation problem which aims to find some 
$\widehat\theta_N$ so that $\|\theta^* - \widehat{\theta}_N\|_2$ is as small as possible. The other aspect is the prediction problem, namely, given an estimator $\widehat{\theta}_N$, we would like to know how it performs on future data compared to $\theta^*$, i.e.
\[
\expect{\l(\dotp{\mathbf x}{\widehat\theta_N} - y\r)^2 - (\dotp{\mathbf x}{\theta^*} - y)^2~\l|~\{\mf x_i, y _i\}_{i=1}^N\r.}.
\]
This is also known as the ``generalization error'' of $\widehat{\theta}_N$. Throughout the thesis, we mainly focus on the estimation problem. 

The classical way one analyzes the performance of \eqref{eq:l2-loss} is as follows (\cite{bartlett2005local}): since $\widehat{\theta}_N\in \Theta$ minimizes 
\eqref{eq:l2-loss}, it must satisfy:
\[
\frac1N\sum_{i=1}^N\widehat\theta_N^T\mathbf{x}_i\mathbf{x}_i^T\widehat\theta_N - \frac2N\sum_{i=1}^Ny_i\mathbf{x}_i^T\widehat\theta_N
\leq \frac1N\sum_{i=1}^N(\theta^*)^T\mathbf{x}_i\mathbf{x}_i^T(\theta^*) - \frac2N\sum_{i=1}^Ny_i\mathbf{x}_i^T\theta^*.
\]
Rearranging the terms gives:
\[
\frac1N\sum_{i=1}^N(\widehat\theta_N -\theta^*)^T\mathbf{x}_i\mathbf{x}_i^T(\widehat\theta_N-\theta^*) 
- \frac2N\sum_{i=1}^N(y_i-\mathbf{x}_i^T\theta^*)\mathbf{x}_i^T(\widehat\theta_N-\theta^*)
\leq 0.
\]
Thus, it follows that:
\begin{multline}\label{eq:bias-variance}
\frac1N\sum_{i=1}^N(\widehat\theta_N -\theta^*)^T\mathbf{x}_i\mathbf{x}_i^T(\widehat\theta_N-\theta^*)
\leq \frac2N\sum_{i=1}^N(y_i-\mathbf{x}_i^T\theta^*)\mathbf{x}_i^T(\widehat\theta_N-\theta^*)
-2\expect{(y_i-\mathbf{x}_i^T\theta^*)\mathbf{x}_i^T(\widehat\theta_N-\theta^*)} \\
+ 2\expect{(y_i-\mathbf{x}_i^T\theta^*)\mathbf{x}_i^T(\widehat\theta_N-\theta^*)}.
\end{multline}
The right hand side corresponds to the classical ``bias-variance decomposition''. When $\expect{y_i} = \expect{\mathbf{x}_i^T\theta^*}$, the last term (which is the bias) is 0 and we only have the variance term. 
It should be kept in mind though that in general this bias term can be non-zero and increasing the bias in some sense can actually help us control the variance, which will be discussed in more details later.

If one believes that the matrix $\frac1N\sum_{i=1}^N\mathbf{x}_i\mathbf{x}_i^T$ is invertible in the range of 
$\Theta-\Theta:=\{\theta_1-\theta_2:~\theta_1,\theta_2\in\Theta\}$, i.e.
\begin{equation}\label{eq:least-eigen}
\inf_{\theta_1,\theta_2\in\Theta}\frac1N\sum_{i=1}^N\frac{(\theta_1-\theta_2)^T\mathbf{x}_i\mathbf{x}_i^T(\theta_1-\theta_2)}
{\|\theta_1-\theta_2\|_2^2}
\geq \sigma_{\min}
\end{equation}
for some \textit{absolute constant}\footnote{Throughout the thesis, an absolute constant is a constant that is independent of parameters of the problem.} 
$\sigma_{\min}>0$ and 
\begin{equation}\label{eq:sup-process}
\sup_{\theta_1,\theta_2\in\Theta}\frac{\l|\frac2N\sum_{i=1}^N(y_i-\mathbf{x}_i^T\theta^*)\mathbf{x}_i^T(\theta_1-\theta_2)
-2\expect{(y_i-\mathbf{x}_i^T\theta^*)\mathbf{x}_i^T(\theta_1-\theta_2)}\r| }{\|\theta_1-\theta_2\|_2}
\leq \gamma 
\end{equation}
for some constant $\gamma>0$. Then, \eqref{eq:bias-variance} implies
\[
 \sigma_{\min}\|\widehat\theta_N-\theta^*\|_2^2\leq \gamma\|\widehat\theta_N-\theta^*\|_2~~\Rightarrow
 \|\widehat\theta_N-\theta^*\|_2\leq \frac{\gamma}{\sigma_{\min}}.
\]
However, there are only limited scenarios where \eqref{eq:least-eigen} holds. It is wrong, for example, when $N<d$ and
$\Theta-\Theta$ spans $\mathbb{R}^d$. Furthermore, the validity of \eqref{eq:sup-process}, which essentially requires 
$\frac2N\sum_{i=1}^N(y_i-\mathbf{x}_i^T\theta^*)\mathbf{x}_i^T(\theta_1-\theta_2)$ to be uniformly concentrated around 
$2\expect{(y_i-\mathbf{x}_i^T\theta^*)\mathbf{x}_i^T(\theta_1-\theta_2)}$, 
is also questionable.

On the other hand, it is obvious that $\frac1N\sum_{i=1}^N\mathbf{x}_i\mathbf{x}_i^T$ has to satisfy some invertibility conditions in order to estimate $\theta^*$. For example, when $\theta^*$ lies in the null space of $\frac1N\sum_{i=1}^N\mathbf{x}_i\mathbf{x}_i^T$, asking for a bound on $\|\widehat\theta_N-\theta^*\|_2$ is meaningless. Over the years, people have been trying to identify minimal conditions so that objectives like \eqref{eq:least-eigen} and \eqref{eq:sup-process} holds true probabilistically, and our goal is to further expand the scope of this line of research.

\subsection{Supremum of an empirical process: binary functions}
It turns out that proving inequalities \eqref{eq:least-eigen} and \eqref{eq:sup-process} belongs to a more general class of problems, namely, bounding the supremum of an empirical process. Historically, such kind of problems originates from the well-known Glivenko-Cantelli theorem. 
\begin{theorem}[Glivenko-Cantelli]
Suppose $X_1,~X_2,~\cdots,~X_N\in\mathbb{R}$ is a sequence of independent and identically distributed (i.i.d.) random variables on the probability space $(\Omega,\Sigma,P)$ with a cumulative distribution function (CDF) $F(t):=P(X\leq t)$. Define the empirical CDF as $F_N(t) := \frac1N\sum_{i=1}^N1_{\{X_i\leq t\}}$, where $1_{\{x\leq t\}}$ is the indicator function which is 1 if $x\leq t$ and 0 otherwise.
Then,
\[
\lim_{N\rightarrow\infty}\sup_{t\in\mathbb{R}}|F_N(t) - F(t)| = 0,
\]
with probability 1.
\end{theorem}
The class of random variables $\{F_N(t) - F(t)\}_{t\in\mathbb{R}}$ is historically called an empirical process. Of course, one can show that  the supremum is measurable (i.e. $\sup_{t\in\mathbb{R}}|F_N(t) - F(t)|$ is a random variable on the space $(\Omega,\Sigma,P)$, see \cite{durrett2019probability}), on which we will not discuss here. We further refer readers to Chapter 1 of \cite{wellner2013weak} for a synthetic treatment of the measurability issue of the supremum. 
In the absence of supremum (i.e. for a fixed $t\in\mathbb{R}$), this is just law of large numbers. 
However, with the supremum, it is not immediately clear why the convergence is still true. More generally, for any class of (measurable) sets 
$\mathcal{S}$,
one can ask if the following supremum always converges to zero: 
$$\lim_{N\rightarrow\infty}\sup_{S\in\mathcal{S}}\l| \frac1N\sum_{i=1}^N1_{\{X_i\in S\}}- \expect{1_{\{X_i\in S\}}}\r|,$$
which turns out to be wrong, as is illustrated in the following simple example:

\begin{remark}[A non-Glivenko-Cantelli class]
Consider the following class of indicator functions:\footnote{This example is from Peter Bartlett's lecture notes: \url{https://www.stat.berkeley.edu/~bartlett/courses/2013spring-stat210b/notes/8notes.pdf}} $\mathcal{F}:= \{1_{S}(x):~|S|<\infty\}$, where
$|S|$ denotes the cardinality of the set $S$. Then, it can be easily seen that for any random variable $X_i$ with a continuous distribution function $F$, $\expect{1_{\{X_i\in S\}}} = P(X_i\in S) = 0$. However, we have $\sup_{S\in\mathcal{S}} \frac1N\sum_{i=1}^N1_{\{X_i\in S\}} = 1$. Thus, the supremum does not converge to 0.
\end{remark}
This example indicates that there has to be some \textit{measure of complexity} which indicates that the class of function $\{1_{S}(x):~\text{Card}(S)<\infty\}$ is ``too large'' for the supremum to converge, whereas $\{1_{\{x\leq t\}}: ~t\in\mathbb{R}\}$ is small. This type of complexity, which appears very often in machine learning theory, is call \textit{Rademacher complexity}. 
\begin{definition}
Consider a set of samples $\{X_i\}_{i=1}^N\subseteq\mathcal{X}$ and a function class $\mathcal{F}$ containing 
$f:\mathcal{X}\rightarrow\{-1,+1\}$. The \textit{empirical Rademacher complexity} of the function class $\mathcal{F}$ given 
$\{X_i\}_{i=1}^N$ is defined as
\[
R_N(\mathcal{F}) := \expect{\l.\sup_{f\in\mathcal{F}}\frac2N\sum_{i=1}^N\varepsilon_if(X_i)\r|~X_1,\cdots,X_N},
\]
where $\varepsilon_i$ being i.i.d. Rademacher random variables (taking $+1$ and $-1$ with equal probability) and independent of 
$\{X_i\}_{i=1}^N$. 
\end{definition}

We have the following general theorem from \cite{bartlett2002rademacher}:
\begin{theorem}[Theorem 5 of \cite{bartlett2002rademacher}]\label{thm:bm}
Let $P$ be a probability distribution on the product space $\mathcal{X}\times\{-1,+1\}$, where $\mathcal{X}\subseteq\mathbb{R}^d$ is a set\footnote{In general this set does not have to be in $\mathbb{R}^d$. We state this way mainly because we only care about finite dimensional spaces in this thesis.}. Let $\mathcal{F}$ be a class of functions containing $f:\mathcal{X}\rightarrow\{-1,+1\}$. Let 
$\{X_i,Y_i\}_{i=1}^N$ be i.i.d. samples drawn according to $P$, then, with probability at least $1-\delta$, for every function 
$f\in\mathcal{F}$,
\[
\l|P(Y\neq f(X)) - \frac1N\sum_{i=1}^N1_{\{Y_i\neq f(X_i)\}}\r|\leq R_N(\mathcal{F}) + \sqrt{\frac{\ln(1/\delta)}{N}},
\]
\end{theorem}

Intuitively, $R_N(\mathcal{F})$ measures the correlations of $\mathcal F$ with random noise, and if $\mathcal F$ can fit noise very well, then, its complexity is high.
To use this theorem, one should be able to compute or upper bound $R_N(\mathcal{F})$. One way is to apply the following theorem.

\begin{theorem}[Theorem 6 of \cite{bartlett2002rademacher}]
Fix any sequence of samples $X_1,~\cdots,~X_N$. For a function class $\mathcal{F}$ containing $f:\mathcal{X}\rightarrow\{-1,+1\}$, define the restriction of $\mathcal{F}$ to the samples as follows:
\begin{equation}\label{eq:F-seq}
\mathcal F|_X:=\{(f(X_1),~\cdots,~f(X_N)):~f\in\mathcal F\}.
\end{equation}
Then, 
\[
R_N(\mathcal{F})\leq L\sqrt{\frac{\log|\mathcal F|_X|}{N}},
\]
where $L$ is an absolute constant and $|\mathcal F|_X|$ denotes the cardinality of the set $\mathcal F|_X$.
\end{theorem}
This theorem can be proved by using the fact that $\varepsilon_i$ is a sub-Gaussian random variable, together with a union bound.
Using this lemma, one can easily prove the Glivenko-Cantelli theorem. To be more specific, we let $\mathcal{F} = \{1_{\{x\leq t\}}: ~t\in\mathbb{R}\}$. One can show that $|\mathcal{F}|_X| = N+1$, and thus, it follows from Theorem \ref{thm:bm} with probability at least $1-\delta$,
\[
\sup_{t\in\mathbb{R}}|F_N(t) - F(t)|\leq L\sqrt{\frac{\log(N+1)}{N}} + \sqrt{\frac{\ln(1/\delta)}{N}}.
\]
By Borel-Cantelli Lemma, we finish the proof. Thus, not only do we prove the Glivenko-Cantelli theorem, we also get the explicit rate of convergence $\mathcal{O}\l(\sqrt{\frac{\log(N+1)}{N}}\r)$, which is otherwise difficult to obtain from ``classical'' proof (for example, in \cite{durrett2019probability}). However, as we shall see, this $\log N$ is in fact not needed.

It turns out for a class of binary functions $\mathcal{F}$, Rademacher complexity can be upper bounded by the well known complexity measure, namely, the Vapnik-Chervonenkis(VC) dimension.
\begin{definition}[VC dimension of sets]
Consider a class of sets $\mathcal{C}$ in $\mathcal X$. 
For a sequence of samples $X_1,~\cdots,~X_N\in\mathcal{X}$, we say $\mathcal C$ shatters $X_1,~\cdots,~X_N$ if 
\[
\Delta(\mathcal C, X_1,~\cdots,~X_N) := |\{C\cap\{X_1,~\cdots,~X_N\}:~C\in\mathcal C\}| = 2^N.
\]
The VC dimension of the class $\mathcal{C}$, denoted as $V(\mathcal C)$, is defined as
\[
V(\mathcal C) = \min\{N\in\mathbb{N}:\max_{X_1,~\cdots,~X_N\in\mathcal X}\Delta(\mathcal C, X_1,~\cdots,~X_N) < 2^N\}.
\]
\end{definition}

We also have the definition of VC dimension for a class of binary functions $\mathcal{F}$:
\begin{definition}[VC dimension for classification functions]
Consider a function class $\mathcal{F}$ containing $f:\mathcal{X}\rightarrow\{-1,+1\}$.
The VC dimension of the class $\mathcal{F}$, denoted as $V(\mathcal F)$, is defined as
\[
V(\mathcal F) = \min\{N\in\mathbb{N}:\max_{X_1,~\cdots,~X_N\in\mathcal X}|\mathcal F|_X| < 2^N\},
\]
where $F|_X$ is defined in \eqref{eq:F-seq}.
\end{definition}
We have the following theorem:
\begin{theorem}[Theorem 7 of \cite{bartlett2002rademacher}]
Fix any sequence of samples $X_1,~\cdots,~X_N$. For a function class $\mathcal{F}$ containing $f:\mathcal{X}\rightarrow\{-1,+1\}$, 
\[
R_N(\mathcal{F})\leq L\sqrt{\frac{V(\mathcal F)}{N}},
\]
where $L$ is an absolute constant.
\end{theorem}
The proof of this theorem is highly non-trivial as it is a delicate combination of Dudley's entropy bound together with Haussler's inequality (see Chapter 2.6-2.7 of \cite{wellner2013weak}).
One can see immediately though by using this theorem instead, we can remove the log factor in the earlier proof of Glivenko-Cantelli theorem.

\subsection{Supremum of an empirical process: General cases}
In this section, we review some key results which bound supremum of a classes of function with range in 
$\mathbb{R}$ instead of $\{+1,-1\}$. During the last 80's and 90's, there has been tremendous progress in empirical process theory, mostly associated with the name of Michel Talagrand, who has made significant contributions on various aspects of concentration of empirical processes including (but not limited to): Talagrand's isoperimetric inequality \cite{talagrand1995concentration}, Talagrand's concentration inequality \cite{massart2000constants}, contraction principle \cite{ledoux2013probability} and generic chaining \cite{talagrand2014upper}. Several of his results will be in use throughout this thesis. 

We will take this opportunity trying to explain why Talagrand's generic chaining is of central importance in modern empirical process theory and how it leads to a tight bound for the supremum of an empirical process. To understand this, we start with the follow basic definition of covering and packing numbers:

\begin{definition}[Covering and packing numbers]
Consider a compact metric space cosisting of a set $T$ and a metric $d:T\times T\rightarrow\mathbb{R}_+$,
\begin{itemize}
\item An $\varepsilon$-covering of $T$ under the metric $d$ is a collection of $\{t_1,\cdots,t_N\}\subseteq T$ such that for all $t\in T$, there exists some $i\in\{1,2,\cdots,N\}$ with $d(t,t_i)\leq \varepsilon$. The $\varepsilon$-covering number $\mathcal N(T,d,\varepsilon)$ is the cardinality of the minimal $\varepsilon$-covering.
\item An $\varepsilon$-packing of $T$ under the metric $d$ is a collection of $\{t_1,\cdots,t_N\}\subseteq T$ such that for all $i\neq j$,  $d(t_i,t_j)\geq \varepsilon$. The $\varepsilon$-packing number $\mathcal M(T,d,\varepsilon)$ is the cardinality of the maximal $\varepsilon$-packing.
\end{itemize}
\end{definition}
It can be shown that covering and packing are (up to constant) the same \cite{wellner2013weak}: 
\[
\mathcal M(T,d,\varepsilon)\leq \mathcal N(T,d,\varepsilon) \leq \mathcal M(T,d,\varepsilon/2).
\]
The covering number can also be expressed in terms of general sets as opposed to metrics.
\begin{definition}[Covering net for general sets]
Let $A,B$ be two sets in $\mathbb{R}^d$, the covering number $\mathcal{N}(A,B)$ is the minimum number of translates of $B$ in order to cover $A$.
\end{definition}
It is obvious that when $A = T\subseteq\mathbb{R}^d$, $B$ is the unit ball under the metric $d$, then, $\mathcal{N}(A,\varepsilon B) = \mathcal N(T,d,\varepsilon)$.

The log of the covering number is also commonly referred to as the entropy number.
A classical way of estimating the covering number in $\mathbb{R}^d$ is the volume argument: Let $A, B$ be a subset of $\mathbb{R}^d$, then, it is not difficult to see that (Proposition 4.2 of \cite{vershynin2010lectures}):
\begin{equation}\label{eq:vol-argument}
\frac{Vol(A)}{Vol(\varepsilon B)}\leq\mathcal N(A,\varepsilon B)\leq\frac{Vol(A+\frac{\varepsilon}{2} B)}{Vol(\frac{\varepsilon}{2} B)},
\end{equation}
where $Vol(A)$ is the Euclidean $\mathbb{R}^d$ volume of the set $A$. In particular, this implies for $B$ being the unit ball under the metric $d$ in $\mathbb{R}^d$,
\[
\l(\frac{1}{\varepsilon}\r)^d\leq \mathcal N(B,d,\varepsilon)\leq \l(2+\frac{1}{\varepsilon}\r)^d
\]
However, in general, the volume argument can be suboptimal and sometimes difficult to compute. A somewhat easier way to bound the covering number is through Sudakov inequality. We need the following definition:
\begin{definition}[Gaussian mean width]
Let $K$ be a set in $\mathbb{R}^d$ and let $\mathbf g\sim\mathcal{N}(0,\mathbf I_{d\times d})$. The Gaussian mean width of the set 
$K$ is $\omega(K) = \expect{\sup_{\mathbf x\in K}\dotp{\mf g}{\mf x}}$. 
\end{definition}
 The quantity $\omega(K)$ is crucial in learning theory. Intuitively, it measures the average width of a set.
One can easily check when $K$ being a unit ball in the $k$ dimensional subset of $\mathbb{R}^d$, $\omega(K)=\sqrt{k}$, and when 
$K$ is the cross-polytope, i.e. $K= \{x\in\mathbb{R}^d,~ \|x\|_1 =1\}$, $\omega(K) = C\sqrt{\log d}$ for some absolute constant $C$.

The following is the well-known Sudakov inequality:

\begin{theorem}[Theorem 2.2 of  \cite{vershynin2010lectures}]
Let $B$ be a unit ball in $\mathbb{R}^d$. For every symmetric convex set $K\subseteq\mathbb{R}^d$, we have 
$\sqrt{\log(K,B)} \leq C\omega(K)$, where $C$ is an absolute constant.
\end{theorem} 

One might wonder if it is possible to reverse the Sudakov inequality and derive an upper bound on
$\expect{\sup_{\mathbf x\in K}\dotp{\mf g}{\mf x}}$, i.e. the supremum of a Gaussian process, in terms of covering numbers. This turns out to be a highly non-trivial task. The technique bounding the supremum via covering nets is commonly referred to as \textit{chaining}. Intuitively, chaining is a method of taking fine-grained union bounds on sets of infinite cardinality through progressively finer covering nets. We start by defining the \textit{sub-Gaussian process}:
\begin{definition}
A zero mean stochastic process $\{X_t\}_{t\in T}$ with respect to a metric $d$ in $T$ is called sub-Gaussian, if for every $t_1,t_2\in T$, and any 
$\lambda\geq0$,
\[
\mathbb{E}{\exp(\lambda(X_{t_1} - X_{t_2}))}\leq \exp\l(\frac{\lambda^2d(t_1,t_2)^2}{2}\r).
\]
\end{definition}
For sub-Gaussian processes, we have the following key result due to R. Dudley. The technique proving this theorem is commonly referred to as Dudley's chaining:

\begin{theorem}[Dudley's entropy integral (Corollary 2.2.8 of \cite{wellner2013weak})]\label{thm:dudley}
Consider a zero mean sub-Gaussian stochastic process $\{X_t\}_{t\in T}$ with respect to a metric $d$ in $T$. Then, 
\[
\expect{\sup_{t\in T}X_{t}}\leq \int_{0}^{\infty}\sqrt{\log \mathcal N(T,d,\varepsilon)}d \varepsilon.
\]
\end{theorem}
One might wonder how tight this bound is. The following (not so trivial) example indicates that this bound is far from being tight.

\begin{remark}[A difficult set for Dudley's entropy integral]
This example can be found as an exercise in Chapter 2.2 of \cite{talagrand2014upper}. 
Consider the Gaussian mean width $\omega(T)$ of the probability simplex:
\begin{equation}\label{eq:prob-simplex}
T = \{t\in\mathbb{R}^d:~t\geq0,~ \|t\|_1= 1\},
\end{equation}
where $t\geq0$ is entrywise.
It is easy to check that $W(K) = C\sqrt{\log d}$ for some absolute constant $C$. Now, compute the Dudley's entropy integral with $d$ being the $\ell_2$-norm. One can show that (somewhat surprisingly)
\[
 \int_{0}^{\infty}\sqrt{\log \mathcal N(T,\ell_2,\varepsilon)}d \varepsilon\geq c (\log d)^{3/2},
\]  
where $c>0$ is some absolute constant. Thus, Dudley's integral is off by a factor of $\log d$. 
\end{remark}

One way to prove the previous remark is to rewrite the Dudley integral in another form. We consider a sequence of subsets 
$T_n\subseteq T,~n=0,1,2,\cdots$ with the condition that $|T_n|\leq N_n$ where
\[
N_0 = 1,~N_n = 2^{2^n},~n\geq1.
\]
For any $t\in T$, define $d(t,T_n) = \inf_{t_n\in T_n}d(t,t_n)$.  Note right away we have $\sqrt{\log N_n} = 2^{n/2}$, $N_n^2 = N_{n+1}$ and the function $\sqrt{\log x}$ is related to the fact that in some sense this is the inverse of the function $\exp(-x^2)$ that governs the size of the tails of a Gaussian random variables. Define the \textit{entropy number} $e_n(T)$ as
\[
e_n(T) = \inf \sup_{t\in T} d(t, T_n),
\]
where the infimum is taken over all possible admissible sequences.
\begin{lemma}[Lemma 2.2.11 of \cite{talagrand2014upper}]\label{lem:entropy-bound}
Under the aforementioned conditions, there exists an absolute constant $L$ such that
\[
 \frac1L\sum_{n\geq 0}2^{n/2}e_n(T)\leq\int_{0}^{\infty}\sqrt{\log \mathcal N(T,d,\varepsilon)}d \varepsilon
 \leq L\sum_{n\geq 0}2^{n/2}e_n(T)
\] 
\end{lemma}

Then, one could lower bound the entropy integral by the left hand side and lower bound the sum by a properly constructed subset of the probability simplex \eqref{eq:prob-simplex} (e.g. one can take subsets $T_n$ of $T$ consisting of sequences $t=[t(i)]_{i=1}^d$ for which 
$t(i)\in\{0,1/n\}$.)

Note that combining Lemma \ref{lem:entropy-bound} with Theorem \ref{thm:dudley} one readily get
\[
\expect{\sup_{t\in T}X_{t}}\leq L\sum_{n\geq 0}2^{n/2}e_n(T) = L\sum_{n\geq 0}2^{n/2}\inf\sup_{t\in T}d(t,T_n)
\]
The key contribution of Talagrand is to realize that, surprisingly, if we exchange $\inf\sup_{t\in T}$ with the sum, then, this bound is tight!
To make this rigorous, we need the following definition of admissible sequence:
\begin{definition}[Admissible sequence]
Given a metric space $(T,d)$. 
We say a sequence of subsets $\{\mathcal A_n\}_{n\geq0}$ of $T$ is increasing if $\mathcal A_n\subseteq \mathcal A_{n+1},~\forall n$. A sequence of subsets $\{\mathcal A_n\}_{n\geq0}$ is admissible if it is increasing and satisfy the condition that $|\mathcal A_n|\leq N_n$ where
\[
N_0 = 1,~N_n = 2^{2^n},~n\geq1.
\]
\end{definition}

\begin{definition}[Talagrand functionals]
Given a constant $\alpha>0$ and a metric space $(T,d)$. The Talagrand $\gamma_\alpha$ functional is defined as
\[
\gamma_\alpha(T) = \inf\sup_{t\in T}\sum_{n\geq 0}2^{n/\alpha}d(t,\mathcal A_n),
\]
where the infimum is taken over all possible admissible sequences $\{\mathcal A_n\}_{n\geq0}$.
\end{definition}
We are now ready to state the main theorem due to Talagrand:
\begin{theorem}[Talagrand majorizing measure theorem]
Consider a centered Gaussian process $\{G_t\}_{t\in T}$ index by the set $T$ and the metric $d$ defined by
\[
d(s,t) = \expect{(G_s- G_t)^2}^{1/2}.
\]
There exists some absolute constant $L>0$ such that
\[
\frac1L\cdot\gamma_2(T)\leq \expect{\sup_{t\in T}G_t}\leq L\cdot\gamma_2(T).
\]
\end{theorem}

Throughout the thesis, the $L_p$-norm of a random variable $X$ is defined as $\|X\|_{L_p}:= \expect{|X|^p}^{1/p}$.
\begin{definition}
A random variable $X$ is $L$ sub-Gaussian if  $p^{-1/2}\|X\|_{L_p}\leq L\|X\|_{L_2},~\forall p\geq1$. The corresponding sub-Gaussian norm ($\psi_2$-norm) is defined as $\|X\|_{\psi_2} := \sup_{p\geq1}p^{-1/2}\|X\|_{L_p}$.
\end{definition}

\begin{definition}[Subgaussian random vector]\label{def:vector.subg.norm}
	A random vector $\mathbf{X}\in\mathbb{R}^d$ is $L$ sub-Gaussian if the collection random variables $\langle\mathbf{X},\mathbf{z}\rangle,\mathbf{z}\in \mathbb{S}^{d-1}$
	are $L$ sub-Gaussian. The corresponding sub-Gaussian norm of the vector $\mathbf{X}$ is then given by
	\[\|\mathbf{X}\|_{\psi_2}=\sup_{\mathbf{z}\in\mathbb{S}^{d-1}}\|\langle\mathbf{X},\mathbf{z}\rangle\|_{\psi_2}.\]
\end{definition}

For sub-Gaussian processes, we have
\begin{theorem}[Theorem 2.2.18 of \cite{talagrand2014upper}]\label{thm:tal-chaining}
Consider a centered sub-Gaussian process $\{X_t\}_{t\in T}$ index by the set $T$ and the metric $d$ defined by
\[
d(s,t) = \expect{(X_s- X_t)^2}^{1/2}.
\]
We have 
\[
\expect{\sup_{t\in T}X_t}\leq L\cdot\gamma_2(T).
\]
and 
\[
P(\sup_{t\in T}X_t\geq Lu\cdot\gamma_2(T))\leq 2\exp(-u^2).
\]
\end{theorem}

Throughout the thesis, we seldom encounter any exact computation and our bounds are always in terms of unspecified absolute constants. Furthermore, the constants (for example, $L$ and $C$) can be different per occurrence.

\subsection{Other key inequalities}
Let $(T,d)$ be a semi-metric space, and let $X_1(t),\cdots,X_m(t)$ be independent stochastic processes indexed by $T$ such that $\mb E|X_j(t)|<\infty$ for all $t\in  T$ and $1\leq j\leq m$. 
We are interested in bounding the supremum of the empirical process
\begin{equation}\label{empirical}
Z_m(t)=\frac1m\sum_{i=1}^m\l[X_i(t)-\expect{X_i(t)}\r].
\end{equation}
The following well-known symmetrization inequality reduces the problem to bounds on a (conditionally) Rademacher process 
$R_m(t)=\frac 1m\sum_{i=1}^m\varepsilon_i X_i(t),~t\in T$, where $\eps_1,\ldots,\eps_m$ are i.i.d. Rademacher random variables (meaning that they take values $\{-1,+1\}$ with probability $1/2$ each), independent of $X_i$'s.
\begin{lemma}[Symmetrization inequalities]\label{lem:symmetry}
\label{symmetrization}
\[
\mb E\sup_{t\in T}|Z_m(t)|
\leq 2\mb E\sup_{t\in T}|R_m(t)|,
\]
and for any $u>0$, we have
\[
\mb P\left(\sup_{t\in T}|Z_m(t)|\geq 2\mb E\sup_{t\in T}|Z_m(t)|+u\right)\leq 4\mb P\left(\sup_{t\in T}|R_m(t)|\geq u/2\right).
\]
\end{lemma}
See Lemmas 6.3 and 6.5 in \cite{ledoux2013probability} for proofs.

\begin{lemma}[Bernstein's inequality \cite{wellner2013weak}]
\label{Bernstein}
Let $X_1,\cdots,X_m$ be a sequence of independent centered random variables. 
Assume that there exist positive constants $\sigma$ and $D$ such that for all integers $p\geq 2$
\[
\frac1m\sum_{i=1}^m\expect{|X_i|^p}\leq\frac{p!}{2}\sigma^2D^{p-2},
\]
then
\[
P\left(\left|\frac1m\sum_{i=1}^m X_i\right|\geq\frac{\sigma}{\sqrt{m}}\sqrt{2u}+\frac{D}{m}u\right)
\leq2\exp(-u).
\]
In particular, if $X_1,\cdots,X_m$ are all sub-exponential random variables, then $\sigma$ and $D$ can be chosen as 
$\sigma=\frac{1}{m}\sum_{i=1}^m\|X_i\|_{\psi_1}$ and $D=\max\limits_{i=1\ldots m}\|X_i\|_{\psi_1}$. 
\end{lemma}

\begin{lemma}[Contraction principle \cite{ledoux2013probability}]\label{lem:contraction}
Let $X_1,\cdots,X_N$ be a sequence of samples in $\mathcal{X}$ and let $\mathcal{F}$ be a class of functions containing $f:\mathcal X\rightarrow\mathbb{R}$.
Let $\Psi_1,\Psi_2,\cdots,\Psi_N:\mathbb{R}\rightarrow\mathbb{R}$ be a sequence of $L$-Lipschitz functions for some $L>0$, then, we have 
\[
\expect{\l.\sup_{f\in \mathcal F}\frac 1N\sum_{i=1}^N\varepsilon_i \Psi(f(X_i))\r|~X_1,\cdots,X_N}
\leq L\cdot \expect{\l.\sup_{f\in \mathcal{F}}\frac 1N\sum_{i=1}^N\varepsilon_i f(X_i)\r|~X_1,\cdots,X_N}
\]  
\end{lemma}

\begin{lemma}[Contraaction principle \cite{ledoux2013probability}]\label{lem:contraction-2}
Let $X_1,\cdots,X_N$ be a sequence of samples in $\mathcal{X}$, let $\mathcal{F}$ be a class of functions containing $f:\mathcal X\rightarrow\mathbb{R}$, and let $\alpha_1,\cdots,\alpha_N$ be a sequence of real numbers (possibly depends on the samples) such that $|\alpha_i|\leq 1$. We have for any $u\geq0$, 
\[
P\l(\l.\sup_{f\in \mathcal{F}}\frac 1N\sum_{i=1}^N\varepsilon_i \alpha_if(X_i)\geq u\r|~X_1,\cdots,X_N\r)
\leq 2 P\l(\l.\sup_{f\in \mathcal{F}}\frac 1N\sum_{i=1}^N\varepsilon_i f(X_i) \geq u\r|~X_1,\cdots,X_N\r)
\]
\end{lemma}

\begin{lemma}[Paley-Zygmund inequality \cite{paley1930some}]
Suppose $Z\geq0$ is a random variable with finite variance and $\theta\in(0,1)$, then,
\[
P(Z\geq \theta\expect{Z}) \geq (1-\theta)^2\frac{\expect{Z}^2}{\expect{Z^2}}.
\]
\end{lemma}

Finally, the following lemma is crucial in the analysis of heavy-tailed processes which is sometimes referred to as the Montgomery-Smith inequality:

\begin{lemma}[\cite{montgomery1990distribution}]\label{lemma:Rademacher}
Let $\mathbf{X} = [X_1,\cdots,X_m]$ be a sequence of scalars. Define the following quantity:
\[
K_{1,2}(\mathbf X, u) := \inf\l\{ \sum_{i\in I} |X_i| + u\l(\sum_{i\not\in I}|X_i|^2\r)^{1/2},~~I\subseteq\l\{1,2,\cdots,m\r\} \r\}.
\]
Then, we have
\begin{align}\label{eq:MS-inequality}
\mathbb P\left(
\left|\sum_{i=1}^m\varepsilon_iX_i\right|\geq K_{1,2}(\mathbf X, u) 
\right)
\leq2\exp(-u^2/2).
\end{align}
Furthermore, there exists a  universal constant $c>0$ such that
\[
c^{-1}K_{1,2}(\mathbf X, u)  \leq \sum_{i=1}^{\lfloor u^2 \rfloor}X_i^* + u\l( \sum_{i=\lfloor u^2 \rfloor+1}^{m}(X_i^*)^2\r)^{1/2}
\leq cK_{1,2}(\mathbf X, u)
\]
where $\{X_i^*\}_{i=1}^m$ is the \textit{non-increasing} rearrangement of $\{|X_i|\}_{i=1}^m$ and $\{\varepsilon_i\}_{i=1}^m$ is a sequence of i.i.d. Rademancher random variables independent of $\{X_i\}_{i=1}^m$.
\end{lemma}

\subsection{Gordon's theorem and bounds on the estimation error}
Let's go back to the least squares ERM discussed at the beginning and see how to perform a rigorous analysis on the estimation error. 
We start with \eqref{eq:bias-variance} and assume the bias is 0. Further assume that $\{\mf x_i\}_{i=1}^N$ are i.i.d. Gaussian vectors from $\mathcal{N}(0, I_{d\times d})$, and the noise $|y_i-\mathbf{x}_i^T\theta^*|\leq b$ for some absolute constant $b>0$.
%
%
Recall the following Gordon's ``escape through the mesh'' theorem:
\begin{theorem}[Gordon's theorem (Corollary 1.2 of \cite{gordon1988milman})]
Let $S$ be a closed subset of unit sphere, and let matrix $G$ be a $N\times d$ entry-wise i.i.d.  random matrix drawn from a standard Gaussian distribution $\mathcal N(0,1)$. Then, for any $u\geq0$,
\[
P\l(\sup_{x\in S}\big| \|Gx\|_2 -\mathbb{E}{\|g_N\|_2}\big| \geq \omega(S) + u\r)\leq \exp(u^2/2)
\]
where $g_N\sim\mathcal{N}(0, I_{N\times N})$.
\end{theorem}
Note that we have $\sqrt{N}\geq\expect{\|g_N\|_2}\geq\frac{N}{\sqrt{N+1}}$. 
Let $r>0$ and $\mathcal S_2(r)$ is the sphere centered at the origin with radius $r$, i.e. $\mathcal S_2(r) = \{x\in\mathbb R^d:~\|x\|_2 = r\}$. Furthermore, define the descent cone of a set $T\subseteq\mathbb{R}^d$ at some point $x$ as 
\[
D(T,x) = \{\lambda(t-x),~\lambda\geq0, ~t\in T\}.
\]  
Note that for any vector $\theta\in\Theta$, $(\theta - \theta^*)/\|\theta - \theta^*\|_2\in D(\Theta, \theta^*)$. Thus, we consider the following infimum:
\[
\inf_{\theta\in D(\Theta, \theta^*)\cap \mathcal S_2(1)} \frac1N\sum_{i=1}^N\dotp{\mf x_i}{\theta}^2.
\]
Using Gordon's theorem, we readily have
\[
\inf_{\theta\in D(\Theta, \theta^*)\cap \mathcal S_2(1)} \frac1N\sum_{i=1}^N\dotp{\mf x_i}{\theta}^2
\geq \l(\sqrt{\frac{N}{N+1}} - \frac{\omega( D(\Theta, \theta^*)\cap \mathcal S_2(1)) + u}{\sqrt N}\r)^2
\]
with probability at least $1- \exp(u^2/2)$. Suppose $N\geq 4(\omega(D(\Theta, \theta^*)\cap \mathcal S_2(1)) + u)^2$, then, the above quantity is no less than $1/2$ and it follows with probability at least $1- \exp(u^2/2)$,
\begin{equation}\label{eq:ch1-quad-low}
\frac1N\sum_{i=1}^N(\widehat\theta_N -\theta^*)^T\mathbf{x}_i\mathbf{x}_i^T(\widehat\theta_N-\theta^*)
\geq\frac12\|\widehat\theta_N -\theta^*\|_2^2.
\end{equation}
On the other hand, for the right hand side of \eqref{eq:bias-variance}, we would like to upper bound
\[
\sup_{\theta\in D(\Theta, \theta^*)\cap \mathcal S_2(1)}
\frac2N\sum_{i=1}^N(y_i-\mathbf{x}_i^T\theta^*)\mathbf{x}_i^T\theta -2\expect{(y_i-\mathbf{x}_i^T\theta^*)\mathbf{x}_i^T\theta}
\]
By symmetrization inequality (Lemma \ref{lem:symmetry}), it is enough to consider 
\[
\sup_{\theta\in D(\Theta, \theta^*)\cap \mathcal S_2(1)}\frac2N\sum_{i=1}^N\varepsilon_i(y_i-\mathbf{x}_i^T\theta^*)\mathbf{x}_i^T\theta,
\]
where $\varepsilon_i$'s are i.i.d Rademacher random variable. Since $|y_i-\mathbf{x}_i^T\theta^*|\leq b$, by contraction principle (Lemma \ref{lem:contraction-2}), it is enough to consider 
\[
b\cdot\sup_{\theta\in D(\Theta, \theta^*)\cap \mathcal S_2(1)}\frac2N\sum_{i=1}^N\varepsilon_i\mathbf{x}_i^T\theta.
\] 
Using Theorem \ref{thm:tal-chaining}, we readily get with probability at least $1-2\exp(-u^2)$,
\[
b\cdot\sup_{\theta\in D(\Theta, \theta^*)\cap \mathcal S_2(1)}\frac2N\sum_{i=1}^N\varepsilon_i\mathbf{x}_i^T\theta
\leq \frac{bLu\cdot \omega(D(\Theta, \theta^*)\cap \mathcal S_2(1))}{\sqrt N},
\]
where $L>0$ is some absolute constant. Thus, with probability at least $1-c\exp(-u^2)$, where $c>0$ is some absolute constant,
\[
\frac2N\sum_{i=1}^N(y_i-\mathbf{x}_i^T\theta^*)\mathbf{x}_i^T\theta -2\expect{(y_i-\mathbf{x}_i^T\theta^*)\mathbf{x}_i^T
(\widehat\theta_N-\theta^*)}
\leq  \frac{bLu\cdot \omega(D(\Theta, \theta^*)\cap \mathcal S_2(1))}{\sqrt N}
\|\widehat\theta_N-\theta^*\|_2.
\]
Overall, combining this inequality with \eqref{eq:ch1-quad-low}, we conclude with the following theorem, which can also be found, for example, in \cite{rudelson2008sparse}:
\begin{theorem}\label{ch1-thm-1}
Suppose $\{\mf x_i\}_{i=1}^N$ are i.i.d. Gaussian vectors from $\mathcal{N}(0, I_{d\times d})$, and the noise $|y_i-\mathbf{x}_i^T\theta^*|\leq b$ for some absolute constant $b>0$. 
For any $u\geq0$, if $N\geq 4(\omega(D(\Theta, \theta^*)\cap \mathcal S_2(1)) + u)^2$, then with probability at least $1-c\exp(-u^2)$, the solution to minimizing \eqref{eq:l2-loss} satisfies
\[
 \|\widehat\theta_N-\theta^*\|_2\leq \frac{bLu\cdot \omega(D(\Theta, \theta^*)\cap \mathcal S_2(1))}{\sqrt N}.
\]
\end{theorem}

Note that such a quantity measures the ``true'' complexity of estimating $\theta_*$ in the sense that the Gaussian mean width of a set can be much smaller than the ambient dimension of that set. For example, one can apply this theorem to sparse recovery problems and easily obtain a minimax optimal rate. More specifically, 
 the work \cite{chandrasekaran2012convex} shows that when taking 
$\Theta = \{\theta\in\mathbb{R}^d:~\|\theta\|_{1}\leq \|\theta^*\|_1\}$, i.e. the ball of $\|\cdot\|_1$ with radius $\|\theta^*\|_1$, and $\theta_*$ is $s$-sparse, we have
$\omega(D(\Theta,\theta_*)\cap \mathcal{S}_2(1))$ is on the order of $\sqrt{s\log(d/s)}$. Thus, instead of having number of samples 
$N$ scales with the dimension $d$, we only need the sample to scale with the sparsity level $s\log(d)$ in order to get an accurate estimation, which is in fact minimax optimal.

\subsection{Theorem \ref{ch1-thm-1} is restrictive}
Despite the simplicity of proving Theorem \eqref{ch1-thm-1}, it is fairly restrictive due to Gaussian measurements and bounded noise assumptions. One might wonder if these two assumptions are really necessary. The short answer is that they cannot be much relaxed if we would like to more or less keep the same idea of analysis. The reason is that proving Gordon's theorem for general measurements is difficult. It is known that one can significantly relax the Gaussian assumption for special sets (For example, unit ball in $\mathbb{R}^d$ \cite{mendelson2012generic}). For general sets, it is recently established in \cite{liaw2017simple} that one can recover Theorem \eqref{ch1-thm-1} using sub-Gaussian measurements, but with inexplicit constants. For measurements that have heavier tails than Gaussian, such a result is not known and likely untrue.

However, a closer look at the proof indicates that only a lower bound of $\frac1N\sum_{i=1}^N\dotp{\mf x_i}{\theta}^2$ is needed whereas Gordon's theorem provides a double sided bound. As a simple example, we look at bounds like
\[
\frac1N\sum_{i=1}^N\dotp{\mf x_i}{\theta}^2\geq\frac12\expect{\dotp{\mf x_i}{\theta}^2},
\]
as oppose to 
\[
\l|\frac1N\sum_{i=1}^N\dotp{\mf x_i}{\theta}^2-\expect{\dotp{\mf x_i}{\theta}^2}\r|\leq\frac12\expect{\dotp{\mf x_i}{\theta}^2}.
\]
Obviously, there are huge differences between these two inequalities. Intuitively, large values on $\dotp{\mf x_i}{\theta}^2$ might ruin the second inequality, it only helps with the first inequality. An example demonstrating this fact is as follows: 
\begin{remark}[Differences between upper and lower bounds \cite{mendelson2014learning}]
Fix an integer $N\geq100$ and consider a sequence of i.i.d. random variables $Z_1,\cdots, Z_N$ such that each $Z_i$ takes $2\sqrt{N}$ with probability $1/N^2$ and takes $1$ with probability $1-1/N^2$. We have 
$$\expect{Z_i^2} = 1-\frac{1}{N^2} + \frac4N.$$
With probability at least $1/2N$, there exists some $i$ such that $Z_i = 2\sqrt N$, which implies $\frac1N\sum_{i=1}^NZ_i^2\geq4$. Thus, we have
\[
Pr\l( \l|\frac1N\sum_{i=1}^NZ_i^2-\expect{Z_i^2}\r|\leq\frac12\expect{Z_i^2} \r)
\leq 1-\frac{1}{2N}.
\]
On the other hand, if we consider the lower bound only, then, using Chernoff's inequality, we obtain
\[
Pr\l( \frac1N\sum_{i=1}^NZ_i^2\geq\frac12\expect{Z_i^2} \r)\geq1-\exp(-c N),
\]
where $c>0$ is an absolute constant.
\end{remark}

An immediate consequence of these observations is that the standard
method of analysis for the estimation problem, which is based on a two-sided
concentration argument that holds with exponential probability, can never
work in heavy-tailed situations. Thus, one must find a different argument
altogether if one wishes to deal with learning problems that include classes
of heavy-tailed functions or with a heavy-tailed target.

\section{Small-ball Method}
\subsection{A general theorem}
A key contribution in \cite{mendelson2014learning,koltchinskii2015bounding} is a completely new method bounding the lower tail on the infimum of the quadratic form $\frac1N\sum_{i=1}^N\dotp{\mf x_i}{\theta}^2$ without concentration. As is mentioned in \cite{mendelson2014learning}, the term ``without concentration'' should be understood in the sense of ``when the concentration is false'', as oppose to ``concentration
methods are not needed and will not take any part in the analysis of ERM''. To state the main theorem, we need the following definition, so called ``small-ball condition''.
\begin{definition}
A random vector $\mathbf{x}$ is said to satisfy the \textbf{small-ball condition} over a set $\mathcal{H}\subseteq\mathbb{R}^d$ if for any $\mathbf v\in\mathcal{H}$, there exist positive constants $\delta$ and $Q$ so that
\[
\inf_{\mathbf{v}\in\mathcal{H}} P\l( \l| \dotp{\mathbf v}{\mathbf{x}} \r| \geq \delta\|\mathbf{v}\|_2 \r) \geq Q.
\]
\end{definition}

To see how weak the small-ball condition is, we consider a random vector $\mathbf x$ satisfying 
$\|\dotp{\mf v}{\mf x}\|_{L_2} = \|\mf v\|_2$ and the $L_4-L_2$ equivalence condition, i.e. $\forall \mathbf v\in\mathcal{H}\subseteq\mathbb{R}^d$, $\|\dotp{\mf v}{\mf x}\|_{L_4}\leq L\|\dotp{\mf v}{\mf x}\|_{L_2}$, where $L>0$ is an absolute constant. By Paley-Zygmund inequality, for any $\eta\in[0,1]$, 
\[
P\l(|\dotp{\mf v}{\mf x}|^2 \geq \eta\|\mf v\|_2^2\r)\geq (1-\eta)^2\frac{\expect{|\dotp{\mf v}{\mf x}|^2}^2}{\expect{|\dotp{\mf v}{\mf x}|^4}}
= (1-\eta)^2\frac{\|\dotp{\mf v}{\mf x}\|_{L_2}^4}{\|\dotp{\mf v}{\mf x}\|_{L_4}^4} \geq \frac{(1-\eta)^2}{L^4}.
\]
Thus, small-ball condition does allow heavy-tailed random vectors. The key theorem by Mendelson is as follows:
\begin{lemma}[\cite{mendelson2014learning}]\label{lem:mendelson}
Let $\mathcal H \subseteq \mathcal S_2(1)$ and define the empirical mean width
\[
\omega_N(\mathcal H) := \expect{\sup_{\mf h\in \mathcal H}\frac{1}{\sqrt{N}} \sum_{i=1}^N\varepsilon_i\dotp{\mf x_i}{\mf h}}.
\]
Suppose $P(\l| \dotp{\mf x}{\mf h} \r|\geq\delta\|\mf h\|_2)\geq Q,~\forall \mf h\in\mathcal{H}$, then, it follows
\[
\inf_{\mf h\in \mathcal H}\l( \sum_{i=1}^N\dotp{\mf x_i}{\mf h}^2  \r)^{1/2}\geq \delta Q\sqrt{N} - 2\omega_N(\mathcal H) - \frac{\delta u}{2},
\]
with probability at least $1-ce^{-u^2}$ for any $u>0$.
\end{lemma}

\subsection{Application to least squares ERM}
Lemma \ref{lem:mendelson} is very powerful and applicable to analysis of many different loss functions. Here, we will show how it helps in the estimation error analysis of minimizing \eqref{eq:l2-loss}. We assume that the measurement $\mf x_i$ satisfies $\|\dotp{\mf v}{\mf x_i}\|_{L_2} = \|\mf v\|_2,~\forall \mf v\in \mathbb{R}^d$ and the $L_4-L_2$ equivalence condition, i.e. $\forall \mathbf v\in\mathbb{R}^d$, $\|\dotp{\mf v}{\mf x_i}\|_{L_4}\leq L\|\dotp{\mf v}{\mf x_i}\|_{L_2}$, where $L>0$ is an absolute constant. 
Again, we consider the following infimum:
\[
\inf_{\theta\in D(\Theta, \theta^*)\cap \mathcal S_2(1)} \frac1N\sum_{i=1}^N\dotp{\mf x_i}{\theta}^2.
\]
By Paley-Zygmund inequality, we have
\[
P\l(|\dotp{\mf v}{\mf x_i}|^2 \geq \frac12\|\mf v\|_2^2\r)\geq \frac{1}{4L^4}
\]

Applying Lemma \ref{lem:mendelson}, we readily have
\[
\inf_{\theta\in D(\Theta, \theta^*)\cap \mathcal S_2(1)} \frac1N\sum_{i=1}^N\dotp{\mf x_i}{\theta}^2
\geq \l(\frac{1}{8L^4} - \frac{2\omega_N( D(\Theta, \theta^*)\cap \mathcal S_2(1))}{\sqrt N} - \frac{u}{4\sqrt{N}}\r)^2
\]
with probability at least $1- \exp(u^2/2)$, where 
\begin{equation}\label{eq:emp-width-1}
\omega_N( D(\Theta, \theta^*)\cap \mathcal S_2(1))  =  \expect{\sup_{\mf h\in D(\Theta, \theta^*)\cap \mathcal S_2(1)}\frac{1}{\sqrt{N}} \sum_{i=1}^N\varepsilon_i\dotp{\mf x_i}{\mf h}}
\end{equation}
is the empirical mean width.
Suppose 
$$N\geq 256 L^8\l(2\omega(D(\Theta, \theta^*)\cap \mathcal S_2(1)) + \frac{u}{4}\r)^2,$$ 
then, it follows with probability at least $1- \exp(u^2/2)$,
\begin{equation}\label{eq:ch1-quad-low-2}
\frac1N\sum_{i=1}^N(\widehat\theta_N -\theta^*)^T\mathbf{x}_i\mathbf{x}_i^T(\widehat\theta_N-\theta^*)
\geq\frac{1}{16L^4}\|\widehat\theta_N -\theta^*\|_2^2.
\end{equation}
On the other hand, define $\xi_i = y_i - \mf x_i^T\theta^*$, and define another empirical width:
\begin{equation}\label{eq:emp-width-2}
\widetilde{\omega}_N(D(\Theta, \theta^*)\cap \mathcal S_2(1)) := \sup_{\theta\in D(\Theta, \theta^*)\cap \mathcal S_2(1)}
\frac{1}{\sqrt N}\sum_{i=1}^N\xi_i\mathbf{x}_i^T\theta -2\expect{\xi_i\mathbf{x}_i^T\theta},
\end{equation}
from which we have
\[
\frac2N\sum_{i=1}^N(y_i-\mathbf{x}_i^T\theta^*)\mathbf{x}_i^T\theta -2\expect{(y_i-\mathbf{x}_i^T\theta^*)\mathbf{x}_i^T
(\widehat\theta_N-\theta^*)}
\leq  \frac{2 \widetilde{\omega}_N(D(\Theta, \theta^*)\cap \mathcal S_2(1))}{\sqrt N}\|\widehat\theta_N-\theta^*\|_2.
\]
Overall, we obtain the following theorem:
\begin{theorem}\label{ch1-thm:heavy-tail}
Suppose $\{\mf x_i\}_{i=1}^N$ are $L_4-L_2$ equivalence condition, i.e. $\forall \mathbf v\in\mathbb{R}^d$, $\|\dotp{\mf v}{\mf x_i}\|_{L_4}\leq L\|\dotp{\mf v}{\mf x_i}\|_{L_2}$, where $L>0$ is an absolute constant. For any $u\geq0$, if 
$$N\geq 256 L^8\l(2\omega_N(D(\Theta, \theta^*)\cap \mathcal S_2(1)) + \frac{u}{4}\r)^2,$$ 
then with probability at least $1-\exp(-u^2)$, the solution to minimizing \eqref{eq:l2-loss} satisfies
\[
 \|\widehat\theta_N-\theta^*\|_2\leq \frac{2 \widetilde{\omega}_N(D(\Theta, \theta^*)\cap \mathcal S_2(1))}{\sqrt N}.
\]
\end{theorem}
Bounds of this flavor via the small-ball method can be found, for example, in \cite{tropp2015convex}.
To apply this theorem to specific problems, we need to compute the two quantities \eqref{eq:emp-width-1} and \eqref{eq:emp-width-2}. One might wonder if anything can be said regarding the general properties of these two empirical quantities. It turns out when both 
$\xi_i$ and $\mf x_i$ are sub-Gaussian, we recover Theorem \ref{ch1-thm-1} up to constant via the following theorem:
\begin{theorem}[Lemma 3.2 of \cite{goldstein2018non}]
Suppose $\mf x_i$ is an isotropic sub-Gaussian random vector and $\xi_i$ is a sub-Gaussian random variable. Suppose 
$N\geq \omega(D(\Theta, \theta^*)\cap \mathcal S_2(1))^2$, then, with probability at least $1- e^{-u^2}$,
\[
\sup_{\theta\in D(\Theta, \theta^*)\cap \mathcal S_2(1)}
\frac{1}{\sqrt N}\sum_{i=1}^N\xi_i\mathbf{x}_i^T\theta -2\expect{\xi_i\mathbf{x}_i^T\theta}
\leq C(\|\xi\|_{\psi_2}^2 + \|\mf x_i\|_{\psi_2}^2)(\omega(D(\Theta, \theta^*)\cap \mathcal S_2(1)) + u^2),
\]
where $C>0$ is an absolute constant.
\end{theorem}
This theorem gives a bound on $\widetilde \omega_N(D(\Theta, \theta^*)\cap \mathcal S_2(1))$. 
For the term $\omega_N(D(\Theta, \theta^*)\cap \mathcal S_2(1))$, one can simply invoke Theorem \ref{thm:tal-chaining},
\[
\expect{\sup_{\theta\in D(\Theta, \theta^*)\cap \mathcal S_2(1)}\frac2N\sum_{i=1}^N\varepsilon_i\mathbf{x}_i^T\theta}
\leq \frac{L\omega(D(\Theta, \theta^*)\cap \mathcal S_2(1))}{\sqrt N},
\]
where $L$ is an absolute constant. Overall, we obtain the following corollary of Theorem \ref{ch1-thm:heavy-tail}.
\begin{corollary}
Suppose $\mf x_i$ is an isotropic sub-Gaussian random vector, $\xi_i$ is a sub-Gaussian random variable, and
$$N\geq C_1\l(\omega_N(D(\Theta, \theta^*)\cap \mathcal S_2(1)) + u\r)^2,$$ 
then, for any $u\geq1$, with probability at least $1-\exp(-u^2)$,
\[
 \|\widehat\theta_N-\theta^*\|_2\leq C_2(\|\xi\|_{\psi_2}^2 + \|\mf x_i\|_{\psi_2}^2)
 \frac{\omega(D(\Theta, \theta^*)\cap \mathcal S_2(1)) + u^2}{\sqrt N},
\]
where $C_1,~C_2$ are absolute constants.
\end{corollary}

However, in general, when $\xi_i$ and $\mf x_i$ exhibit heavier tails than Gaussian, it is highly non-trivial to bound \eqref{eq:emp-width-1} and \eqref{eq:emp-width-2} in terms of Gaussian mean width. It is an active research area and we will introduce several methods later to bound them.

\section{Organization of the Thesis}
The rest of the thesis is organized as follows.
In Chapter 2,  we introduce a new adaptively thresholded ERM for generalized linear model with a new analysis framework, which refines the results from an earlier draft \cite{wei2018structured}. Special attention is devoted to recovering an approximately sparse vector in $\ell_1$-ball as well as bounded sparse vectors with the minimax statistical rates under a rather weak assumption that the design vector has more than $15$ moments. This result significantly improves the previously known results which require 
$\mathcal{O}(\log d)$ moments ($d$ being the dimension of the vector). 
In Chapter 3, we show that if one knows the design vectors are sampled from a specific class of distributions, then, a somewhat simpler analysis with even weaker assumptions is possible \cite{goldstein2016structured}\cite{goldstein2019non}. In particular, we show that when the design vectors are elliptical symmetric with more than 2 moments, then, one can recovery a structured signal (up to constant scaling) with minimax rate from measurements with unknown nonlinear transformations. Finally, in Chapter 4, we look at a problem with a somewhat different flavor, namely, the robust covariance matrix estimation. We show that a Huber-type estimator achieves the minimax optimal statistical rate with more than 4 moments on the samples \cite{wei2017estimation}\cite{minsker2020robust}.



\chapter{Optimal Statistical Rate in Generalized Linear Models under Weak Moment Assumptions}

In this Chapter, we consider the scenario of high-dimensional estimation in generalized linear models (GLMs). While high-dimensional recovery problems have been studied extensively under the sub-Gaussian assumption, much less is known in the case of heavy-tailed measurements, such as those with moments of only constant order. In this paper, we propose and analyze new thresholding methods recovering high-dimensional structured vectors from nonlinear measurements under very weak assumptions on the underlying distributions. In particular, we show that, by solving a convex program, the proposed method achieves the minimax statistical rate of estimation in $\ell_1$-ball with only 
$(15+\delta)$ moments on the design vectors. Our results improve upon the best known analysis on the convex methods for ordinary linear models, i.e. LASSO type estimators, which require 
$\mathcal{O}(\log d)$ moments to achieve the minimax optimal statistical rate.

\section{Introduction}
We study a general model where the response $y\in\mathbb{R}$ is linked to the covariate $\mathbf{x}\in\mathbb{R}^d$ via a generalized linear model through a canonical link function. More specifically, we assume $y$ satisfies the following distribution
\begin{equation}\label{eq:glms}
Pr(y~|\mathbf x;~\theta^*,\sigma) \propto \exp\l(\frac{y\dotp{\mf x}{\theta^*} - g(\dotp{\mf x}{\theta^*})}{c(\sigma)}\r), 
\end{equation}
where $\sigma$ is a known scalar parameter and $c$ is a known mapping. The vector $\theta^*\in\mathbb R^d$ is unknown to be estimated and $g:\mathbb{R}\rightarrow\mathbb{R}$ is the link function. Using the standard properties of an exponential family (\cite{brown1986fundamentals}), we know that the function $g$ is twice differentiable and $g''$ is \textit{strictly positive} on the realline.
In particular, this implies the function $g$ is a strictly convex function.
\footnote{This should be distinguished from the more restricted class of \textit{strongly convex functions} for which there is a positive lower bound $c$ such that $g''(x)\geq c,~\forall x\in\mathbb R$. On the other hand, for a strictly convex function, there is no such a uniform lower bound.} Some examples of GLMs are as follows:
\begin{itemize}
\item The ordinary linear model, i.e. $y = \dotp{\mf x}{\theta_*} + \xi$ with $\xi\sim\mathcal{N}(0,1)$, corresponds to the condition distribution of $y$ being a Gaussian distribution with mean $\dotp{\mf x}{\theta^*}$ and variance $\sigma^2$. More specifically, we have 
$g(\dotp{\mf x}{\theta^*}) = (\dotp{\mf x}{\theta^*})^2/2$ and $c(\sigma) =\sigma^2$.

\item The logistic regression model corresponds to $y$ being a Bernoulli random variable (taking values in $\{0,1\}$). More specifically, we have 
$g(\dotp{\mf x}{\theta^*}) = \log(1+ \exp(\dotp{\mf x}{\theta^*} ))$ and $c(\sigma) = 1$. In particular, we have
\[
Pr(y=1 ~| \mathbf x;~\theta^*) = \frac{\exp(\dotp{\mf x}{\theta^*} )}{1+\exp(\dotp{\mf x}{\theta^*} )}.
\]
\item The poisson regression model corresponds to $y$ being a Poisson distribution taking values in $\mathbb N$ and 
$g(\dotp{\mf x}{\theta^*}) = \exp(\dotp{\mf x}{\theta^*} )$ and $c(\sigma) = 1$. 
\end{itemize}

The goal is to estimate the true parameter $\theta^*\in\mathbb R^d$ from a sequence of $N$ samples $\{(\mf x_i,y_i)\}_{i=1}^N$. When assuming $\theta^*$ possesses certain structure which tends to make the corresponding norm function $\Psi(\theta^*)$ small, one proposes to estimate $\theta^*$ via the following maximum likelihood (ML) with regularization:
\begin{equation}\label{eq:reg-glm}
\widehat{\theta}_N := \argmin_{\theta\in\mathbb R^d} -\frac1N\sum_{i=1}^Ny_i\dotp{\mf x_i}{\theta} 
+ \frac1N\sum_{i=1}^Ng(\dotp{\mf x_i}{\theta}) + \lambda\Psi(\theta).
\end{equation}
In particular, if $\theta^*$ is an approximately sparse vector, then, the usual choice for $\Psi$ is $\Psi(\theta) = \|\theta\|_1$.  

Note that in general, there is a sharp contrast between ordinary linear model and the GLMs from an analysis perspective. For linear model, the analysis in the previous chapter demonstrates that an important step of controlling the error is to argue that the smallest eigenvalue of the covariance matrix $\frac1N\sum_{i=1}^N\mf x_i\mf x_i^T$ is away from zero in certain restricted area. However, the same argument does not work here since the quadratic component in least squares ERM is now replaced by 
$\frac1N\sum_{i=1}^Ng(\dotp{\mf x_i}{\theta})$, where $g$ is only approximately quadratic on compact sets and it is not always possible to bound $g(\dotp{\mf x_i}{\theta})$ by a quadratic form.

The difference is even more significant if we further assume that the covariance matrix of $\mf x_i$ is known, i.e. we know 
$\Sigma = \expect{\mf x_i\mf x_i^T}$ and it is positive definite. Consider again the ordinary linear problem. Since we know the covariance, instead of \eqref{eq:l2-loss}, we consider using the following ERM problem:
\begin{equation}\label{eq:l2-loss-cp2}
\widehat\theta_N = \argmin_{\theta\in T}\mathcal L_m(\theta) = 
\theta^T\Sigma\theta - \frac2N\sum_{i=1}^Ny_i\mathbf{x}_i^T\theta.
\end{equation}
We then show this objective is much easier to analyze. To start, we have
\[
\widehat\theta_N^T\Sigma\widehat\theta_N - \frac2N\sum_{i=1}^Ny_i\mathbf{x}_i^T\widehat\theta_N
\leq \theta_*^T\Sigma\theta_* - \frac2N\sum_{i=1}^Ny_i\mathbf{x}_i^T\theta_*.
\]
Rearranging terms gives
\[
(\widehat\theta_N - \theta_*)^T\Sigma(\widehat\theta_N-\theta_*)\leq  \frac2N\sum_{i=1}^N\l(y_i\mathbf{x}_i^T(\widehat\theta_N-\theta_*)
-\expect{y_i \mathbf{x}_i^T(\widehat\theta_N-\theta_*)}\r),
\]
where the expectation is taken given the $N$ samples $\{(\mf x_i, y_i)\}$ and we use the fact that 
$\expect{y_i \mathbf{x}_i^T\theta} = \expect{\theta_*^T\mathbf{x}_i\mathbf{x}_i^T\theta} = \theta_*^T\Sigma\theta$. Since the covariance matrix is positive definite, we have
\begin{align*}
\|\widehat\theta_N - \theta_*\|_2\leq \frac{1}{\lambda_{\min}(\Sigma)}
\sup_{\theta\in T}\frac2N\sum_{i=1}^N\frac{y_i\mathbf{x}_i^T(\theta-\theta_*)
-\expect{y_i \mathbf{x}_i^T(\theta-\theta_*)}}{\|\theta-\theta_*\|_2}.
\end{align*}
As a consequence, we refrain from bounding the smallest eigenvalue of the empirical covariance matrix completely and small-ball method is never needed. This method was first proposed in the seminal work \cite{koltchinskii2011nuclear} which deals with a low-rank matrix regression.
However, this very method cannot be extended to analyzing objectives with general convex functions such as \eqref{eq:reg-glm}.

Of course knowing the covariance matrix and solving problems like \eqref{eq:l2-loss-cp2} can be unrealistic depending on the application. For example, in a typical image classification problem \cite{deng2009imagenet}, we are given a series of image samples and several class hypotheses. We would like to known which class they belong to. In such a scenario, it is unclear how one is able to obtain the population covariance of the samples and the notion of ``population covariance'' might not even be well-defined.

\subsection{Related works}

The ordinary linear model with $\theta_*$ being an $s$-sparse vector and $\Psi(\cdot) = \|\cdot\|_1$ corresponds to the classical compressed sensing problem. Over the past two decades, compressed sensing has been thoroughly studied under the assumption that the measurement vectors are isotropic subgaussian and the noise is also subgaussian, e.g. \cite{tibshirani1996regression,candes2006stable,candes2008restricted, bickel2009simultaneous, hastie2015statistical}. 
It is shown that when each row of the the measurement matrix 
$\mf \Gamma = [\mathbf{x}_1, \mathbf{x}_2,\cdots, \mathbf{x}_N]^T$ is sub-Gaussian, 
$
N\gtrsim s\log(d/s)
$, 
then, the restricted isometric property (RIP) holds 
over all $s$-sparse vectors $\mathbf v\in\mathbb{R}^d$, i.e. there exists a fixed constant $\delta\in(0,1)$, 
$
(1-\delta)\|\mathbf v\|_2\leq \|\mathbf{\Gamma} \mathbf{v}\|_2/\sqrt{N} \leq(1+\delta)\|\mathbf v\|_2.
$.
Then, one can  show that by solving the LASSO:
$\widehat{\theta} := \argmin_{\theta\in\mathbb{R}^d}\|\mf \Gamma\theta - \mf y\|_2^2 + \lambda\|\theta\|_1$, one can achieve the following optimal error rate:
$
\| \widehat{\theta} - \theta_* \|_2\lesssim \sqrt{s\log d/N}.
$
Estimation of sparse vectors in generalized linear model via \eqref{eq:reg-glm} with a similar statistical rate is also proved in the work \cite{negahban2012unified}.

As is mentioned in the previous chapter, the sub-Gaussian assumption is restrictive, but RIP does not necessarily hold with the optimal sample rate $N\gtrsim s\log(d/s)$ when the tail of $\dotp{\mathbf v}{\mathbf x}$ decays slower than sub-Gaussian. The crux lies in the fact that 
RIP simultaneously requires upper bounds on the quadratic form, which is not needed in the proof of performance in sparse recovery. 
Extending the small-ball method originally proposed in \cite{koltchinskii2015bounding}, the work \cite{lecue2017sparse} shows that by assuming the condition that $\mathbf{x}$ has sub-Gaussian property up to only $\mathcal O(\log d)$ moments, i.e.
$
\expect{\l|\dotp{\mathbf v}{\mathbf x} \r|^p}^{1/p}\leq C\sqrt{p}\cdot\expect{\l| \dotp{\mathbf v}{\mathbf x} \r|^2}^{1/2},~\forall 2\leq p \leq c_1\log d,
$
where $c_1>0$ is an absolute constant,
one can achieve the same aforementioned sample and error rates with high probability by solving the LASSO. Furthermore, the work 
\cite{lecue2017regularization} shows that the same $\mathcal O(\log d)$ moments assumption also leads to minimax optimal estimation of an approximate sparse signal in the $\ell_1$-ball instead of exact sparse signals. Outlier robust methods for sparse recovery based on the median-of-mean (MOM) estimators is also proposed and analyzed in several works (e.g. \cite{lecue2017robust, lugosi2016risk}) but they generally require solving a highly non-convex program with $\mathcal O(\log d)$ type moment assumptions on the measurement vectors in order to get the optimal rate.

Our goal in this chapter is to further relax $\mathcal O(\log d)$ moment assumption for optimal $\ell_1$-ball recovery to just a constant moment requirement, which we termed ``weak moment assumption'', and at the same time allow GLMs instead of just ordinary linear model. Recently, the works \cite{Fan-robust-estimation-2017} and \cite{Fan-huber-regression-2017} propose a new class of thresholded estimators for sparse recovery, based on the earlier work \cite{catoni2012challenging} on adaptive shrinkage for heavy-tailed mean estimation. While their methods are quite effective when dealing with the heavy-tailed noise, the sample rate is suboptimal when it comes to heavy-tailed measurement vectors.

\section{Main Results}
\subsection{Optimal Estimation in $\ell_1$-ball}
Throughout the chapter, we adopt the following assumption on the measurements:
\begin{assumption}\label{assumption:moment}
The samples $\{(\mathbf{x}_i,y_i)\}_{i=1}^N$ are i.i.d. copies of $(\mathbf{x},y)$ with $\expect{\mathbf{x}} = 0$, satisfying the model \eqref{eq:glms}. For some absolute constants $q>15,~q'>5$, there exist corresponding constants $\nu,\nu_q, \nu_{q'}, \kappa>0$ such that\\
1. Bounded kutosis: $\sup_{\mathbf{v}\in S_2(1)}\expect{|\dotp{\mathbf{x}}{\mathbf{v}}|^4}\leq \nu$.\\
2. Bounded moments: $\|x_i\|_{L_q}:=\expect{|x_i|^q}^{1/q}\leq \nu_q$ and $\|y-g'(\dotp{\mathbf{x}}{\theta_*})\|_{L_{q'}}\leq \nu_{q'}$~$\forall i\in\{1,2,\cdots,d\}$.\\
3. Non-degeneracy: $\inf_{\mathbf{v}\in S_2(1)}\expect{|\dotp{\mathbf{x}}{\mathbf{v}}|^2}\geq\kappa$.
\end{assumption}
Our result in this section concerns with the estimation in $\ell_1$-ball:
\begin{assumption}\label{assumption:1ball}
The true parameter $\theta_*\in \mathcal B_1(R) := \{\theta\in\mathbb R^d:~\|\theta\|_1\leq R\}$.
\end{assumption}
Note that the set $\mathcal B_1(R)$ includes all bounded vectors that tend to be small in the $\ell_1$-norm ball (but not necessarily exactly sparse). 
The benchmark we will compare to is the following minimax lower bound on estimation within $\mathcal B_1(R)$ via Gaussian measurements:
\begin{theorem}[Theorem 1 of \cite{raskutti2011minimax}]\label{thm:lower-bound}
Consider the ordinary linear model, i.e. $y = \dotp{\mf x}{\theta_*} + \xi$ with $\xi\sim\mathcal{N}(0,1)$ and $\mf x\sim\mathcal N(0, I_{d\times d})$. Suppose Assumption \ref{assumption:1ball} holds and $R\sqrt{\log (ed)/N}< c_1$ for some absolute constant $c_1>0$, then,
\[
\min_{\widehat\theta}\max_{\theta_*\in\mathcal B_1(R)}\expect{\|\widehat\theta - \theta_*\|_2^2}
\geq c_2R\sqrt{\frac{\log ed}{N}},
\]
for some absolute constant $c_2>0$.
\end{theorem}
Note that an underlying assumption in this theorem (which is not explicit in \cite{raskutti2011minimax}) is that the the number of of measurements $N\leq c_3 d^2/R^2$ for some absolute constant $c_3>0$.\footnote{It is easy to see when $N> d^2/R^2$, $R(\log ed/N)^{1/2}\geq d/N$ and the minimax lower bound in this region should be $d/N$, which is achieved by the least squares regression. }
Our goal would be to design an estimator achieving this rate for GLMs \eqref{eq:glms} under Assumption \ref{assumption:moment} and \ref{assumption:1ball}.
Our robust estimator involves generating the adapted truncated measurements $\{(\widetilde{\mathbf{x}}_i,y_i)\}_{i=1}^N$ from the samples $\{(\mathbf{x}_i,y_i)\}_{i=1}^N$ and solving the following problem:
\begin{equation}\label{eq:reg-threshold-glm}
\widehat{\theta}_N := \argmin_{\theta\in\mathbb R^d} -\frac1N\sum_{i=1}^Ny_i\dotp{\widetilde{\mf x}_i}{\theta} 
+ \frac1N\sum_{i=1}^Ng(\dotp{\widetilde{\mf x}_i}{\theta}) + \lambda\Psi(\theta).
\end{equation}
where $\lambda$ is a trade-off parameter to be determined later and $\Psi(\theta) = \|\theta\|_1$ for the $\ell_1$-ball recovery problem. We take
$\widetilde{\mathbf{x}}_i$ such that 
\begin{equation}\label{eq:trunc1}
\widetilde{x}_{ij} = \sign\l(x_{ij}\r)\l(|x_{ij}|\wedge\tau\r),~~\forall j\in\{1,2,\cdots,d\},
\end{equation}
where $\tau = \l(N/\log \l(ed\r)\r)^{1/4}$. 

Next, we will describe conditions on the link function $g$ in \eqref{eq:glms}, which trivially holds for the ordinary linear models.
\begin{assumption}\label{assumption:link-function}
There exists some constant $M_g>0$ such that the Hessian of the cumulant function is uniformly bounded, i.e. $\|g''\|_\infty\leq D_{\max}$.
\end{assumption}

The following is our main result. 

\begin{theorem}\label{thm:sparse-recovery}
Suppose Assumptions \ref{assumption:moment}, \ref{assumption:1ball}, \ref{assumption:link-function} hold. 
Let 
$$D_{\min} := \min_{z\in [-c_1(\nu,\kappa)\|\theta_*\|_1, ~c_1(\nu,\kappa)\|\theta_*\|_1]}g''(z).$$
Suppose  
$N\geq C_1(\nu,\nu_q,\nu_{q'},\kappa)\beta^2(\|\theta_*\|_1^2+1)\log(ed)$,
$\lambda \geq C_2(\nu,\nu_q,\nu_{q'},\kappa) (wu^2v+w\beta^{3/4})\sqrt{\frac{\log(ed)}{N}}$. Then, with probability at least 
\begin{multline*}
1-c' \l(e^{-\beta} + e^{-v^2}
+u^{-q}(ed)^{-(\frac c2-1)}+(u^{-q/4}+u^{-q'})(ed)^{-c/2}\r.\\
\l.+(eN)^{-\frac{q}{10}+1}(\log(eN))^{q/5}w^{-q/5} +  (eN)^{-(\frac{q'}{4}-1)}(\log(eN))^{q'/2}w^{-q'}\r),
\end{multline*}
for some absolute constant $c,c'>0$,
we have 
\begin{align*}
\|\widehat\theta_N - \theta_*\|_2^2\leq& 
\lambda\|\theta_*\|_1
\end{align*}
for any $\beta,u,v,w>7$, where 
$C_i(\nu, \nu_q,\nu_{q'},\kappa),~i=1,2,3$ and $c_1(\nu,\kappa)$ are constants depending polynomially on the parameters $\nu, \nu_q,\nu_{q'},\kappa$.
\end{theorem}
\begin{remark}
Theorem \ref{thm:sparse-recovery} shows that our proposed method can attain the minimax statistical rate when $N\geq\mathcal{O}(\log ed)$, and it does so without knowing how large $R$ is. This result also (up to constants)  matches previous bounds on $\ell_1$-ball estimation which in general require stronger moment assumptions. For example, Theorem 4.2 of \cite{lecue2017regularization} shows when the model is linear and $N\geq\log ed$, one can attain the minimax rate with $\mathcal O(\log d)$ moments on the measurement vector $\{\mathbf{x}_i\}_{i=1}^N$.
\end{remark}

\subsection{Optimal Estimation of Bounded Sparse Vectors}
In this section, we show a result regarding optimal estimation of sparse vectors in a bounded range in the presence of heavy-tailed measurements. More specifically, we consider the following set of vectors:
\begin{assumption}\label{assumption:sparse}
The true parameter $\theta_*\in \Sigma_s\cap S_2(0,1)$, where $\Sigma_s$ denotes the set of $s$-sparse vectors and $S_2(0,1)$ is the unit $\ell_2$-norm ball.
\end{assumption}
The benchmark we compare to is the following lower bound:
\begin{theorem}\label{thm:lower-bound-sparse}
Consider the ordinary linear model, i.e. $y = \dotp{\mf x}{\theta_*} + \xi$ with $\xi\sim\mathcal{N}(0,1)$ and $\mf x\sim\mathcal N(0, I_{d\times d})$. Suppose  $\theta_*\in \Sigma_s\cap S_2(0,1)$, $s\leq d/4$, and $\big(1+\sqrt{s\log (d/s)}\big)/\sqrt{N}< c_1$ for some absolute constant $c_1>0$, then,
\[
\min_{\widehat\theta}\max_{\theta_*\in\Sigma_s\cap S_2(0,1)}\expect{\|\widehat\theta - \theta_*\|_2}
\geq c_2\cdot\sqrt{\frac{s\log(d/s)}{N}},
\]
for some absolute constant $c_2>0$.
\end{theorem}
This lower bound is somewhat different from known lower bounds (e.g. \cite{raskutti2011minimax}) in the sense that it considers 
a restricted candidate set of sparse vectors in a bounded set $S_2(0,1)$ instead of all sparse vectors. Nevertheless, Theorem \ref{thm:lower-bound-sparse} shows that imposing such a restriction does not make the problem easier. To show why it is true, we need the following definition:
\begin{definition}[Local packing number]
Given a set $K\subseteq\mathbb{R}^d$, the local packing number $P_t,~t>0$ is the packing number of $K\cap B_2(0,t)$ with balls of radius $t/10$. 
\end{definition}

Theorem \ref{thm:lower-bound-sparse} is a corollary of the following theorem:
\begin{theorem}[Theorem 4.2 of \cite{plan2016high}]
Assume that $\theta_*\in K$ where $K$ is a star-shaped subset of $\mathbb R^d$. Assume that $y = \dotp{\mf x}{\theta_*} + \xi$ with $\xi\sim\mathcal{N}(0,\sigma^2)$ and $\mf x\sim\mathcal N(0, I_{d\times d})$. Let  
\[
\delta_* := \inf_{t>0}\l\{ t+\frac{\sigma}{\sqrt N}\l(1+\sqrt{\log P_t}\r) \r\}.
\]
Then, there exists an absolute constant $c>0$ such that any estimator $\widehat\theta$ which depends only on the observations $y_i$ and $\mf x_i$ satisfies
\[
\sup_{\mf x\in K} \expect{\|\widehat{\theta} - \theta_*\|_2}\geq c\min\{\delta_*,\text{diam}(K)\}.
\]
\end{theorem}
Now, using this theorem, it is enough to compute $P_t$ in our problem with $K = \Sigma_s\cap S_2(0,1)$ and $\sigma = 1$, for which one can show the following:
\begin{lemma}\label{lem:packing}
When $s\leq d/4$ and $t\leq 1$,
$
P_t\geq \exp\l(cs\log d/s \r),
$
where $c>0$ is an absolute constant. 
\end{lemma}
\begin{proof}[Proof of Lemma \ref{lem:packing}]
The proof of this lemma follows from ideas in Section 4.3 of \cite{plan2016high}. 
To compute $P_t$ for $t\leq 1$, it is enough to consider $1/10$ packing of 
$\Sigma_s\cap S_2(0,1)$. Consider a set $\mathcal N\subseteq \Sigma_s\cap S_2(0,1)$,  
which contains vectors of $s$ cardinality, where each nonzero entry is equal to $s^{-1/2}$. Thus, $|\mathcal N| = {d\choose s}$. We will show that there exists a subset $\mathcal X\subseteq\mathcal N$ such that $\forall x, y \in\mathcal X$, $\|x-y\|_2>1/10$. Consider picking vectors $x,y\in\mathcal N$ uniformly at random and compute the probability of the event
$
\|x-y\|_2^2\leq 1/100
$. When the event happens, it requires $x$ and $y$ to have at least $0.99s$ matching non-zero coordinates. Assume without loss of generality that $0.01s$ is an integer, this event happens with probability
\[
\l.{s\choose 0.99s}{d-0.99s\choose 0.01s}\r/{d\choose s}.
\]
Using Stirling's approximation and $s\leq n/4$, we have $Pr(\|x-y\|_2^2\leq 1/100)\leq \exp(-c's\log d/s)$, where $c'>0$ is an absolute constant. This implies the claim when choose $\mathcal X$ to have $cs\log d/s$ uniformly chosen vectors from $\mathcal N$, which satisfies $\forall x, y \in\mathcal X$, $\|x-y\|_2>1/10$ with a constant probability.
\end{proof}

Thus, by Lemma \ref{lem:packing}, it follows that
\[
\inf_{t\in(0,1]}t+\frac{1}{\sqrt N}\l(1+\sqrt{cs\log d/s}\r)  = \frac{1}{\sqrt N}\l(1+\sqrt{cs\log d/s}\r).
\]
When $N \geq c_1(1+ \log d/s)$, the claim in Theorem \ref{thm:lower-bound-sparse} follows.

Our main result in this section is the following theorem:
\begin{theorem}\label{thm:sparse-recovery-2}
Suppose Assumption \ref{assumption:moment}, \ref{assumption:link-function}, \ref{assumption:sparse} hold.
Let
$s_0 = \frac{\sqrt{\nu}}{\delta^2Q^2}s\leq d$, where $\delta = \frac12\sqrt{\frac{\kappa}{2}}$, $Q=\frac{\kappa^2}{8\nu}$
and $D_{\min} := \min_{z\in [-c_1(\nu,\kappa)\sqrt s, ~c_1(\nu,\kappa)\sqrt s]}g''(z)$.  
Suppose  
$N\geq C_1(\nu,\nu_q,\nu_{q'},\kappa)\beta^2(s_0+1)\log(ed)$,
$\lambda = C_2(\nu,\nu_q,\nu_{q'},\kappa) (wu^2v+w\beta^{3/4})\frac{D_{\max}+1}{D_{\min}}\sqrt{\frac{\log(ed)}{N}}$. Then, with probability at least 
\begin{multline*}
1-c' \l(e^{-\beta} + e^{-v^2}
+u^{-q}(ed)^{-(\frac c2-1)}+(u^{-q/4}+u^{-q'})(ed)^{-c/2}\r.\\
\l.+(eN)^{-\frac{q}{10}+1}(\log(eN))^{q/5}w^{-q/5} +  (eN)^{-(\frac{q'}{4}-1)}(\log(eN))^{q'/2}w^{-q'}\r),
\end{multline*}
for some absolute constant $c,c'>0$,
we have 
\begin{align*}
\|\widehat\theta_N - \theta_*\|_2\leq& \frac{D_{\max}+1}{D_{\min}}C_3(\nu, \nu_q,\nu_{q'},\kappa) (wu^2v+w\beta^{3/4})
\sqrt{\frac{s\log(ed)}{N}}
\end{align*}
for any $\beta,u,v,w>7$, where 
$C_i(\nu, \nu_q,\nu_{q'},\kappa),~i=1,2,3$ and $c_1(\nu,\kappa)$ are constants depending polynomially on the parameters $\nu, \nu_q,\nu_{q'},\kappa$.
\end{theorem}

\section{Proof of Theorems: A Heavy-tailed Framework}

In this section, we provide a general analysis on ERM of the form \eqref{eq:reg-threshold-glm} which can also be applied to problems beyond $\ell_1$-regularization, and show that to control the estimation error, it is enough to control local complexities around the true vector $\theta_*$. Our procedure here is an extension of the small-ball method proposed in the works \cite{lecue2017regularization, lecue2017sparse, mendelson2014learning}, and the difference lies in the treatment of a general function $g(\cdot)$ as well as the bias caused by the thresholding.

For the rest of the paper, the notations $B_\Psi(\mf x, r)$, $B_2(\mf x, r)$ denote the ball of radius $r$ centered at $\mf x$ for $\Psi$-norm, 2-norm respectively, and $S_\Psi(\mf x, r)$, $S_2(\mf x, r)$ denote the sphere of radius $r$ centered at $\mf x$ for $\Psi$-norm, 2-norm respectively. We omit $\mf x$ if they are centered at the origin.

We start with the usual optimality analysis of the ERM. Since $\widehat{\theta}_N$ is the solution to \eqref{eq:reg-threshold-glm}, we have
\[
\frac1N\sum_{i=1}^N\l(g\l(\dotp{\widetilde{\mathbf{x}}_i}{\widehat{\theta}_N}\r) - y_i\dotp{\widetilde{\mathbf{x}}_i}{\widehat{\theta}_N}\r)  + \lambda\Psi\l(\widehat{\theta}_N\r)
\leq \frac1N\sum_{i=1}^N\l(g\l(\dotp{\widetilde{\mathbf{x}}_i}{\theta_*}\r) - y_i\dotp{\widetilde{\mathbf{x}}_i}{\theta_*} \r) + \lambda\Psi(\theta_*)
\]
Simple algebraic manipulations give
\begin{multline}\label{eq:opt1}
\frac{1}{N}\sum_{i=1}^N\l(g\l(\dotp{\widetilde{\mathbf{x}}_i}{\widehat{\theta}_N}\r) - g(\dotp{\widetilde{\mathbf{x}}_i}{\theta_*})
 - g'(\dotp{\widetilde{x}_i}{\theta_*})\dotp{\widetilde{\mf x}_i}{\widehat\theta_N - \theta_*}\r)\\
\frac{1}{N}\sum_{i=1}^N\dotp{\widetilde{\mathbf{x}}_i}{\widehat{\theta}_N-\theta_*}
\l(y_i - g'(\dotp{\widetilde{\mf x}_i}{\theta_*})\r) + \lambda\l( \Psi\l(\widehat{\theta}_N\r) - \Psi\l(\theta_*\r)\r)\leq0.
\end{multline}
To simplify the notations, for any $\mathbf{v} \in\mathbb{R}^d$, define 
\begin{align*}
\mathcal{Q}_{\mathbf{v}}(\mf x)&:= g\l(\dotp{\widetilde{\mathbf{x}}}{\theta_*+\mathbf{v}}\r) - g(\dotp{\widetilde{\mathbf{x}}}{\theta_*})
 - g'(\dotp{\widetilde{x}}{\theta_*})\dotp{\widetilde {\mf x}}{\mathbf{v}} \\
\mc{M}_{\mathbf{v}}(\mf x)&:=\l(y-g'(\dotp{\wt{\mathbf{x}}}{\theta_*})\r)\dotp{\wt{\mf x}}{\mathbf{v}} - \expect{\l(y -\dotp{\wt{\mathbf{x}}}{\theta_*}\r)\dotp{\wt{\mf x}}{\mathbf{v}}}  \\
\mc{V}_{\mathbf{v}}& := \expect{\l(y - g'(\dotp{\wt {\mf x}}{\theta_*})\r)\dotp{\wt{\mf{x}}}{\mathbf{v}}} 
\end{align*}
In addition, for any Borel measurable function $G:\mathbb{R}^d\rightarrow\mb{R}$, $\mc{P}_NG:=\frac{1}{N}\sum_{i=1}^NG(\mf{x}_i)$. Let \begin{equation}\label{eq:L}
\mathcal{L}^{\lambda}_{\mathbf{v}}(\mf x) := \mathcal{Q}_{\mathbf{v}}(\mf x) - \mc{M}_{\mathbf{v}}(\mf x) 
-  \mc{V}_{\mathbf{v}} + \lambda\l( \Psi\l(\theta_*+\mathbf{v}\r) - \Psi\l(\theta_*\r) \r)
\end{equation}
Having defined these notations, the criterion \eqref{eq:opt1} simply implies $\mc{P}_N\mathcal{L}^{\lambda}_{\widehat{\theta}_N,\theta_*}\leq0$. Our goal is then to show that for any $\theta\in\mb{R}^d$ such that $\|\theta-\theta_*\|_2>r$, where $r>0$ is a certain bounding radius, then, 
$$\mc{P}_N\mathcal{L}^{\lambda}_{\theta-\theta_*} =  
\mc{P}_N\mathcal{Q}_{\theta-\theta_*} - \mc{P}_N\mc{M}_{\theta-\theta_*} -  \mc{V}_{\theta-\theta_*} + \lambda\l( \Psi\l(\theta\r) - \Psi\l(\theta_*\r) \r)> 0. $$
The intuition why one would expect this to happen is as follows. Suppose $\Psi(\cdot)$ is not a smooth function near $\theta_*$ and the set of sub-differentials of the norm function $\Psi(\cdot)$ near $\theta_*$ (which we denote as $\partial\Psi(\theta_*)$) is ``large'', then, the set of descent directions i.e. 
$D_{\Psi}(\theta_*):= \l\{ \theta\in\mathbb{R}^d:~\Psi(\theta)\leq \Psi(\theta_*) \r\}$ would be relatively small.\footnote{The descent cone and the cone of sub-differentials are dual to each other.} This implies 
\begin{itemize}
\item For $\theta\in\mathbb{R}^d$ not in the descent directions, $\Psi(\theta)> \Psi(\theta_*)$, and for an appropriate choice of $\lambda$, the possibly negative linear terms $- \mc{P}_N\mc{M}_{\theta-\theta_*} -  \mc{V}_{\theta-\theta_*}$ would be dominated by $\Psi(\theta)- \Psi(\theta_*)$.
\item For the set of $\theta\in\mathbb{R}^d$ in the descent directions, we would expect the term $\mc{P}_N\mathcal{Q}_{\theta-\theta_*}$ to dominate the linear terms $- \mc{P}_N\mc{M}_{\theta-\theta_*} -  \mc{V}_{\theta-\theta_*}$. Using the strictly convex property, for a sufficiently small set of descent directions intersecting with a bounded region, $\mc{P}_N\mathcal{Q}_{\theta-\theta_*}$ would be a non-degenerated quadratic form (i.e. 
$\mc{P}_N\mathcal{Q}_{\theta-\theta_*}\geq c\|\theta-\theta_*\|_2^2$ for some constant $c>0$), which dominates the linear terms $\mc{P}_N\mc{M}_{\theta-\theta_*}$ and $ \mc{V}_{\theta,\theta_*}$ for all $\theta$ sufficiently away from $\theta_*$ within this bounded region. We then extend this result to any vector sufficiently away from $\theta_*$ via convexity of $g(\cdot)$.
\end{itemize}

To this point, we invoke an idea from 
 \cite{lecue2017sparse} and consider the intersection of an $\ell_2$-ball $B_{2}(\theta_*,r)$ and a $\Psi$-ball $B_\Psi(\theta_*,\rho)$, with a properly chosen $\rho>0$, and we aim to show that if $\theta$ is outside of $B_{2}(\theta_*,r)\cap B_\Psi(\theta_*,\rho)$ with appropriate choices of $r$ and $\rho$, then,  
$\mc{P}_N\mathcal{L}^{\lambda}_{\theta-\theta_*}>0$. As is shown in Fig. \ref{fig:geometry},  having this intersection essentially divides the space outside of $B_{2}(\theta_*,r)\cap B_\Psi(\theta_*,\rho)$ into two types of regions: 1. The region containing the set of descent directions $D_{\Psi}(\theta_*)$, where the term $\mc{P}_N\mathcal{Q}_{\theta-\theta_*}$ is expected to take effect. 2. The region where $\Psi(\theta)>\Psi(\theta_*)$, and the term $\lambda(\Psi(\theta) - \Psi(\theta_*))$ is expected to take effect.

\begin{figure}[htbp]
 \centering
   \includegraphics[width=4in]{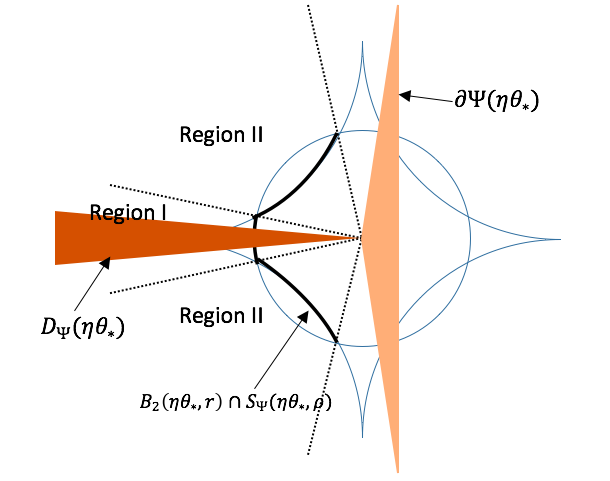} 
   \caption{($\eta=1$) A geometric interpretation that $\theta\not\in B_{2}(\theta_*,r)\cap B_\Psi(\theta_*,\rho)$ implies $\mc{P}_N\mathcal{L}^{\lambda}_{\theta-\theta_*}>0$: When the set of sub-differentials $\partial\Psi(\theta_*)$ is large, the set of descent directions 
   $D_\Psi(\theta_*)$ is small. Then, region I contains $D_\Psi(\theta_*)$, in which $\Psi(\theta)\leq\Psi(\theta_*)$, and the quadratic term $\mc{P}_N\mathcal{Q}_{\theta-\theta_*}$ is expected to dominate $-\mc{P}_N\mc{M}_{\theta-\theta_*} - \mc{V}_{\theta-\theta_*}$. On the other hand, any vector $\theta$ in region II has $\Psi(\theta)>\Psi(\theta_*)$, which gives sufficient increase of norm values to dominate $- \mc{P}_N\mc{M}_{\theta-\theta_*} - \mc{V}_{\theta-\theta_*}$.}
   \label{fig:geometry}
\end{figure}

Let $\Lambda_Q,~\Lambda_M$ and $\Lambda_{\mc{V}}$ be three positive constants.
For chosen $\rho>0$ and $p_{\mathcal{Q}},p_{\mathcal{M}}\in(0,1)$, 
we define three critical radiuses:
\begin{align*}
r_{\mathcal{Q}} :=& \inf\l\{ r>0: Pr\l( \inf_{\theta\in S_{2}(\theta_*,r)\cap B_\Psi(\theta_*,\rho)}\mc{P}_N\mathcal{Q}_{\theta-\theta_*}
\geq \Lambda_Qr^2 \r)\geq1-p_{\mathcal{Q}} \r\},\\
r_{\mc{V}}:=& \inf\l\{ r>0:\sup_{\theta\in B_{2}(\theta_*,r)\cap B_\Psi(\theta_*,\rho)}\l| \mc{V}_{\theta-\theta_*} \r| \leq \Lambda_{\mc{V}}r^2 \r\},\\
r_{M}:= &\inf\l\{ r>0:~Pr\l( \sup_{\theta\in B_{2}(\theta_*,r)\cap B_\Psi(\theta_*,\rho)}\l|\mc{P}_N\mc{M}_{\theta-\theta_*}  \r|  
\leq \Lambda_M r^2\r)\geq 1-p_{\mathcal{M}} \r\},
\end{align*}

We then set 
$$r(\rho) := \max\l\{ r_{\mathcal{Q}}, r_{\mathcal{M}}, r_{\mc{V}} \r\}.$$
Define the set of sub-differentials of the norm function $\Psi(\cdot)$ near $\theta_*$ (i.e. within $\Psi$-radius of $\rho/4$) as
\begin{equation}\label{eq:sub-diff}
\Gamma_\Psi(\theta_*,\rho) := \l\{ \mf{z}\in\mathbb{R}^d: \Psi(\mf u+\Delta \mf u) - \Psi(\mf u)\geq \dotp{\mf{z}}{\Delta\mf{u}} ,~~ \exists \mf u\in B_{\Psi}\l(\theta_*,\frac{\rho}{4}\r),~\forall \Delta \mf u \in\mathbb{R}^d \r\}.
\end{equation}
Then, the set $\Gamma_\Psi(\theta_*,\rho)$ being ``large'' is characterized by the following quantity:
\[
\Delta(\theta_*,\rho) := \inf_{\theta\in B_{2}(\theta_*,r)\cap S_\Psi(\theta_*,\rho)}
~\sup_{\mf{z}\in\Gamma_\Psi(\theta_*,\rho)}\dotp{\mf z}{\theta-\theta_*}
\]
It characterizes the minimum amount of increase of the norm function $\Psi(\cdot)$ from $\Psi(\theta_*)$ on the boundary of region II in Fig. \ref{fig:geometry}, and the set of sub-differentials $\Gamma_\Psi(\theta_*,\rho)$ being ``large'' means for any $\theta\in B_{2}(\theta_*,r)\cap S_\Psi(\theta_*,\rho)$, there exists a vector in $\Gamma_\Psi(\theta_*,\rho)$ which is close to the sub-differential of $\theta-\theta_*$.
Our goal is to show that when $\theta\not\in B_{2}(\theta_*,r(\rho))\cap B_\Psi(\theta_*,\rho)$ and $\Delta(\theta_*,\rho)$ is comparable to $\rho$, then, one has $\mc{P}_N\mathcal{L}^{\lambda}_{\theta-\theta_*}>0$, as is shown in the following theorem.

\begin{theorem}\label{thm:master}
Suppose there exists $\rho>0$ and $c_1\frac{r(\rho)^2}{\rho}\leq\lambda\leq c_2\frac{r(\rho)^2}{\rho}$ for some constant $c_1,c_2$, such that 
$\Lambda_Q> \Lambda_M+\Lambda_{\mc{V}}+c_2$, $c_1\geq8(\Lambda_M+\Lambda_{\mc{V}})$ and
 $\Delta(\eta\theta_*,\rho)\geq\frac34\rho$.
Then, for any $\theta\not\in B_{2}(\eta\theta_*,r(\rho))\cap B_\Psi(\eta\theta_*,\rho)$, $\mc{P}_N\mathcal{L}^{\lambda}_{\theta-\eta\theta_*}>0$ with probability at least $1-p_{\mathcal{Q}}-p_{\mathcal{M}}$. 

Furthermore, if
$\rho= c\Psi(\theta_*)$ for some absolute constant $c>4$, then, for $\lambda\geq c_1\frac{r(\rho)^2}{\rho}$ such that $c_1>8(\Lambda_M+\Lambda_{\mc{V}})$ and 
$\Lambda_Q> \Lambda_M+\Lambda_{\mc{V}}$. 
Then, with probability at least $1-p_{\mathcal{Q}}-p_{\mathcal{M}}$, 
\[
\|\widehat{\theta}_N-\theta_*\|_2\leq \max\l\{ r(\rho), \frac{\lambda}{r(\rho)(\Lambda_Q- \Lambda_M-\Lambda_{\mc{V}})}\Psi(\theta_*) \r\}
\]
\end{theorem}
\begin{remark}
This theorem shows that the desired estimation error follows readily from tight bounds on $r_Q,~r_M$ and $r_{\mathcal V}$. Furthermore, in the second scenario when $\rho= c\Psi(\theta_*)$ for $c>4$,  the set $B_{\Psi}\l(\theta_*,\frac{\rho}{4}\r)$ contains the origin, in which case $\Gamma_\Psi(\theta_*,\rho)$ must contain the unit ball in the dual norm and $\Delta(\theta_*,\rho) \geq \rho$.
\end{remark}

To prove this theorem we need the following simple preliminary lemma:
\begin{lemma}\label{lem:convex}
For any $\mathbf v\in\mathbb{R}^d$, 
$\mathcal{Q}_{\gamma\mathbf{v}} \geq\gamma  \cdot\mathcal{Q}_{\mathbf{v}}$.
\end{lemma}
\begin{proof}[Proof of Lemma \ref{lem:convex}]
First of all, by convexity of the function $g(\cdot)$,
\[
\frac{1}{\gamma}\cdot g\l(\dotp{\widetilde{\mathbf{x}}}{\theta_*+\gamma\mathbf{v}}\r) + \frac{\gamma-1}{\gamma}\cdot 
g\l(\dotp{\widetilde{\mathbf{x}}}{\theta_*}\r)
\geq g\l(\dotp{\widetilde{\mathbf{x}}}{\theta_*+\mathbf{v}}\r).
\]
Rearranging the terms gives 
\[
g\l(\dotp{\widetilde{\mathbf{x}}}{\theta_*+\gamma\mathbf{v}}\r) - g\l(\dotp{\widetilde{\mathbf{x}}}{\theta_*}\r)
\geq \gamma\cdot \big(g\l(\dotp{\widetilde{\mathbf{x}}}{\theta_*+\mathbf{v}}\r) - g\l(\dotp{\widetilde{\mathbf{x}}}{\theta_*}\r)\big).
\]
Substituting this relation into the definition of $\mathcal{Q}_{\gamma\mathbf{v}}(\mf x)$ gives
\begin{align*}
\mathcal{Q}_{\gamma\mathbf{v}}(\mf x)
&= g\l(\dotp{\widetilde{\mathbf{x}}}{\theta_*+\gamma\mathbf{v}}\r) - g(\dotp{\widetilde{\mathbf{x}}}{\theta_*})
 - g'(\dotp{\widetilde{x}}{\theta_*})\dotp{\widetilde {\mf x}}{\gamma\mathbf{v}}\\
 &\geq\gamma\cdot\big(
 g\l(\dotp{\widetilde{\mathbf{x}}}{\theta_*+\mathbf{v}}\r) - g(\dotp{\widetilde{\mathbf{x}}}{\theta_*})
 - g'(\dotp{\widetilde{x}}{\theta_*})\dotp{\widetilde {\mf x}}{\mathbf{v}}\big)\\
 &=\gamma \mathcal{Q}_{\mathbf{v}}(\mf x),
 \end{align*}
 finishing the proof.
\end{proof}

\begin{proof}[Proof of Theorem \ref{thm:master}]
First of all, we have for any $\theta\in\mb{R}^d$
\[
\mc{P}_N\mathcal{L}^{\lambda}_{\theta-\theta_*}
\geq \mc{P}_N\mathcal{Q}_{\theta-\theta_*} - |\mc{P}_N\mc{M}_{\theta-\theta_*}| -  |\mc{V}_{\theta-\theta_*}| + \lambda\l( \Psi\l(\theta\r) - \Psi\l(\theta_*\r) \r)
\]
We now prove the first part of the lemma, which is divided into the following three steps.
\begin{enumerate}
\item Consider first that $\|\theta-\theta_*\|_2>r(\rho)$ and $\Psi(\theta-\theta_*)\leq\rho$. 
By Lemma \ref{lem:convex} and then the definition of $r(\rho)$, we have
\[
 \mc{P}_N\mathcal{Q}_{\theta-\theta_*} = \frac{\|\theta-\theta_*\|_2}{r(\rho)}\cdot  
  \mc{P}_N\mathcal{Q}_{\frac{\theta-\theta_*}{\|\theta-\theta_*\|_2}r(\rho)} 
 \geq \Lambda_Q\|\theta-\theta_*\|_2r(\rho),
\]
with probability at least $1-p_{\mathcal{Q}}$, and 
\[
|\mc{P}_N\mc{M}_{\theta-\theta_*}| = \l|\mc{P}_N\mc{M}_{\frac{\theta-\theta_*}{\|\theta-\theta_*\|_2}r(\rho)}  \r|\cdot\frac{\|\theta-\theta_*\|_2}{r(\rho)}
\leq \Lambda_M\|\theta-\theta_*\|_2r(\rho),
\]
with probability at least $1-p_{\mathcal{M}}$. Also,
\[
|\mc{V}_{\theta-\theta_*}| = \l|\mc{V}_{\frac{\theta-\theta_*}{\|\theta-\theta_*\|_2}r(\rho)}  \r|\cdot\frac{\|\theta-\theta_*\|_2}{r(\rho)}
\leq \Lambda_{\mc V}\|\theta-\theta_*\|_2r(\rho).
\]
Thus,
\begin{equation}\label{eq:inter-11}
\mc{P}_N\mathcal{L}^{\lambda}_{\theta-\theta_*}
\geq ( \Lambda_Q -\Lambda_M- \Lambda_{\mc V})\|\theta-\theta_*\|_2r(\rho) + \lambda\l( \Psi\l(\theta\r) - \Psi\l(\theta_*\r) \r)
\end{equation}
For $\lambda\leq c_2\frac{r(\rho)^2}{\rho}$, we have 
\begin{equation}\label{eq:interinter}
\lambda(\Psi(\theta) - \Psi(\theta_*))\geq -c_2\frac{r(\rho)^2}{\rho}\cdot\Psi(\theta-\theta_*)
\geq -c_2r(\rho)^2\geq - c_2\|\theta-\theta_*\|_2r(\rho).
\end{equation}
By the assumption that $\Lambda_Q> \Lambda_M+\Lambda_{\mc{V}}+c_2$, we know that $ \mc{P}_N\mathcal{L}^{\lambda}_{\theta-\theta_*}>0$ with probability at least $1-p_{\mathcal{Q}}-p_{\mathcal{M}}$. 

%

\item Consider the case $\|\theta-\theta_*\|_2\leq r(\rho)$ and $\Psi(\theta-\theta_*)>\rho$, then, for any specific $\theta$ satisfying the aforementioned conditions,
\begin{align*}
\mc{P}_N\mathcal{L}^{\lambda}_{\theta-\theta_*}
\geq&   - |\mc{P}_N\mc{M}_{\theta-\theta_*}| -  |\mc{V}_{\theta-\theta_*}| + \lambda\l( \Psi\l(\theta\r) - \Psi\l(\theta_*\r) \r)\\
=& \l(- \l|\mc{P}_N\mc{M}_{\frac{\theta-\theta_*}{\Psi(\theta-\theta_*)}\rho}\r|-  \l|\mc{V}_{\frac{\theta-\theta_*}{\Psi(\theta-\theta_*)}\rho}\r|\r)\cdot\frac{\Psi(\theta-\theta_*)}{\rho}  + \lambda\l( \Psi\l(\theta\r) - \Psi\l(\theta_*\r) \r)\\
\geq& - (\Lambda_M+\Lambda_{\mc{V}})r(\rho)^2\cdot\frac{\Psi(\theta-\theta_*)}{\rho}  + \lambda\l( \Psi\l(\theta\r) - \Psi\l(\theta_*\r) \r).
\end{align*}


Let $\mf u\in B_{\Psi}(\theta_*,\rho/4)$ be the vector containing a sub-dfferential $\mf z\in\partial\Psi(\mf u)$ such that $\dotp{\mf z}{\theta-\theta_*}\geq \frac34\Psi(\theta-\theta_*) $. Note that this is possible because by the assumption that $\Delta(\theta_*,\rho)\geq\frac34\rho$, we have there exists $\mf u\in B_{\Psi}(\theta_*,\rho/4)$ with a sub-dfferential $\mf z\in\partial\Psi(\mf u)$ such that 
$
\dotp{\mf z}{\frac{\theta-\theta_*}{\Psi(\theta-\theta_*)}\rho}\geq\frac34\rho.
$
Thus, for the same choice of $\mf u$ and $\mf z$, $\Psi(\theta-\theta_*)>\rho$ implies
\begin{equation}\label{sub-diff-bound}
\dotp{\mf z}{\theta-\theta_*}=\dotp{\mf z}{\frac{\theta-\theta_*}{\Psi(\theta-\theta_*)}\rho}\cdot\frac{\Psi(\theta-\theta_*)}{\rho} \geq\frac34\Psi(\theta-\theta_*) .
\end{equation}
This implies 
\begin{align*}
\mc{P}_N\mathcal{L}^{\lambda}_{\theta-\theta_*}
\geq&
- (\Lambda_M+\Lambda_{\mc{V}})r(\rho)^2\cdot\frac{\Psi(\theta-\theta_*)}{\rho}  
+ \lambda\l( \Psi\l(\theta\r) - \Psi(\theta_*+\mf u - \mf u)\r)\\
\geq&
- (\Lambda_M+\Lambda_{\mc{V}})r(\rho)^2\cdot\frac{\Psi(\theta-\theta_*)}{\rho}  + \lambda\l( \Psi\l(\theta\r) - \Psi(\mf u) - \frac{\rho}{4}\r)\\
\geq& - (\Lambda_M+\Lambda_{\mc{V}})r(\rho)^2\cdot\frac{\Psi(\theta-\theta_*)}{\rho}  +\lambda\l( \dotp{\mf z}{\theta-\mf u} - \frac{\rho}{4} \r)\\
\geq& - (\Lambda_M+\Lambda_{\mc{V}})r(\rho)^2\cdot\frac{\Psi(\theta-\theta_*)}{\rho}  +\lambda\l( \dotp{\mf z}{\theta-\theta_*} - \frac{\rho}{2} \r)\\
\geq& \l(- (\Lambda_M+\Lambda_{\mc{V}})r(\rho)^2  +\lambda\cdot\frac{\rho}{4}\r)\cdot\frac{\Psi(\theta-\theta_*)}{\rho},
\end{align*}
where the second inequality follows from $\mf u\in B_{\Psi}(\theta_*,\rho/4)$, the third inequality follows from the definition of sub-differential, the fourth inequality follows from Holder's inequality $\dotp{\mf z}{\theta_*-\mf u}\leq \Psi^*(\mf z)\Psi(\theta_*-\mf u)\leq\frac{\rho}{4} $ and the final inequality follows from the preceding argument \eqref{sub-diff-bound}. Now, we use the assumption that $\lambda\geq c_1\frac{r(\rho)^2}{\rho}$ and $c_1\geq8(\Lambda_M+\Lambda_{\mc{V}})$ to conclude that $\mc{P}_N\mathcal{L}^{\lambda}_{\theta-\eta\theta_*}>0$.


\item The case $\|\theta-\theta_*\|_2> r(\rho)$ and $\Psi(\theta-\theta_*)>\rho$. If $\frac{\|\theta-\theta_*\|_2}{\Psi(\theta-\theta_*)}>\frac{r(\rho)}{\rho}$, then, let $\alpha = \frac{\Psi(\theta-\theta_*)}{\rho}$. We have by Lemma \ref{lem:convex} and then \eqref{eq:inter-11}, \eqref{eq:interinter} in step 1,
\[
\mc{P}_N\mathcal{L}^{\lambda}_{\theta-\theta_*}\geq \alpha\mc{P}_N\mc{Q}_{\frac{\theta-\theta_*}{\Psi(\theta-\theta_*)}\rho} 
- \alpha\l(\l| \mc{P}_N\mc{M}_{\frac{\theta-\theta_*}{\Psi(\theta-\theta_*)}\rho}  \r| + \l|\mc{V}_{\frac{\theta-\theta_*}{\Psi(\theta-\theta_*)}\rho}\r|  \r) -\lambda(\Psi(\theta) - \Psi(\theta^*))>0.
\]
On the other hand, if $\frac{\|\theta-\theta_*\|_2}{\Psi(\theta-\theta_*)}\leq\frac{r(\rho)}{\rho}$, then, let $\alpha = \frac{\|\theta-\theta_*\|_2}{r(\rho)}$ and we have
\[
\mc{P}_N\mathcal{L}^{\lambda}_{\theta-\theta_*}\geq 
- 2\alpha\l(\l| \mc{P}_N\mc{M}_{\frac{\theta-\theta_*}{\|\theta-\theta_*\|_2}r(\rho)}  \r| + \l|\mc{V}_{\frac{\theta-\theta_*}{\|\theta-\theta_*\|_2}r(\rho)}\r|\r)
+\lambda(\Psi(\theta) - \Psi(\theta_*))>0,
\]
by step 2. 
\end{enumerate}
This finishes the proof of the first part.

For the second part of the claim, one first considers the case 
$\|\theta-\theta_*\|_2>r(\rho)$ and $\Psi(\theta-\theta_*)\leq \rho$. Using the fact that $\widehat\theta_N$ is a minimizer of $\mathcal P_N\mathcal L_{\theta-\theta_*}^\lambda$, we get $\mathcal P_N\mathcal L_{\widehat\theta_N-\theta_*}^\lambda\leq 0$. By \eqref{eq:inter-11} in step 1 of the proof, 
\[
\big(\Lambda_Q - \Lambda_M-\Lambda_{\mathcal V}\big)\|\widehat{\theta}_N-\theta_*\|_2r(\rho)\leq \lambda\Psi(\theta_*).
\]
This implies 
\[
\|\widehat{\theta}_N - \theta_*\|_2\leq \frac{\lambda\Psi(\theta_*)}{r(\rho)(\Lambda_Q - \Lambda_M-\Lambda_{\mathcal V})}.
\]


For the case $\|\theta-\theta_*\|_2\leq r(\rho)$ and $\Psi(\theta-\theta_*)>\rho$, one can invoke step 2 of the above proof. Instead of using the assumption 
$\Delta(\theta^*,\rho)\geq3\rho/4$. We consider the following argument: 
Since $\rho/4>\Psi(\theta_*)$, the set $B_{\Psi}\l(\theta_*,\frac{\rho}{4}\r)$ must contain the origin. Thus, one can take $\mf u$ in \eqref{eq:sub-diff} to be 0 and by Hahn-Banach theorem, the set $\Gamma_\Psi(\theta_*,\rho)$ must contain the unit ball of the dual norm, i.e. for any $\mf v\in\mathbb{R}^d$, there exists a vector $\mf z\in \mathbb{R}^d$ such that $\Psi^*(\mf z) = 1$ and $\dotp{\mf z}{\mf v} = \Psi(\mf v)$. As a consequence,
$\Delta(\theta^*,\rho)\geq\rho$ and we have for any $\theta$, there exists a $\mf z\in \Gamma_\Psi(\theta_*,\rho)$, such that 
$\dotp{\mf z}{\theta-\theta_*} = \Psi(\theta-\theta_*)>\rho$. The rest of step 2 and step 3 carry through.
Overall, we finish the proof.
\end{proof}

\section{Proof of Theorem \ref{thm:sparse-recovery}: Computing Local Complexities}

\subsection{Bounding $r_Q$: Preliminary estimates}
In this section, we bound the local complexity $r_Q$. We let $\Psi(\cdot)$ to be the $\ell_1$-norm. Note first that 
\begin{align}
\mathcal{Q}_{\theta - \theta_*}(\mf x)&= g\l(\dotp{\widetilde{\mathbf{x}}}{\theta}\r) - g(\dotp{\widetilde{\mathbf{x}}}{\theta_*})
-g'(\dotp{\widetilde{\mathbf{x}}}{\theta_*})\dotp{\widetilde{\mathbf{x}}}{\theta - \theta_*} \nonumber\\
&= g''\big(\dotp{\widetilde{\mathbf{x}}}{\theta_*} + \alpha \dotp{\widetilde{\mathbf{x}}}{\theta - \theta_*} \big)\dotp{\widetilde{\mf{x}}}{\theta - \theta_*}, \label{eq:def-Q-2}
\end{align}
where $\alpha\in[0,1]$. Define the constants $\delta = \frac12\sqrt{\frac{\kappa}{2}}$ and $Q=\frac{\kappa^2}{8\nu}$, where $\kappa,\nu$ are defined in Assumption \ref{assumption:moment}. Let 
$s_0$ be a constant less than $d$ (to be defined later) and define $\mathcal G_{s_0}$ to be the set of vectors with $s_0$ cardinality.

Our goal is to show that  the intersection of the following three sets, when taking infimum over $\mf v_1\in \mathcal G_{s_0}\cap S_2(1)$ and
$\mf v_2\in S_{2}(0,r)\cap B_\Psi(0,\rho)$ is sufficiently large:
\[
\{i:~|\dotp{\widetilde{\mathbf{x}}}{\mf v_1}|\geq\delta\}
\cap\{i:~|\dotp{\widetilde{\mathbf{x}}}{\mf v_2}|\leq 32(\nu_q^2+\nu_q+1)\rho/Q\}
\cap\{i:~|\dotp{\widetilde{\mathbf{x}}}{\theta_*}|\leq 32\nu_q\|\theta_*\|_1/Q\},
\]
where $c$ is an absolute constant.

\begin{lemma}\label{lem:bound-upper-1}
Let $u\geq 1$ and $N\geq 1024u/Q^2 + c\log ed$ for some absolute constant $c>0$. With probability at least $1-e^{-u}$,
\[
\sup_{\mf v\in  B_\Psi(0,\rho)}\frac1N\sum_{i=1}^N1_{\l\{|\dotp{\mf v}{\widetilde {\mf x}_i}|\geq 32(\nu_q^2+\nu_q+1)\rho/Q\r\}}
\leq \frac{Q}{16}.
\]
\end{lemma}
\begin{proof}[Proof of Lemma \ref{lem:bound-upper-1}]
Let $\tau_1 = 32(\nu_q^2+\nu_q+1)/Q$ and define $\psi(t) = t/\tau_1\rho$. First, by finite difference inequality, we have with probability 
$1-e^{-u}$
\[
\sup_{\mf v\in  B_\Psi(0,\rho)}\frac1N\sum_{i=1}^N1_{\l\{|\dotp{\mf v}{\widetilde {\mf x}_i}|\geq \tau_1\rho\r\}}\leq
\expect{\sup_{\mf v\in  B_\Psi(0,\rho)}\frac1N\sum_{i=1}^N1_{\l\{|\dotp{\mf v}{\widetilde {\mf x}_i}|\geq \tau_1\rho\r\}}} + \sqrt{\frac{u}{N}}.
\]
Thus, it is enough to bound the expected supremum. We have
\begin{align*}
&\expect{\sup_{\mf v\in  B_\Psi(0,\rho)}\frac1N\sum_{i=1}^N1_{\l\{|\dotp{\mf v}{\widetilde {\mf x}_i}|\geq 32(\nu+1)\rho/Q\r\}}}\\
\leq& \expect{\sup_{\mf v\in  B_\Psi(0,\rho)}\frac1N\sum_{i=1}^N\psi\l(|\dotp{\mf v}{\widetilde {\mf x}_i}|\r)}\\
 = &\expect{\sup_{\mf v\in  B_\Psi(0,\rho)}\frac1N\sum_{i=1}^N\psi\l(|\dotp{\mf v}{\widetilde {\mf x}_i}|\r)
 -\expect{\psi\l(|\dotp{\mf v}{\widetilde {\mf x}_i}|\r)} + \expect{\psi\l(|\dotp{\mf v}{\widetilde {\mf x}_i}|\r)}} \\
 \leq& 2\expect{\sup_{\mf v\in  B_\Psi(0,\rho)}\frac1N\sum_{i=1}^N\varepsilon_i\psi\l(|\dotp{\mf v}{\widetilde {\mf x}_i}|\r)}
 + \sup_{\mf v\in  B_\Psi(0,\rho)}\expect{\psi\l(|\dotp{\mf v}{\widetilde {\mf x}_i}|\r)} \\
 \leq&\underbrace{  \frac{2}{\tau_1\rho}\expect{\sup_{\mf v\in  B_\Psi(0,\rho)}\frac1N\sum_{i=1}^N\varepsilon_i\dotp{\mf v}{\widetilde {\mf x}_i}}   }_{\text{(I)}}
 + \underbrace{\sup_{\mf v\in  B_\Psi(0,\rho)}\frac{\expect{|\dotp{\mf v}{\widetilde {\mf x}_i}|}}{\tau_1\rho}}_{\text{(II)}},
\end{align*}
where the first inequality follows from $\psi(t)\geq 1_{\{t\geq \tau_1\rho\}}$, the second inequality follows from symmetrization inequality and the last inequality follows from Talagrand contraction principle. Now, we bound the two terms respectively. 

\begin{itemize}
\item Bounding (I):
First, by Bernstein's ineuqality, 
\[
Pr\l(\frac1N\sum_{i=1}^N\varepsilon_i \widetilde{x}_{ij}\geq\sqrt{\frac{2\nu_q^2u}{N}} + \frac{u}{(\log ed)^{1/4}N^{3/4}}\r)\leq 2 e^{-u/2}.
\]
Taking a union bound over $j\in\{1,2,\cdots, d\}$,
\[
Pr\l(\max_{j}\frac1N\sum_{i=1}^N\varepsilon_i \widetilde{x}_{ij}\geq\sqrt{\frac{2\nu_q^2u\log ed}{N}} + \frac{2u(\log ed)^{3/4}}{N^{3/4}}\r)\leq 2 e^{-u/2}.
\]
Since $N\geq \log ed$ and $u\geq1$, this implies
\[
Pr\l(\max_{j}\frac1N\sum_{i=1}^N\varepsilon_i \widetilde{x}_{ij}\geq(\sqrt 2\nu_q + 2)u\sqrt{\frac{\log ed}{N}} \r)\leq 2 e^{-u/2}.
\]
Thus,
\[
\expect{\max_{j}\frac1N\sum_{i=1}^N\varepsilon_i \widetilde{x}_{ij}}
\leq (1+ 4\sqrt e)(\sqrt 2\nu_q + 2)u\sqrt{\frac{\log ed}{N}},
\]
and
\[
\frac{2}{\tau_1\rho}\expect{\sup_{\mf v\in  B_\Psi(0,\rho)}\frac1N\sum_{i=1}^N\varepsilon_i\dotp{\mf v}{\widetilde {\mf x}_i}}
\leq \frac{2}{\tau_1}\expect{\max_{j}\frac1N\sum_{i=1}^N\varepsilon_i \widetilde{x}_{ij}}
\leq \frac{2(1+ 4\sqrt e)(\sqrt 2\nu_q + 2)}{\tau_1}\sqrt{\frac{\log ed}{N}}.
\]
\item Bounding (II):
\[
\sup_{\mf v\in  B_\Psi(0,\rho)}\frac{\expect{|\dotp{\mf v}{\widetilde {\mf x}_i}|}}{\tau_1\rho}
\leq \frac{1}{\tau_1\rho}\sup_{\mf v\in  B_\Psi(0,\rho)} \expect{|\dotp{\mf v}{\mf x_i}|} +  \expect{|\dotp{\mf v}{\mf x_i - \widetilde{\mf x}_i}|}\\
\leq \frac{1}{\tau_1}\l(\nu_q + \nu_q^2\l(\frac{\log ed}{N}\r)^{1/4}\r),
\]
where the last inequality follows from:
\[
\sup_{\mf v\in  B_\Psi(0,\rho)} \expect{|\dotp{\mf v}{\mf x_i}|}
\leq \sup_{\mf v\in  B_\Psi(0,\rho)} \expect{\dotp{\mf v}{\mf x_i}^2}^{1/2}
\leq \rho \max_{j,k}\expect{|x_{ij}x_{ik}|}^{1/2}\leq \rho\nu_q, 
\]
and the following derivation:
\begin{multline*}
\expect{|\dotp{\mf v}{\mf x_i - \widetilde{\mf x}_i}|}
=\expect{\l| \sum_{j=1}^d v_j(x_{ij} - \widetilde{x}_{ij}) \r|}
\leq \sum_{j=1}^d |v_j|\cdot \expect{|x_{ij} - \widetilde{x}_{ij}|}
\leq \rho \max_j \expect{|x_{ij} - \widetilde{x}_{ij}|}
\end{multline*}
and
\begin{multline*}
\expect{|x_{ij} - \widetilde{x}_{ij}|}
\leq \expect{|x_{ij}| 1_{\l\{|x_{ij}|>(N/\log ed)^{1/4}\r\}}}
\leq \expect{|x_{ij}|^2}^{1/2}Pr(|x_{ij}|>(N/\log ed)^{1/4})^{1/2}\\
\leq \expect{|x_{ij}|^2}^{1/2}\expect{|x_{ij}|^2}^{1/2}\l(\frac{\log ed}{N}\r)^{1/4}
\leq \nu_q^2\l(\frac{\log ed}{N}\r)^{1/4},
\end{multline*}
where the first inequality follows from the definition that $\widetilde{x}_{ij} = \sign(x_{ij})|x_{ij}|\wedge(N/\log ed)^{1/4}$, the second inequality follows from Holder's inequality and  the third inequality follows from Markov inequality.
\end{itemize}
Overall, we obtain with probability $1-e^{-u}$,
\[
\sup_{\mf v\in  B_\Psi(0,\rho)}\frac1N\sum_{i=1}^N1_{\l\{|\dotp{\mf v}{\widetilde {\mf x}_i}|\geq \tau_1\rho\r\}}\leq
\frac{2(1+ 4\sqrt e)(\sqrt 2\nu_q + 2)}{\tau_1}\sqrt{\frac{\log ed}{N}} +  \frac{1}{\tau_1}\l(\nu_q + \nu_q^2\l(\frac{\log ed}{N}\r)^{1/4}\r) + 
\sqrt{\frac{u}{N}}.
\]
Since $N\geq 1024u/Q^2 + c\log ed$ and $\tau_1 = 32(\nu_q^2+\nu_q+1)/Q$, it follows for $c$ large enough, we have
\[
\sup_{\mf v\in  B_\Psi(0,\rho)}\frac1N\sum_{i=1}^N1_{\l\{|\dotp{\mf v}{\widetilde {\mf x}_i}|\geq \tau_1\rho\r\}}\leq\frac{Q}{16},
\]
finishing the proof.
\end{proof}

\begin{lemma}\label{lem:bound-upper-2}
Let $u\geq 1$ and $N\geq 1024u/Q^2$. With probability at least $1-e^{-u}$,
\[
\frac1N\sum_{i=1}^N1_{\l\{|\dotp{\theta_*}{\widetilde {\mf x}_i}|\geq 32\nu_q\|\theta_*\|_1/Q\r\}}
\leq \frac{Q}{16}.
\]
\end{lemma}
\begin{proof}[Proof of Lemma \ref{lem:bound-upper-2}]
First of all, note that
\[
\expect{|\dotp{\theta_*}{\widetilde {\mf x}_i}|}\leq\expect{|\dotp{\theta_*}{\widetilde {\mf x}_i}|^2}^{1/2}\leq \sum_{i=1}^d|\theta_{*.i}|\cdot\expect{|\widetilde{x}_{ij}|^2}^{1/2}
\leq \|\theta_*\|_1\nu_q.
\] 
By Markov inequality, 
\[
\expect{1_{\l\{|\dotp{\theta_*}{\widetilde {\mf x}_i}|\geq 32\nu_q\|\theta_*\|_1/Q\r\}}} = Pr\l(|\dotp{\theta_*}{\widetilde {\mf x}_i}| \geq \frac{32\nu_q\|\theta_*\|_1}{Q}\r)\leq \frac{Q}{32}.
\]
By bounded difference inequality,
\[
Pr\l( \frac1N\sum_{i=1}^N1_{\l\{|\dotp{\theta_*}{\widetilde {\mf x}_i}|\geq 32\nu_q\|\theta_*\|_1/Q\r\}}\geq
\expect{1_{\l\{|\dotp{\theta_*}{\widetilde {\mf x}_i}|\geq 32\nu_q\|\theta_*\|_1/Q\r\}}} + \sqrt{\frac{u}{N}} \r)\leq e^{-u}.
\]
Thus, it follows when $N\geq 1024u/Q^2$, the desired inequality holds.
\end{proof}

\subsection{Weak small-ball estimates for small $N$}
In this section, we consider lower bounding the cardinality of the set $\{i:~|\dotp{\widetilde{\mathbf{x}}}{\mf v_1}|\geq\delta\},~\mf v_1\in\mathcal G_{s_0}\cap S_2(1)$ when $s_0\leq d$. This holds when $\frac{N}{\log ed}\cdot\frac{Q}{\nu}\leq d$ which is
$N\leq \frac{\nu}{Q}d\log ed$.

We start with the following small-ball estimate via Paley-Zygmund inequality:
\begin{lemma}\label{lem:small-ball}
Under Assumption \ref{assumption:moment}, let $\delta = \frac12\sqrt{\frac{\kappa}{2}}$ and $Q=\frac{\kappa^2}{8\nu}$, then, we have
\[
\inf_{\mathbf{v}\in\mathbb{R}^{d}} Pr\l( \Big| \dotp{\mathbf{x}_i}{\mathbf{v}} \Big| \geq 2\delta\|\mathbf{v}\|_2 \r) \geq 2Q.
\]
\end{lemma}
\begin{proof}
By Paley-Zygmund inequality, we know for any nonnegative real valued random variable $Z$,
\[
Pr(Z>t\expect{Z})\geq (1-t)^2\frac{\expect{Z}^2}{\expect{Z^2}},
\]
for any $t\geq0$. Now, fix any $\mathbf{v}\in\mathbb{R}^{d}$, we take $Z=| \dotp{\mathbf{x}_i}{\mathbf{v}} |^2$, $t = 1/2$, and obtain
\[
Pr\l(  \Big| \dotp{\mathbf{x}_i}{\mathbf{v}} \Big|^2 \geq \frac12 \expect{\Big| \dotp{\mathbf{x}_i}{\mathbf{v}} \Big|^2}\r)
\geq\frac14\frac{\expect{| \dotp{\mathbf{x}_i}{\mathbf{v}} |^2}^2}{\expect{| \dotp{\mathbf{x}_i}{\mathbf{v}} |^4}}
\] 
Recall from Assumption \ref{assumption:moment}, $\lambda_{\min}(\mathbf{\mathbf{\Sigma}}_X)>\kappa$, thus, $\expect{\Big| \dotp{\mathbf{x}_i}{\mathbf{v}} \Big|^2}\geq\kappa\|\mathbf{v}\|_2^2$
for any $\mathbf{v}\in\mathbb{R}^{d}$, and it follows,
\begin{align*}
\inf_{\mathbf{v}\in\mathbb{R}^{d}} Pr\l(  \Big| \dotp{\mathbf{x}_i}{\mathbf{v}} \Big| \geq \sqrt{\frac{\kappa}{2}}\|\mathbf{v}\|_2 \r)
&\geq \inf_{\mathbf{v}\in\mathbb{R}^{d}} Pr\l(  \Big| \dotp{\mathbf{x}_i}{\mathbf{v}} \Big|^2 \geq \frac12 \expect{\Big| \dotp{\mathbf{x}_i}{\mathbf{v}} \Big|^2}\r)\\
&\geq \inf_{\mathbf{v}\in\mathbb{R}^{d}}  \frac14\expect{\Big| \dotp{\mathbf{x}_i}{\mathbf{v}} \Big|^2}^2\left/\expect{\Big| \dotp{\mathbf{x}_i}{\mathbf{v}} \Big|^4}\right.\\
&\geq  \frac14\frac{ \inf_{\mathbf{v}\in S_2(1)} \expect{| \dotp{\mathbf{x}_i}{\mathbf{v}} |^2}^2}{\sup_{\mathbf{v}\in S_2(1)}
\expect{| \dotp{\mathbf{x}_i}{\mathbf{v}}  |^4}}\geq \frac{\kappa^2}{4\nu},
\end{align*}
where the last inequality follows from Assumption \ref{assumption:moment}. Taking $\delta = \frac12\sqrt{\frac{\kappa}{2}}$ and $Q=\frac{\kappa^2}{8\nu}$ finishes the proof.
\end{proof}

We see from Lemma \ref{lem:small-ball} that indeed such a small-ball condition is easily satisfied merely under a bounded moment assumption. The following lemma is the key to our analysis in this step. It says a somewhat ``weak'' small-ball condition is preserved under adaptive truncation.

\begin{lemma}\label{lem:weak-small-ball}
Let $s_0$ be a positive integer such that $1\leq s_0\leq d$. Let $\mathcal{G}_{s_0}$ be the set of all vectors in $\mathbb{R}^d$ with $s_0$ cardinality of the support set. Suppose Assumption \ref{assumption:moment} holds and $N\geq \frac{\nu}{Q}s_0\log(ed) $, then,
for any $\mathbf{v}\in \mathcal{G}_{s_0}$,
\[
Pr\l( \Big| \dotp{\widetilde{\mathbf{x}}_i}{\mathbf{v}} \Big| \geq \delta\|\mathbf{v}\|_2 \r) \geq Q.
\]
\end{lemma}

\begin{proof}
First, note that for any vector $\mathbf{v}\in \mathcal{G}_{s_0}$,
\begin{align*}
\l|  \dotp{\widetilde{\mathbf{x}}_i}{\mathbf{v}} \r| = \l|  \dotp{\widetilde{\mathbf{x}}_i - \mathbf{x}_i}{\mathbf{v}} +  \dotp{\mathbf{x}_i}{\mathbf{v}} \r|
\geq\l| \dotp{\mathbf{x}_i}{\mathbf{v}} \r| - \l| \dotp{\widetilde{\mathbf{x}}_i - \mathbf{x}_i}{\mathbf{v}}  \r|.
\end{align*}
Thus, it follows
\begin{align}
Pr\l( \l| \dotp{\widetilde{\mathbf{x}}_i}{\mathbf{v}} \r| \geq \delta\|\mathbf{v}\|_2 \r) 
\geq& Pr\l( \l| \dotp{\mathbf{x}_i}{\mathbf{v}} \r| \geq \delta\|\mathbf{v}\|_2 +  \l| \dotp{\widetilde{\mathbf{x}}_i - \mathbf{x}_i}{\mathbf{v}}  \r| \r) \nonumber\\
\geq& Pr\l( \l\{ \l| \dotp{\mathbf{x}_i}{\mathbf{v}} \r| \geq 2\delta\|\mathbf{v}\|_2 \r\}   
  \cap   \l\{ \l| \dotp{\widetilde{\mathbf{x}}_i - \mathbf{x}_i}{\mathbf{v}}  \r|\leq \delta\|\mathbf{v}\|_2 \r\}  \r) \nonumber\\
\geq&  Pr\l( \l| \dotp{\mathbf{x}_i}{\mathbf{v}} \r| \geq 2\delta\|\mathbf{v}\|_2 \r)  
- 
Pr\l( \l| \dotp{\widetilde{\mathbf{x}}_i - \mathbf{x}_i}{\mathbf{v}}  \r|\geq\delta\|\mathbf{v}\|_2 \r), \label{inter-0}
\end{align}
where the last inequality follows from the fact that for any two measurable set $A,B$ in a probability space $(\Omega,\mathcal{E},\mathbb{P})$, $Pr(A\cap B) = Pr(A \setminus (B^c\cap A) )\geq Pr(A) - Pr(B^c\cap A)\geq Pr(A) - Pr(B^c)$. By Lemma \ref{lem:small-ball}, $Pr\l( \Big| \dotp{\mathbf{x}_i}{\mathbf{v}} \Big| \geq 
2\delta\|\mathbf{v}\|_2 \r)\geq 2Q$. It remains to bound $Pr\l( \l| \dotp{\widetilde{\mathbf{x}}_i - \mathbf{x}_i}{\mathbf{v}}  \r|\geq\delta\|\mathbf{v}\|_2 \r)$ from above. To this point, let $\mathcal{P}_{\mathbf{v}}\mathbf{x}$ be the orthogonal projection of a vector $\mathbf{x}\in\mathbb{R}^d$ onto the non-zero coordinates of $\mathbf{v}$. Then, by Holder's inequality, we have
\begin{align*}
Pr\l( \l| \dotp{\widetilde{\mathbf{x}}_i - \mathbf{x}_i}{\mathbf{v}}  \r|\geq\delta\|\mathbf{v}\|_2 \r)
\leq& Pr\l(  \|\mc{P}_{\mf v}(\widetilde{\mathbf{x}}_i - \mathbf{x}_i)\|_{\infty}\|\mathbf{v}\|_1  \geq\delta\|\mathbf{v}\|_2 \r)\\
=& Pr\l(  \| \mc{P}_{\mf v} (\widetilde{\mathbf{x}}_i - \mathbf{x}_i)\|_{\infty}  \geq\delta\frac{\|\mathbf{v}\|_2}{\|\mathbf{v}\|_1} \r)\\
\leq& Pr\l( \| \mc{P}_{\mf v}\mathbf{x}_i \|_{\infty}>\tau \r),
\end{align*}
where the last inequality follows from the definition of $\wt{\mf x}_i$ in \eqref{eq:trunc1} that if every entry of $\mc{P}_{\mf v}\mathbf{x}_i $ is bounded by $\tau$, then
$\mc{P}_{\mf v}\mathbf{x}_i = \mc{P}_{\mf v}\wt{\mathbf{x}}_i $. Furthermore,
\begin{multline*}
Pr\l( \| \mc{P}_{\mf v}\mathbf{x}_i \|_{\infty}>\tau \r)
\leq Pr\l( \l(\sum_{j\in\mc G_{\mf v}}x_{ij}^4\r)^{\frac14}> \tau \r)
= Pr\l( \sum_{j\in\mc G_{\mf v}}x_{ij}^4>\tau^4\r)\\
\leq\frac{\expect{\sum_{j\in\mc  G_{\mf v}}x_{ij}^4}}{\tau^4}
\leq \frac{s_0\nu\log (ed) }{N},
\end{multline*}
where the second from the last inequality follows from Markov inequality and the last inequality follows from the definition of $\tau = (N/\log (ed))^{1/4}$ and  the assumption that $\expect{x_{ij}^4}\leq \nu$. Since $N\geq \frac{\nu}{Q}s_0\log (ed) $ by assumption, we have $Pr\l( \| \mc{P}_{\mf v}\mathbf{x}_i \|_{\infty}>\tau \r)\geq Q$ and the proof is finished.
\end{proof}

Using the previous lemma one can show the following via a book-keeping VC dimension argument.

\begin{lemma}\label{lem:VC}
Consider any integer $s_0$ such that $1\leq s_0\leq d$.
Suppose $N\geq \frac{\nu}{Q}s_0\log(ed) $, then,
with probability at least $1-c\exp(-u)$,
\[
 \inf_{\mathbf{v}\in \mathcal{G}_{s_0}\cap S_2(1)}\frac1N\sum_{i=1}^N\mathbf{1}_{  \l\{ \l|\dotp{\widetilde{\mathbf{x}}_i}{\mathbf{v}}\r| \geq\delta/2 \r\}  }
 \geq Q - L\sqrt{s_0\log(ed)/N} - \sqrt{u/N},
\]
where $L,c\geq 1$ are absolute constants.
\end{lemma}

\begin{proof}[Proof of Lemma \ref{lem:VC}]
First of all, by Lemma \ref{lem:weak-small-ball}, for any $i\in\{1,2,\cdots,N\}$ and $\mathbf{v}\in\mathcal{G}_{s_0}\cap S_2(1)$, we have
\[
\expect{\mathbf{1}_{  \l\{ \l|\dotp{\widetilde{\mathbf{x}}_i}{\mathbf{v}}\r| \geq\delta \r\}  }} = Pr\l( \Big| \dotp{\widetilde{\mathbf{x}}_i}{\mathbf{v}} \Big| \geq \delta\|\mathbf{v}\|_2 \r) \geq Q.
\]
Let $\widetilde{\mathbf{x}}_1^N:= \l[ \widetilde{\mathbf{x}}_1,\cdots, \widetilde{\mathbf{x}}_N \r]$, and
define the following process parametrized by $\mathbf{v}\in\mathcal{G}_{s_0}\cap S_2(1)$:
\[
R\l( \widetilde{\mathbf{x}}_1^N,\mathbf{v} \r) = \frac1N\sum_{i=1}^N\mathbf{1}_{  \l\{ \l|\dotp{\widetilde{\mathbf{x}}_i}{\mathbf{v}}\r| \geq\delta/2 \r\}  }
- \expect{\mathbf{1}_{  \l\{ \l|\dotp{\widetilde{\mathbf{x}}_i}{\mathbf{v}}\r| \geq\delta/2 \r\}  }},
\]
and we aim to bound the following supremum
\[
\sup_{\mathbf{v}\in\mathcal{G}_{s_0}\cap S_2(1)}\l| R\l( \widetilde{\mathbf{x}}_1^N,\mathbf{v} \r) \r|.
\]
Define the following class of indicator functions:
\[
\mathcal{F} := \l\{ \mathbf{1}_{\l\{ |\dotp{\cdot}{\mathbf{v}}| \geq \delta/2\r\}}, ~\mathbf{v}\in \mathcal{G}_{s_0}\cap S_2(1)  \r\},
\]
By the standard symmetrization argument and then Dudley's entropy estimate (see, for example, \cite{van1996weak} for details of VC theory), we have
\begin{equation}\label{inter-2-1}
\expect{\sup_{\mathbf{v}\in\mathcal{G}_{s_0}\cap S_2(1)}\l| R\l( \widetilde{\mathbf{x}}_1^N,\mathbf{v} \r) \r|} 
\leq \frac{C_0}{\sqrt{N}}\int_0^2\sqrt{\log  \mathcal{N}\l( \varepsilon,\mathcal{F}, \|\cdot\|_{L_2(\mu_N)} \r) }d\varepsilon,
\end{equation}
where $C_0$ is a constant and $\mathcal{N}\l( \varepsilon,\mathcal{F}, \|\cdot\|_{L_2(\mu_N)} \r)$ is the $\varepsilon$-covering number of $\mathcal{F}$ under the norm 
$\|f-g\|_{L_2(\mu_N)}:= \sqrt{\frac{1}{N}\sum_{i=1}^N(f(\mathbf{x}_i) - g(\mathbf{x}_i))^2}$.

Consider, without loss of generality, a particular subspace $K_{s_0}$ of $\mathbb{R}^d$ consisting of all vectors whose first $s_0$ coordinates are non-zero. Note that for any fixed number $c\in \mathbb{R}$, the VC dimension of the set of halfspaces
$\mathcal{H}:=\{\dotp{\cdot}{\mathbf{v}}\geq c, ~\mathbf{v}\in K_{s_0}\cap\mathbb{S}^{s_0-1} \}$ is $VC(\mathcal{H}) = s_0$. Thus, by classical VC theorem, for any distinctive $p$ points in $\mathbb{R}^d$, the number distinctive projections from $\mathcal{H}$ to these $p$ points is $\sum_{i=0}^{s_0}{p \choose i}\leq (p+1)^{s_0}$.
Furthermore, any set in $\mathcal{H}':=\{|\dotp{\cdot}{\mathbf{v}}|\geq c, ~\mathbf{v}\in K_{s_0}\cap\mathbb{S}^{s_0-1} \}$ is the intersection of two sets in $\mathcal{H}$,
thus, the number of distinctive projections from $\mathcal{H}'$ to those $p$ points is at most 
$${(p+1)^{s_0} \choose 2}\leq \frac{e^2(p+1)^{2s_0}}{4}\leq 2(p+1)^{2s_0}.$$
This implies $VC(\mathcal{H}')\leq cs_0\log(s_0)$ for some absolute constant $c>0$.

Thus, the following class of indicator functions 
\[
\mathcal{F}_{\delta,K_{s_0}} := \l\{ \mathbf{1}_{\l\{ |\dotp{\cdot}{\mathbf{v}}| \geq \delta\r\}}, ~\mathbf{v}\in K_{s_0}\cap\mathbb{S}^{s_0-1} \r\}
\]
has VC dimension $VC(\mathcal{F}_{\delta,K_{s_0}})\leq cs_0\log(s_0)$. By Haussler's inequality, we have the $\varepsilon$ covering number of $\mathcal{F}_{\delta,K_{s_0}}$ can be bounded as
\begin{align*}
\mathcal{N}\l( \varepsilon,\mathcal{F}_{\delta,K_{s_0}}, \|\cdot\|_{L_2(\mu_N)} \r) 
\leq Cs_0(16e)^{cs_0\log(s_0)}\varepsilon^{-2cs_0\log(s_0)},
\end{align*}
where $C>0$ is an absolute constant.
Furthermore, $\mathcal{F}$ is the union of ${ d \choose s_0}$ different subspaces $K_{s_0}$. Thus, the $\varepsilon$ covering number of $\mathcal{F}$ can be bounded as
 \begin{align*}
 \mathcal{N}\l( \varepsilon,\mathcal{F}, \|\cdot\|_{L_2(\mu_N)} \r)
 &\leq { d \choose s_0}Cs_0(16e)^{cs_0\log(s_0)}\varepsilon^{-2cs_0\log(s_0)}\\
 &\leq \l(ed/s_0\r)^{s_0}Cs_0(16e)^{cs_0\log(s_0)}\varepsilon^{-2cs_0\log(s_0)}.
 \end{align*}
 Substituting this bound into \eqref{inter-2-1} gives 
 \begin{align*}
 \expect{\sup_{\mathbf{v}\in\mathcal{G}_{s_0}\cap S_2(1)}\l| R\l( \widetilde{\mathbf{x}}_1^N,\mathbf{v} \r) \r|}
\leq  L\sqrt{s_0\log(ed)/N},
 \end{align*}
 for some absolute constant $L>0$. By bounded difference inequality, we have
 \begin{align*}
 \sup_{\mathbf{v}\in\mathcal{G}_{s_0}\cap S_2(1)}\l| R\l( \widetilde{\mathbf{x}}_1^N,\mathbf{v} \r) \r| 
 \leq \expect{\sup_{\mathbf{v}\in\mathcal{G}_{s_0}\cap S_2(1)}\l| R\l( \widetilde{\mathbf{x}}_1^N,\mathbf{v} \r) \r|} 
 + \sqrt{u/N},
 \end{align*}
 with probability at least $1-ce^{-u}$ for some constant $c>0$ any $u\geq0$, which implies 
 \begin{align*}
 \inf_{\mathbf{v}\in \mathcal{G}_{s_0}\cap S_2(1)}\frac1N\sum_{i=1}^N\mathbf{1}_{  \l\{ \l|\dotp{\widetilde{\mathbf{x}}_i}{\mathbf{v}}\r| \geq\delta/2 \r\}  }
 \geq Q - L\sqrt{s_0\log(ed)/N} - \sqrt{u/N},
 \end{align*}
 with probability at least $1-ce^{-u}$. This implies the claim of the lemma.
\end{proof}

Combining Lemma \ref{lem:VC} with Lemma \ref{lem:bound-upper-1} and \ref{lem:bound-upper-2} we obtain the following lemma:

\begin{lemma}\label{lem:counting-number}
Let $u\geq1$, $N\geq 1024u/Q^2 + c'\log ed$, where $c'>0$ is an absolute constant and $N\leq\max\{\frac{\nu}{Q},\frac{64L^2}{Q^2}\}d\log ed$, where $L$ is the constant defined in Lemma \ref{lem:VC}. Let 
$s_0 = \frac{N}{\log ed}\min\{\frac{Q}{\nu},\frac{Q^2}{16L^2}\}$. then, with probability at least $1-ce^{-u}$ for some absolute constant $c>0$, there exists a set of indices $\mathcal I\in\{1,2,\cdots,N\}$ such that $|\mathcal I|\geq\frac{Q}{4}N$ and for any $i\in \mathcal I$, $\forall \mf v_1 \in \mathcal G_{s_0}\cap S_2(1),~\forall \mf v_2 \in S_1(\rho)$,
\[
|\dotp{\widetilde{\mathbf{x}}_i}{\mf v_1}|\geq\delta/2,
~~|\dotp{\widetilde{\mathbf{x}}_i}{\mf v_2}|\leq 32(\nu_q^2+\nu_q+1)\rho/Q,
~~|\dotp{\widetilde{\mathbf{x}}_i}{\theta_*}|\leq 32\nu_q\|\theta_*\|_1/Q.
\]
\end{lemma}

\begin{proof}[Proof of Lemma \ref{lem:counting-number}]
First of all, by Lemma \ref{lem:VC}, $s_0 = \frac{N}{\log ed}\min\{\frac{Q}{\nu},\frac{Q^2}{64L^2}\}$ and $N\geq 1024u/Q^2 + c'\log ed$, we have with probability at least $1-e^{-u}$,
\[
\inf_{\mathbf{v}\in \mathcal{G}_{s_0}\cap S_2(1)}\frac1N\sum_{i=1}^N\mathbf{1}_{  \l\{ \l|\dotp{\widetilde{\mathbf{x}}_i}{\mathbf{v}}\r|\r\}} 
\geq \frac{Q}{2}.
\]
On the other hand, by Lemma \ref{lem:bound-upper-1} and \ref{lem:bound-upper-2}, we have 
\[
\inf_{\mf v\in  B_\Psi(0,\rho)}\frac1N\sum_{i=1}^N1_{\l\{|\dotp{\mf v}{\widetilde {\mf x}_i}|\geq 32(\nu_q^2+\nu_q+1)\rho/Q\r\}}
\geq1- \frac{Q}{16},
\]
and
 \[
\frac1N\sum_{i=1}^N1_{\l\{|\dotp{\theta_*}{\widetilde {\mf x}_i}|\leq 32\nu_q\|\theta_*\|_1/Q\r\}}
\geq 1- \frac{Q}{16},
 \]
 with probability at least $1-2e^{-u}$. Combining the above three bounds, we have there exists a set of indices $\mathcal I\subseteq\{1,2,\cdots,N\}$ of cardinality at least $\frac{Q}{4} - \frac{Q}{16} - \frac{Q}{16} >\frac{Q}{4}$ such that the claim in the lemma holds.
\end{proof}

The following theorem bounds $r_Q$:
\begin{theorem}\label{thm:r-Q-1}
Let $u\geq1$, $D_{\min} := \min_{z\in [-c_1(\nu,\kappa)R, ~c_1(\nu,\kappa)R]}g''(z)$, $N\geq 1024u/Q^2 + c\log ed$ and $N\leq\max\{\frac{\nu}{Q},\frac{64L^2}{Q^2}\}d\log ed$, where $L$ is the constant defined in Lemma \ref{lem:VC}. Let 
$s_0 = \frac{N}{\log ed}\min\{\frac{Q}{\nu},\frac{Q^2}{16L^2}\}$, $\rho = c\|\theta_*\|_1$, and $\Lambda_Q = D_{\min}\delta^2Q^2/32$, then, 
\[
r_Q^2\leq \frac{C\nu}{\delta^2Q^2}\l(\sqrt{\nu} + (\sqrt{\nu}+1)\beta\sqrt{\frac{8\log ed}{QN}}\r)
\cdot \max\l\{\frac{\nu}{Q},\frac{64L^2}{Q^2}\r\}\cdot\frac{\|\theta_*\|_1^2\log ed}{N},
\]
with $p_Q = c_1e^{-\beta}$, where $c,c_1,C$ are absolute constants.
\end{theorem}

To prove Theorem \ref{thm:r-Q-1}, we need the the following useful lower bound on the random quadratic form, which comes from \cite{lecue2017sparse}. Lower bounds of this sort via Maurey's empirical method originate from \cite{oliveira2013lower}.

\begin{lemma}[Lemma 2.7 of \cite{lecue2017sparse}]\label{lem:quad-form}
Let $\bold{\Gamma}: \mathbb{R}^{d}\rightarrow\mathbb{R}^m$. Let $s_0$ be a positive integer such that $1< s_0\leq d$. Assume for any $\mathbf{v}\in\mathcal{G}_{s_0}$, 
$\l\| \bold{\Gamma}\mathbf{v} \r\|_2\geq \xi\|\mathbf{v}\|_2$ for some absolute constant $\xi>0$. If $\mathbf{x}\in\mathbb{R}^d$ is a non-zero vector and $\mu_j = |x_j|/\|\mathbf{x}\|_1$, then,
\[
\l\| \bold{\Gamma}\mathbf{x} \r\|_2^2\geq \xi^2\|\mathbf{x}\|_2^2 - \frac{\|\mathbf{x}\|_1^2}{s_0-1}\l( \sum_{j=1}^d \l\|\bold{\Gamma}\mathbf{e}_j \r\|_2^2\mu_j - \xi^2 \r),
\] 
where $\{\mathbf{e}_j\}_{j=1}^d$ is the standard basis in $\mathbb{R}^d$.
\end{lemma}

Denote $\mathcal I$ in Lemma \ref{lem:counting-number} to be $\mathcal I = \{i_1,\cdots,i_{|\mathcal I|}\}$ and let $\wt{\mf{\Gamma}} := \l[\wt{\mf x}_{i_1},~\wt{\mf x}_{i_2},\cdots,~\wt{\mf x}_{i_{|\mathcal I|}}\r]^T/\sqrt{N}$.
We then deduce a lower bound for 
In view of the previous lemma, we also need an upper bound for 
$\max_{1\leq j\leq d}\l\|  \widetilde{\bold{\Gamma}}\mathbf{e}_j \r\|_2^2$:

\begin{lemma}\label{lem:upper-1}
For any $u\geq1$ chosen by the thresholding parameter $\tau$, we have with probability at least $1-e^{-\beta}$,
\[
\max_{1\leq j\leq d}\l\|  \widetilde{\bold{\Gamma}}\mathbf{e}_j \r\|_2^2 \leq \sqrt{\nu} + C (\sqrt{\nu}+1)\beta\sqrt{\frac{8\log (ed)}{QN}}, 
\]
where $C>0$ is an absolute constant.
\end{lemma}

\begin{proof}[Proof of Lemma \ref{lem:upper-1}]
By Bernstein's inequality, we have for any $t\geq0$,
\begin{align*}
Pr\l( \l| \frac{1}{|\mathcal I|}\sum_{i\in\mathcal{I}}\widetilde{x}_{ij}^2 - \expect{\widetilde{x}_{ij}^2} \r| 
\geq C\l( \sqrt{\frac{2\sigma_j^2 t}{|\mathcal I|}} + \frac{bt}{|\mathcal I|} \r) \r)
\leq \exp(-t),
\end{align*}
where 
$$\sigma_j^2 = \expect{\l(\widetilde{x}_{ij}^2-\expect{\wt x_{ij}^2}\r)^2}\leq \expect{| x_{ij}|^4}\leq \sup_{\mathbf{v}\in\mathbb{S}^{d-1}} \expect{\l| \dotp{\mathbf{v}}{\mathbf{x}_i} \r|^4}\leq \nu,$$ 
$|\mathcal I| \geq \frac{Q}{8}N$, $b = \tau^2 = \sqrt{\frac{N}{\log (ed)}}$, and $\expect{\widetilde{x}_{ij}^2}\leq \expect{|\widetilde{x}_{ij}|^4}^{1/2}\leq \sqrt{\nu}$. Thus, it follows for any $j\in\{1,2,\cdots,d\}$,
\[
 \frac{1}{|\mathcal I|}\sum_{i\in\mathcal I}\widetilde{x}_{ij}^2  \leq \sqrt{\nu} + C\l(\sqrt{\frac{8\nu t}{QN}} + \frac{2t}{\sqrt{QN\log (ed)}} \r),
\]
with probability at least $1-\exp(-t)$. Take a union bound over $j\in\{1,2,\cdots,d\}$ and let $t=\beta \log (ed) $ give
\[
 \max_{1\leq j\leq d} \frac{1}{|\mathcal I|}\sum_{i\in\mathcal I}\widetilde{x}_{ij}^2\leq \sqrt{\nu} + C (\sqrt{\nu}+1)\beta\sqrt{\frac{8\log (ed)}{QN}},
\]
with probability at least $1-e^{-\beta}$,
for some absolute constant $C>0$. This finishes the proof.
\end{proof}

\begin{proof}[Proof of Theorem \ref{thm:r-Q-1}]
First of all, by \eqref{eq:def-Q-2} and Lemma \ref{lem:counting-number}, we have with probability at least $1-ce^{-u}$,
\[
\inf_{\theta\in B_1(\theta_*,\rho)\cap S_2(\theta_*,r)}\mc P_N\mc Q_{\theta-\theta_*}
\geq 
D_{\min} \inf_{\theta\in B_1(\theta_*,\rho)\cap S_2(\theta_*,r)}\frac{1}{N}\sum_{i\in\mathcal I}|\dotp{\wt{\mf x}_i}{\theta - \theta_*}|^2.
\]
Since $|\dotp{\widetilde{\mathbf{x}}_i}{\mf v_1}|\geq\delta/2$, we have
\[
\inf_{\mf v\in\mathcal G_{s_0}\cap S_2(1)}\frac{1}{|\mathcal I|}\sum_{i\in\mathcal I}|\dotp{\wt{\mf x}_i}{\mf v}|^2\geq 
\frac{\delta^2Q}{4}.
\]
By Lemma \ref{lem:quad-form} and \ref{lem:upper-1}, we have 
\[
\inf_{\theta\in B_1(\theta_*,\rho)\cap S_2(\theta_*,r)}\mc P_N\mc Q_{\theta-\theta_*}\geq 
D_{\min}\l(\frac{\delta^2Q^2}{8}r^2 - \frac{\rho^2}{s_0-1}
\l( \sqrt\nu + C\l(\sqrt{\nu}+1\r)\beta\sqrt{\frac{4\log (ed)}{QN}} \r)\r).
\]
Note that $s_0 = \frac{N}{\log ed}\min\{\frac{Q}{\nu},\frac{Q^2}{16L^2}\}$, $\rho = c\|\theta_*\|_1$, and $\Lambda_Q = D_{\min}\delta^2Q^2/32$.
The infimum of $r>0$ such that the right hand side is greater than $\Lambda_Qr^2 = \frac{\delta^2Q^2}{32}D_{\min}r^2$ can be obtained by 
letting the right hand side equal to $\frac{\delta^2Q^2}{32}D_{\min}r^2$ and solve for $r$, which gives 
\[
r^2 = \frac{C\|\theta_*\|_1^2\log ed}{\delta^2Q^2N}\nu\l(\sqrt\nu + \l(\sqrt{\nu}+1\r)\beta\sqrt{\frac{\log (ed)}{QN}}\r)\cdot \max\l\{\frac{\nu}{Q},\frac{64L^2}{Q^2}\r\},
\]
for some absolute constant $C$. 
It then follows from the definition of $r_{\mathcal{Q}}$ that $r_{\mathcal{Q}}$ must be bounded above by this value.
\end{proof}

\subsection{Applying Mendelson's small-ball method for large $N$}
In this section, we consider lower bounding the cardinality of the set $\{i:~|\dotp{\widetilde{\mathbf{x}}}{\mf v_1}|\geq\delta\},~\mf v_1\in S_2(1)$ when $N> \frac{\nu}{Q}d\log ed$. In this case, suppose Assumption \ref{assumption:moment} holds, by Lemma \ref{lem:weak-small-ball}, we have for any $\mf v\in\mathbb{R}^d$, 
\begin{equation}\label{eq:small-ball-2}
Pr(|\dotp{\wt{\mf x}_i}{\mf v}|\geq \|\mf v\|_2\delta)\geq Q.
\end{equation}
We have the following lemma:
\begin{lemma}\label{lem:counting-number-2}
Let $u\geq1$, $\rho = c\|\theta_*\|_1$ for some absolute constant $c>0$, 
$$N\geq \frac{4c^2(2+\sqrt 2\nu)^2(1+4\sqrt e)^2\|\theta_*\|_1^2\log ed/r^2 + 4u}{Q}$$ 
and $N>\frac{\nu}{Q}d\log ed$, then, with probability at least $1-c_1e^{-u}$ for some absolute constant $c_1>0$, there exists a set of indices $\mathcal I\in\{1,2,\cdots,N\}$ such that $|\mathcal I|\geq\frac{Q}{4}N$ and for any $i\in \mathcal I$, $\forall \mf v_1 \in \mathcal B_1(0,\rho/r)\cap S_2(0,1),~\forall \mf v_2 \in S_1(\rho)$,
\[
|\dotp{\widetilde{\mathbf{x}}_i}{\mf v_1}|\geq\delta/2,
~~|\dotp{\widetilde{\mathbf{x}}_i}{\mf v_2}|\leq 32(\nu_q^2+\nu_q+1)\rho/Q,
~~|\dotp{\widetilde{\mathbf{x}}_i}{\theta_*}|\leq 32\nu_q\|\theta_*\|_1/Q.
\]
\end{lemma}

\begin{proof}[Proof of Lemma \ref{lem:counting-number-2}]
The proof of this lemma almost follows from that of Lemma \ref{lem:mendelson} from \cite{mendelson2014learning}, the only difference is that we need to take care of  indices $i$ such that
$|\dotp{\widetilde{\mathbf{x}}_i}{\mf v_2}|\leq 32(\nu_q^2+\nu_q+1)\rho/Q,
~~|\dotp{\widetilde{\mathbf{x}}_i}{\theta_*}|\leq 32\nu_q\|\theta_*\|_1/Q,$ which are Lemma \ref{lem:bound-upper-1} and \ref{lem:bound-upper-2}. We consider the quantity
\[
\inf_{\mf v\in  \mathcal B_1(0,\rho/r)\cap S_2(0,1)}\frac{\delta}{N}\sum_{i=1}^N1_{\l\{|\dotp{\wt{\mf x}_i}{\mf v}|\geq \delta\r\}}.
\]
By the same argument as that of Theorem 5.4 in \cite{mendelson2014learning} (using \eqref{eq:small-ball-2}), one obtains with probability at least $1-e^{-u/2}$,
\[
\inf_{\mf v\in  \mathcal B_1(0,\rho/r)\cap S_2(0,1)}\frac{\delta}{N}\sum_{i=1}^N1_{\l\{|\dotp{\wt{\mf x}_i}{\mf v}|\geq \delta\r\}}
\geq Q - \frac{2}{\sqrt{N}}\omega_N(\mathcal B_1(0,\rho/r)\cap S_2(0,1))
-\sqrt{\frac{u}{N}}.
\]
where for any $\mathcal H\subseteq S_2(0,1)$,
\[
\omega_N(\mathcal H) := \expect{\sup_{\mf h\in \mathcal H}\frac{1}{\sqrt{N}} \sum_{i=1}^N\varepsilon_i\dotp{\wt{\mf x}_i}{\mf h}}.
\]
Similar to bounding term (I) is Lemma \ref{lem:bound-upper-1}, one obtains
\begin{align*}
\frac{1}{\sqrt{N}}\omega_N(\mathcal B_1(0,\rho/r)\cap S_2(0,1))
\leq& \frac{1}{\sqrt{N}}\frac{\rho}{r}\expect{\max_j \frac{1}{\sqrt{N}} \sum_{i=1}^N\varepsilon_i\wt{\mf x}_{ij}}\\
\leq&(2+\sqrt 2\nu)(1+4\sqrt e)\sqrt{\frac{\log ed}{N}}\cdot \frac{\rho}{r}\\
=& c(2+\sqrt 2\nu)(1+4\sqrt e)\sqrt{\frac{\log ed}{N}}\cdot\frac{\|\theta_*\|_1}{r},
\end{align*}
where the last inequality follows from the fact that $\rho = c\|\theta_*\|_1$. When 
\[
N\geq \frac{4c^2(2+\sqrt 2\nu)^2(1+4\sqrt e)^2\|\theta_*\|_1^2\log ed/r^2 + 4u}{Q},
\]
we have
\[
\frac1N\sum_{i=1}^N1_{\l\{\l|\dotp{\wt{\mf x}_i}{\mf v}\r|\geq \delta\r\}}\geq \frac{Q}{2},
\]
with probability $1-e^{-u/2}$. Combining this result with Lemma \ref{lem:bound-upper-1} and \ref{lem:bound-upper-2} finishes the proof.
\end{proof}

\begin{theorem}\label{thm:r-Q-2}
Let $u\geq1$, $D_{\min} := \min_{z\in [-c_2(\nu,\kappa)R, ~c_2(\nu,\kappa)R]}g''(z)$, 
$\rho = c\|\theta_*\|_1$ for some absolute constant $c>0$, 
$N\geq \frac{8u}{Q}$, 
 and $N>\frac{\nu}{Q}d\log ed$. Suppose
$\Lambda_Q = D_{\min}\delta^2Q^2/4$, then, 
\[
r_Q^2\leq \frac{8c^2(2+\sqrt 2\nu)^2(1+4\sqrt e)^2\|\theta_*\|_1^2\log ed}{N},
\]
with $p_Q = c_1e^{-u}$, where $c_1$ is absolute constant.
\end{theorem}
\begin{proof}[Proof of Theorem \ref{thm:r-Q-2}]
First, note that when $N\geq \frac{8u}{Q}$ and $r = r_Q$ satisfying the condition asserted in the theorem, then, 
\[
N\geq \frac{4c^2(2+\sqrt 2\nu)^2(1+4\sqrt e)^2\|\theta_*\|_1^2\log ed/r^2 + 4u}{Q}.
\]
For any $\theta\in B_1(\theta_*,\rho)\cap B_2(0, r)$, let 
$\mf v = (\theta - \theta^*)/\|\theta - \theta^*\|_2\in B_1(0,\rho/r)\cap B_2(0,1)$ and with probability at least $1-c_1e^{-u}$, 
\begin{multline*}
\mathcal P_NQ_{\theta - \theta_*}\geq \frac{D_{\min}}{N}\sum_{i\in \mathcal I}|\dotp{\wt{\mf x}_i}{\theta- \theta_*}|^2\\
=\frac{D_{\min}r^2}{N}\sum_{i\in \mathcal I}|\dotp{\wt{\mf x}_i}{\mf v}|^2
\geq \frac{D_{\min}r^2\delta^2}{N |\mathcal I|}\l(\sum_{i\in \mathcal I}1_{\l\{|\dotp{\wt{\mf x}_i}{\mf v}|\geq \delta\r\}}\r)^2
\geq \frac{Q^2\delta^2}{4}D_{\min}r^2,
\end{multline*}
where the first inequality follows from Lemma \ref{lem:counting-number-2} by taking the corresponding $\mathcal I$, the second from the last inequality follows from
\[
\l(\frac{1}{|\mathcal I|}\sum_{i\in \mathcal I}|\dotp{\wt{\mf x}_i}{\mf v}|^2\r)^{1/2}
\geq \frac{1}{|\mathcal I|}\sum_{i\in \mathcal I}|\dotp{\wt{\mf x}_i}{\mf v}|\geq \frac{1}{|\mathcal I|}\sum_{i\in \mathcal I}1_{\l\{|\dotp{\wt{\mf x}_i}{\mf v}|\geq \delta\r\}},
\]
and the last inequality follows from Lemma \ref{lem:counting-number-2} again.
\end{proof}

\subsection{Bounding $r_M$ via Montgomery-Smith inequality}

The main objective is the following bound on $|\mc P_N \mc M_{\theta-\theta_*}|$:

\begin{lemma}\label{lem:bound-PM}
Suppose $N\geq \|\theta_*\|_1^2\log(ed) + \log(ed)$ and Assumption \ref{assumption:moment}, \ref{assumption:link-function} hold.
For any $\beta,u,v,w>7$, we have with probability at least 
\begin{multline*}
1-2e^{-\beta}-2e^{-v^2}
-c'\l(u^{-q}(ed)^{-(\frac c2-1)}+(u^{-q/4}+u^{-q'})(ed)^{-c/2}\r.\\
\l.+(eN)^{-\frac{q}{10}+1}(\log(eN))^{q/5}w^{-q/5} +  (eN)^{-(\frac{q'}{4}-1)}(\log(eN))^{q'/2}w^{-q'}\r).
\end{multline*}
where $c,c'>2$ are absolute constants,
\[
\sup_{\theta\in B_1(\theta_*,\rho)\cap B_2(\theta_*,r)}\l| \mc P_N \mc M_{\theta-\theta_*} \r|\leq C(\nu_q, \nu_{q'})
(D_{\max} + 1)\l(wu^2v+w\beta^{3/4}+\beta\r)\rho
\sqrt{\frac{\log (ed)}{N}},
\]
where $C(\nu_q, \nu_{q'})$ depends polynomially on $\nu_q$ and $\nu_{q'}$.
\end{lemma}


\begin{proof}[Proof of Lemma \ref{lem:bound-PM}]
First of all, by symmetrization inequality, it is enough to bound 
\begin{align*}
&\sup_{\theta\in B_1(\theta_*,\rho)\cap B_2(\theta_*,r)}
\l| \frac1N\sum_{i=1}^N\varepsilon_i\l(y_i - g'(\dotp{\widetilde{\mf x}_i}{\theta_*})\r)\dotp{\wt{\mf x}_i}{\theta-\theta_*} \r| \\
=& \sup_{\mf v\in B_1(0,\rho)\cap B_2(0,r)}
\l| \frac1N\sum_{i=1}^N\varepsilon_i\l(y_i - g'(\dotp{\widetilde{\mf x}_i}{\theta_*})\r)\dotp{\wt{\mf x}_i}{\mf v} \r|
\end{align*}
We define $\mf z := \frac1N\sum_{i=1}^N\varepsilon_i\l(y_i - g'(\dotp{\widetilde{\mf x}_i}{\theta_*})\r)\wt{\mf x}_i$ and note that
\begin{equation}\label{eq:multi-bound-0}
\sup_{\mf v\in B_1(0,\rho)\cap B_2(0,r)}\leq \rho \cdot\max_{j\in\{1,2,\cdots,d\}}|z_j|.
\end{equation}


Now for each $|z_j|$,
\begin{multline*}
N|z_j| = \l|\sum_{i=1}^N\varepsilon_i\l(y_i - g'(\dotp{\widetilde{\mf x}_i}{\theta_*})\r)\wt x_{ij}  \r|\\
\leq 
\l|\sum_{i=1}^N\varepsilon_i \l(y_i - g'(\dotp{\mf x_i}{\theta_*})\r) \wt x_{ij}  \r| + \l|\sum_{i=1}^N\varepsilon_i
\l(g'(\dotp{\mf x_i}{\theta_*}) - g'(\dotp{\widetilde{\mf x}_i}{\theta_*})\r) \wt x_{ij}  \r|
\end{multline*}
Thus, it follows 
\begin{multline}\label{eq:decomp-1}
N\cdot\max_{j\in\{1,2,\cdots,d\}}|z_j| \leq
+ \max_{j\in\{1,2,\cdots,d\}} \l|\sum_{i=1}^N\varepsilon_i\l(y_i - g'(\dotp{\widetilde{\mf x}_i}{\theta_*})\r)\wt x_{ij}  \r|\\
+ \max_{j\in\{1,2,\cdots,d\}}\l|\sum_{i=1}^N\varepsilon_i
\big(g'(\dotp{\mf x_i}{\theta_*}) - g'(\dotp{\widetilde{\mf x}_i}{\theta_*})\big) \wt x_{ij}  \r|
\end{multline}

Then, we need to bound the three terms on the right hand side of \eqref{eq:decomp-1} separately. 
\\

1. \textbf{Bounding the terms} $\max_{j\in\{1,2,\cdots,d\}}\l|\sum_{i=1}^N\varepsilon_i
\big(g'(\dotp{\mf x_i}{\theta_*}) - g'(\dotp{\widetilde{\mf x}_i}{\theta_*})\big) \wt x_{ij}  \r|$:

Let 
$\wt \phi_i = g'(\dotp{\mf x_i}{\theta_*}) - g'(\dotp{\widetilde{\mf x}_i}{\theta_*})$. A crucial first step analyzing such a Rademacher sum (see, for example, \cite{mendelson2016upper, goldstein2016structured}) is to apply Montgomery-Smith inequality from, i.e. Lemma \ref{lemma:Rademacher}, conditioned on $\wt{\mf{x}}_i$, which results in
\[
\l|\sum_{i=1}^N\varepsilon_i\wt \phi_i \wt x_{ij}  \r|
\leq \sum_{i=1}^k\l|\wt\phi_i^{\sharp}\wt x_{ij}^{\sharp} \r|
+v\l( \sum_{i>k}\l|\wt\phi_i^{\sharp}\wt x_{ij}^{\sharp} \r|^2 \r)^{1/2},
\]
with probability at least $1-e^{-v^2}$, 
where $k$ is any chosen integer within $\l\{0,1,2,\cdots,N \r\}$ and $\l(\wt\phi_i^{\sharp}\r)_{i=1}^N$, $\l(\wt x_{ij}^{\sharp}\r)_{i=1}^N$ are non-increasing rearrangements of  $\l(|\wt\phi_i|\r)_{i=1}^N$, $\l(|\wt x_{ij}|\r)_{i=1}^N$. We define the former sum to be 0 when $k=0$.

By Holder's inequality, we have
\begin{equation*}
\l|\sum_{i=1}^N\varepsilon_i\wt \phi_i \wt x_{ij}  \r|
\leq \l(\sum_{i=1}^k\l|\wt\phi_i^{\sharp}\r|^2\r)^{1/2}\l(\sum_{i=1}^k\l|\wt x_{ij}^{\sharp} \r|^2\r)^{1/2}
+v\l( \sum_{i>k}\l|\wt\phi_i^{\sharp} \r|^{2r} \r)^{1/(2r)} 
\l(\sum_{i>k}\l|\wt x_{ij}^{\sharp}\r|^{2r'}\r)^{1/(2r')},
\end{equation*}
for some positive constants $r,r'$ such that $\frac1r+\frac{1}{r'}=1$. Take a union bound for all $j\in\{1,2,\cdots,d\}$, gives with probsability at least $1-e^{-v^2}$,
\begin{multline}\label{eq:master-bound-2}
\max_{j\in\{1,2,\cdots,d\}}\l|\sum_{i=1}^N\varepsilon_i\wt \phi_i \wt x_{ij}  \r|
\leq \l(\sum_{i=1}^k\l|\wt\phi_i^{\sharp}\r|^2\r)^{1/2}
\max_{j\in\{1,2,\cdots,d\}}\l(\sum_{i=1}^k\l|\wt x_{ij}^{\sharp} \r|^2\r)^{1/2}\\
+v\sqrt{\log d}\l( \sum_{i>k}\l|\wt\phi_i^{\sharp} \r|^{2r} \r)^{1/(2r)} 
\max_{j\in\{1,2,\cdots,d\}}\l(\sum_{i>k}\l|\wt x_{ij}^{\sharp}\r|^{2r'}\r)^{1/(2r')},
\end{multline}
where $k$ is to be chosen.

Now we bound the four terms in \eqref{eq:master-bound-2} respectively.

\begin{lemma}\label{lem:supp-11}
Let $k = \lfloor\frac{c\log(ed)}{\log(eN/c\log(ed))}\rfloor$ for some absolute constant $c>2$, and suppose 
$N\geq \|\theta_*\|_1^2\log(ed)$, then, we have
\[
\l(\sum_{i=1}^k\l|\wt\phi_i^{\sharp}\r|^2\r)^{1/2}
\leq CD_{\max}\nu_q^{5}w\sqrt{e\log(ed)},
\]
with probability at least $1 - c' (eN)^{-\frac{q}{10}+1}(\log(eN))^{\frac q5}w^{-\frac q5}$ for any $w>6$ and some absolute constant $C, c'>1$.
\end{lemma}

\begin{proof}[Proof of Lemma \ref{lem:supp-11}]
First of all, using Binomial estimates, we have for any $i$, and any positive constant $c_i$,
\begin{align*}
Pr\l(\l|\wt\phi_i^{\sharp}\r|\geq c_i\|\wt \phi_i\|_{L_p}\r)
&\leq {N \choose i}Pr(\l|\wt\phi_i\r|\geq c_i\|\wt \phi_i\|_{L_p})^i\\
\leq& \l(\frac{eN}{i}\r)^iPr(\l|\wt\phi_i\r|\geq c_i\|\wt \phi_i\|_{L_p})^i\\
\leq&\l(\frac{eN}{i}\r)^i \frac{\expect{\l|\wt\phi_i\r|^p}^i}{c_i^{pi}\l\|\wt\phi_i\r\|_{L_p}^{pi}}
=\l(\frac{eN}{i}\r)^ic_i^{-pi},
\end{align*}
where we define $\l\|\wt\phi_i\r\|_{L_p}:=\expect{\l|\wt\phi_i\r|^p}^{1/p}$ and $p>2$ is a chosen positive constant. Then, we choose $c_i:=\frac{w}{\log(eN/i)}\l(\frac{eN}{i}\r)^{\frac12}$, which implies
\[
Pr\l(\l|\wt\phi_i^{\sharp}\r|\geq \frac{w}{\log(eN/i)}\l(\frac{eN}{i}\r)^{1/2}\|\wt \phi_i\|_{L_p}\r)
\leq \l(\frac{i}{eN}\r)^{i\l(\frac{p}{2}-1\r)}w^{-pi}\l(\log(eN/i)\r)^{pi}.
\]
Thus, it follows,
\begin{multline}\label{eq:inter-111}
\sum_{i=1}^k\l| \wt\phi_i^{\sharp} \r|^2\leq \sum_{i=1}^N\l| \wt\phi_i \r|^2\leq \sum_{i=1}^N\frac{w^2}{(\log(eN/i))^2}\l(\frac{eN}{i}\r)\|\wt \phi_i\|_{L_p}^2\\
\leq w^2\|\wt \phi_i\|_{L_p}^2eN \int_0^N\frac{1}{x(\log(eN)-\log x)^2}dx 
\leq Cw^2\|\wt \phi_i\|_{L_p}^2eN
\end{multline}
with probability at least 
$$
1 - \sum_{i=1}^N\l(\frac{i}{eN}\r)^{i\l(\frac{p}{2}-1\r)}w^{-pi}\l(\log(eN/i)\r)^{pi}.
$$
Note that for $w>7$ and $p$ chosen to be $p := q/5>3$, the above sum is a geometrically decreasing sequence, specifically, it is easy to verify that 
$\l(\frac{i}{eN}\r)^{\l(\frac{p}{2}-1\r)}w^{-p}\l(\log(eN/i)\r)^{p}<(7/6)^{-p},~\forall i \in\{1,2,3,4,\cdots,N\}$.
Thus, 
it follows the above probability is at least 
\[
1 - c'\l(eN\r)^{-\l(\frac{p}{2}-1\r)}\l(\log(eN)\r)^{p}w^{-p},
\]
for some absolute constant $c'>1$. Now, we bound the term $\|\wt \phi_i\|_{L_p}$. We choose $p= \frac{q}{5}$. Then, under the condition that $q>15$, $p=\frac q5>3$, and $\expect{|x_{ij}|^{5p}}<\infty,~\forall i\in\{1,2,\cdots,N\},~j\in\{1,2,\cdots,d\}$. Furthermore, we have by Assumption \ref{assumption:link-function},
\begin{align*}
\|\wt \phi_i\|_{L_p} =\|g'(\dotp{\mf x_i}{\theta_*}) - g'(\dotp{\widetilde{\mf x}_i}{\theta_*})\|_{L_p}\leq D_{\max}\cdot
\|\dotp{\mf x_i-\wt{\mf x}_i}{\theta_*}\|_{L_p}
\end{align*}
Note that 
\begin{multline*}
\|\dotp{\mf x_i-\wt{\mf x}_i}{\theta_*}\|_{L_p}
=\expect{\l|\sum_{n=1}^d(x_{in} - \wt x_{in})\theta_{*,j}\r|^p}^{1/p}\\
\leq \sum_{n=1}^d\expect{|x_{in} - \wt x_{in}|^p}^{1/p}|\theta_{*,j}|
\leq \max_{n}\expect{|x_{in} - \wt x_{in}|^p}^{1/p}\|\theta_*\|_1
\end{multline*}
where the first inequality follows from Minkowski's inequality. Now, for each $n$, we have
\begin{multline*}
\l\| x_{in}-\wt x_{in}  \r\|_{L_p}\leq\l\| x_{in}\cdot1_{\l\{|x_{in}|>\tau\r\}}  \r\|_{L_p} 
\leq \expect{|x_{in}|^{p}\cdot1_{\l\{|x_{in}|>\tau\r\}}}^{1/p}\\
\leq \expect{|x_{in}|^{5p}}^{1/5p}Pr(|x_{in}|>\tau)^{4/5p}\leq  \expect{|x_{in}|^{5p}}^{1/5p}\l(\frac{\expect{|x_{in}|^{5p}}}{\tau^{5p}}\r)^{4/5p},
\end{multline*}
where the second from the last inequality follows from Holder's inequality and the last inequality follows from Markov inequality. Thus, we obtain,
\[
\|\wt \phi_i\|_{L_p} \leq D_{\max}\|\theta_*\|_1 \max_n\expect{|x_{in}|^{5p}}^{1/p}\tau^{-4}
\leq  D_{\max}\|\theta_*\|_1\nu_q^5\frac{\log ed}{N}\leq D_{\max}\nu_q^5\sqrt{\frac{\log ed}{N}},
\]
for some constant $C$ and $\tau = \l( \frac{N}{\log(ed)} \r)^{1/4}\geq \|\theta_*\|^{1/2}$. Overall, substituting the above bound into \eqref{eq:inter-111}, we have with probability at least $1 - c'\l(eN\r)^{-\l(\frac{p}{2}-1\r)}\l(\log(eN)\r)^{p}w^{-p}$, where $p=q/5$,
\[
\sum_{i=1}^k\l| \wt\phi_i^{\sharp} \r|^2\leq CD_{\max}^2\nu_q^{10}w^2eN\cdot \frac{\log(ed)}{N} = CD_{\max}^2\nu_q^{10}w^2e\log(ed),
\]
for some constant $C>1$.
\end{proof}

\begin{lemma}\label{lem:supp-12}
Let $k = \lfloor\frac{c\log(ed)}{\log(eN/c\log(ed))}\rfloor$ for some absolute constant $c>2$, and suppose 
$N\geq \|\theta_*\|_1^2\log(ed)$, then, we have
$$\max_{j\in\{1,2,\cdots,d\}}\l(\sum_{i=1}^k\l|\wt x_{ij}^{\sharp} \r|^2\r)^{1/2}
\leq C\l(\nu_q^2\log(ed) + \nu_q^2\sqrt{\beta}\log(ed)+\sqrt{\frac{N}{\log(ed)}}(\beta+\log(ed))\r)^{1/2},$$
with probability at least $1-e^{-\beta}$ for any $\beta>1$ and some constant $C>1$.
\end{lemma}
\begin{proof}[Proof of Lemma \ref{lem:supp-12}]
First, for any set of $k$ random variables $x_{1j},~x_{2j},~\cdots,~x_{kj}$ we have by Bernstein's inequality,
\[
Pr\l( \sum_{i=1}^k\l|\wt x_{ij} \r|^2\geq k\expect{\wt x_{ij}^2} + C\l( \sqrt{2\sigma_2^2kt}+b_2t\r) \r)\leq e^{-t},
\]
for some constant $C$,
where $\sigma_2^2 := \expect{  \l(\wt x_{ij}^2-\expect{\wt x_{ij}^2} \r)^2  }\leq \expect{x_{ij}^4}\leq\nu_q^4$, $b_2:= \l( N/\log(ed) \r)^{1/2}$ and 
$\expect{\wt x_{ij}^2}\leq \expect{x_{ij}^2}\leq \nu_q^2$. Take a union bound over all ${N \choose k}$ different combinations from $x_{1j},~x_{2j},\cdots,~x_{Nj}$, we obtain,
\[
Pr\l( \sum_{i=1}^k\l|\wt x_{ij}^{\sharp} \r|^2\geq k\expect{\wt x_{ij}^2} + C\l(\sqrt{2\sigma_2^2kt}+b_2t \r) \r)\leq {N \choose k}e^{-t}\leq \l(\frac{eN}{k}\r)^ke^{-t}.
\]
Taking a union bound over all $j\in\{1,2,\cdots,d\}$, we get
\[
Pr\l( \max_{j\in\{1,2,\cdots,d\}}\sum_{i=1}^k\l|\wt x_{ij}^{\sharp} \r|^2\geq k\expect{\wt x_{ij}^2} + C\l( \sqrt{2\sigma_2^2kt}+b_2t \r) \r)\leq
d\l(\frac{eN}{k}\r)^ke^{-t}
\]
Substituting the definition of $k =\lfloor\frac{c\log(ed)}{\log(eN/\log(ed))}\rfloor\leq \frac{c\log(ed)}{\log(eN/\log(ed))}$, we get 
\begin{multline*}
d\l(\frac{eN}{k}\r)^ke^{-t} = \exp\l(-t + k\log(eN/k) + \log d\r)\\
\leq \exp\l(-t + \frac{c\log(ed)}{\log(eN/c\log(ed))}\log\l(\frac{eN}{c\log(ed)}\cdot \log\l(\frac{eN}{c\log(ed)}\r)\r)  + \log d\r)\\
\leq\exp(-t+(2c+1)\log(ed)).
\end{multline*}
Setting $\beta = t - (2c+1)\log(ed)$ and rearranging the terms gives the claim.
\end{proof}

\begin{lemma}\label{lem:supp-13}
Let $k = \lfloor\frac{c\log(ed)}{\log(eN/c\log(ed))}\rfloor$ for some absolute constant $c>2$, and suppose 
$N\geq \|\theta_*\|_1^2\log(ed)$, then, we have with probability at least $1-c'u^{-q/3}(ed)^{-c/2}$, for some absolute constant $c'>0$,
\[
\l(\sum_{i> k}\l|\wt\phi_i^{\sharp}\r|^{2r}\r)^{1/2r}\leq  CD_{\max}u \nu_q^3 N^{1/2r},
\]
for $5/4\leq r<q/12$, any $u>2$, and some absolute constant $C>0$. 
\end{lemma}
\begin{proof}[Proof of Lemma \ref{lem:supp-13}]
Let $p = q/4$, then, $p>3r$. 
Using Binomial estimates, we have for any $i>k$, and any $\alpha>0$,
\[
Pr\l(\l|\wt\phi_i^{\sharp}\r| > \alpha\r)\leq {N\choose i} Pr(|\wt\phi_i|>\alpha)^i\leq   {N\choose i}\l(\frac{\expect{|\wt\phi_i|^{p}}}{\alpha^{p}}\r)^i
\leq \l(\frac{eN}{ i}\frac{\expect{|\wt\phi_i|^{p}}}{\alpha^{p}}\r)^i,
\]
where the second inequality follows from Markov inequality. We choose $\alpha = \|\wt\phi\|_{L_{p}}u\l(\frac{eN}{i}\r)^{3/2p}$ and get
\[
Pr\l(\l|\wt\phi_i^{\sharp}\r| > \|\wt\phi\|_{L_{p}}u\l(\frac{eN}{i}\r)^{3/2p} \r)\leq u^{-pi}\l(\frac{eN}{i}\r)^{-i/2}.
\]
Thus, it follows
\begin{multline*}
Pr\l(\exists i >k,~s.t. \l|\wt\phi_i\r|> \|\wt\phi\|_{L_{p}}u\l(\frac{eN}{i}\r)^{2/p} \r)\leq \sum_{i>k}u^{-pi}\l(\frac{eN}{i}\r)^{-i/2}\\
\leq c'u^{-(k+1)p}\l(\frac{eN}{k+1}\r)^{-(k+1)/2}\leq c'u^{-p}\l(\frac{eN}{k+1}\r)^{-(k+1)/2},
\end{multline*}
for some absolute constant $c'>0$, 
where the second from the last inequality follows from the fact that for any $u>2$, the summand is a geometrically decreasing sequence since $N\geq i$. Plugging in $k+1 \geq \frac{c\log(ed)}{\log(eN/\log(ed))}$ and using the fact that $N\geq k+1$ give 
\begin{multline*}
\l(\frac{eN}{k+1}\r)^{-(k+1)/2} \leq \exp\l(-\frac{c\log(ed)}{2\log(eN/c\log(ed))}  \log\l(\frac{eN}{c\log(ed)}\log\l(\frac{eN}{c\log(ed)}\r)\r)    \r)\\
\leq \exp(-c\log(ed)/2) = (ed)^{-c/2},
\end{multline*}
Thus, it follows with probability at least $1- c_0u^{-p}(ed)^{-c}$, we have 
\begin{equation}\label{eq:inter-1111}
\l(\sum_{i>k}\l|\wt\phi_i^{\sharp}\r|^{2r}\r)^{1/2r}\leq \|\wt\phi\|_{L_{p}}u\l(\sum_{i>k}\l(\frac{eN}{i}\r)^{3r/p} \r)^{1/2r}
\end{equation}
Since $p = q/4>3r$, it follows 
\[
\sum_{i>k}\l(\frac{1}{i}\r)^{3r/p}\leq \int_0^N\l(\frac{1}{x}\r)^{3r/p}dx = \frac{1}{1-3r/p}N^{1-\frac{3r}{p}}.
\]
Thus, with probability at least $1- c_0u^{-q/3}(ed)^{-c/2}$, 
\begin{equation}\label{eq:inter-13}
\l(\sum_{i>k}\l|\wt\phi_i^{\sharp}\r|^{2r}\r)^{1/2r}\leq C\|\wt\phi\|_{L_{p}}u N^{1/2r},
\end{equation}
for some constant $C$. It remains to bound $\|\wt\phi\|_{L_{p}}$. By Assumption \ref{assumption:link-function},
\begin{align*}
\|\wt \phi_i\|_{L_p} =\|g'(\dotp{\mf x_i}{\theta_*}) - g'(\dotp{\widetilde{\mf x}_i}{\theta_*})\|_{L_p}\leq D_{\max}\cdot
\|\dotp{\mf x_i-\wt{\mf x}_i}{\theta_*}\|_{L_p}
\end{align*}
Note that 
\begin{align*}
\|\dotp{\mf x_i-\wt{\mf x}_i}{\theta_*}\|_{L_p}
=\expect{\l|\sum_{n=1}^d(x_{in} - \wt x_{in})\theta_{*,j}\r|^p}^{1/p}
\leq \sum_{n=1}^d\expect{|x_{in} - \wt x_{in}|^p}^{1/p}|\theta_{*,j}|
\leq \max_{n}\expect{|x_{in} - \wt x_{in}|^p}^{1/p}\|\theta_*\|_1
\end{align*}
where the first inequality follows from Minkowski's inequality. Now, for each $n$, we have
\begin{multline*}
\l\| x_{in}-\wt x_{in}  \r\|_{L_p}\leq\l\| x_{in}\cdot1_{\l\{|x_{in}|>\tau\r\}}  \r\|_{L_p} 
\leq \expect{|x_{in}|^{p}\cdot1_{\l\{|x_{in}|>\tau\r\}}}^{1/p}\\
\leq \expect{|x_{in}|^{3p}}^{1/3p}Pr(|x_{in}|>\tau)^{2/3p}\leq  \expect{|x_{in}|^{3p}}^{1/3p}\l(\frac{\expect{|x_{in}|^{3p}}}{\tau^{3p}}\r)^{2/3p},
\end{multline*}
where the second from the last inequality follows from Holder's inequality and the last inequality follows from Markov inequality. Thus, we obtain,
\[
\|\wt \phi_i\|_{L_p} \leq D_{\max}\|\theta_*\|_1 \max_n\expect{|x_{in}|^{3p}}^{1/p}\tau^{-2}
\leq  D_{\max}\|\theta_*\|_1\nu_q^3\sqrt{\frac{\log ed}{N}}\leq D_{\max}\nu_q^3,
\]
for some constant $C$ and $\tau = \l( \frac{N}{\log(ed)} \r)^{1/4}\geq \|\theta_*\|^{1/2}$. 
Combining this bound with \eqref{eq:inter-13} finishes the proof.
\end{proof}

\begin{lemma}\label{lem:supp-14}
Let $k = \lfloor\frac{c\log(ed)}{\log(eN/c\log(ed))}\rfloor$ for some absolute constant $c>2$, and suppose 
$N\geq cs\log(ed)$, then, we have with probability at least $1-c'u^{-q}(ed)^{-(\frac{c}{2}-1)}$, for some absolute constant $c'>0$.
$$\max_{j\in\{1,2,\cdots,d\}}\l(\sum_{i>k}\l|\wt x_{ij}^{\sharp} \r|^{2r'}\r)^{1/2r'}
\leq Cu\nu_qN^{1/2r'},$$
for some constant absolute constant $C>0$ and $r'\in(\frac{q}{q-12},5]$.
\end{lemma}
\begin{proof}
First, following the same procedure as that of Lemma \ref{lem:supp-13} up to \eqref{eq:inter-1111}, with $p = q$, we have with probability at least $1-c'u^{-q}(ed)^{-c/2}$,
\[
\l(\sum_{i>k}\l|\wt x_{ij}^{\sharp}\r|^{2r'}\r)^{1/2r'}\leq \|\wt x_{ij}\|_{L_{q}}u\l(\sum_{i>k}\l(\frac{eN}{i}\r)^{3r'/q} \r)^{1/2r'}.
\]
Note that  $\|\wt x_{ij}\|_{L_{q}}\leq\| x_{ij}\|_{L_{q}}\leq\nu_q$ by the assumption and $r'\in(\frac{q}{q-12},5]$, thus, $3r'/q<1$ and we have with probability at least $1-c'u^{-q}(ed)^{-c/2}$,
\[
\l(\sum_{i>k}\l|\wt x_{ij}^{\sharp}\r|^{2r'}\r)^{1/2r'}\leq Cu\nu_qN^{1/2r'}.
\]
Finally, taking a union bound over all $j\in\{1,2,\cdots,d\}$ finishes the proof.
\end{proof}

Finally, substituting Lemma \ref{lem:supp-11},~\ref{lem:supp-12},~\ref{lem:supp-13},~\ref{lem:supp-14} into \eqref{eq:master-bound-2} with $r=5/4, r'=5$ gives with probability at least $1-e^{-\beta}-e^{-v^2}-c'\l(u^{-q}(ed)^{-(\frac{c}{2}-1)}+u^{-q/4}(ed)^{-c/2}+e^{-\frac{q}{10}}N^{-\frac{q}{10}+1}(\log(eN))^{q/5}w^{-q/5}\r)$,
\begin{multline}\label{eq:master-bound-3.5}
\max_{j\in\{1,2,\cdots,d\}}\l|\sum_{i=1}^N\varepsilon_i\wt \phi_i \wt x_{ij}  \r|\\ 
\leq CD_{\max}\l(\nu_q^5+\nu_q^{3}+\nu_q\r)w\l( \log(ed)\beta^{1/4} + N^{1/4}(\log(ed))^{3/4}\beta^{1/2} + vu^2\sqrt{N\log d}  \r).
\end{multline}

2. \textbf{Bounding the terms} $\max_{j\in\{1,2,\cdots,d\}} \l|\sum_{i=1}^N\varepsilon_i \l(y_i - g'(\dotp{\widetilde{\mf x}_i}{\theta_*})\r) \wt x_{ij}  \r|$: \\
The proving techniques in this part are essentially the same as those of the last part but with a slight change of exponents when applying Holder's inequality adapting to the moment condition of the term $y_i - g'(\dotp{\widetilde{\mf x}_i}{\theta_*})$. For simplicity of notations, let 
$$\xi_i: = y_i - g'(\dotp{\widetilde{\mf x}_i}{\theta_*}).$$
Similar as before, one can employ the inequality from \cite{montgomery1990distribution}, conditioned on $\wt{\mf{x}}_i$, which results in
\[
\l|\sum_{i=1}^N\varepsilon_i \xi_i \wt x_{ij}  \r|
\leq
\sum_{i=1}^k\l|\xi_i^{\sharp}\wt x_{ij}^{\sharp} \r|
+v\l( \sum_{i>k}\l|\xi_i^{\sharp}\wt x_{ij}^{\sharp} \r|^2 \r)^{1/2},
\]
with probability at least $1-e^{-v^2}$, where $k$ is any chosen integer within $\l\{0,1,2,\cdots,N \r\}$ and $\l(\xi_i^{\sharp}\r)_{i=1}^N$, $\l(\wt x_{ij}^{\sharp}\r)_{i=1}^N$ are non-increasing rearrangements of  $\l(|\xi_i|\r)_{i=1}^N$, $\l(|\wt x_{ij}|\r)_{i=1}^N$. We define the former sum to be 0 when $k=0$. 
By Holder's inequality, we have
\begin{equation*}
\l|\sum_{i=1}^N\varepsilon_i\xi_i \wt x_{ij}  \r|
\leq \l(\sum_{i=1}^k\l|\xi_i^{\sharp}\r|^4\r)^{1/4}\l(\sum_{i=1}^k\l|\wt x_{ij}^{\sharp} \r|^{4/3}\r)^{3/4}
+v\l( \sum_{i>k}\l|\xi_i^{\sharp} \r|^{2r} \r)^{1/(2r)} 
\l(\sum_{i>k}\l|\wt x_{ij}^{\sharp}\r|^{2r'}\r)^{1/(2r')},
\end{equation*}
for some positive exponents $r,r'$ such that $\frac1r+\frac{1}{r'}=1$. Take a union bound for all $j\in\{1,2,\cdots,d\}$, gives with probsability at least $1-e^{-v^2}$,
\begin{multline}\label{eq:master-bound-3}
\max_{j\in\{1,2,\cdots,d\}}\l|\sum_{i=1}^N\varepsilon_i\xi_i \wt x_{ij}  \r|
\leq \l(\sum_{i=1}^k\l|\xi_i^{\sharp}\r|^4\r)^{1/4}
\max_{j\in\{1,2,\cdots,d\}}\l(\sum_{i=1}^k\l|\wt x_{ij}^{\sharp} \r|^{4/3}\r)^{3/4}\\
+v\sqrt{\log d}\l( \sum_{i>k}\l|\xi_i^{\sharp} \r|^{2r} \r)^{1/(2r)} 
\max_{j\in\{1,2,\cdots,d\}}\l(\sum_{i>k}\l|\wt x_{ij}^{\sharp}\r|^{2r'}\r)^{1/(2r')},
\end{multline}
Again, our goal is to bound the four terms in \eqref{eq:master-bound-3} separately.

\begin{lemma}\label{lem:xi-1}
Let $k = \lfloor\frac{c\log(ed)}{\log(eN/c\log(ed))}\rfloor$ for some absolute constant $c>2$, and suppose 
$N\geq \|\theta_*\|_1^2\log(ed)$, then, we have
\[
\l(\sum_{i=1}^k\l|\xi_i^{\sharp}\r|^4\r)^{1/4}
\leq C\nu_{q'}w N^{1/4},
\]
with probability at least $1 - c' (eN)^{-\frac{q'}{4}+1}(\log(eN))^{\frac {q'}{2}}w^{-q'}$ for any $w>4$ and some absolute constant $C, c'>1$, where $\|\xi_i\|_{L_{q'}}\leq\nu_{q'}$ with $q'>5$ is defined in Assumption \ref{assumption:moment}.
\end{lemma}
\begin{proof}[Proof of Lemma \ref{lem:xi-1}]
First of all, by Markov inequality, 
\begin{multline*}
Pr\l(\l|\xi_i^{\sharp}\r|\geq c_i\|\xi\|_{L_{q'}}\r)\leq {N \choose i}Pr\l(\l|\xi_i\r|\geq c_k\|\xi\|_{L_{q'}}\r)^i\\
\leq \l(\frac{eN}{i}\r)^iPr\l(\l|\xi_i\r|\geq c_k\|\xi\|_{L_{q'}}\r)^i
\leq  \l(\frac{eN}{i}\r)^i \frac{\expect{\l|\xi_i\r|^{q'}}^i}{c_i^{q'i}\l\|\xi\r\|_{L_{q'}}^{q'i}}
=\l(\frac{eN}{i}\r)^ic_i^{-q'i}.
\end{multline*}
Choosing $c_i = w(eN/i)^{1/4}(\log(eN/i))^{1/2}$ gives
\[
Pr\l( \l|\xi_i^{\sharp}\r|\geq \l(\frac{eN}{i}\r)^{1/4}\frac{w}{(\log(eN/i))^{1/2}}\l\|\xi\r\|_{L_{q'}} \r)
\leq \l( \frac{i}{eN} \r)^{i(\frac{q'}{4}-1)}w^{-q'i}\l(\log\frac{eN}{i}\r)^{\frac{q'}{2}i}.
\]
Thus, it follows
\[
\sum_{i=1}^k\l|\xi_i^{\sharp}\r|^4\leq \sum_{i=1}^N\l|\xi_i^{\sharp}\r|^4\leq \sum_{i=1}^N\frac{eN}{i}\frac{w^4}{(\log(eN/i))^2}\l\|\xi_i\r\|_{L_{q'}} 
\leq Cw^4\l\|\xi_i\r\|_{L_{q'}}eN\leq Cw^4\nu_{q'}eN, 
\]
with probability at least 
\[
1 - \sum_{i=1}^N\l( \frac{i}{eN} \r)^{i(\frac{q'}{4}-1)}w^{-q'i}\l(\log\l(\frac{eN}{i}\r)\r)^{\frac{q'}{2}i}.
\]
Since for any $w>4$ and $q'>5$, the above summand is a geometrically decreasing sequence. Specifically, it is easy to show that 
$\l( \frac{i}{eN} \r)^{(\frac{q'}{4}-1)}w^{-q'}\l(\log\l(\frac{eN}{i}\r)\r)^{\frac{q'}{2}}<\l(4/\sqrt{10}\r)^{-q'},~\forall i\in\{1,2,\cdots,N\}$. 
Thus, it follows the probability is at least
\[
1- c'\l( eN \r)^{-(\frac{q'}{4}-1)}w^{-q'}\l(\log\l(eN\r)\r)^{\frac{q'}{2}}
\]
for some absolute constant $c'>0$.
\end{proof}

\begin{lemma}\label{lem:x-1}
Let $k = \lfloor\frac{c\log(ed)}{\log(eN/c\log(ed))}\rfloor$ for some absolute constant $c>2$, then, we have
$$\max_{j\in\{1,2,\cdots,d\}}\l(\sum_{i=1}^k\l|\wt x_{ij}^{\sharp} \r|^{4/3}\r)^{3/4}
\leq C\l(\nu_q^{4/3}\log(ed) + \nu_q^{4/3}\sqrt{\beta}\log(ed)+\l(\frac{N}{\log(ed)}\r)^{1/3}(\beta+\log(ed))\r)^{3/4},$$
with probability at least $1-e^{-\beta}$ for any $\beta>1$ and some constant $C>1$.
\end{lemma}
\begin{proof}[Proof of Lemma \ref{lem:x-1}]
First, for any set of $k$ random variables $x_{1j},~x_{2j},~\cdots,~x_{kj}$ we have by Bernstein's inequality,
\[
Pr\l( \sum_{i=1}^k\l|\wt x_{ij} \r|^{4/3}\geq k\expect{|\wt x_{ij}|^{4/3}} + C\l( \sqrt{2\sigma_2^2kt}+b_2t\r) \r)\leq e^{-t},
\]
for some constant $C$,
where $\sigma_2^2 := \expect{  \l(|\wt x_{ij}|^{4/3}-\expect{|\wt x_{ij}|^{4/3}} \r)^2  }\leq \expect{|x_{ij}|^{8/3}}\leq\nu_q^{8/3}$, $b_2:= \l( N/\log(ed) \r)^{1/3}$ and 
$\expect{|\wt x_{ij}|^{4/3}}\leq \expect{|x_{ij}|^{4/3}}\leq \nu_q^{4/3}$. Take a union bound over all ${N \choose k}$ different combinations from $x_{1j},~x_{2j},\cdots,~x_{Nj}$, we obtain,
\[
Pr\l( \sum_{i=1}^k\l|\wt x_{ij}^{\sharp} \r|^{4/3}\geq k\expect{|\wt x_{ij}|^{4/3}} + C\l(\sqrt{2\sigma_2^2kt}+b_2t \r) \r)\leq {N \choose k}e^{-t}\leq \l(\frac{eN}{k}\r)^ke^{-t}.
\]
Taking a union bound over all $j\in\{1,2,\cdots,d\}$, we get
\[
Pr\l( \max_{j\in\{1,2,\cdots,d\}}\sum_{i=1}^k\l|\wt x_{ij}^{\sharp} \r|^{4/3}\geq k\expect{|\wt x_{ij}|^{4/3}} + C\l( \sqrt{2\sigma_2^2kt}+b_2t \r) \r)\leq
d\l(\frac{eN}{k}\r)^ke^{-t}
\]
Substituting the definition of $k =\lfloor\frac{c\log(ed)}{\log(eN/\log(ed))}\rfloor\leq \frac{c\log(ed)}{\log(eN/\log(ed))}$, we get 
\begin{multline*}
d\l(\frac{eN}{k}\r)^ke^{-t} = \exp\l(-t + k\log(eN/k) + \log d\r)\\
\leq \exp\l(-t + \frac{c\log(ed)}{\log(eN/c\log(ed))}\log\l(\frac{eN}{c\log(ed)}\cdot \log\l(\frac{eN}{c\log(ed)}\r)\r)  + \log d\r)\\
\leq\exp(-t+(2c+1)\log(ed)).
\end{multline*}
Setting $\beta = t - (2c+1)\log(ed)$ and rearranging the terms gives the claim.
\end{proof}

\begin{lemma}\label{lem:xi-2}
Let $k = \lfloor\frac{c\log(ed)}{\log(eN/c\log(ed))}\rfloor$ for some absolute constant $c>2$, then, we have with probability at least $1-c'u^{-q'}(ed)^{-c/2}$, for some absolute constant $c'>0$,
\[
\l(\sum_{i> k}\l|\xi_i^{\sharp}\r|^{2r}\r)^{1/2r}\leq  Cu\nu_{q'} N^{1/2r},
\]
for $r\leq5/4$, any $u>2$, and some absolute constant $C>0$. 
\end{lemma}

\begin{proof}
Following from the same proof as that of Lemma \ref{lem:supp-13} up to \eqref{eq:inter-1111} with $p=q'$, we have with probability at least $1- c_0u^{-q'}(ed)^{-c/2}$,  
\begin{equation}
\l(\sum_{i>k}\l|\xi_i^{\sharp}\r|^{2r}\r)^{1/2r}\leq \|\xi\|_{L_{q'}}u\l(\sum_{i>k}\l(\frac{eN}{i}\r)^{3r/q'} \r)^{1/2r}.
\end{equation}
Since $q'>5\geq4r$ by assumption, it follows,
\[
\sum_{i>k}\l(\frac{1}{i}\r)^{3r/q'}\leq \int_0^N\l(\frac{1}{x}\r)^{3r/q'}dx = \frac{1}{1-3r/q'}N^{1-\frac{3r}{q'}},
\]
which implies the claim.
\end{proof}

Also, by Lemma \ref{lem:supp-14}, we have with probability at least $1-c'u^{-q}(ed)^{-(\frac{c}{2}-1)}$, for some absolute constant $c'>0$,
\begin{equation}\label{eq:x-2}
\max_{j\in\{1,2,\cdots,d\}}\l(\sum_{i>k}\l|\wt x_{ij}^{\sharp} \r|^{2r'}\r)^{1/2r'}
\leq Cu\nu_qN^{1/2r'},
\end{equation}
for some constant absolute constant $C>0$ and $r'\in(\frac{q}{q-12},5]$.

Overall, substituting Lemma \ref{lem:xi-1}, \ref{lem:x-1}, \ref{lem:xi-2}, and \eqref{eq:x-2} into \eqref{eq:master-bound-3} with $r = 5/4, r' = 5$ gives with probability at least 
$$1- e^{-\beta} - e^{-v^2} - c'\l( (eN)^{-(\frac{q'}{4}-1)}(\log(eN))^{q'/2}w^{-q'} +u^{-q}(ed)^{-(\frac{c}{2}-1)} + u^{-q'}(ed)^{-\frac{c}{2}} \r),$$
the following holds
\begin{multline}\label{eq:master-bound-4}
\max_{j\in\{1,2,\cdots,d\}}\l|\sum_{i=1}^N\varepsilon_i\xi_i \wt x_{ij}  \r|
\leq C\nu_{q'}\l( vu^2\nu_qN^{1/2}(\log(ed))^{1/2} + w\nu_q(\log(ed))^{3/4}N^{1/4} \r.\\
\l.+ w\nu_q\beta^{3/8}N^{1/4}(\log(ed))^{3/4} + w\beta^{3/4}N^{1/2}(\log(ed))^{1/2} \r)
\end{multline}


Overall, substituting the bounds \eqref{eq:master-bound-3.5} and \eqref{eq:master-bound-4} into \eqref{eq:decomp-1} gives
\begin{multline*}
N\cdot\max_{j\in\{1,2,\cdots,d\}}|z_j| \leq 
CD_{\max}\l(\nu_q^5+\nu_q^{3}+\nu_q\r)w\l( \log(ed)\beta^{1/4} + N^{1/4}(\log(ed))^{3/4}\beta^{1/2} 
+ vu^2\sqrt{N\log d}  \r) \\
+ C\nu_{q'}\l( \nu_q +1 \r)(vu^2+w+w\beta^{3/8}+w\beta^{3/4})\l(\sqrt{\beta N\log(ed)} + \beta N^{1/4}\l(\log(ed)\r)^{3/4} \r),
\end{multline*}
with probability at least 
\begin{multline*}
1-2e^{-\beta}-2e^{-v^2}
-c'\l(u^{-q}(ed)^{-(\frac c2-1)}+(u^{-q/4}+u^{-q'})(ed)^{-c/2}\r.\\
\l.+(eN)^{-\frac{q}{10}+1}(\log(eN))^{q/5}w^{-q/5} +  (eN)^{-(\frac{q'}{4}-1)}(\log(eN))^{q'/2}w^{-q'}\r).
\end{multline*}
This implies the claim when combining \eqref{eq:multi-bound-0} and the fact that $N\geq \log(ed)$.
\end{proof}

The following lemma gives a bound on $r_{\mathcal{M}}$ in terms of $\rho$.
\begin{lemma}\label{lem:bound-M}
Suppose $N\geq \|\theta_*\|_1^2\log(ed) + \log(ed)$, $\Lambda_M = \frac{\delta^2Q^2}{128}D_{\min}$ and Assumption \ref{assumption:moment}, \ref{assumption:link-function} hold.
For any $\beta,u,v,w>7$, we have 
\[
r_{M}^2\leq \frac{C(\nu_q, \nu_{q'},\kappa,\nu)(D_{\max} + 1)\l(wu^2v+w\beta^{3/4}\r)\rho}{D_{\min}}
\sqrt{\frac{\log (ed)}{N}},
\]
for any $m\in\{1,2,\cdots,d\}$,~
where $C(\nu_q, \nu_{q'},\kappa,\nu)$ depends polynomially on $\nu_q$, $\nu_{q'}$, $\kappa$ and $\nu$, when taking  
\begin{multline*}
p_M =2e^{-\beta} + 2e^{-v^2}
+c'\l(u^{-q}(ed)^{-(\frac c2-1)}+(u^{-q/4}+u^{-q'})(ed)^{-c/2}\r.\\
\l.+(eN)^{-\frac{q}{10}+1}(\log(eN))^{q/5}w^{-q/5} +  (eN)^{-(\frac{q'}{4}-1)}(\log(eN))^{q'/2}w^{-q'}\r),
\end{multline*}
where $c,c'>2$ are absolute constants.
\end{lemma}

\begin{proof}[Proof of Lemma \ref{lem:bound-M}]
Since $\Lambda_M = \frac{\delta^2Q^2}{128}D_{\min}$, the infimum of $r>0$ such that the right hand side of Lemma \ref{lem:bound-PM} is less than $\Lambda_M = \frac{\delta^2Q^2}{128}D_{\min}r^2$ can be achieved by setting the right hand side equal to $\Lambda_M = \frac{\delta^2Q^2}{128}D_{\min}r^2$, which gives,
\[
\Lambda_Mr^2 = \frac{\delta^2Q^2}{128}D_{\min} r^2= C(\nu_q, \nu_{q'})
(D_{\max} + 1)\l(wu^2v+w\beta^{3/4}\r)\rho
\sqrt{\frac{\log (ed)}{N}},
\]
which implies the claim.
\end{proof}

\subsection{Bounding the radius $r_{\mathcal{V}}$}
\begin{lemma}\label{lem:bound-V}
Suppose $N\geq \|\theta_*\|_1^2\log(ed)$ and $\Lambda_{\mc V} = D_{\min}\delta^2Q^2/128$, then, 
\[
r_{\mc V}^2\leq\frac{128D_{\max}\nu_q^6}{D_{\min}\delta^2Q^2}\rho\sqrt{\frac{\log(ed)}{N}},
\]
\end{lemma}

\begin{proof}[Proof of Lemma \ref{lem:bound-V}]
First of all,
\[
\sup_{\theta\in B_{2}(\theta_*,r)\cap B_1(\theta_*,\rho)}  
\l|\mc{V}_{\theta-\theta_*}\r| :=
\sup_{\mf v\in B_{2}(0,r)\cap B_1(0,\rho)} \expect{\l(y - g'(\dotp{\wt {\mf x}}{\theta_*})\r)\dotp{\wt{\mf{x}}}{\mf{v}}}. 
\]
For each $\mf v$, we have
\begin{multline*}
\expect{\l(y - g'(\dotp{\wt {\mf x}}{\theta_*})\r)\dotp{\wt{\mf{x}}}{\mf{v}}}
=
| \expect{\l(y-g'(\dotp{\mf x}{\theta_*})\r)\dotp{\wt{\mf{x}}}{\mf{v}}}| +
|\expect{(g'(\dotp{\mf x}{\theta_*}) - g'(\dotp{\wt{\mf x}}{\theta_*}))\dotp{\wt{\mf{x}}}{\mf{v}}}|\\
\leq 
\rho\|\expect{(g'(\dotp{\mf x}{\theta_*}) - g'(\dotp{\wt{\mf x}}{\theta_*}))\wt{\mf{x}}}\|_{\infty},
\end{multline*}
where we use the fact that the conditional expectation
\[
\expect{y-g'(\dotp{\mf x}{\theta_*}) ~|~\mf x} = 0.
\]
Note that for any $j\in\{1,2,\cdots,d\}$, by Cauchy-Schwarz inequality,
\begin{align*}
|\expect{(g'(\dotp{\mf x}{\theta_*}) - g'(\dotp{\wt{\mf x}}{\theta_*}))\wt{x}_j}|
\leq&
\expect{(g'(\dotp{\mf x}{\theta_*}) - g'(\dotp{\wt{\mf x}}{\theta_*}))^2}^{1/2}\expect{\wt{x}_j^2}^{1/2}\\
\leq&D_{\max}\expect{(\dotp{\mf x}{\theta_*} - \dotp{\wt{\mf x}}{\theta_*})^2}^{1/2}\expect{\wt{x}_j^2}^{1/2}\\
=&D_{\max}\expect{\l(\sum_{i=1}^d(x_i - \wt{x}_i)\theta_{*,i}\r)^2}^{1/2}\expect{\wt{x}_j^2}^{1/2}\\
\leq&
D_{\max}\sum_{i=1}^d\expect{(x_i - \wt{x}_i)^2}^{1/2}|\theta_{*,i}| \cdot\expect{\wt{x}_j^2}^{1/2}\\
\leq&
D_{\max}\|\theta_*\|_1\max_{i}\expect{(x_i - \wt{x}_i)^2}^{1/2} \cdot\expect{\wt{x}_j^2}^{1/2},
\end{align*}
where the second inequality follows from Assumption \ref{assumption:link-function}, and the third inequality follows from Minkowski's inequality.  Now, for each $i$, we have
\[
\expect{(x_i - \wt{x}_i)^2}^{1/2}\leq \expect{x_i^21_{\{|x_i|\geq\tau\}}}^{1/2}
\leq \expect{x_i^{10}}^{1/10}Pr(|x_i|\geq\tau)^{2/5},
\]
where 
\[
Pr(|x_i|\geq\tau)\leq \frac{\expect{|x_i|^{10}}}{\tau^{-10}}\leq \expect{|x_i|^{10}}\l(\frac{\log ed}{N}\r)^{5/2}.
\]
Thus,
\[
\expect{(x_i - \wt{x}_i)^2}^{1/2}\leq \expect{x_i^{10}}^{1/2}\frac{\log ed}{N}\leq \nu_q^5\frac{\log ed}{N},
\]
and we have
\[
|\expect{(g'(\dotp{\mf x}{\theta_*}) - g'(\dotp{\wt{\mf x}}{\theta_*}))\wt{x}_j}|
\leq D_{\max}\|\theta_*\|_1\nu_q^6\frac{\log ed}{N}\leq D_{\max}\nu_q^6\sqrt{\frac{\log ed}{N}},
\]
where we use the assumption that $N\geq \|\theta_*\|_1^2\log ed$.
Overall, we get 
\[
\sup_{\theta\in B_{2}(\theta_*,r)\cap B_\Psi(\theta_*,\rho)}  
\l|\mc{V}_{\theta-\theta_*}\r|
\leq \nu_q^6D_{\max}\rho\sqrt{\frac{\log(ed)}{N}}
\]
Since $\Lambda_{\mc V}=\frac{\delta^2Q^2}{128}D_{\min}$, let
\[
\frac{\delta^2Q^2}{128}D_{\min}r^2 =  \nu_q^6D_{\max}\rho\sqrt{\frac{\log(ed)}{N}},
\]
which results in 
$$r^2=\frac{128D_{\max}\nu_q^6}{D_{\min}\delta^2Q^2}\rho\sqrt{\frac{\log(ed)}{N}},$$ 
and $r_{\mc V}^2$ must be bounded above by this value.
\end{proof}

\subsection{Putting everything together}
\begin{proof}[Proof of Theorem \ref{thm:sparse-recovery}]
We choose $\Lambda_Q = \frac{\delta^2Q^2}{32}D_{\min}$, $\Lambda_M = \frac{\delta^2Q^2}{128}D_{\min}$ and
$\Lambda_V = \frac{\delta^2Q^2}{128}D_{\min}$. Then, $\Lambda_Q> \Lambda_M + \Lambda_{\mc V}$. By 
Theorem \ref{thm:r-Q-1} and \ref{thm:r-Q-2},
\[
r_Q^2\leq C_1(\nu_q,\nu_{q'},\nu,\kappa)\beta\frac{\|\theta_*\|_1^2\log ed}{N}.
\]
with $p_Q = c_1e^{-\beta}$, when $N\geq 1024\frac{\beta}{Q^2} + \beta^2\frac{\nu+1}{Q^2}\log d$.
By Lemma \ref{lem:bound-M}, we have
\[
r_{M}^2\leq \frac{C_2(\nu_q, \nu_{q'},\kappa,\nu)(D_{\max} + 1)\l(wu^2v+w\beta^{3/4}\r)\rho}{D_{\min}}
\sqrt{\frac{\log (ed)}{N}},
\]
with 
\begin{multline*}
p_M =2e^{-\beta} + 2e^{-v^2}
+c'\l(u^{-q}(ed)^{-(\frac c2-1)}+(u^{-q/4}+u^{-q'})(ed)^{-c/2}\r.\\
\l.+(eN)^{-\frac{q}{10}+1}(\log(eN))^{q/5}w^{-q/5} +  (eN)^{-(\frac{q'}{4}-1)}(\log(eN))^{q'/2}w^{-q'}\r),
\end{multline*}
when $N\geq \|\theta_*\|_1^2\log(ed) + \log(ed)$. Finally, by Lemma \ref{lem:bound-V},
\[
r_{\mc V}^2\leq \frac{128D_{\max}\nu_q^6}{D_{\min}\delta^2Q^2}\rho\sqrt{\frac{\log(ed)}{N}}
\]
when $N\geq \|\theta_*\|_1^2\log(ed)$. Thus, when $N\geq c_0(\|\theta_*\|_1^2\log(ed) + \log(ed)+\frac{\beta}{Q^2})$ for some absolute constant $c_0$,
\[
r(\rho)^2\leq  \frac{C_3(\nu_q, \nu_{q'},\kappa,\nu)(D_{\max} + 1)\l(wu^2v+w\beta^{3/4}+\beta\r)\rho}{D_{\min}}
\]

Now, we choose $\rho = c\|\theta_*\|_1$ for some $c>4$ and 
\[
\lambda \geq \frac{C_3(\nu_q, \nu_{q'},\kappa,\nu)(D_{\max} + 1)\l(wu^2v+w\beta^{3/4}+\beta\r)}{D_{\min}}
\sqrt{\frac{\log (ed)}{N}},
\]
By Theorem \ref{thm:master}, we have the estimator satisfies
\[
\|\widehat{\theta}_N - \theta_*\|_2^2\leq 
\frac{C_3(\nu_q, \nu_{q'},\kappa,\nu)(D_{\max} + 1)\l(wu^2v+w\beta^{3/4}+\beta\r)}{D_{\min}}\|\theta_*\|_1
\sqrt{\frac{\log (ed)}{N}},
\]
and we finish the proof.
\end{proof}

\section{Proof of Theorem \ref{thm:sparse-recovery-2}: Computing Local Complexities}
In this section, we prove Theorem \ref{thm:sparse-recovery-2} in a similar manner as that of Theorem \ref{thm:sparse-recovery}. Building upon previous intermediate results, the proof will be relatively simpler.
\subsection{Bounding radius $r_Q$}

\begin{lemma}\label{lem:counting-number-2}
Let $u\geq1$, $N\geq 1024u/Q^2 + (\frac{\nu}{Q}+ 64L^2) s_0\log ed $ where $L>0$ is the absolute constant defined in Lemma \ref{lem:VC}, then, with probability at least $1-ce^{-u}$ for some absolute constant $c>0$, there exists a set of indices $\mathcal I\in\{1,2,\cdots,N\}$ such that $|\mathcal I|\geq\frac{Q}{4}N$ and for any $i\in \mathcal I$, $\forall \mf v_1 \in \mathcal G_{s_0}\cap S_2(1),~\forall \mf v_2 \in S_1(\rho)$,
\[
|\dotp{\widetilde{\mathbf{x}}_i}{\mf v_1}|\geq\delta/2,
~~|\dotp{\widetilde{\mathbf{x}}_i}{\mf v_2}|\leq 32(\nu_q^2+\nu_q+1)\rho/Q,
~~|\dotp{\widetilde{\mathbf{x}}_i}{\theta_*}|\leq 32\nu_q\|\theta_*\|_1/Q.
\]
\end{lemma}

\begin{proof}[Proof of Lemma \ref{lem:counting-number-2}]
First of all, by Lemma \ref{lem:VC}, $N\geq 1024u/Q^2 + (\frac{\nu}{Q}+ 64L^2) s_0\log ed $, we have with probability at least $1-e^{-u}$,
\[
\inf_{\mathbf{v}\in \mathcal{G}_{s_0}\cap S_2(1)}\frac1N\sum_{i=1}^N\mathbf{1}_{  \l\{ \l|\dotp{\widetilde{\mathbf{x}}_i}{\mathbf{v}}\r|\r\}} 
\geq \frac{Q}{4}.
\]
On the other hand, by Lemma \ref{lem:bound-upper-1} and \ref{lem:bound-upper-2}, we have 
\[
\inf_{\mf v\in  B_\Psi(0,\rho)}\frac1N\sum_{i=1}^N1_{\l\{|\dotp{\mf v}{\widetilde {\mf x}_i}|\geq 32(\nu_q^2+\nu_q+1)\rho/Q\r\}}
\geq1- \frac{Q}{16},
\]
and
 \[
\frac1N\sum_{i=1}^N1_{\l\{|\dotp{\theta_*}{\widetilde {\mf x}_i}|\leq 32\nu_q\|\theta_*\|_1/Q\r\}}
\geq 1- \frac{Q}{16},
 \]
 with probability at least $1-2e^{-u}$. Combining the above three bounds, we have there exists a set of indices $\mathcal I\subseteq\{1,2,\cdots,N\}$ of cardinality at least $\frac{Q}{2} - \frac{Q}{16} - \frac{Q}{16} >\frac{Q}{4}$ such that the claim in the lemma holds.
\end{proof}

\begin{theorem}\label{thm:bound-Q}
Suppose $N\geq C_0\l( \frac{s_0}{Q^2} + \frac{\nu+1}{\nu} \r)\beta^2\log(ed)+ \frac{\nu}{Q}s_0\log(ed)$ for some absolute constant $C_0>0$, and $s_0 = \frac{\sqrt{\nu}}{\delta^2Q^2}s\leq d$ and $\Lambda_Q = D_{\min}\delta^2Q^2/16$, then,
\[
r_{\mathcal{Q}}\leq\frac{8}{\sqrt s}\rho
\]
when taking $p_{\mathcal{Q}} = ce^{-\beta}$ in the definition of $r_{\mathcal{Q}}$ for $\beta\geq1$.
\end{theorem}

\begin{proof}[Proof of Theorem \ref{thm:bound-Q}]
First of all, recall that $\wt{\mf{\Gamma}} := \l[\wt{\mf x}_{i_1},~\wt{\mf x}_{i_2},\cdots,~\wt{\mf x}_{i_{|\mathcal I|}}\r]^T/\sqrt{N}$.
By Lemma \ref{lem:counting-number-2}, and the assumption $N\geq C_0 \frac{s_0}{Q^2}\beta^2\log(ed)+ \frac{\nu}{Q}s_0\log(ed)$ for some large enough absolute constant $C_0$, we have 
$$\inf_{\mf v\in \mc G_{s_0}\cap \mc S^{d-1}}\l\| \widetilde{\bold{\Gamma}}\mathbf{v} \r\|_2^2\geq\frac{\delta^2Q^2}{8},$$
with probability at least $1-e^{-\beta}$.
Thus, it follows from Lemma \ref{lem:quad-form} and \ref{lem:upper-1} that
\[
\inf_{\theta\in B_1(\theta_*,\rho)\cap S_2(\theta_*,r)}\mc P_N\mc Q_{\theta-\theta_*}\geq 
D_{\min}\l(\frac{\delta^2Q^2}{8}r^2 - \frac{\rho^2}{s_0-1}
\l( \sqrt\nu + C\l(\sqrt{\nu}+1\r)\beta\sqrt{\frac{\log (ed)}{N}} \r)\r),
\]
where $C>0$ is an absolute constant.
By assumption that $N\geq C_0\frac{\nu+1}{\nu}\beta^2\log (ed)$ for some $C_0$ large enough, then,
\[
\inf_{\theta\in B_1(\theta_*,\rho)\cap S_2(\theta_*,r)}\mc P_N\mc Q_{\theta-\theta_*}
\geq D_{\min}\l(\frac{\delta^2Q^2}{8}r^2 - \frac{2\sqrt{\nu}}{s_0-1}\rho^2\r)\geq D_{\min}
\l(\frac{\delta^2Q^2}{8}r^2 - \frac{4\sqrt{\nu}}{s_0}\rho^2\r).
\]
Using the assumption that $s_0 = \frac{\sqrt{\nu}}{\delta^2Q^2}s$, we obtain 
\[
\inf_{\theta\in B_1(\theta_*,\rho)\cap S_2(\theta_*,r)}\mc P_N\mc Q_{\theta-\theta_*}\geq \frac{\delta^2Q^2D_{\min}}{8}\l(r^2 - \frac{32\rho^2}{s}\r).
\]
The infimum of $r>0$ such that the right hand side is greater than $\frac{\delta^2Q^2D_{\min}}{16}r^2$ can be obtained by 
letting the right hand side equal to $\frac{\delta^2Q^2D_{\min}}{16}r^2$ and solve for $r$, which gives $r = \frac{8}{\sqrt s}\rho$. It then follows from the definition of $r_{\mathcal{Q}}$ that $r_{\mathcal{Q}}$ must be bounded above by this value.
\end{proof}

\subsection{Bounding $r_M$ and $r_{\mathcal{V}}$}
The main objective is the following bound on $|\mc P_N \mc M_{\theta-\theta_*}|$:

\begin{lemma}\label{lem:bound-PM-2}
Suppose $N\geq cs\log(ed)$ for some absolute constant $c>1$ and Assumption \ref{assumption:moment}, \ref{assumption:link-function} and \ref{assumption:sparse} hold.
For any $\beta,u,v,w>7$, we have with probability at least 
\begin{multline*}
1-2e^{-\beta}-2e^{-v^2}
-c'\l(u^{-q}(ed)^{-(\frac c2-1)}+(u^{-q/4}+u^{-q'})(ed)^{-c/2}\r.\\
\l.+(eN)^{-\frac{q}{10}+1}(\log(eN))^{q/5}w^{-q/5} +  (eN)^{-(\frac{q'}{4}-1)}(\log(eN))^{q'/2}w^{-q'}\r).
\end{multline*}
where $c,c'>2$ are absolute constants, 
\[
\sup_{\theta\in B_1(\theta_*,\rho)\cap B_2(\theta_*,r)}\l| \mc P_N \mc M_{\theta-\theta_*} \r|\leq C(\nu_q, \nu_{q'})
(D_{\max} + 1)\l(wu^2v+w\beta^{3/4}+\beta\r)(r\sqrt{m} + \rho)
\sqrt{\frac{\log (ed)}{N}},
\]
for any $m\in\{1,2,\cdots,d\}$,~
where $C(\nu_q, \nu_{q'})$ depends polynomially on $\nu_q$ and $\nu_{q'}$.
\end{lemma}

\begin{proof}[Proof of Lemma \ref{lem:bound-PM-2}]
First of all, by symmetrization inequality, it is enough to bound 
\[
\sup_{\theta\in B_1(\theta_*,\rho)\cap B_2(\theta_*,r)}
\l| \frac1N\sum_{i=1}^N\varepsilon_i( y_i - g'(\dotp{\wt{\mf x_i}}{\theta_*}))\dotp{\wt{\mf x}_i}{\theta-\theta_*} \r| 
= \sup_{\mf v\in B_1(0,\rho)\cap B_2(0,r)}
\l| \frac1N\sum_{i=1}^N\varepsilon_i( y_i - g'(\dotp{\wt{\mf x}_i}{\theta_*}))\dotp{\wt{\mf x}_i}{\mf v} \r|
\]
We define $\mf z := \frac1N\sum_{i=1}^N\varepsilon_i(y_i - g'(\dotp{\wt{\mf x}_i}{\theta_*}))\wt{\mf x}_i$. Let $J$ be any group of coordinates in $\{1,2,\cdots,d\}$ with $m$ largest coordinates of $\l\{|z_j|\r\}_{j=1}^N$ for $m\in\{1,2,\cdots,d\}$. Then, it follows
\begin{multline}\label{eq:master-bound-0}
\sup_{\mf v\in B_1(0,\rho)\cap B_2(0,r)}
\dotp{\mf z}{\mf v} \leq \sup_{\mf v\in B_1(0,\rho)\cap B_2(0,r)}\sum_{j\in J}v_jz_j + \sup_{\mf v\in B_1(0,\rho)\cap B_2(0,r)}\sum_{j\in J^c}v_jz_j \\
\leq  \sup_{\mf v\in B_2(0,r)}\sum_{j\in J}v_jz_j +  \sup_{\mf v\in B_1(0,\rho)}\sum_{j\in J^c}v_jz_j
= r \cdot \l(\sum_{j\leq m}\l(z_j^{\sharp}\r)^2 \r)^{1/2} + \rho\cdot \max_{j>m}\l| z_j^\sharp \r|\\
\leq  \max_{j}\l| z_j \r|\cdot \l( r\sqrt m  + \rho \r)
\end{multline}
for any $m$, 
where $\l\{ z_j^{\sharp} \r\}_{j=1}^d$ denotes the non-increasing ordering of $\l\{ |z_j| \r\}_{j=1}^d$. Now for each $|z_j|$, let 
$\xi_i = y_i - g'(\dotp{\mf x_i}{\theta_*})$, 
\begin{equation*}
N|z_j| = \l|\sum_{i=1}^N\varepsilon_i( \wt y_i - g'(\dotp{\wt{\mf x}_i}{\theta_*}))\wt x_{ij}  \r|
\leq  
\l|\sum_{i=1}^N\varepsilon_i \xi_i \wt x_{ij}  \r| + \l|\sum_{i=1}^N\varepsilon_i(g'(\dotp{\wt{\mf x}_i}{\theta_*}) - g'(\dotp{\mf x_i}{\theta_*})) \wt x_{ij}  \r|
\end{equation*}
Thus, it follows 
\begin{multline}\label{eq:decomp-1}
N\cdot\max_{j\in\{1,2,\cdots,d\}}|z_j| \leq
\max_{j\in\{1,2,\cdots,d\}} \l|\sum_{i=1}^N\varepsilon_i \xi_i \wt x_{ij}  \r| 
+ \max_{j\in\{1,2,\cdots,d\}}\l|\sum_{i=1}^N\varepsilon_i(g'(\dotp{\wt{\mf x}_i}{\theta_*}) - g'(\dotp{\mf x_i}{\theta_*})) \wt x_{ij}  \r|
\end{multline}
By the same analysis as that of Lemma \ref{lem:bound-PM}, we obtain
\begin{multline*}
N\cdot\max_{j\in\{1,2,\cdots,d\}}|z_j| \leq 
C(D_{\max}+1)\l(\nu_q^5+\nu_q^{3}+\nu_q\r)w\l( \log(ed)\beta^{1/4} + N^{1/4}(\log(ed))^{3/4}\beta^{1/2} 
+ vu^2\sqrt{N\log d}  \r) \\
+ C\nu_{q'}\l( \nu_q +1 \r)(vu^2+w+w\beta^{3/8}+w\beta^{3/4})\l(\sqrt{\beta N\log(ed)} + \beta N^{1/4}\l(\log(ed)\r)^{3/4} \r),
\end{multline*}
with probability at least 
\begin{multline*}
1-2e^{-\beta}-2e^{-v^2}
-c'\l(u^{-q}(ed)^{-(\frac c2-1)}+(u^{-q/4}+u^{-q'})(ed)^{-c/2}\r.\\
\l.+(eN)^{-\frac{q}{10}+1}(\log(eN))^{q/5}w^{-q/5} +  (eN)^{-(\frac{q'}{4}-1)}(\log(eN))^{q'/2}w^{-q'}\r).
\end{multline*}
This implies the claim when combining with \eqref{eq:master-bound-0}.
\end{proof}

\begin{lemma}\label{lem:bound-M-2}
Suppose $N\geq cs\log(ed)$ for some absolute constant $c>1$, Assumption \ref{assumption:moment}, \ref{assumption:link-function} and \ref{assumption:sparse} hold and $\Lambda_M=\delta^2Q^2D_{\min}/128$.
Then, we have
\[
r_{\mathcal{M}}\leq C(\nu_q,\nu_{q'})\frac{D_{\max}+1}{D_{\min}}\l( \frac{wu^2v+w\beta^{3/4}+\beta}{\delta^2Q^2}\sqrt{\frac{s\log(ed)}{N}} + \sqrt{\rho\frac{wu^2v+w\beta^{3/4}+\beta}{\delta^2Q^2}}\l(\frac{s\log(ed)}{N}\r)^{1/4} \r),
\]
when taking
\begin{multline*}
p_{\mathcal{M}} = 2e^{-\beta}-2e^{-v^2}
-c'\l(u^{-q}(ed)^{-(\frac c2-1)}+(u^{-q/4}+u^{-q'})(ed)^{-c/2}\r.\\
\l.+(eN)^{-\frac{q}{10}+1}(\log(eN))^{q/5}w^{-q/5} +  (eN)^{-(\frac{q'}{4}-1)}(\log(eN))^{q'/2}w^{-q'}\r),
\end{multline*}
for some absolute constant $c'>1$ and any $\beta,u,v,w>7$,
where $C(\nu_q, \nu_{q'})$ depends polynomially on $\nu_q$ and $\nu_{q'}$.
\end{lemma}

\begin{proof}[Proof of Lemma \ref{lem:bound-M-2}]
Since $\Lambda_M =\delta^2Q^2D_{\min}/128$, let $m=s$ in Lemma \ref{lem:bound-PM-2} and the infimum of the $r>0$ such that the right hand side of Lemma \ref{lem:bound-PM-2} is less than $\delta^2Q^2D_{\min}r^2/128$ can be achieved by setting the right hand side equal to $\delta^2Q^2D_{\min}r^2/128$, which gives,
\[
\frac{\delta^2Q^2}{128}r^2 = C(\nu_q, \nu_{q'})
(D_{\max} + 1)\l(wu^2v+w\beta^{3/4}+\beta\r)(r\sqrt{m} + \rho)
\sqrt{\frac{\log (ed)}{N}}.
\]
Solving the above quadratic equation gives
\[
r = C(\nu_q,\nu_{q'})\frac{D_{\max}+1}{D_{\min}}\l( \frac{wu^2v+w\beta^{3/4}+\beta}{\delta^2Q^2}\sqrt{\frac{s\log(ed)}{N}} + \sqrt{\rho\frac{wu^2v+w\beta^{3/4}+\beta}{\delta^2Q^2}}\l(\frac{s\log(ed)}{N}\r)^{1/4} \r).
\]
Thus, the defined $r_{\mathcal{M}}$ must be bounded above by this value and the lemma is proved.
\end{proof}

\begin{lemma}\label{lem:bound-V-2}
Suppose $N\geq s\log(ed)$ and $\Lambda_{\mc V}=D_{\min}\delta^2Q^2/128$, then, 
\[
r_{\mc V}^2\leq\frac{128D_{\max}\nu_q^6}{D_{\min}\delta^2Q^2}\rho\sqrt{\frac{\log(ed)}{N}}.
\]
\end{lemma}
The proof is the same as that of Lemma \ref{lem:bound-V}.

\subsection{Putting everything together}
\begin{lemma}\label{lem:sparse-eq}
Suppose $\|\theta_* - \theta_0\|_1\leq\rho/4$, where $\theta_0$ an $s$-sparse vector and $\rho\geq 8r(\rho)\sqrt{s}$, then, 
$\Delta(\eta\theta_*,\rho)\geq 3\rho/4$.
\end{lemma}

\begin{proof}[Proof of Lemma \ref{lem:sparse-eq}]
Let $\mc G_s$ be the set of nonzero coordinates of $\theta_0$, then, for any vector $\mf v\in B_{2}(0,r)\cap S_\Psi(0,\rho)$, we have
$\mf v = \mc P_{\mc G_s}\mf v + \mc P_{\mc G_s^c}\mf v$ and since $\|\theta_* - \theta_0\|_1\leq\rho/4$, by definition of $\Gamma_{\Psi}(\theta_*,\rho)$ in \eqref{eq:sub-diff},
 there exists a sub-differential $\mf z^*\in\Gamma_{\Psi}(\theta_*,\rho)$ such that $\dotp{\mf z^*}{\theta_0} = \|\theta_0\|_1$ and
$\dotp{\mf z^*}{\mc P_{\mc G_s^c}\mf v} = \|\mc P_{\mc G_s^c}\mf v\|_1$. Thus, it follows,
\begin{multline*}
\dotp{\mf z^*}{\mf v} = \dotp{\mf z^*}{\mc P_{\mc G_s}\mf v} + \dotp{\mf z^*}{\mc P_{\mc G_s^c}\mf v} \geq  \|\mc P_{\mc G_s^c}\mf v\|_1 - \|\mc P_{\mc G_s}\mf v\|_1\\
\geq \|\mf v\|_1 - 2\|\mc P_{\mc G_s}\mf v\|_1 \geq \rho - 2\sqrt{s}\|\mc P_{\mc G_s}\mf v\|_2
\geq  \rho - 2r(\rho)\sqrt{s},
\end{multline*}
where the second from the last inequality follows from $\mf v\in B_{2}(0,r)\cap S_\Psi(0,\rho)$ that $\|\mf v\|_1= \rho$ and the last inequality follows from $\|\mc P_{\mc G_s}\mf v\|_2\leq \|\mf v\|_2\leq r(\rho)$. The above bound is greater than $3\rho/4$ when $\rho\geq 8r(\rho)\sqrt{s}$.
\end{proof}

Finally, we are ready to prove the main theorem .
\begin{proof}[Proof of Theorem \ref{thm:sparse-recovery-2}]
We set $\Lambda_Q = D_{\min}\delta^2Q^2/32$, $\lambda_M = D_{\min}\delta^2Q^2/128$ and 
$\lambda_{\mc V} = D_{\min}\delta^2Q^2/128$. 
By Theorem \ref{thm:bound-Q}, Lemma \ref{lem:bound-PM-2} and \ref{lem:bound-V-2}, we have 
\begin{multline*}
r(\rho)\leq \frac{8}{\sqrt s}\rho + \sqrt{\frac{128D_{\max}\nu_q^6}{D_{\min}\delta^2Q^2}}\rho^{1/2}\l(\frac{\log(ed)}{N}\r)^{1/4} + \\
C(\nu_q,\nu_{q'})\frac{D_{\max}+1}{D_{\min}}\l( \frac{wu^2v+w\beta^{3/4}+\beta}{\delta^2Q^2}\sqrt{\frac{s\log(ed)}{N}} + \sqrt{\rho\frac{wu^2v+w\beta^{3/4}+\beta}{\delta^2Q^2}}\l(\frac{s\log(ed)}{N}\r)^{1/4} \r)
\end{multline*}
By Lemma \ref{lem:sparse-eq}, the condition $\Delta(\eta\theta_*,\rho)\geq 3\rho/4$ is satisfied for any $\rho\geq 8r(\rho)\sqrt s$. Take equality in the above bound and choose $\rho$ to be
\[
\rho = C(\nu_q,\nu_{q'})\frac{D_{\max}+1}{D_{\min}}\frac{wu^2v+w\beta^{3/4}+\beta}{\delta^2Q^2}s\sqrt{\frac{\log(ed)}{N}},
\] 
where $C(\nu_q,\nu_{q'})$ depends polynomially on$\nu_q,\nu_{q'}$ and $\rho\geq 8r(\rho)\sqrt s$ is satisfied. This implies
\[
r(\rho)\leq   C'(\nu_q,\nu_{q'})\frac{D_{\max}+1}{D_{\min}}\frac{wu^2v+w\beta^{3/4}+\beta}{\delta^2Q^2}\sqrt{\frac{s\log(ed)}{N}}.
\]
Taking 
\[
\lambda = C''(\nu_q,\nu_{q'})\frac{D_{\max}+1}{D_{\min}}\frac{wu^2v+w\beta^{3/4}+\beta}{\delta^2Q^2}\sqrt{\frac{\log(ed)}{N}}
\]
finishes the proof.
\end{proof}



\chapter{Structured Recovery from Non-linear and Heavy-tailed Measurements}
In this chapter, we show that when the design vectors are selected from a specific class of distributions, then, one can simultaneously relax the moment condition as well as treat more general structured problems. 
We study high-dimensional signal recovery from non-linear measurements with design vectors having elliptically symmetric distribution. 
Special attention is devoted to the situation when the unknown signal belongs to a set of low statistical complexity, while both the measurements and the design vectors are heavy-tailed. 
We propose and analyze a new estimator that adapts to the structure of the problem, while being robust both to the possible model misspecification characterized by arbitrary non-linearity of the measurements as well as to data corruption modeled by the heavy-tailed distributions. Moreover, this estimator has low computational complexity. Theoretically, our results are expressed in the form of exponential concentration inequalities relying on an improved generic chaining method. 
Numerically, we conduct simulation experiments demonstrating that our estimator outperforms existing alternatives when data is heavy-tailed. 

\section{Introduction}
In many practical settings, exact measurements from linear models or GLMs \eqref{eq:glms} are not available. Instead, the data one observes are often subject to unknown distortions such as quantization and hard thresholding. Furthermore, one might not even know the exact model \eqref{eq:glms}. Is it possible to perform faithful parameter estimation in these imperfect scenarios?
This chapter treats this problem with a more general setup than that of \eqref{eq:glms}.
Instead of adopting a specific model, we assume the link function is unknown. More specifically,
let $(\mf{x},y)\in \mb R^d\times \mb R$ be a random couple satisfying the \textit{semi-parametric single index model}
\begin{equation}
\label{model}
y=f(\langle\mathbf{x},\theta_*\rangle,\delta),
\end{equation}
where $\mf{x}$ is a measurement vector with marginal distribution $\Pi$, $\delta$ is a noise variable that is assumed to be independent of $\mf{x}$, $\theta_\ast \in \mb R^d$ is a fixed but otherwise unknown signal (``index vector''), and $f:\mathbb{R}^2\mapsto\mathbb{R}$ is an unknown link function; here and in what follows, $\dotp{\cdot}{\cdot}$ denotes the Euclidean dot product.  
We impose no explicit conditions on $f$, and in particular it is not assumed that $f$ is convex, or even continuous.
Our goal is to estimate the signal $\theta_\ast$ from a sequence of samples $(\mf{x}_1,y_1),\ldots,(\mf{x}_N,y_N)$ which are copies of $(\mf{x},y)$. 
As $f(a^{-1}\langle {\bf x},a \theta_*\rangle, \delta)=f(\langle {\bf x},\theta_*\rangle, \delta)$ for any $a>0$, the best one can hope for is to recover 
$\theta_*$ up to a scaling factor. 
Hence, without loss of generality, we will assume that $\theta_\ast$ satisfies 
$\|\mf\Sigma^{1/2}\theta_\ast\|^2_2:=\dotp{\mf\Sigma^{1/2}\theta_\ast}{\mf\Sigma^{1/2}\theta_\ast}=1$, where $\mf\Sigma=\mb E(\mf x-\mb E\mf x)(\mf x-\mb E\mf x)^T$ is the covariance matrix of $\mf x$.
Instead of being sparse or approximately sparse, in this chapter, we will assume that $\theta_*$ is an element of a closed set 
$\Theta\subseteq\mathbb{R}^d$ of small statistical complexity that is characterized by its Gaussian mean width. 


 Due to the ambiguity of $f$, such a task can easily fail regardless of the algorithms \cite{ai2014one}. As an example, consider the model $y_i=\sign(\dotp{\mf x_i}{\theta_*})$. Consider two sparse 
 vectors: $\theta_1 = [1,0,0,~\cdots,0]$, $\theta_2 = [1,-0.5,0,~\cdots,0]$, and i.i.d. Bernoulli design vectors $\mf x_i$, where each entry takes +1 and -1 with equal probabilities. It is obvious that 
 for $\theta_*=\theta_1$ and $\theta_*=\theta_2$, the responses $y_i$ are identical and the model cannot distinguish between $\theta_1$ and $\theta_2$.
Thus, one has to pose extra assumptions on the design vector itself so that the problem is well-defined. 

Generally, the task of estimating the index vector requires approximating the link function $f$ \cite{hardle1993optimal} or its derivative, assuming that it exists (the so-called Average Derivative Method), see \cite{stoker1986consistent,hristache2001direct}. 
However, when the measurement vector $\mf{x}$ is Gaussian, a somewhat surprising result states that one can estimate $\theta_\ast$ directly, avoiding preliminary link function estimation step completely.  
More specifically, \cite{Brillinger1983A-generalized-l00} proved that 
$\eta\theta_\ast = \argmin_{\theta\in\mb R^d}\mb E\l( y - \dotp{\theta}{\mf{x}}\r)^2$, where $\eta = \mb E\dotp{y\mf{x}}{\theta_\ast}$. 
Later, \cite{li1989regression} extended this result to the more general case of elliptically symmetric distributions, which includes the Gaussian as a special case; see Lemma \ref{lemma:mean-consistency}.

Our work was partly inspired by the work of 
Y. Plan, R. Vershynin and E. Yudovina \cite{plan2014high,plan2016generalized}, who presented the non-asymptotic study for the case of Gaussian measurements in the context of high-dimensional structured estimation; also, see \cite{genzel2016high,ai2014one,thrampoulidis2015lasso,yi2015optimal} for further details. 
On a high level, these works show that when $\mf{x}_j$'s are Gaussian, nonlinearity can be treated as an additional noise term. 
To give an example, \cite{plan2016generalized} and \cite{plan2014high} demonstrate that under the same model as \eqref{model}, when
$\mathbf{x}_j\sim\mathcal{N}(0,\mathbf{I}_{d\times d})$, $\theta_*\in\Theta$, and $y_j$ is sub-Gaussian for $j=1,\ldots,N$, solving the constrained problem
\[
\widehat{\theta}=\argmin_{\mathbf{\theta}\in \Theta}~\|\mathbf{y}-\mathbf{X}\theta\|_2^2,
\]
with $\mathbf{y}=[y_1~\cdots~y_N]^T$ and $\mathbf{X}=\frac{1}{\sqrt{N}}[\mathbf{x}_1~\cdots~\mathbf{x}_N]^T$, recovers $\theta_*$ up to a scaling factor $\eta$ with high probability: namely, for all $\beta \ge 2$,
\begin{align}
\label{bound-1}
&
\mb P\left[\left\|\widehat{\theta}-\eta\theta_*\right\|_2\geq 
C\frac{\omega(D(\Theta,\eta\theta_*)\cap S_2(1))+\beta}{\sqrt{N}}\right]\leq ce^{-\beta^2/2},
\end{align}
where, with formal definitions to follow in Section \ref{sec:back}, $ S_2(1)$ is the unit sphere in $\mathbb{R}^d$, $D(\Theta,\theta)$ is the descent cone of $\Theta$ at point $\theta$ and $\omega(T)$ is the Gaussian mean width of a subset $T \subset \mathbb{R}^d$.
A different approach to estimation of the index vector in model \eqref{model} with similar recovery guarantees has been developed in \cite{yi2015optimal}. 
However, the key assumption adopted in all these works that the vectors $\mf{x}_j$ follow Gaussian distributions preclude situations where the measurements are heavy tailed, and hence might be overly restrictive for some practical applications; for example, noise and outliers observed in high-dimensional image recovery often exhibit heavy-tailed behavior, see \cite{face-recognition}. The works \cite{yang2017stein} and \cite{yang2017learning} later consider using Stein's identity to perform nonlinear recovery under the assumption that the distribution of the sensing vector is known, both the distribution function and the nonlinear transform must satisfy certain smoothness assumptions,

As we mentioned above, \cite{li1989regression} have shown that direct consistent estimation of $\theta_\ast$ is possible when $\Pi$ belongs to a family of elliptically symmetric distributions.
Our main contribution is the non-asymptotic analysis for this scenario, with a particular focus on the case when $d>n$ and $\theta_\ast$ possesses special structure, such as sparsity. 
Moreover, we make very mild assumptions on the tails of the response variable $y$: for example, when the link function satisfies 
$f(\dotp{\mf{x}}{\theta_\ast},\delta)=\tilde f(\dotp{\mf{x}}{\theta_\ast})+\delta$, it is only assumed that $\delta$ possesses $2+\eps$ moments, for some $\eps>0$. 
\cite{plan2016generalized} present analysis for the Gaussian case and ask ``Can the same kind of accuracy be expected for random non-Gaussian matrices?'' In this chapter, we give a positive answer to their question. 
To achieve our goal, we propose a Lasso-type estimator that admits tight probabilistic guarantees in spirit of \eqref{bound-1} despite weak tail assumptions (see Theorem \ref{master-bound} below for details).


\section{Definitions and Background Material.} 
\label{sec:back}

This section introduces main notation and the key facts related to elliptically symmetric distributions, convex geometry and empirical processes. 
The results of this section will be used repeatedly throughout the chapter.
\\
For the unified treatment of vectors and matrices, it will be convenient to treat a vector $v\in \mb R^{d\times 1}$ as a $d\times 1$ matrix. 
Let $d_1,d_2\in \mb N$ be such that $d_1 d_2=d$. 
Given $v_1,v_2\in \mb R^{d_1\times d_2}$, the Euclidean dot product is then defined as $\dotp{v_1}{v_2}=\tr(v_1^T v_2)$, where $\tr(\cdot)$ stands for the trace of a matrix and $v^T$ denotes the transpose of $v$. \\
The $\ell_1$-norm of $v\in \mb R^d$ is defined as $\|v\|_1=\sum_{j=1}^d |v_j|$. 
The nuclear norm of a matrix $v\in \mb R^{d_1\times d_2}$ is  
$\|v\|_\ast = \sum_{j=1}^{\min(d_1,d_2)} \sigma_j(v)$, where $\sigma_j(v), \ j=1,\ldots,\min(d_1,d_2)$ stand for the singular values of $v$, and 
the operator norm is defined as $\|v\|=\max_{j=1,\ldots,\min(d_1,d_2)} \sigma_j(v)$.  

\subsection{Elliptically symmetric distributions.} 

A centered random vector $\mathbf{x}\in\mathbb{R}^d$ has elliptically symmetric (alternatively, elliptically contoured or just elliptical) distribution with parameters $\mathbf{\Sigma}$ and $F_{\mu}$, denoted $\mathbf{x}\sim\mathcal{E}(0,~\mathbf{\Sigma},~F_{\mu})$, 
if 
\begin{equation}
\label{elliptical-definition}
\mathbf{x}\stackrel{d}{=}\mu\mathbf{B}U,
\end{equation}
where $\stackrel{d}{=}$ denotes equality in distribution, $\mu$ is a scalar random variable with cumulative distribution function $F_{\mu}$, $\mathbf{B}$ is a fixed $d\times d$ matrix such that $\mathbf{\Sigma}=\mathbf{B}\mathbf{B}^T$, and $U$ is uniformly distributed over the unit sphere $ S_2(1)$ and independent of $\mu$. 
Note that distribution $\mathcal{E}(0,~\mathbf{\Sigma},~F_{\mu})$ is well defined, as if $\mathbf{B}_1\mathbf{B}_1^T=\mathbf{B}_2\mathbf{B}_2^T$, then there exists a unitary matrix $\mathbf{Q}$ such that $\mathbf{B}_1=\mathbf{B}_2\mathbf{Q}$, and $\mathbf{Q}U\stackrel{d}{=}U$. Along these same lines, we note that representation \eqref{elliptical-definition} is not unique, as one may replace the pair $(\mu,~\mathbf{B})$ with $\left(c\mu,~\frac{1}{c}\mathbf{B}\mathbf{Q}\right)$ for any constant $c>0$ and any orthogonal matrix $\mathbf{Q}$. 
To avoid such ambiguity, in the following we allow $\mathbf{B}$ to be any matrix satisfying $\mathbf{B}\mathbf{B}^T=\mathbf{\Sigma}$, and noting that the covariance matrix of $U$ is a multiple of the identity, we further impose the condition that the covariance matrix of $\mathbf{x}$ is equal to $\mathbf{\Sigma}$, i.e. $\expect{\mathbf{x}\mathbf{x}^T}=\mathbf{\Sigma}$.

Alternatively, the mean-zero elliptically symmetric distribution can be defined uniquely via its characteristic function  
\[
\mathbf{s}\rightarrow\psi\left(\mathbf{s}^T\mathbf{\Sigma}\mathbf{s}\right),~\mathbf{s}\in\mathbb{R}^d,
\]
where $\psi:\mathbb{R}^+\rightarrow\mathbb{R}$ is called the characteristic generator of $\mathbf{x}$. 
For further details information about elliptically  distribution, see \cite{elliptical-paper} for details.

An important special case of the family $\mathcal{E}(0,~\mathbf{\Sigma},~F_{\mu})$
of elliptical distributions is the Gaussian distribution $\mathcal{N}(0,\mathbf{\Sigma})$, where $\mu=\sqrt{z}$ with $z \stackrel{d}{=} \chi_d^2$, and the characteristic generator is $\psi(x)=e^{-x/2}$.

The following elliptical symmetry property, generalizing the well known fact for the conditional distribution of the multivariate Gaussian, plays an important role in our subsequent analysis, see \cite{elliptical-paper}:
\begin{proposition}
\label{elliptical-theorem}
Let $\mathbf{x}=[\mathbf{x}_1,~\mathbf{x}_2]\sim\mathcal{E}_d(0,\mathbf{\Sigma},F_\mu)$, where are of dimension $d_1$ and $d_2$ respectively, with $d_1+d_2=d$. 
Let $\mf{\Sigma}$ be partitioning accordingly as
\[
\mathbf{\Sigma}=
\left[
\begin{array}{cc}
\mathbf{\Sigma}_{11}  & \mathbf{\Sigma}_{12}  \\
\mathbf{\Sigma}_{21} & \mathbf{\Sigma}_{22}  
\end{array}
\right].
\]
Then, whenever $\mathbf{\Sigma}_{22}$ has full rank,
the conditional distribution of ${\bf x}_1$ given ${\bf x}_2$ is elliptical 
$\mathcal{E}_{d_1}(0,\mathbf{\Sigma}_{1|2},F_{\mu_{1|2}})$, where
\[
\mathbf{\Sigma}_{1|2}=\mathbf{\Sigma}_{11}-\mathbf{\Sigma}_{12}\mathbf{\Sigma}_{22}^{-1}\mathbf{\Sigma}_{21},
\]
and $F_{\mu_{1|2}}$ is the cumulative distribution function of $(\mu^2-\mathbf{x}_2^T\mathbf{\Sigma}_{22}^{-1}\mathbf{x}_2)^{1/2}$ given 
$\mathbf{x}_2$.
\end{proposition}
Note that $\mu^2-\mathbf{x}_2^T\mathbf{\Sigma}_{22}^{-1}\mathbf{x}_2$ is always nonnegative, hence $F_{\mu_{1|2}}$ is well defined, since by \eqref{elliptical-definition} we have
\begin{align*}
\mathbf{x}_2^T\mathbf{\Sigma}_{22}^{-1}\mathbf{x}_2
=\mu^2(\mathbf{B}_2 U)^T(\mathbf{B}_2\mathbf{B}_2^T)^{-1}(\mathbf{B}_2 U)
=\mu^2 U^T\mathbf{B}_2^T(\mathbf{B}_2\mathbf{B}_2^T)^{-1}\mathbf{B}_2 U
\leq\mu^2 U^T U=\mu^2,
\end{align*}
where $\mathbf{B}_2$ is the matrix consisting of the last $d_2$ rows of $\mathbf{B}$ in \eqref{elliptical-definition}, and where the inequality holds due to the fact that $\mathbf{B}_2^T(\mathbf{B}_2\mathbf{B}_2^T)^{-1}\mathbf{B}_2$ is a  projection matrix. 
The following corollary is easily deduced from the theorem above:
\begin{corollary}
\label{elliptical-corollary}
If $\mathbf{x}\sim\mathcal{E}_d(0,\mathbf{\Sigma},F_\mu)$ with $\mathbf{\Sigma}$ of full rank, then for any two fixed vectors $\mathbf{y}_1,\mathbf{y}_2\in\mathbb{R}^d$ with $\|\mathbf{y}_2\|_2=1$,
\[\expect{\langle\mathbf{x},\mathbf{y}_1\rangle~|~\langle\mathbf{x},\mathbf{y}_2\rangle}
=\langle\mathbf{y}_1,\mathbf{y}_2\rangle\langle\mathbf{x},\mathbf{y}_2\rangle.\]
\end{corollary}
\begin{proof}
Let $\{\mathbf{v}_1,\cdots,\mathbf{v}_d\}$ be an orthonormal basis in $\mathbb{R}^d$ such that $\mathbf{v}_d=\mathbf{y}_2$. 
Let $\mathbf{V}=[\mathbf{v}_1
~\mathbf{v}_2~\cdots~\mathbf{v}_d]$ and consider the linear transformation
\[
\widetilde{\mathbf{x}}=\mathbf{V}^T\mathbf{x}.
\]
Then, by \eqref{elliptical-definition}, $\widetilde{\mathbf{x}}=\mu\mathbf{V}^T\mathbf{B} U$, which is centered elliptical with full rank covariance matrix 
$\mathbf{V}^T\mathbf{\Sigma}\mathbf{V}$. 
Applications of Theorem \ref{elliptical-theorem} with $\mathbf{x}_1=[\langle\mathbf{x},\mathbf{v}_1\rangle,~\cdots,~\langle\mathbf{x},\mathbf{v}_{d-1}\rangle]$ and $\mathbf{x}_2=\langle\mathbf{x},\mathbf{v}_{d}\rangle=\langle\mathbf{x},\mathbf{y}_2\rangle$ yields
\begin{align*}
\expect{\langle\mathbf{x},\mathbf{y}_1\rangle~|~\langle\mathbf{x},\mathbf{y}_2\rangle}
=&\expect{\left.\sum_{i=1}^d\langle\mathbf{x},\mathbf{v}_i\rangle\langle\mathbf{y}_1,\mathbf{v}_i\rangle~\right|~\langle\mathbf{x},\mathbf{v}_d\rangle}\\
=&\expect{\left.\sum_{i=1}^{d-1}\langle\mathbf{x},\mathbf{v}_i\rangle\langle\mathbf{y}_1,\mathbf{v}_i\rangle~\right|~\langle\mathbf{x},\mathbf{v}_d\rangle}
+\langle\mathbf{x},\mathbf{v}_d\rangle\langle\mathbf{y}_1,\mathbf{v}_d\rangle\\
=&\langle\mathbf{x},\mathbf{v}_d\rangle\langle\mathbf{y}_1,\mathbf{v}_d\rangle
=\langle\mathbf{y}_1,\mathbf{y}_2\rangle\langle\mathbf{x},\mathbf{y}_2\rangle,
\end{align*}
where in the second to last equality we have used the fact that the conditional distribution of  $[\langle\mathbf{v}_1,\mathbf{x}\rangle,~\cdots,~\langle\mathbf{v}_{d-1},\mathbf{x}\rangle]$ given $\langle\mathbf{x},\mathbf{v}_d\rangle$ is 
elliptical with mean zero.
\end{proof}

\subsection{Geometry.}
\begin{definition}[Restricted set]
\label{def:restricted.set}
Given $c_0>1$, the $c_0$-restricted set of the norm $\|\cdot\|_{\mathcal{K}}$ at $\theta\in\mathbb{R}^d$ is defined as
\begin{align}
\label{eq:rest-set}
&
\mathbb{S}_{c_0}(\theta):=\mathbb{S}_{c_0}(\theta;\m K)
=\left\{\mathbf{v}\in\mathbb{R}^d:~\|\theta+\mathbf{v}\|_{\mathcal{K}}
\leq \|\theta\|_{\mathcal{K}}+\frac{1}{c_0}\|\mathbf{v}\|_{\mathcal{K}}\right\}.
\end{align}
\end{definition}
\begin{definition}[Restricted compatibility]\label{def:restricted.compatibility}
The restricted compatibility constant of a set $A \subseteq \mathbb{R}^d$ with respect to the norm $\|\cdot\|_{\mathcal{K}}$ is given by
\[
\Psi(A):=\Psi(A;\m K) =
\sup_{\mathbf{v}\in A\backslash \{0\}}\frac{\|\mathbf{v}\|_{\mathcal{K}}}{\|\mathbf{v}\|_2}.
\]
\end{definition}
\begin{remark}
The restricted set from the definition \ref{def:restricted.set} is not necessarily convex. 
However, if the norm $\|\cdot\|_{\mathcal{K}}$ is decomposable (see definition \ref{def:decomposable}), then the restricted set is contained in a convex cone, and the corresponding restricted compatibility constant is easier to estimate. 
Decomposable norms have been introduced by \cite{general-m-estimator} and later appeared in a number of works, e.g. \cite{m-estimator-2} and references therein. 
For reader's convenience, we provide a self-contained discussion in Appendix \ref{app-B}. 
\end{remark}

\section{Main Results}

In this section, we define a version of Lasso estimator that is well-suited for heavy-tailed measurements, and state its performance guarantees. 


We will assume that $\mathbf{x}_1,~\mathbf{x}_2,~\ldots,~\mathbf{x}_N\in \mb R^d$ are i.i.d. copies of an \textbf{isotropic} vector $\mathbf{x}$ with spherically symmetric distribution $\mathcal{E}_d(0,\mathbf{I_{d\times d}},F_{\mu})$. 
If $\mf{x} \sim \mathcal{E}_d(0,\mf{\Sigma},F_{\mu})$ for some positive definite matrix $\mf\Sigma$, then by definition $\mf{x}\stackrel{d}{=}\mu\mathbf{\Sigma}^{1/2}U$, and 
$\dotp{\mf{x}}{\theta_\ast}=\dotp{\mf{\Sigma}^{-1/2}\mf{x}}{\mf\Sigma^{1/2}\theta_\ast}$, where $\mf{\Sigma^{-1/2}\mf{x}} = \mu U \sim \mathcal{E}_d(0,\mathbf{I_{d\times d}},F_{\mu})$. 
Hence, if we set $\tilde\theta_\ast:=\mf\Sigma^{1/2}\theta_\ast$, then all results that we establish for isotropic measurements hold with 
$\theta_\ast$ replaced by $\tilde\theta_\ast$; remark after Theorem \ref{master-bound} includes more details. 

\subsection{Description of the proposed estimator.}

We first introduce an estimator under the scenario that $\theta_*\in\Theta$, for some known closed set $\Theta\subseteq\mathbb{R}^d$. Define the loss function $L_N^0(\cdot)$ as 
\begin{align}
\label{eq:loss0}
&
L^0_{N}(\theta):=\|\theta\|_2^2  - \frac{2}{N}\sum_{i=1}^N \dotp{y_i \mathbf{x}_i }{\theta},
\end{align}
which is the unbiased estimator of
\[
L^0(\theta) := \|\theta\|_2^2 - 2\mb E\dotp{y\mathbf{x}}{\theta} = \mb E\l( y - \dotp{x}{\theta}\r)^2 - \mb E y^2,
\]
where the last equality follows since $x$ is isotropic. 
Clearly, minimizing $L^0(\theta)$ over any set $\Theta\subseteq \mb R^d$ is equivalent to minimizing the quadratic loss
$\mb E\l(  y - \dotp{\mathbf{x}}{\theta} \r)^2$.
If distribution $F_\mu$ has heavy tails, the sample average 
$\frac{1}{N}\sum_{i=1}^N y_i \mathbf{x}_i$ might not concentrate sufficiently well around its mean, hence we replace it by a more ``robust'' version obtained via truncation. 
Let $\mu\in \mb R$, $U\in  S_2(1)$ be such that $\mathbf{x} = \mu U$ (so that $\mu=\|\mf{x}\|_2)$, and set
\begin{align}
\label{transformation}
&
\widetilde{U}=\sqrt{d}U,\\
&\nonumber
q=\mu y/\sqrt{d}, 
\end{align}
so that $q \widetilde U = y \mathbf{x}$ and $\widetilde U$ is uniformly distributed on the sphere of radius $\sqrt d$, implying that its covariance matrix is $I_d$, the identity matrix. 
Next, define the truncated random variables
\begin{align}
\label{transformation-2}
\widetilde{q}_i=\sign{(q_i)}(|q_i|\wedge\tau), \ i=1,\ldots, m,
\end{align}
where $\tau=N^{\frac{1}{2(1+\kappa)}}$ for some $\kappa\in(0,1)$ that is chosen based on the integrability properties of $q$, see \eqref{eq:determine.kappa}.
Finally, set 
\begin{align}
\label{eq:loss}
L^\tau_{N}(\theta)=\|\theta\|_2^2  - \frac{2}{N}\sum_{i=1}^N\dotp{\widetilde q_i\widetilde U_i}{\theta},
\end{align}
and define the estimator $\widehat \theta_N$ as the solution to the constrained optimization problem:
\begin{align}
\label{eq:truncated-version}
\widehat\theta_N:=\argmin\limits_{\theta\in \Theta}L^\tau_N(\theta).
\end{align}
We will also denote  
\begin{align}
\label{eq:exp-loss}
L^\tau(\theta) := \mb E L^\tau_N(\theta) = \|\theta\|_2^2 - 2\mb E \dotp{\widetilde q\widetilde U}{\theta}. 
\end{align}
For the scenarios where structure on the unknown $\theta_\ast$ is induced by a norm $\|\cdot\|_{\mathcal{K}}$ (e.g., if $\theta_\ast$ is sparse, then $\|\cdot\|_\m K$ could be the $\|\cdot\|_1$ norm), we will also consider the estimator $\widehat{\theta}^{\lambda}_{m}$ defined via 
\begin{equation}
\label{eq:unconstrained-version}
\widehat{\theta}^{\lambda}_N:=\argmin_{\theta\in\mathbb{R}^d}\Big[ L_N^{\tau}(\theta)+\lambda\|\theta\|_{\mathcal{K}} \Big],
\end{equation}
where $\lambda>0$ is a regularization parameter to be specified, and $L_N^{\tau}(\theta)$ is defined in \eqref{eq:loss}.

Let us note that truncation approach has previously been successfully implemented by \cite{truncation-paper} to handle heavy-tailed noise in the context of matrix recovery with sub-Gaussian design.
In the present chapter, we show that truncation-based approach is also useful in the situations where the measurements are heavy-tailed.

\begin{remark}
Note that our estimator \eqref{eq:unconstrained-version} is in general much easier to implement than some other popular alternatives, such as the usual Lasso estimator \cite{tibshirani1996regression}. 
For example, when the signal $\theta$ is sparse, our estimator takes the form
\[
\widehat{\theta}^{\lambda}_N:=\argmin_{\theta\in\mathbb{R}^d}\Big[ \|\theta\|_2^2  - \frac{2}{N}\sum_{i=1}^N\dotp{\widetilde q_i\widetilde U_i}{\theta} + \lambda\|\theta\|_{1} \Big],
\]
which yields a closed form solution in the form of ``soft-thresholding''. 
Specifically, let 
$\mathbf{b}=\frac{1}{N}\sum_{i=1}^N\widetilde q_i\widetilde U_i$, then, the $k$-th entry of 
$\widehat{\theta}^{\lambda}_N$ takes the form:
\begin{align}\label{ST-estimator}
\left(\widehat{\theta}^{\lambda}_N\right)_k
=\begin{cases}
b_k-\lambda/2,~~&\textrm{if}~~b_k\geq\lambda/2,\\
0,~~&\textrm{if}~~   -\lambda/2\leq b_k\leq\lambda/2,\\
b_k+\lambda/2,~~&\textrm{if}~~b_k\leq   -\lambda/2.
\end{cases}
\end{align}
We should note however that such simplification comes at the cost of knowing the distribution of measurement vector $\mf{x}$. 
Despite being of low computational complexity, our estimator can still exploit the structure of the problem, while being robust both to the possible model misspecification as well as to data corruption modeled by the heavy-tailed distributions. 
We demonstrate this in the following sections. 
\end{remark}
\begin{remark}[Non-isotropic measurements]
\label{rmk:non-isotropic}
When $\mf x\sim  \mathcal{E}_d(0,\mf{\Sigma},F_{\mu})$ for some $\Sigma\succ 0$, then estimator \eqref{eq:truncated-version} has to be replaced by 
\begin{align}
\label{non-isotropic1}
\widehat\theta_N:=\argmin\limits_{\theta\in \Theta} 
\Big[
\| \mf\Sigma^{1/2}\theta\|_2^2  - \frac{2}{N}\sum_{i=1}^N\dotp{\widetilde q_i\widetilde U_i}{ \mf\Sigma^{1/2}\theta}
\Big],
\end{align}
which is equivalent to 
\[
\tilde\theta_N:=\argmin\limits_{\theta\in \mf\Sigma^{1/2}\Theta} 
\Big[
\|\theta\|_2^2  - \frac{2}{N}\sum_{i=1}^N\dotp{\widetilde q_i\widetilde U_i}{\theta}
\Big],
\]
is a sense that $\tilde\theta_m = \mf\Sigma^{1/2}\hat\theta_m$. 
Hence, results obtained for isotropic measurements easily extend to the more general case. 
Similarly, estimator \eqref{eq:unconstrained-version} should be replaced by
\begin{align}
\label{non-isotropic2}
\hat{\theta}^{\lambda}_N:=\argmin_{\theta\in\mathbb{R}^d}\Big[ 
\| \mf\Sigma^{1/2}\theta\|_2^2  - \frac{2}{N}\sum_{i=1}^N\dotp{\widetilde q_i\widetilde U_i}{\Sigma^{1/2}\theta} + \lambda\| \mf\Sigma^{1/2}\theta\|_{\mathcal{K}} 
\Big],
\end{align}
which is equivalent to 
\[
\tilde{\theta}^{\lambda}_N:=\argmin_{\theta\in\mathbb{R}^d}\Big[ 
\|\theta\|_2^2  - \frac{2}{N}\sum_{i=1}^N\dotp{\widetilde q_i\widetilde U_i}{\theta} + \lambda\|\theta\|_{\mf\Sigma^{1/2}\mathcal{K}} 
\Big],
\]
meaning that $\tilde{\theta}^{\lambda}_m = \mf\Sigma^{1/2}\hat{\theta}^{\lambda}_N$. 
\end{remark}


\subsection{Estimator performance guarantees.}

In this section, we present the probabilistic guarantees for the performance of the estimators $\widehat \theta_N$ and $\widehat \theta^\lambda_m$ defined by \eqref{eq:truncated-version} and \eqref{eq:unconstrained-version} respectively. \\
Everywhere below, $C,c,C_j$ denote numerical constants; when these constants depend on parameters of the problem, we specify this dependency by writing $C_j=C_j(\text{parameters})$. 
Let 
\begin{align}
\label{eq:eta}
&
\eta=\mb E\dotp{y\mf{x}}{\theta_\ast},
\end{align} 
and assume that $\eta\ne0$ and $\eta\theta_\ast\in \Theta$. 
\begin{theorem}
\label{master-bound}
Suppose that $\mathbf{x}\sim\mathcal{E}(0,~\mathbf{I}_{d\times d},~F_{\mu})$. 
Moreover, suppose that for some $\kappa>0$ 
\bea 
\label{eq:determine.kappa}
\phi:=\mb E|q|^{2(1+\kappa)}<\infty.
\ena
Then there exist constants $C_1=C_1(\kappa,\phi),C_2=C_2(\kappa,\phi)>0$ such that $\widehat{\theta}_N$ satisfies
\begin{align*}
\mb P\left(\left\|\widehat{\theta}_N-\eta\theta_*\right\|_2\geq C_1\frac{(\omega(D(\Theta,\eta\theta_*)\cap S_2(1))+1)\beta}{\sqrt{N}}\right)\leq C_2e^{-\beta/2},
\end{align*}
for any $\beta\geq8$ and $N\geq \beta^2\l( \omega(D(\Theta,\eta\theta_*)\cap S_2(1))+1 \r)^2$.
\end{theorem}
\begin{remark}
\begin{enumerate}
\item Unknown link function $f$ enters the bound only through the constant $\eta$ defined in \eqref{eq:eta}. 
\item 
Aside from independence, conditions on the noise $\delta$ are implicit and follow from assumptions on $y$. 
In the special case when the error is additive, that is, when $y=f(\dotp{\mathbf{x}}{\theta_*})+\delta$, the moment condition \eqref{eq:determine.kappa} becomes 
$\mb E \big|\|\mf{x}\|_2 f(\dotp{\mathbf{x}}{\theta_*})+\|\mf{x}\|_2\delta\big|^{2(1+\kappa)}<\infty$, 
for which it is sufficient to assume that  $\mb E\Big|\|\mf{x}\|_2 f(\dotp{\mathbf{x}}{\theta_*})\Big|^{2(1+\kappa)}<\infty$ and 
$\mb E\left|\|\mf{x}\|_2\delta\right|^{2(1+\kappa)}<\infty$. 
\item 
Theorem \ref{master-bound} is mainly useful when $\eta\theta_*$ lies on the boundary of the set $\Theta$. 
Otherwise, if $\eta\theta_*$ belongs to the relative interior of $\Theta$, the descent cone $D(\Theta,\eta\theta_*)$ is the affine hull of 
$\Theta$  (which will often be the whole space $\mathbb{R}^d$). 
Thus, in such cases the Gaussian mean width $\omega(D(\Theta,\eta\theta_*)\cap S_2(1))$ can be on the order of $\sqrt{d}$, which is prohibitively large when $d\gg m$. 
We refer the reader to \cite{plan2016generalized,plan2014high} for a discussion of related result and possible ways to tighten them. 
\end{enumerate}
\end{remark}

\noindent 
Next, we present performance guarantees for the unconstrained estimator \eqref{eq:unconstrained-version}.
\begin{theorem}
\label{master-bound-2}
Assume that the norm $\|\cdot\|_{\mathcal{K}}$ dominates the 2-norm, i.e. $\|\mathbf{v}\|_{\mathcal{K}}\geq\|\mathbf{v}\|_{2},~\forall \mathbf{v}\in\mathbb{R}^d$. 
Let $\mathbf{x}\sim\mathcal{E}(0,~\mathbf{I}_{d\times d},~F_{\mu})$, and suppose that for some $\kappa>0$ 
\[
\phi:=\mb E |q|^{2(1+\kappa)}<\infty.
\]
Then there exist constants $C_3=C_3(\kappa,\phi),C_4=C_4(\kappa,\phi)>0$ such that for all 
$\lambda\geq \frac{C_3\beta}{\sqrt{N}}\l(1+\omega(\mathcal{G}) \r)$ 
\begin{align*}
\mb P\left(
\left\|\widehat{\theta}_N^{\lambda}-\eta\theta_*\right\|_2\geq \frac{3}{2}\lambda\cdot\Psi\left(\mb{S}_2\left(\eta\theta_*\right)\right)
\right)\leq C_4e^{-\beta/2},
\end{align*}
for any $\beta \geq 8$ and $N \geq (\omega(\mathcal{G})+1)^2\beta^2$, where 
$\mathcal{G}:=\{\mathbf{x}\in\mathbb{R}^d:~\|\mathbf{x}\|_{\mathcal{K}}\leq1\}$ is the unit ball of $\|\cdot\|_\m K$ norm,
and $\mb{S}_2(\cdot)$ and $\Psi(\cdot)$ are given in Definitions \ref{def:restricted.set} and \ref{def:restricted.compatibility} respectively. 
\end{theorem}
\noindent
\begin{remark}[Non-isotropic measurements]
It follows from remark \ref{rmk:non-isotropic} and \eqref{non-isotropic1} that, whenever $\mf x\sim  \mathcal{E}_d(0,\mf{\Sigma},F_{\mu})$, 
inequality of Theorem \ref{master-bound} has the form
\begin{align*}
\mb P\left(\left\| \mf\Sigma^{1/2}\l( \widehat{\theta}_N-\eta\theta_* \r) \right\|_2\geq C_1\frac{\l(\omega \l( \mf\Sigma^{1/2} D(\Theta,\eta\theta_*)\cap S_2(1)\r)+1 \r) \beta}{\sqrt{N}}\right)\leq C_2e^{-\beta/2},
\end{align*}
which can be further combined with the bound 
\[
\omega \l( \mf\Sigma^{1/2} D(\Theta,\eta\theta_*)\cap S_2(1)\r)\leq \| \mf\Sigma^{1/2}\| \cdot \|\mf\Sigma^{-1/2}\| \, \omega \l( D(\Theta,\eta\theta_*)\cap S_2(1) \r),
\]
that follows from remark 1.7 in \cite{plan2016generalized}. 
Similarly, the inequality of Theorem \ref{master-bound-2} holds with 
\[
\m G_{\mf\Sigma^{1/2}} := \{\mathbf{x}\in\mathbb{R}^d:~\|\mathbf{x}\|_{ \mf\Sigma^{1/2}\mathcal{K}}\leq1\},
\]
the unit ball of $\|\cdot\|_{\mf\Sigma^{1/2}\m K}$ norm, in place of $\m G$. 
Namely, for all $\lambda\geq \frac{C_3\beta}{\sqrt{N}} \l(1+\omega(\mathcal{G}_{\mf \Sigma^{1/2}}) \r)$, 
\[
\mb P\left(
\left\| \mf \Sigma^{1/2} \l( \widehat{\theta}_N^{\lambda}-\eta\theta_* \r) \right\|_2 
\geq 
\frac{3}{2}\lambda \cdot \Psi \l( \mb{S}_2 \l( \eta \mf \Sigma^{1/2} \theta_*\r); \mf \Sigma^{1/2}\m K \r)
\right)\leq C_4e^{-\beta/2}
\]
Note that $\omega \l(	\mathcal{G}_{\mf \Sigma^{1/2}}\r) \leq \| \mf \Sigma^{1/2} \| \, \omega(\mathcal{G})$. 
Moreover, we show in Appendix \ref{app-B} that for a class of decomposable norms (which includes $\|\cdot\|_1$ and nuclear norm), the upper bounds for $\Psi \l( \mb{S}_2 \l( \eta \mf \Sigma^{1/2} \theta_*\r); \mf \Sigma^{1/2}\m K \r)$ and 
$\Psi \l(\mb S_2(\eta\theta_\ast) \r)$ differ by the factor of $\l\| \mf \Sigma^{-1/2} \r\|$. 
\end{remark}

\subsection{Examples.}

We discuss two popular scenarios: estimation of the sparse vector and estimation of the low-rank matrix. 
\\
\textbf{Estimation of the sparse signal. } Assume that there exists $J\subseteq \l\{1,\ldots,d\r\}$ of cardinality $s\leq d$ such that $\theta_{\ast,j}=0$ for $j\notin J$. 
Let $\Theta = \l\{ \theta\in \mb R^d: \ \|\theta\|_1\leq \| \eta\theta_\ast\|_1 \r\}$, with $\eta$ defined in \eqref{eq:eta}. 
In this case, it is well-known that $\omega^2\l( D(\Theta,\eta\theta_\ast)\cap  S_2(1) \r)\leq 2 s\log(d/s)+\frac{5}{4}s$, see proposition 3.10 in \cite{chandrasekaran2012convex}, hence Theorem \ref{master-bound} implies that, with high probability,  
\begin{align}
\label{eq:bound-m1}
&
\left\| \widehat{\theta}_N-\eta\theta_* \right\|_2\lesssim \sqrt{\frac{s\log(d/s)}{N}}
\end{align}
as long as $m\gtrsim s\log(d/s)$. 
\\
We compare this bound to result of Theorem \ref{master-bound-2} for constrained estimator. 
Let $\|\cdot\|_\m K$ be the $\ell_1$ norm. 
It is well-know that $\omega(\m G)=\mb E\max_{j=1,\ldots, d}|g_j|\leq \sqrt{2\log(2d)}$, where $\mf g\sim \m N(0,\mf I_{d\times d})$. 
Moreover, we show in Appendix \ref{app-B} that $\Psi\left(\mb{S}_2\left(\eta\theta_*\right) \right)\leq 4\sqrt{s}$. 
Hence, for $\lambda \simeq \sqrt{\frac{\log(2d)}{N}}$, Theorem \ref{master-bound-2} implies that 
\[
\left\|\widehat{\theta}_N^{\lambda}-\eta\theta_*\right\|_2\lesssim \sqrt{\frac{s\log(d)}{N}}
\] 
with high probability whenever $m\gtrsim \log(2d)$. 
This bound is only marginally weaker than \eqref{eq:bound-m1} due to the logarithmic factor, however, definition of 
$\widehat{\theta}_N^{\lambda}$ does not require the knowledge of $\l\|  \eta\theta_\ast \r\|_1$, as we have already mentioned before. 
\\
\textbf{Estimation of a low-rank matrix. } Assume that $d=d_1d_2$ with $d_1\leq d_2$, and $\theta_\ast\in \mb R^{d_1\times d_2}$ has rank $r\leq \min(d_1,d_2)$. 
Let $\Theta = \l\{ \theta\in \mb R^{d_1\times d_2}: \ \|\theta\|_\ast\leq \|\eta\theta_\ast \|_\ast \r\}$. 
Then the Gaussian mean width of the intersection of a descent cone with a unit ball is bounded as 
$\omega^2\l( D(\Theta,\eta\theta_\ast)\cap  S_2(1) \r) \leq  3r(d_1+d_2 - r)$, see proposition 3.11 in \cite{chandrasekaran2012convex}, hence 
Theorem \ref{master-bound} yields that, with high probability, 
\[
\left\| \widehat{\theta}_N-\eta\theta_* \right\|_2\lesssim \sqrt{\frac{r(d_1+d_2)}{N}}
\] 
as long as the number of observations satisfies $m\gtrsim r(d_1+d_2)$. 
\\
Finally, we derive the corresponding bound from Theorem \ref{master-bound-2}. 
The Gaussian mean width of the unit ball in the nuclear norm is bounded by $2(\sqrt{d_1}+\sqrt{d_2})$, see proposition 10.3 in \cite{vershynin2015estimation}. 
It follows from results in Appendix \ref{app-B} that $\Psi\left(\mb{S}_2\left(\eta\theta_*\right) \right)\leq 4\sqrt{2r}$. 
Theorem \ref{master-bound-2} now implies that with high probability
\[
\left\| \widehat{\theta}_N-\eta\theta_* \right\|_2\lesssim \sqrt{\frac{r(d_1+d_2)}{N}},
\] 
which matches the bound of Theorem \ref{master-bound}.

\section{Numerical Experiments}

In this section, we demonstrate the performance of proposed robust estimator \eqref{eq:unconstrained-version} for one-bit compressed sensing model. 
The model takes the following form:
\begin{equation}\label{one-bit}
y=sign(\dotp{\mathbf{x}}{\theta_*})+\delta,
\end{equation}
where $\delta$ is the additive noise and the parameter $\theta^*$ is assumed to be $s$-sparse. 
This model is highly non-linear because one can only observe the sign of each measurement.

The 1-bit compressed sensing model was previously discussed extensively in a number of works \cite{plan2014high,ai2014one,plan2016generalized}. 
It was shown that when the measurement vectors are either Gaussian or sub-Gaussian, the Lasso estimator recovers the support of $\theta^*$ with high probability. 
Here, we show that under the heavy-tailed elliptically distributed measurements, our estimator numerically outperforms the standard Lasso estimator
\[
\theta_{\mbox{Lasso}}=\argmin_{\theta\in\mathbb{R}^d}~~\|\mathbf{X}\theta-\mathbf{y}\|_2^2+\lambda \| \theta \|_1,
\]
 while taking the form of a simple soft-thresholding as explained in \eqref{ST-estimator}.

In the first numerical experiment, 
data are simulated in the following way: $\mathbf{x}_1,~\mathbf{x}_2,~\cdots,~\mathbf{x}_{128}\in\mathbb{R}^{512}$ are i.i.d. with spherically symmetric distribution $\mathbf{x}_i\stackrel{d}{=}\mu_i U_i, \ i=1,\ldots,N$. 
The random vectors $U_i\in\mathbb{R}^{512}$ are i.i.d. with uniform distribution over the sphere of radius $\sqrt{512}$, and the random variables $\mu_i\in\mathbb{R}$ are also i.i.d., independent of $U_i$ and such that 
\begin{equation}\label{heavy-noise}
\mu_i\stackrel{d}{=}\frac{1}{\sqrt{2c(q)}}(\xi_{i,1}-\xi_{i,2}),
\end{equation}
where $\xi_{i,1}$ and $\xi_{i,2}$,~$i=1,2,\cdots,128$ are i.i.d. with Pareto distribution, meaning that their probability density function is given by
\[
p(t;q) = \frac{q}{(1+t)^{1+q}}I_{\{t>0\}},
\]
$c(q):=\var(\xi)=\frac{q}{(q-1)^2(q-2)}$, and $q=2.1$. 
The true signal $\theta^*$ has sparsity level $s=5$, with index of each non-zero coordinate chosen uniformly at random, and the magnitude having uniform distribution on $[0,1]$. 

Since we can only recover the original signal $\theta^*$ up to scaling, define the relative error for any estimator 
$\hat\theta$ with respect to $\theta^*$ as follows:
\begin{equation}\label{relative-error}
\textrm{Relative~error} = \left|\frac{\hat\theta}{\|\hat\theta\|_2}-\frac{\theta^*}{\|\theta^*\|_2}\right|.
\end{equation}
In each of the following two scenarios, we run the experiment 200 times for both the Lasso estimator and the estimator defined in \eqref{eq:unconstrained-version} with $\|\cdot\|_\m K$ being the $\|\cdot\|_1$ norm. 
We set the truncation level as $\tau = c m^{\frac{1}{2(1+\kappa)}}$, and the values of $c$ and regularization parameter $\lambda$ are obtained via the standard 2-fold cross validation for the relative error \eqref{relative-error}.  
We then plot the histogram of obtained results over 200 runs of the experiment. 

In the first scenario, we set the additive error $\delta_i=0,~i=1,2,\cdots,128$ in the 1-bit model \eqref{one-bit} and plot the histogram in Fig. \ref{fig:Stupendous1}. 
We can see from the plot that the robust estimator \eqref{eq:unconstrained-version} noticeably outperforms the Lasso estimator.

In the second scenario, we set the additive error $\delta_i,~i=1,2,\cdots,128$ to be i.i.d. heavy tailed noise with signal-to-noise ratio (SNR)\footnote{The signal-to-noise ratio (dB) is defined as $\textrm{SNR}:=10\log_{10}(\sigma^2_{\textrm{signal}}/\sigma^2_{\textrm{noise}})$. In our case, since $\langle\mathbf{x}_i,\theta^*\rangle$ can be positive or negative with equal probability, $\sigma^2_{\textrm{signal}}=1$, and thus, $\sigma^2_{\textrm{noise}}=1/10$.} equal to 10dB, so that the noise has the distribution
\[\delta_i\stackrel{d}{=}h_i/\sqrt{10},\]
and $h_i,~i=1,2,\cdots, 128$ are i.i.d. random variables with Pareto distribution, see \eqref{heavy-noise}. 
The results are plotted in Fig. \ref{fig:Stupendous2}. 
The histogram shows that, while performance of the Lasso estimator becomes worse, results of robust estimator \eqref{eq:unconstrained-version} are relatively stable. 

\begin{figure}[htbp]
   \minipage{0.49\textwidth}
   \includegraphics[width=\linewidth]{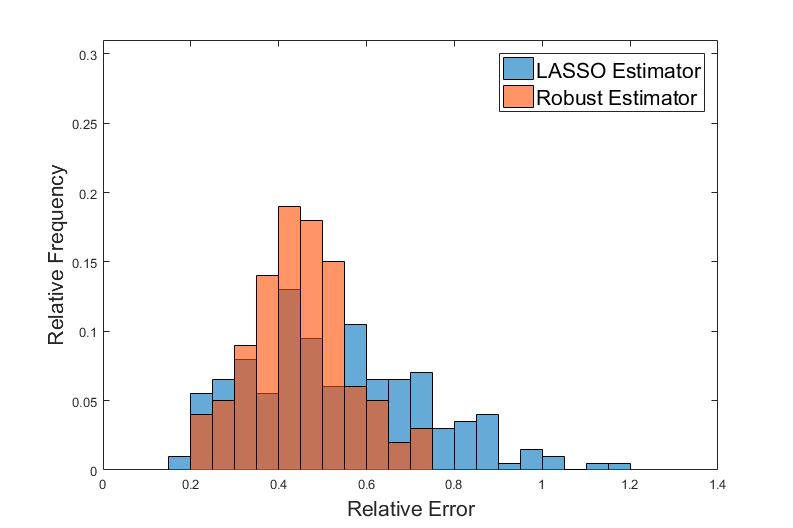} 
   \caption{Lasso vs robust estimator without additive noise. \mbox{                }}
   \label{fig:Stupendous1}
   \endminipage\hfill
 \minipage{0.49\textwidth}
   \includegraphics[width=\linewidth]{chapter3/200-one-bit-5dB-heavynoise} 
   \caption{Lasso vs robust estimator under heavy-tailed noise with signal-to-noise ratio(SNR) equal to $10dB$.}
   \label{fig:Stupendous2}
   \endminipage
\end{figure}

In the second simulation study, the simulation framework similar to the second scenario above, the only difference being the increased sample size $N$. 
The results are plotted in Fig. \ref{fig:awesome_image1}-\ref{fig:awesome_image3} with sample sizes $m=128,~256$ and 512, respectively. 

\begin{figure}[!htb]
\minipage{0.49\textwidth}
  \includegraphics[width=\linewidth]{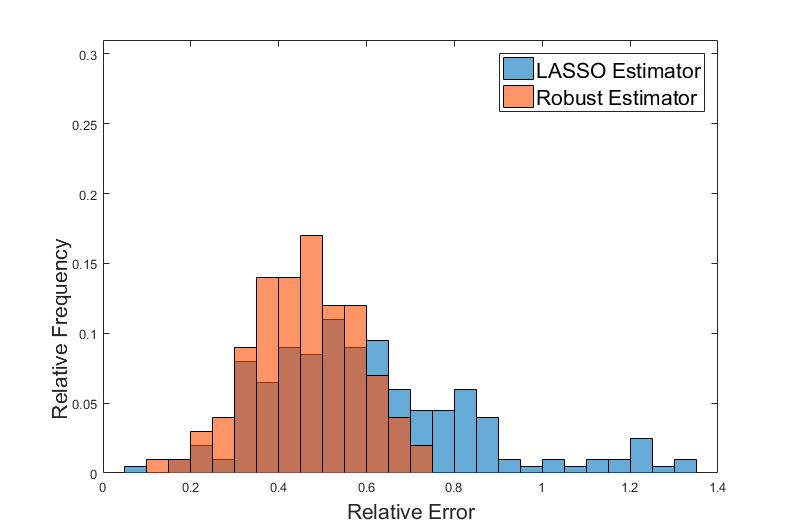}
  \caption{$m=128$}\label{fig:awesome_image1}
\endminipage\hfill
\minipage{0.49\textwidth}
  \includegraphics[width=\linewidth]{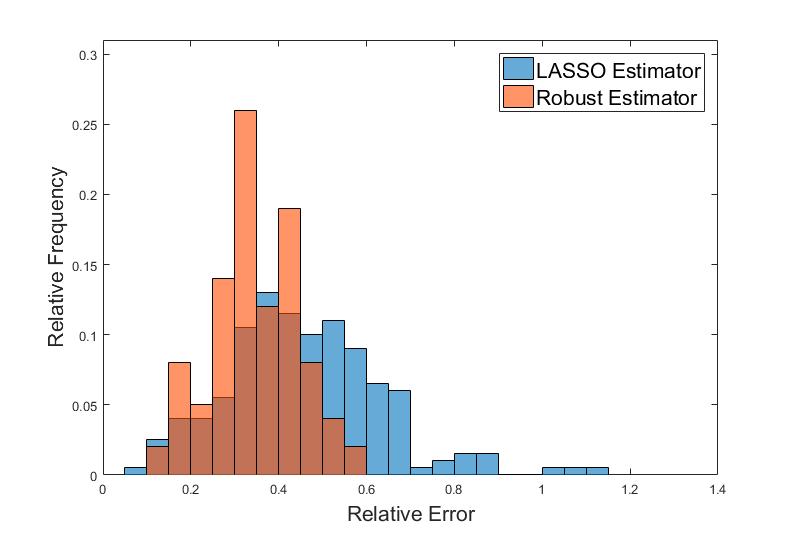}
  \caption{$m=256$}\label{fig:awesome_image2}
\endminipage\hfill
\minipage{0.5\textwidth}%
  \includegraphics[width=\linewidth]{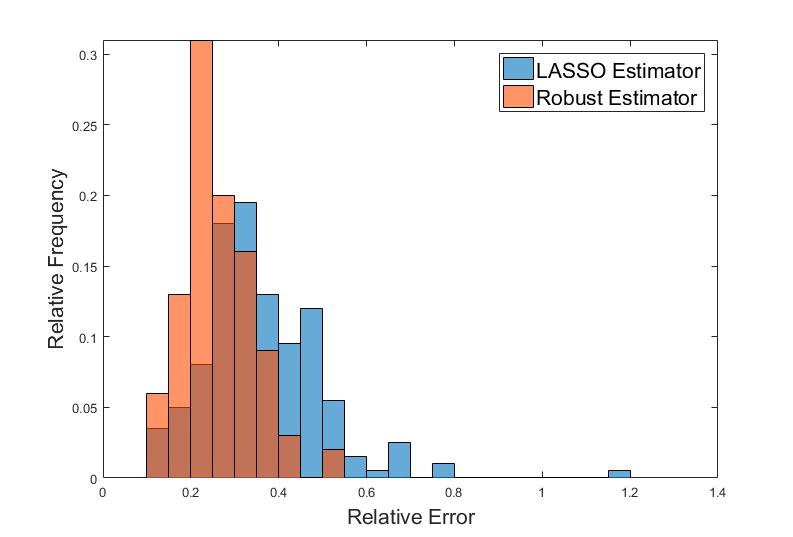}
  \caption{$m=512$}\label{fig:awesome_image3}
\endminipage
\end{figure}

\section{Proofs.}

This  section is devoted to the proofs of Theorems \ref{master-bound} and \ref{master-bound-2}. 

\subsection{Preliminaries.}

We recall several useful facts from probability theory that we rely on in the subsequent analysis. 
\\
The following well-known bound shows that the uniform distribution on a high-dimensional sphere enjoys strong concentration properties.
\begin{lemma}[Lemma 2.2 of \cite{convex-geometry-Ball}]
\label{ball-lemma-0}
Let $U$ have the uniform distribution on $ S_2(1)$.
Then for any $\Delta\in(0,1)$ and any fixed $\mathbf{v}\in S_2(1)$, 
\[ 
\mb P\left(\langle U,\mathbf{v}\rangle\geq\Delta \right) \leq e^{-d\Delta^2/2}.
\]
\end{lemma}
\noindent 
Next, we state several useful results from the theory of empirical processes. 
\begin{definition}[$\psi_q$-norm]
For $q \ge 1$, the $\psi_q$-norm of a random variable $\xi\in \mb R$ is given by
\[\|\xi\|_{\psi_q}=\sup_{p\geq1}p^{-\frac{1}{q}}(\expect{|X|^p})^{\frac1p}.\]
Specifically, the cases $q=1$ and $q=2$ are known as the sub-exponential and sub-Gaussian norms respectively.  
We will say that $\xi$ is sub-exponential if $\|\xi\|_{\psi_1}<\infty$, and $X$ is sub-Gaussian if $\|\xi\|_{\psi_2}<\infty$.
\end{definition}
\begin{remark}
\label{norm-justify}
It is easy to check that $\psi_q$-norm is indeed a norm. 
\end{remark}
\begin{remark}
A useful property, equivalent to the previous definition of a sub-Gaussian random variable $\xi$, is that there exists a positive constant $C$ such that
\[
\mb P\l( |\xi|\geq u\r)\leq \exp(1-Cu^2).
\]
For the proof, see Lemma 5.5 in \cite{introduction-to-random-matrix}.
\end{remark}

\begin{definition}[sub-Gaussian random vector]
A random vector $\mathbf{x}\in\mathbb{R}^d$ is called sub-Gaussian if there exists $C>0$ such that 
$\|\langle\mathbf{x},\mathbf{v}\rangle\|_{\psi_2}\leq C$ for any $\mathbf{v}\in S_2(1)$. 
The corresponding sub-Gaussian norm is then
\[
\|\mathbf{x}\|_{\psi_2}:=\sup_{\mathbf{v}\in S_2(1)}\|\langle\mathbf{x},\mathbf{v}\rangle\|_{\psi_2}.
\]
\end{definition}

Next, we recall the notion of the generic chaining complexity. 
Let $(T,d)$ be a metric space. 
We say a collection $\{\mathcal{A}_l\}_{l=0}^{\infty}$ of subsets of $T$ is increasing when $\mathcal{A}_{l} \subseteq \mathcal{A}_{l+1}$ for all $l \ge 0$.
\begin{definition}[Admissible sequence]
An increasing sequence of subsets  $\{\mathcal{A}_l\}_{l=0}^{\infty}$ of $T$ is admissible if $|\mathcal{A}_l|\leq N_l,~\forall l$, where $N_0=1$ and $N_l=2^{2^l},~\forall l\geq1$.
\end{definition}
For each $\mathcal{A}_l$, define the map $\pi_l:T\rightarrow \mathcal{A}_l$ as 
$\pi_l(t)=\textrm{arg}\min_{s\in\mathcal{A}_l}d(s,t),~\forall t\in T$. 
Note that, since each $\mathcal{A}_l$ is a finite set, the minimum is always achieved. 
When the minimum is achieved for multiple elements in $\mathcal{A}_l$, we break the ties arbitrarily.
The generic chaining complexity $\gamma_2$ is defined as
\bea 
\label{eq:def.gamma2}
\gamma_2(T,d):=\inf\sup_{t\in T}\sum_{l=0}^{\infty}2^{l/2}d(t,\pi_l(t)),
\ena
where the infimum is over all admissible sequences. 
The following theorem tells us that $\gamma_2$-functional controls the ``size'' of a Gaussian process. 

\begin{lemma}[Theorem 2.4.1 of \cite{Talagrand-book-2}]
\label{mmt}
Let $\{G(t), \ t\in T\}$ be a centered Gaussian process indexed by the set $T$, and let 
\[
d(s,t)=\expect{(G(s)-G(t))^2}^{1/2},~\forall s,t\in T.
\]
Then, there exists a universal constant $L$ such that 
\[
\frac1L\gamma_2(T,d)\leq\expect{\sup_{t\in T}G(t)}\leq L\gamma_2(T,d).
\]
\end{lemma}
Let $(T,d)$ be a semi-metric space, and let $X_1(t),\cdots,X_m(t)$ be independent stochastic processes indexed by $T$ such that $\mb E|X_j(t)|<\infty$ for all $t\in \mb T$ and $1\leq j\leq m$. 
We are interested in bounding the supremum of the empirical process
\begin{equation}\label{empirical}
Z_N(t)=\frac1N\sum_{i=1}^N\l[X_i(t)-\expect{X_i(t)}\r].
\end{equation}
The following well-known symmetrization inequality reduces the problem to bounds on a (conditionally) Rademacher process 
$R_N(t)=\frac 1N\sum_{i=1}^N\varepsilon_i X_i(t),~t\in T$, where $\eps_1,\ldots,\eps_m$ are i.i.d. Rademacher random variables (meaning that they take values $\{-1,+1\}$ with probability $1/2$ each), independent of $X_i$'s.
\begin{lemma}[Symmetrization inequalities]
\label{symmetrization}
\[
\mb E\sup_{t\in T}|Z_N(t)|
\leq 2\mb E\sup_{t\in T}|R_N(t)|,
\]
and for any $u>0$, we have
\[
\mb P\left(\sup_{t\in T}|Z_N(t)|\geq 2\mb E\sup_{t\in T}|Z_N(t)|+u\right)\leq 4\mb P\left(\sup_{t\in T}|R_N(t)|\geq u/2\right).
\]
\end{lemma}
\begin{proof}
See Lemmas 6.3 and 6.5 in \cite{Talagrand-book}
\end{proof}

Finally, we recall Bernstein's concentration inequality. 
\begin{lemma}[Bernstein's inequality]
\label{Bernstein}
Let $X_1,\cdots,X_m$ be a sequence of independent centered random variables. 
Assume that there exist positive constants $\sigma$ and $D$ such that for all integers $p\geq 2$
\[
\frac1N\sum_{i=1}^N\expect{|X_i|^p}\leq\frac{p!}{2}\sigma^2D^{p-2},
\]
then
\[
\mb P\left(\left|\frac1N\sum_{i=1}^N X_i\right|\geq\frac{\sigma}{\sqrt{N}}\sqrt{2u}+\frac{D}{N}u\right)
\leq2\exp(-u).
\]
In particular, if $X_1,\cdots,X_N$ are all sub-exponential random variables, then $\sigma$ and $D$ can be chosen as 
$\sigma=\frac{1}{N}\sum_{i=1}^N\|X_i\|_{\psi_1}$ and $D=\max\limits_{i=1\ldots N}\|X_i\|_{\psi_1}$. 
\end{lemma}

\subsection{Roadmap of the proof of Theorem \ref{master-bound}.}

We outline the main steps in the proof of Theorem \ref{master-bound}, and postpone some technical details to sections \ref{section:consistency} and \ref{section:chaining}.\\
As it will be shown below in Lemma \ref{lemma:mean-consistency}, 
$\argmin\limits_{\theta\in \Theta}L^0(\theta) = \eta \theta_\ast$ for $\eta = \mb E\l( \dotp{y\mathbf{x}}{\theta_\ast}\r)$ and $L^0(\widehat \theta_N) - L^0(\eta\theta_\ast) = \|\widehat\theta_N - \eta\theta_\ast\|_2^2$, hence 
\begin{align}
\|\widehat\theta_N - \eta\theta_\ast\|_2^2 & =
L^\tau(\widehat \theta_N) - L^\tau(\eta\theta_\ast) + \l( L^0(\widehat\theta_N) - L^\tau(\widehat\theta_N) - L^0(\eta\theta_\ast) + L^\tau(\eta\theta_\ast)  \r) \nonumber\\
&
=L^\tau(\widehat\theta_N) - L^\tau(\eta\theta_\ast) +(L_N^\tau(\widehat\theta_N) - L_N^\tau(\eta\theta_\ast))  
\nonumber\\
& 
\quad - (L_N^\tau(\widehat\theta_N) - L_N^\tau(\eta\theta_\ast))  - 2\mb E_N \dotp{y\mathbf{x} - \widetilde q \widetilde U}{\widehat\theta_N - \eta\theta_\ast}, \label{split-bound}
\end{align}
where $\mb E_N(\cdot)$ stands for the conditional expectation given $(\mathbf{x}_i,y_i)_{i=1}^N$, and where we used the equality 
$L^0(\widehat\theta_N) - L^\tau(\widehat\theta_N) - L^0(\eta\theta_\ast) + L^\tau(\eta\theta_\ast)  = 
-2 \mb E_N\left( \dotp{y\mathbf{x} - \widetilde q \widetilde U}{\widehat\theta_N - \eta\theta_\ast}\right)
$ in the last step. 
Since $\widehat\theta_N$ minimizes $L_N^\tau$, $L_N^\tau(\widehat\theta_N) - L_N^\tau(\eta\theta_\ast)\leq 0$, and
\begin{align*}
\|\widehat\theta_N - \eta\theta_\ast\|_2^2 \leq & \,
\frac{2}{N}\sum_{i=1}^N \l( \dotp{ \widetilde q_i \widetilde U_i}{\widehat\theta_N - \eta\theta_\ast}    - \mb E_N \l(\dotp{\widetilde q \widetilde U}{\widehat\theta_N - \eta\theta_\ast}\r)\r)  \\
& 
-2\mb E_N\left(  \dotp{y\mathbf{x} - \widetilde q \widetilde U}{\widehat\theta_N - \eta\theta_\ast} \right).
\end{align*}
Note that $\widehat\theta_N - \eta\theta_\ast \in D(\Theta,\eta\theta_\ast)$; 
dividing both sides of the inequality by $\|\widehat\theta_N - \eta\theta_\ast\|_2$, we obtain
\begin{align}
\label{eq:main}
&
\|\widehat\theta_N - \eta\theta_\ast\|_2 \leq 
\sup_{\mathbf{v}\in D(\Theta,\eta\theta_\ast)\cap S_2(1)}\l| \frac{2}{N}\sum_{i=1}^N \dotp{\widetilde q_i\widetilde U_i}{\mathbf{v}} - 
\mb E  \dotp{\widetilde q\widetilde U}{\mathbf{v}}\r| +
2\sup_{\mathbf{v}\in S_2(1)}\mb E \dotp{y\mathbf{x} - \widetilde q \widetilde U}{\mathbf{v}}.
\end{align}
To get the desired bound, it remains to estimate two terms above.  
The bound for the first term is implied by Lemma \ref{concentration-multiplier}: setting $T=D(\Theta,\eta\theta_\ast)\cap  S_2(1)$, and observing that the diameter $\Delta_d(T) := \sup_{t\in T}\|t\|_2=1$, we get that with probability $\geq 1 - ce^{-\beta/2}$,
\[
\sup_{\mathbf{v}\in D(\Theta,\eta\theta_\ast)\cap  S_2(1)}
\l| \frac{2}{N}\sum_{i=1}^N \dotp{\widetilde q_i\widetilde U_i}{\mathbf{v}} - 
\mb E  \dotp{\widetilde q\widetilde U}{\mathbf{v}}\r| 
\leq 
C\frac{(\omega(T)+1)\beta}{\sqrt{N}}.
\]
To estimate the second term, we apply Lemma \ref{lemma:truncation-bias}: 
\[
2\sup_{\mathbf{v}\in  S_2(1)}\mb E  \dotp{y\mathbf{x} - \widetilde q \widetilde U}{\mathbf{v}} \
\leq
\frac{\tilde C}{\sqrt N}.
\]
Result of Theorem \ref{master-bound} now follows from the combination of these bounds.
\qed 


\subsection{Roadmap of the proof of Theorem \ref{master-bound-2}.}

Once again, we will present the main steps while skipping the technical parts. 
Lemma \ref{lemma:mean-consistency} implies that 
$\argmin\limits_{\theta\in \Theta}L^0(\theta) = \eta \theta_\ast$ for 
$\eta = \mb E \dotp{y\mathbf{x}}{\theta_\ast}$ and 
\[
L^0(\widehat \theta_N^\lambda) - L^0(\eta\theta_\ast) = \|\widehat\theta_N^\lambda - \eta\theta_\ast\|_2^2.
\] 
Thus, arguing as in \eqref{split-bound},
\begin{align*}
\|\widehat\theta_N^\lambda - \eta\theta_\ast\|_2^2 & =
L^\tau(\widehat\theta_N^\lambda) - L^\tau(\eta\theta_\ast) +(L_N^\tau(\widehat\theta_N^\lambda) - L_N^\tau(\eta\theta_\ast))  \\
& 
\quad - (L_N^\tau(\widehat\theta_N^\lambda) - L_N^\tau(\eta\theta_\ast))  - 2\mb E_N  \dotp{y\mathbf{x} - \widetilde q \widetilde U}{\widehat\theta_N^\lambda - \eta\theta_\ast}.
\end{align*}
Since $\widehat{\theta}^{\lambda}_N$ is a solution of problem \eqref{eq:unconstrained-version}, it follows that
\begin{align*}
L_N^{\tau}(\theta_N^{\lambda})+\lambda\l\| \theta_N^{\lambda} \r\|_{\mathcal{K}}\leq
L_N^{\tau}\left(\eta\theta_*\right)+\lambda\l \|\eta\theta_* \r\|_{\mathcal{K}},
\end{align*}
which further implies that
\begin{align}
\|\widehat{\theta}_N^\lambda-\eta\theta_*\|_2^2 \leq
&
\frac{2}{N}\sum_{i=1}^N \l( \dotp{\widetilde q_i\widetilde U_i}{\widehat\theta_N^\lambda - \eta\theta_\ast}  - 
\mb E_N \dotp{\widetilde q\widetilde U}{\widehat\theta_N^\lambda - \eta\theta_\ast} \r) 
- 2\mb E_N  \dotp{y\mathbf{x} - \widetilde q \widetilde U}{\widehat\theta_N^\lambda - \eta\theta_\ast}  
\nonumber\\
&
+\lambda\left(\|\eta\theta_*\|_{\mathcal{K}}-\|\widehat{\theta}_N^\lambda\|_{\mathcal{K}}\right)\nonumber\\
=&
  \dotp{\frac{2}{N}\sum_{i=1}^N\widetilde q_i\widetilde U_i -\expect{\widetilde q\widetilde U}}{\widehat\theta_N^\lambda - \eta\theta_\ast} - 2\mb E_N \dotp{y\mathbf{x} - \widetilde q \widetilde U}{\widehat\theta_N^\lambda - \eta\theta_\ast}  
 \nonumber\\
&
+\lambda\left(\|\eta\theta_*\|_{\mathcal{K}}-\|\widehat{\theta}_N^\lambda\|_{\mathcal{K}}\right).
\label{sth-1}
\end{align}
Letting $\|\cdot\|_{\mathcal{K}}^*$ be the dual norm of $\|\cdot\|_{\mathcal{K}}$ 
(meaning that $\|\mf{x}\|_{\m K}^\ast = \sup\l\{ \dotp{\mf{x}}{\mf{z}}, \ \|\mf{z}\|_\m K\leq 1 \r\}$),
the first term in \eqref{sth-1} can be estimated as 
\begin{align}\label{sth-2}
\dotp{\frac{1}{N}\sum_{i=1}^N\widetilde q_i\widetilde U_i -\expect{\widetilde q\widetilde U}}{\widehat\theta_N^\lambda - \eta\theta_\ast}
\leq\left\|\frac{1}{N}\sum_{i=1}^N\widetilde q_i\widetilde U_i -\expect{\widetilde q\widetilde U}\right\|_{\mathcal{K}}^*\cdot\|\widehat{\theta}_N^\lambda-\eta\theta_*\|_{\mathcal{K}}.
\end{align}
Since 
\[
\left\|\frac{1}{N}\sum_{i=1}^N\widetilde q_i\widetilde U_i -\expect{\widetilde q\widetilde U}\right\|_{\mathcal{K}}^*=\sup_{\|t\|_{\mathcal{K}}\leq1}\dotp{\frac{1}{N}\sum_{i=1}^N\widetilde q_i\widetilde U_i -\expect{\widetilde q\widetilde U}}{t},
\]
lemma \ref{concentration-multiplier} applies 
with $T=\mathcal{G}:=\{\mathbf{x}\in\mathbb{R}^d:~\|\mathbf{x}\|_{\mathcal{K}}\leq1\}$. 
Together with an observation that
$\Delta_d(T)\leq\sup_{t\in T}\|t\|_{\mathcal{K}}=1$ (due to the assumption $\|\mathbf{v}\|_2\leq\|\mathbf{v}\|_{\mathcal{K}}, \ \forall \mathbf{v}\in\mathbb{R}^d$), this yiels
\[
\mb P\left( 
\sup_{\|t\|_{\mathcal{K}}\leq1}\left|\dotp{\frac{1}{N}\sum_{i=1}^N\widetilde q_i\widetilde U_i -\expect{\widetilde q\widetilde U}}{t}\right|
\geq C'\frac{\left(\omega(\mathcal{G})+1\right)\beta}{\sqrt{N}} \right)
\leq c'e^{-\beta/2},
\]
for any $\beta\geq8$ and some constants $C', c>0$. For the second term in \eqref{sth-1}, we use Lemma \ref{lemma:truncation-bias} to obtain
\[
2\mb E_N \dotp{y\mathbf{x} - \widetilde q \widetilde U}{\widehat\theta_N^\lambda - \eta\theta_\ast} 
\leq\frac{C''}{\sqrt{N}}\|\widehat\theta_N^\lambda-\eta\theta_*\|_2
\leq\frac{C''}{\sqrt{N}}\|\widehat\theta_N^\lambda-\eta\theta_*\|_{\mathcal{K}},\]
for some constant $C''>0$, where we have again applied the inequality $\|\mathbf{v}\|_2\leq\|\mathbf{v}\|_{\mathcal{K}}$.
Combining the above two estimates gives that with probability at least $1-ce^{-\beta/2}$,
\begin{align}
\label{sth-3}
\|\widehat{\theta}_N^\lambda-\eta\theta_*\|_2^2
\leq C\frac{\left(\omega(\mathcal{G})+1\right)\beta}{\sqrt{N}}\|\widehat{\theta}_N^\lambda-\eta\theta_*\|_{\mathcal{K}}+\lambda\left(\|\eta\theta_*\|_{\mathcal{K}}-\|\widehat{\theta}_N^\lambda\|_{\mathcal{K}}\right),
\end{align}
for some constant $C>0$ and any $\beta\geq8$. 
Since $\lambda\geq 2C \left(\omega(\mathcal{G})+1\right)\beta/\sqrt{N}$ by assumption, and the right hand side of \eqref{sth-3} is nonnegative, it follows that
\[
\frac12\|\widehat{\theta}_N^\lambda-\eta\theta_*\|_{\mathcal{K}}+\|\eta\theta_*\|_{\mathcal{K}}-\|\widehat{\theta}_N^\lambda\|_{\mathcal{K}}\geq 0.
\]
This inequality implies that $\widehat{\theta}_N^\lambda-\eta\theta_*\in S_2(\eta\theta_*)$. 
Finally, from \eqref{sth-3} and the triangle inequality,
\begin{align*}
\|\widehat{\theta}_N^\lambda-\eta\theta_*\|_2^2
\leq \frac32\lambda\|\widehat{\theta}_N^\lambda-\eta\theta_*\|_{\mathcal{K}}.
\end{align*}
Dividing both sides by $\|\widehat{\theta}_N^\lambda-\eta\theta_*\|_2$ gives
\begin{align*}
\|\widehat{\theta}_N^\lambda-\eta\theta_*\|_2
\leq \frac32\lambda\frac{\|\widehat{\theta}_N^\lambda-\eta\theta_*\|_{\mathcal{K}}}{\|\widehat{\theta}_N^\lambda-\eta\theta_*\|_2}
\leq\frac32\lambda\cdot\Psi\left( S_2(\eta\theta_*)\right).
\end{align*}
This finishes the proof of Theorem \ref{master-bound-2}.

\subsection{Bias of the truncated mean.}
\label{section:consistency}

The following lemma is motivated by and is similar to Theorem 2.1 in \cite{li1989regression}.
\begin{lemma}
\label{lemma:mean-consistency}
Let $\eta=\mb E \langle y\mathbf{x},\theta_\ast\rangle$. 
Then 
\[
\eta\theta_\ast = \argmin_{\theta\in \Theta} L^0(\theta),
\]
and for any $\theta\in \Theta$,
\[
L^0(\theta) - L^0(\eta\theta_\ast) = \|\theta - \eta\theta_\ast\|_2^2.
\]
\end{lemma}
\begin{proof}
Since $y=f(\dotp{\mathbf{x}}{\theta_*},\delta)$, we have that for any $\theta\in\mathbb{R}^d$ 
\begin{align*}
\mb E\dotp{y\mathbf{x}}{\theta}
=&\mb E \langle\mathbf{x},\theta\rangle f(\langle\mathbf{x},\theta_*\rangle,\delta)\\
=&\mb E
\expect{\langle\mathbf{x},\theta\rangle f(\langle\mathbf{x},\theta_*\rangle,\delta)
~|~\langle\mathbf{x},\theta_*\rangle,\delta}\\
=&\mb E
\mb E \l(\langle\mathbf{x},\theta\rangle 
~|~\langle\mathbf{x},\theta_*\rangle \r) \cdot f(\langle\mathbf{x},\theta_*\rangle,\delta)\\
=&\mb E \Big( \langle \theta_*,\theta\rangle\langle\mathbf{x},\theta_*\rangle
 f(\langle\mathbf{x},\theta_*\rangle,\delta) \Big) \\
 =&\eta\langle\theta_*,\theta\rangle,
\end{align*}
where the third equality follows from the fact that the noise $\delta$ is independent of the measurement vector $\mathbf{x}$, the second to last equality from the properties of elliptically symmetric distributions (Corollary \ref{elliptical-corollary}), 
and the last equality from the definition of $\eta$. 
Thus,
\begin{align*}
L^0(\theta)=&\|\theta\|_2^2-2\expect{\dotp{y\mathbf{x}}{\theta}} =\|\theta\|_2^2-2\eta\langle\theta_*,\theta\rangle
=\|\theta-\eta\theta_*\|_2^2-\|\eta\theta_*\|_2^2,
\end{align*}
which is minimized at $\theta=\eta\theta^*$. Furthermore, $L^0(\eta\theta^*)=-\|\eta\theta_*\|_2^2$, hence
\[
L^0(\theta) - L^0(\eta\theta_\ast) = \|\theta - \eta\theta_\ast\|_2^2,
\]
finishing the proof.
\end{proof}

Next, we estimate the ``bias term'' $\sup_{\mathbf{v}\in  S_2(1)}\mb E \dotp{y\mathbf{x} - \widetilde q \widetilde U}{\mathbf{v}}$ in inequality \eqref{eq:main}. 
In order to do so, we need the following preliminary result.
\begin{lemma}
\label{lemma:ball}
If $\mathbf{x}\sim\mathcal{E}(0,~\mathbf{I}_{d\times d},~F_{\mu})$, 
then the unit random vector $\mathbf{x}/\|\mathbf{x}\|_2$ is uniformly distributed over the unit sphere $ S_2(1)$. Furthermore, $\widetilde{U}=\sqrt{d}\mathbf{x}/\|\mathbf{x}\|_2$ is a sub-Gaussian random vector with sub-Gaussian norm $\|\widetilde{U}\|_{\psi_2}$ independent of the dimension $d$.
\end{lemma}
\begin{proof}
First, we use decomposition \eqref{elliptical-definition} for elliptical distribution together with our assumption that $\mathbf{\Sigma}$ is the identity matrix, to write $\mathbf{x}\stackrel{d}{=}\mu U$, which implies that
\[
\mathbf{x}/\|\mathbf{x}\|_2
\stackrel{d}{=}\textrm{sign}(\mu)U/\| U\|_2
=\textrm{sign}(\mu)U \stackrel{d}{=}U,
\]
with the final distributional equality holding as $ S_2(1)$, and hence its uniform distribution, is invariant with respect to reflections across any hyperplane through the origin.

To prove the second claim, it is enough to show that 
$\left\|\dotp{\widetilde{U}}{\mathbf{v}}\right\|_{\psi_2}\leq C,~\forall \mathbf{v}\in S_2(1)$ 
with constant $C$ independent of $d$. 
By the first claim and Lemma \ref{ball-lemma-0}, we have 
\[
\mb P\left( 
\langle\mathbf{x},\mathbf{v}\rangle/\|\mathbf{x}\|_2\geq \Delta \right)
\leq e^{-d\Delta^2/2},~\forall \mathbf{v}\in S_2(1).
\]
Choosing $\Delta=u/\sqrt{d}$ gives
\[
\mb P \left( 
\dotp{\widetilde{U}}{\mathbf{v}}\geq u \right)\leq e^{-u^2/2},~\forall \mathbf{v}\in S_2(1),~\forall u>0.
\]
By an equivalent definition of sub-Gaussian random variables (Lemma 5.5 of \cite{introduction-to-random-matrix}), this inequality implies that 
$\left\|\dotp{\widetilde{U}}{\mathbf{v}}\right\|_{\psi_2}\leq C$, hence finishing the proof.
\end{proof}
With the previous lemma in hand, we now establish the following result.

\begin{lemma}
\label{lemma:truncation-bias}
Under the assumptions of Theorem \ref{master-bound}, there exists a constant $C=C(\kappa,\phi)>0$ such that
\[
\left|\mb E\dotp{y\mathbf{x}-\widetilde{q}\widetilde{U}}{\mathbf{v}}\right|\leq C/\sqrt{N},
\]
for all $\mathbf{v}\in S_2(1)$.
\end{lemma}
\begin{proof}
By \eqref{transformation}, we have that $y\mathbf{x}=q\widetilde{U}$, thus the claim is equivalent to
\[
\left|\expect{\dotp{\widetilde{U}}{\mathbf{v}}(\widetilde{q}-q)}\right|\leq C/\sqrt{N}.
\]
Since $\widetilde{q}=\textrm{sign}(q)(|q|\wedge\tau)$, we have $|\widetilde{q}-q|=(|q|-\tau){\bf 1}(|q| \ge \tau) \le |q|{\bf 1}(|q| \ge \tau)$, and it follows that
\begin{align*}
\left| \mb E\dotp{\widetilde{U}}{\mathbf{v}} (\widetilde{q}-q)\right|
\leq& \mb E\left|\dotp{\widetilde{U}}{\mathbf{v}} (\widetilde{q}-q)\right| \\
\leq& \mb E \l( \left|\dotp{\widetilde{U}}{\mathbf{v}} q\right|\cdot\mathbf{1}_{\{|q|\geq\tau\}} \r) \\
\leq&\expect{\left|\dotp{\widetilde{U}}{\mathbf{v}} q\right|^2}^{1/2} \mb P\l( |q|\geq\tau \r)^{1/2}\\
\leq&\expect{\left|\dotp{\widetilde{U}}{\mathbf{v}}\right|^{\frac{2(1+\kappa)}{\kappa}}}^{\frac{\kappa}{2(1+\kappa)}}
\expect{|q|^{2(1+\kappa)}}^{\frac{1}{2(1+\kappa)}} \mb P\l( |q|\geq\tau \r)^{1/2},
\end{align*}
where the second to last inequality uses Cauchy-Schwarz, and the last inequality follows from H\"{o}lder's inequality.

For the first term, by Lemma \ref{lemma:ball}, $\widetilde{U}$ is sub-Gaussian with $\|\widetilde{U}\|_{\psi_2}$ independent of $d$. Thus, by the definition of the $\|\cdot\|_{\psi_2}$ norm and the fact that ${\bf v} \in  S_2(1)$,
\[
\expect{ \left|\dotp{\widetilde U}{\mathbf{v}}\right|^{\frac{2(1+\kappa)}{\kappa}}}^{\frac{\kappa}{2(1+\kappa)}}
\leq \sqrt{\frac{2(1+\kappa)}{\kappa}} \|\widetilde U \|_{\psi_2}.
\]

Recall that $\phi= \mb E |q|^{2(1+\kappa)}$.  
Then, the second term is bounded by $\phi^{\frac{1}{2(1+\kappa)}}$. 
For the final term, since
$\tau=m^{\frac{1}{2(1+\kappa)}}$, Markov's inequality implies that
\begin{align*}
  \l( \mb P \l( |q|>\tau \r) \r) ^{1/2}\leq\left(\frac{\mb E |q|^{2(1+\kappa)}}{\tau^{2(1+\kappa)}}\right)^{1/2}
 \leq\frac{\phi^{1/2}}{\sqrt{N}}.
\end{align*}
Combining these inequalities yields
\[
\left| \mb E\dotp{y\mathbf{x}-\widetilde{q}\widetilde{U}}{\mathbf{v}}\right|
\leq 
\frac{\sqrt{\frac{2(1+\kappa)}{\kappa}} \|\widetilde U \|_{\psi_2}\phi^{\frac{2+\kappa}{2(1+\kappa)}}}{\sqrt{N}}:=C(\kappa,\phi)/\sqrt{N},\]
completing the proof.
\end{proof}


\subsection{Concentration via generic chaining.}
\label{section:chaining}
In the following sections, we will use $c,C,C',C''$ to denote constants that are either absolute, or depend on underlying parameters $\kappa$ and $\phi$ (in the latter case, we specify such dependence). 
To make notation less cumbersome, constants denoted by the same letter ($c,C,C'$, etc.) might be different in various parts of the proof.  

The goal of this subsection is to prove the following inequality:
\begin{lemma}
\label{concentration-multiplier}
Suppose $\widetilde U_i$ and $\widetilde q_i$ are as defined according to \eqref{transformation} and \eqref{transformation-2} respectively. 
Then, for any bounded subset $T\subset\mathbb{R}^d$,
\begin{align*}
\mb P\left( \sup_{t\in T}\left|\frac1N\sum_{i=1}^N\dotp{\widetilde{U}_i}{t} \widetilde{q}_i
-\expect{\dotp{\widetilde{U}}{t} \widetilde{q}}\right|
\geq C\frac{(\omega(T)+\Delta_d(T))\beta}{\sqrt{N}} \right)
\leq ce^{-\beta/2},
\end{align*}
for any $\beta\geq8$, a positive constant $C=C(\kappa,\phi)$ and an absolute constant $c>0$. 
Here
\bea 
\label{def:Deltad}
\Delta_d(T):=\sup_{t\in T}\|t\|_2.
\ena
\end{lemma}
The main technique we apply is the generic chaining method developed by M. Talagrand \cite{Talagrand-book-2} for bounding the supremum of stochastic processes. 
Recently, \cite{Mendelson-1} and \cite{tail-bound-chaining} advanced the technique to obtain a sharp bound for supremum of processes index by squares of functions. 
More recently, \cite{Mendelson-2} proved a concentration result for the supremum of multiplier processes
under weak moment assumptions. In the current work, we show that exponential-type concentration inequalities for multiplier processes, such as the one in Lemma \ref{concentration-multiplier}, are achievable by applying truncation under a bounded $2(1+\kappa)$-moment assumption.

Define
\begin{align*} 
\overline{Z}(t)=&\frac1N\sum_{i=1}^N\dotp{\widetilde{U}_i}{t} \widetilde{q}_i
-\expect{\dotp{\widetilde{U}}{t} \widetilde{q}},\\
Z(t)=&\frac1N\sum_{i=1}^N\varepsilon_i \widetilde{q}_i\dotp{\widetilde{U}_i}{t},~\forall t\in T,
\end{align*}
where $T$ is a bounded set in $\mathbb{R}^d$ and $\{\varepsilon_i\}_{i=1}^m$ is a sequence i.i.d. Rademacher random variables taking values $\pm 1$  with probability $1/2$ each, and independent of $\{\widetilde{U}_i,\widetilde{q}_i, \ i=1, \ldots,m\}$. 
Result of Lemma \ref{concentration-multiplier} easily follows from the following concentration inequality:
\begin{lemma}
For any $\beta\geq8$,
\begin{equation}\label{major-criterion}
\mb P\left[\sup_{t\in T}\left|Z(t)\right|
\geq C\frac{(\omega(T)+\Delta_d(T))\beta}{\sqrt{N}}\right]
\leq ce^{-\beta/2},
\end{equation}
where $C=C(\kappa,\phi)$ is another constant possibly different from that of Lemma \ref{concentration-multiplier}, and $c>0$ is an absolute constant.
\end{lemma}
To deduce the inequality of Lemma \ref{concentration-multiplier}, we first apply the symmetrization inequality (Lemma \ref{symmetrization}), followed by Lemma \ref{basic-inequality} with $\beta_0=8$. 
It implies that
\[
\expect{\sup_{t\in T}\left|\overline{Z}(t)\right|}
\leq2\expect{\sup_{t\in T}\left|Z(t)\right|}\leq2C\left(8+2ce^{-4}\right)\frac{\omega(T)+\Delta_d(T)}{\sqrt{N}}.
\]
Application of the second bound of the symmetrization lemma with $u=2C(\omega(T)+\Delta_d(T))\beta/\sqrt{N}$ and \eqref{major-criterion} completes the proof of Lemma \ref{concentration-multiplier}.


It remains to justify \eqref{major-criterion}.
We start by picking an arbitrary point $t_0\in T$ such that there exists an admissible sequence $\{t_0\}=\mathcal{A}_0\subseteq\mathcal{A}_1\subseteq\mathcal{A}_2\subseteq\cdots$ satisfying
\begin{equation}\label{inter-1}
\sup_{t\in T}\sum_{l=0}^\infty2^{l/2}\|\pi_l(t)-t\|_2\leq 2\gamma_2(T),
\end{equation}
where we recall that $\pi_l$ is the closest point map from $T$ to $\mathcal{A}_l$ and the factor 2 is introduced so as to deal with the case where the infimum in the definition \eqref{eq:def.gamma2} of $\gamma_2(T)$ is not achieved. Then, write $Z(t)-Z(t_0)$ as the telescoping sum:
\[Z(t)-Z(t_0)=\sum_{l=1}^{\infty}Z(\pi_l(t))-Z(\pi_{l-1}(t))
=\sum_{l=1}^{\infty}\frac1N\sum_{i=1}^N\varepsilon_i\widetilde{q}_i\dotp{\widetilde{U}_i}{\pi_l(t)-\pi_{l-1}(t)}.\]
We claim that the telescoping sum converges with probability 1 for any $t\in T$. 
Indeed, note that for each fixed set of realizations of $\{\mathbf{x}_i\}_{i=1}^N$ and $\{\varepsilon_i\}_{i=1}^N$,
each summand is bounded as
\[
|\varepsilon_i \widetilde{q}_i \langle\widetilde{U}_i,\pi_l(t)-\pi_{l-1}(t)\rangle |
\leq|\widetilde{q}_i|\|\widetilde{U}_i\|_2\|\pi_l(t)-\pi_{l-1}(t)\|_2
\leq|\widetilde{q}_i|\|\widetilde{U}_i\|_2(\|\pi_l(t)-t\|_2+\|\pi_{l-1}(t)-t\|_2).
\]
Furthermore, since $T$ is a compact subset of $\mathbb{R}^d$, its Gaussian mean width is finite. 
Thus, by lemma \ref{mmt}, $\gamma_2(T)\leq L\omega(T)<\infty$. This inequality further implies that the sum on the left hand side of \eqref{inter-1} converges with probability 1. 

Next, with $\beta\geq 8$ being fixed, we split the index set $\{l\geq1\}$ into the following three subsets: 
\begin{align*}
I_1&=\{l\geq1:2^{l}\beta<\log eN\};\\
I_2&=\{l\geq1:\log eN\leq 2^{l}\beta< N\};\\
I_3&=\{l\geq1:2^{l}\beta\geq N\}.
\end{align*}
By the assumptions in Theorem \ref{master-bound} and the bound $\beta\geq8$, we have that $m\geq(\omega(T)+1)^2\beta^2\geq64$, implying that 
$\log eN=1+\log N <N$, and hence these three index sets are well defined. 
Depending on $\beta$, some of them might be empty, but this only simplifies our argument by making the partial sum over such an index set equal 0. 

The following argument yields a bound for $Z(\pi_l(t))-Z(\pi_{l-1}(t))$, assuming all three index sets are nonempty. 
Specifically, we show that
\begin{equation}
\label{chaining-goal}
\mb P \left(\sup_{t\in T}\left|\sum_{l\in I_j}\left(Z(\pi_l(t))-Z(\pi_{l-1}(t))\right)\right|
\geq C\frac{\gamma_2(T)\beta}{\sqrt{N}} \right)
\leq ce^{-\beta/2},
\end{equation}
for $C=C(\kappa,\phi)$ and $j=1,2,3$, respectively.

\subsubsection{The case $l\in I_1$.}\label{first-chaining}
\begin{proof}[Proof of inequality \eqref{chaining-goal} for the index set $I_1$]
Recall that $\tau=N^{\frac{1}{2(1+\kappa)}}$. \\
For each $t\in T$ we apply Bernstein's inequality (Lemma \ref{Bernstein}) to estimate each summand 
\[
Z(\pi_l(t))-Z(\pi_{l-1}(t))=\frac1N\sum_{i=1}^N\varepsilon_i\widetilde{q}_i\dotp{\widetilde{U}_i}{\pi_l(t)-\pi_{l-1}(t)}.
\] 
For any integer $p\geq2$, we have the following chains of inequalities:
%
\begin{align*}
&\expect{\left|\varepsilon\widetilde{q}\dotp{\widetilde{U}}{\pi_l(t)-\pi_{l-1}(t)} \right|^p}\\
\leq&\expect{\left|\varepsilon\dotp{\widetilde{U}}{\pi_l(t)-\pi_{l-1}(t)} \right|^p q^{2}
	\cdot|\widetilde{q}|^{p-2}}\\
\leq&\expect{\left|\dotp{\widetilde{U}}{\pi_l(t)-\pi_{l-1}(t)} \right|^p q^2}
\cdot \tau^{p-2}\\
\leq& \tau^{p-2}
\expect{ \left|\dotp{\widetilde{U}}{\pi_l(t)-\pi_{l-1}(t)}\right|
	^{\frac{1+\kappa}{\kappa}p}}^{\frac{\kappa}{1+\kappa}}\expect{ q^{2(1+\kappa)}}^{\frac{1}{1+\kappa}}\\
\leq&  \tau^{p-2}\|\widetilde{U}\|_{\psi_2}^p\left(\frac{(1+\kappa)p}{\kappa}\right)^{p/2}\phi^{\frac{1}{1+\kappa}}
\|\pi_l(t)-\pi_{l-1}(t)\|_2^p,
\end{align*}
where the second inequality follows from the truncation bound, the third from H\"{o}lder's inequality, and the last from the assumption that 
$\expect{q^{2(1+\kappa)}}\leq\phi$ and the following bound: by Lemma \ref{lemma:ball}, $\widetilde{U}_i$ is sub-Gaussian, hence for any $p\geq 2$
\[
\l(\mb E\dotp{\widetilde{U}_i}{\mathbf{v}}^{\frac{1+\kappa}{\kappa}p} \r)^{\frac{\kappa}{(1+\kappa)p}}
\leq\left(\frac{(1+\kappa)p}{\kappa}\right)^{1/2}\|\widetilde{U}_i\|_{\psi_2}\|\mathbf{v}\|_2,~\forall \mathbf{v}\in\mathbb{R}^d.
\]
We also note that $\|\widetilde{U}_i\|_{\psi_2}$ does not depend on $d$ by Lemma \ref{lemma:ball}. 
Next, by Stirling's approximation, 
$p!\geq\sqrt{2\pi}\sqrt{p}(p/e)^p$, thus there exist constants
$C'=C'(\kappa,\phi)$ and $C''=C''(\kappa)$ such that
\begin{align*}
\mb E\left|\varepsilon \widetilde{q} \dotp{\widetilde{U}}{\pi_l(t)-\pi_{l-1}(t)} \right|^p
\leq \frac{p!}{2}C'\|\pi_l(t)-\pi_{l-1}(t)\|_2^2(C''\tau\|\pi_l(t)-\pi_{l-1}(t)\|_2)^{p-2}.
\end{align*}
Bernstein's inequality (Lemma \ref{Bernstein}), with $\sigma=C'\|\pi_l(t)-\pi_{l-1}(t)\|_2$, $D=C''\tau\|\pi_l(t)-\pi_{l-1}(t)\|_2$ with $\tau=N^{1/2(1+\kappa)}$ now implies
\begin{align*}
\mb P \left(
\left|\frac1N\sum_{i=1}^N\varepsilon_i\widetilde{q}_i\dotp{\widetilde{U}_i}{\pi_l(t)-\pi_{l-1}(t)}\right|\geq\left(\frac{C'\sqrt{2u}}{\sqrt{N}}+\frac{C'' u}{m^{1-\frac{1}{2(1+\kappa)}}}\right)\|\pi_l(t)-\pi_{l-1}(t)\|_2
\right)
\leq 2e^{-u},
\end{align*}
for any $u>0$. Taking $u=2^l\beta$,  noting that as $\beta\geq8$ by assumption, 
we have $m\geq(\omega(T)+1)^2\beta^2\geq64$, and since $l \in I_1$, $2^l\leq2^l\beta < \log em$. 
In turn, this implies
\begin{align*}
\frac{2^l}{m^{1-\frac{1}{2(1+\kappa)}}}=\frac{2^{l/2}}{m^{1/2}}\cdot\frac{2^{l/2}}{m^{\kappa/2(1+\kappa)}}
\leq\frac{2^{l/2}}{m^{1/2}}\cdot\sqrt{\frac{\log em}{m^{\kappa/(1+\kappa)}}}
\leq\sqrt{\frac{1+\kappa}{\kappa}}\frac{2^{l/2}}{m^{1/2}},
\end{align*}
where the last inequality follows from the fact that $\log em$ is dominated by $\frac{1+\kappa}{\kappa}m^{\kappa/(1+\kappa)}$ for all $m \ge 1$.
This inequality implies that there exists a positive constant $C=C(\kappa,\phi)$ such that for any $\beta \geq 8$
\begin{align}\label{inter-2}
\mb P \l( \Omega_{l,t} \r) \leq2\exp(-2^l\beta),
\end{align}
where for all $l \ge 1$ and $t \in T$ we let 
\begin{align*}
\Omega_{l,t}=\left\{\omega: \left|\frac1N\sum_{i=1}^N\varepsilon_i\widetilde{q}_i\dotp{\widetilde{U}_i}{\pi_l(t)-\pi_{l-1}(t)}\right|\geq C\frac{2^{l/2}\beta}{\sqrt{N}}\|\pi_l(t)-\pi_{l-1}(t)\|_2\right\}.
\end{align*}
Notice that for each $l \ge 1$ the number of pairs $(\pi_l(t),\pi_{l-1}(t))$ appearing in the sum in \eqref{chaining-goal} can be bounded by 
$|\mathcal{A}_l|\cdot|\mathcal{A}_{l-1}|\leq2^{2^{l+1}}$. Thus, by a union bound and \eqref{inter-2},
\begin{align*}
\mb P\left(\bigcup_{t\in T}\Omega_{l,t}\right)
\leq \
2\cdot 2^{2^{l+1}}\exp(-2^l\beta),
\end{align*}
and hence,
\begin{align*}
\mb P \l( \bigcup_{l\in I_1,t\in T}\Omega_{l,t} \r) \leq
&\sum_{l\in I_1}2\cdot 2^{2^{l+1}}\exp(-2^l\beta)\\
\leq&\sum_{l\in I_1}2\cdot 2^{2^{l+1}}\exp\left(-2^{l-1}\beta-\beta/2\right)
\leq ce^{-\beta/2},
\end{align*}
for some absolute constant $c>0$, where in the last inequality we use the fact $\beta\geq 8$ to get a geometrically decreasing sequence. 
Thus, on the complement of the event $\cup_{l\in I_1,t\in T}\Omega_{l,t}$, we have that with probability at least $1-ce^{-\beta/2}$,
\begin{align*}
\sup_{t\in T}\left|\sum_{l\in I_1}\left(Z(\pi_l(t))-Z(\pi_{l-1}(t))\right)\right|
\leq&\sup_{t\in T}\sum_{l\in I_1}\left|Z(\pi_l(t))-Z(\pi_{l-1}(t))\right|\\
\leq&\sup_{t\in T}C\sum_{l\in I_1}\frac{2^{l/2}\beta}{\sqrt{N}}\|\pi_l(t)-\pi_{l-1}(t)\|_2\\
\leq&\sup_{t\in T}C\sum_{l=1}^{\infty}\frac{2^{l/2}\beta}{\sqrt{N}}\|\pi_l(t)-\pi_{l-1}(t)\|_2\\
\leq&4C\frac{\gamma_2(T)\beta}{\sqrt{N}},
\end{align*}
for $C=C(\kappa,\phi)$, where the last inequality follows from triangle inequality $\|\pi_l(t)-\pi_{l-1}(t)\|_2\leq\|\pi_{l-1}(t)-t\|_2+\|\pi_l(t)-t\|_2$ and \eqref{inter-1}. 
This proves the inequality \eqref{chaining-goal} for $l\in I_1$.
\end{proof}

\subsubsection{The case $l\in I_2$.}


This is the most technically involved case of the three. 
For any fixed $t\in T$ and $l\in I_2$, we let $X_i=\widetilde{q}_i \dotp{\widetilde{U}_i}{\pi_l(t)-\pi_{l-1}(t)} $ and $w_i=\langle\widetilde{U}_i,\pi_l(t)-\pi_{l-1}(t)\rangle$. Then $X_i=\widetilde{q}_i w_i$ and
\begin{equation}
\label{inter-average}
Z(\pi_l(t))-Z(\pi_{l-1}(t))=\frac1N\sum_{i=1}^N\varepsilon_iX_i
=\frac1N\sum_{i=1}^N\varepsilon_iw_i\widetilde{q}_i.
\end{equation}
For every fixed $k\in\{1,2,\cdots,N-1\}$ and fixed $u>0$, we bound the summation using the following inequality
\begin{align*}
\mb P\left(
\left|\sum_{i=1}^N\varepsilon_iX_i\right|\geq\sum_{i=1}^kX^*_i+u\left(\sum_{i=k+1}^N(X_i^*)^2\right)^{1/2}
\right)
\leq2\exp(-u^2/2),
\end{align*}
where $\{X_i^*\}_{i=1}^N$ is the \textit{non-increasing} rearrangement of $\{|X_i|\}_{i=1}^N$ and $\{\varepsilon_i\}_{i=1}^N$ is a sequence of i.i.d. Rademancher random variables independent of $\{X_i\}_{i=1}^N$.
\begin{remark}
This bound was first stated and proved in \cite{Rademancher-sums} with a sequence of fixed constants $\{X_i\}_{i=1}^N$. The current form can be obtained using independence property and conditioning on $\{X_i\}_{i=1}^N$. 
Furthermore, \cite{Rademancher-sums} tells us that the optimal choice of $k$ is at $\mathcal{O}(u^2)$
Applications of this inequality to generic chaining-type arguments were previously introduced by \cite{Mendelson-2}.
\end{remark}
Letting $J$ be the set of indices of the variables corresponding to the $k$ largest coordinates of $\{|w_i|\}_{i=1}^m$ and of $\{|\widetilde{q}_i|\}_{i=1}^m$, we have $|J|\leq2k$ and with probability at least $1 - 2\exp(-u^2/2)$
\begin{align}
\left|\sum_{i=1}^N\varepsilon_iX_i\right|
&\leq\sum_{i\in J}X^*_i+u\left(\sum_{i\in J^c}(X_i^*)^2\right)^{1/2}\nonumber\\
&\leq2\sum_{i=1}^kw_i^*\widetilde{q}_i^*+u\left(\sum_{i\in J^c}(w_i^*\widetilde{q}_i^*)^2\right)^{1/2}\nonumber\\
&\leq2\left(\sum_{i=1}^k(w_i^*)^2\right)^{1/2}\left(\sum_{i=1}^k(\widetilde{q}_i^*)^2\right)^{1/2}
+u\left(\sum_{i= k+1}^N(w_i^*)^{\frac{2(1+\kappa)}{\kappa}}\right)^{\frac{\kappa}{2(1+\kappa)}}\left(\sum_{i=k+1}^N(\widetilde{q}_i^*)^{2(1+\kappa)}\right)^{\frac{1}{2(1+\kappa)}}\nonumber\\
&\leq2\left(\sum_{i=1}^k(w_i^*)^2\right)^{1/2}\left(\sum_{i=1}^N\widetilde{q}_i^2\right)^{1/2}
+u\left(\sum_{i= k+1}^N(w_i^*)^{\frac{2(1+\kappa)}{\kappa}}\right)^{\frac{\kappa}{2(1+\kappa)}}\left(\sum_{i=1}^N\widetilde{q}_i^{2(1+\kappa)}\right)^{\frac{1}{2(1+\kappa)}}\label{inter-3}
\end{align}
where the second to last inequality is a consequence of H\"{o}lder's inequality. 
We take $u=2^{(l+1)/2}\sqrt{\beta}$. The key is to pick an appropriate cut point $k$ for each $l\in I_2$. Here, we choose $k=\lfloor2^l\beta/\log(eN/2^l\beta)\rfloor$, which makes $k=\mathcal{O}(2^l\beta)$ and also guarantees that $k\in\{1,2,\cdots,N-1\}$; see Lemma \ref{support-2}. Under this choice, we have the following lemma:

\begin{lemma}\label{cut-point-1}
Let $k=\lfloor2^l\beta/\log(eN/2^l\beta)\rfloor$, $w_i=\dotp{\widetilde{U}_i}{\pi_l(t)-\pi_{l-1}(t)}$ and $\{w_i^*\}_{i=1}^N$ be the nonincreasing rearrangement of $\{|w_i|\}_{i=1}^N$. Then there exists an absolute constant $C>1$ such that for all $\beta \ge 8$,
\[
\mb P \l( 
\left(\sum_{i=1}^k(w_i^*)^2\right)^{1/2}\geq C2^{l/2}\|\pi_l(t)-\pi_{l-1}(t)\|_2\sqrt{\beta}
\right)
\leq2\exp(-2^{l}\beta).
\]
\end{lemma}
\begin{proof}
By Lemma \ref{lemma:ball}, we know that $\{w_i\}_{i=1}^N$ are i.i.d. sub-Gaussian random variables. 
Thus, by Lemma \ref{prop-1}, $w_i^2$ is sub-exponential with norm 
\begin{equation}\label{inter-subexp}
\|w_i^2\|_{\psi_1}=2\|w_i\|_{\psi_2}^2 \leq 2\|\widetilde{U}_i\|_{\psi_2}^2\|\pi_l(t)-\pi_{l-1}(t)\|_2^2.
\end{equation}
It then follows from Bernstein's inequality (Lemma \ref{Bernstein}) that for any fixed set $J\subseteq \{1,2,\cdots,N\}$ with $|J|=k$,
\begin{align*}
\mb P \left(
\left| \frac{1}{k} \sum_{i\in J} \left(w_i^2-\expect{w_i^2}\right) \right| 
\geq 2\| \widetilde{U}_i\|_{\psi_2}^2\|\pi_l(t)-\pi_{l-1}(t) \|_2^2
\left(\sqrt{\frac{2u}{k}}+\frac{u}{k}\right)
\right)
\leq2\exp(-u).
\end{align*}
We choose $u=4\cdot2^{l}\beta=2^{l+2}\beta$. 
Since $2^l\beta\geq \lfloor2^l\beta/\log(eN/2^l\beta)\rfloor=k\geq1$, the factor $u/k$
dominates the right hand side. 
Noting that $\expect{w_i^2}=\|\pi_l(t)-\pi_{l-1}(t)\|_2^2$, we obtain
\begin{align*}
\mb P \l(
\left(\sum_{i\in J}w_i^2\right)^{1/2}\geq C2^{l/2}\|\pi_l(t)-\pi_{l-1}(t)\|_2\sqrt{\beta}
\right)
\leq2\exp(-4\cdot2^l\beta),
\end{align*}
where $C\leq 4\| \widetilde{U}_i\|_{\psi_2}$; note that the upper bound for $C$ is independent of $d$ by Lemma \ref{ball-lemma-0}. 
Thus, 
\begin{align*}
&\mb P \left(
\left(\sum_{i = 1}^k(w_i^*)^2\right)^{1/2}\geq C2^{l/2}\|\pi_l(t)-\pi_{l-1}(t)\|_2\sqrt{\beta}
\right)
\\
=& \mb P \left(
\exists J\subseteq\{1,\cdots,N\},~|J|=k:
\left(\sum_{i\in J}w_i^2\right)^{1/2}\geq C2^{l/2}\|\pi_l(t)-\pi_{l-1}(t)\|_2\sqrt{\beta}
\right)
\\
\leq&
{N \choose k}\cdot 
\mb P\left(
\left(\sum_{i\in J}w_i^2\right)^{1/2}\geq C2^{l/2}\|\pi_l(t)-\pi_{l-1}(t)\|_2\sqrt{\beta}
\right) 
\\
\leq&2{N \choose k}\exp(-4\cdot2^l\beta)\\
\leq&2\left(\frac{eN}{k}\right)^k\exp(-4\cdot2^l\beta)
\leq2\exp(-2^l\beta),
\end{align*}
where the last step follows from $\left(\frac{eN}{k}\right)^k\leq\exp(3\cdot2^l\beta)$, an inequality proved in Appendix \ref{app-A}.
\end{proof}

\begin{lemma}\label{cut-point-2}
Let $k=\lfloor2^l\beta/\log(eN/2^l\beta)\rfloor$, $w_i=\dotp{\widetilde{U}_i}{\pi_l(t)-\pi_{l-1}(t)}$ and $\{w_i^*\}_{i=1}^N$ be the non-increasing rearrangement of $\{|w_i|\}_{i=1}^N$. 
Then
\[
\mb P \left(
\left(\sum_{i=k+1}^N(w_i^*)^{\frac{2(1+\kappa)}{\kappa}}\right)^{\frac{\kappa}{2(1+\kappa)}}\geq C(\kappa) N^{\frac{\kappa}{2(1+\kappa)}}\|\pi_l(t)-\pi_{l-1}(t)\|_2
\right)
\leq\exp(-2^{l}\beta),
\]
for any $\beta\geq8$ and some constant $C(\kappa)>0$.
\end{lemma}
\begin{proof}
To avoid possible confusion, we use $i$ to index the nonincreasing rearrangement and $j$ for the original sequence.
We start by noting that $\{w_j\}_{j=1}^m$ are i.i.d. sub-Gaussian random variables with $\|w_j\|_{\psi_2} \le \|\widetilde{U}_j\|_{\psi_2}\|\pi_l(t)-\pi_{l-1}(t)\|_2$. By an equivalent definition of sub-Gaussian random variables (Lemma 5.5. of \cite{introduction-to-random-matrix}), we have for any fixed $j\in\{1,2,\ldots,N\}$,
\bea \label{eq:boundw_j.change.u}
\mb P \l(
|w_j|-\expect{|w_j|}\geq Cu\|\widetilde{U}_j\|_{\psi_2}\|\pi_l(t)-\pi_{l-1}(t)\|_2
\right)
\leq e^{-u^2},
\ena
for any $u>0$ and an absolute constant $C>0$.

To establish the claim of the lemma, we bound each
$w_i^*$ separately for $i=1,2\ldots,m$ and then combine individual bounds. 
Instead of using a fixed value of $u$ in \eqref{eq:boundw_j.change.u}, our choice of $u$ will depend on the index $i$. 
Specifically, for each $w_i^*$, we choose 
$u=c_\kappa(N/i)^{\kappa/4(1+\kappa)}$ with
\bea \label{def:ckappa}
c_\kappa:=\max\left\{\frac{\sqrt{5}\left(2+\frac{4}{\kappa}\right)^{\frac{2+\kappa}{4(1+\kappa)}}}{e^{1/2(1+\kappa)}},\sqrt{\frac{4(1+\kappa)}{\kappa}}\right\}.
\ena
The reason for this choice will be clear as we proceed. 

First, for a fixed nonincreasing rearrangement index $i>k$, 
by \eqref{eq:boundw_j.change.u} and the fact that 
\[
\expect{|w_j|}\leq\expect{w_j^2}^{1/2}=\|\pi_l(t)-\pi_{l-1}(t)\|_2,~\forall j\in\{1,2,\cdots,N\},
\]
we have
\begin{align*}
\mb P \left(
|w_j|\geq \left(1+Cc_\kappa\|\widetilde{U}_j\|_{\psi_2}\right)\left(\frac{N}{i}\right)^{\frac{\kappa}{4(1+\kappa)}}\|\pi_l(t)-\pi_{l-1}(t)\|_2
\right)
\leq \exp\left(-c_\kappa^2\left(\frac{N}{i}\right)^{\frac{\kappa}{2(1+\kappa)}}\right),&\\
\forall j\in\{1,2,\cdots,N\}.&
\end{align*}
To simplify notation, let $C'=1+Cc_\kappa\|\widetilde{U}_j\|_{\psi_2}$ (note that it depends only on $\kappa$). 
It then follows that 
\begin{align*}
& \mb P \left(
w_{i}^*\geq C'\left(\frac{N}{i}\right)^{\frac{\kappa}{4(1+\kappa)}}\|\pi_l(t)-\pi_{l-1}(t)\|_2
\right) \\
=& \mb P\left(
\exists J\subseteq\{1,\cdots,N\},~|J|=i:~w_j\geq C'\left(\frac{N}{i}\right)^{\frac{\kappa}{4(1+\kappa)}}\|\pi_l(t)-\pi_{l-1}(t)\|_2,~\forall j\in J
\right) \\
\leq&{N\choose k}
\mb P\left(
|w_j|\geq C'\left(\frac{N}{i}\right)^{\frac{\kappa}{4(1+\kappa)}}\|\pi_l(t)-\pi_{l-1}(t)\|_2\right)^i
\\
\leq&{N\choose k}\exp\left(- c^2N^{\frac{\kappa}{2(1+\kappa)}}i^{\frac{2+\kappa}{2(1+\kappa)}}\right)\\
\leq&\left(\frac{eN}{i}\right)^i
\exp\left(- c^2N^{\frac{\kappa}{2(1+\kappa)}}i^{\frac{2+\kappa}{2(1+\kappa)}}\right).
\end{align*}
By a union bound, we have
\begin{align*}
& \mb P\left(
\exists i>k: w_{i}^*\geq C'\left(\frac{N}{i}\right)^{\frac{\kappa}{4(1+\kappa)}}\|\pi_l(t)-\pi_{l-1}(t)\|_2
\right) 
\\
\leq&\sum_{i=k+1}^N\left(\frac{eN}{i}\right)^i
\exp\left(- c^2N^{\frac{\kappa}{2(1+\kappa)}}i^{\frac{2+\kappa}{2(1+\kappa)}}\right)\\ =
&\sum_{i=k+1}^N\exp\left(i\log\left(\frac{eN}{i}\right) 
- c^2N^{\frac{\kappa}{2(1+\kappa)}}i^{\frac{2+\kappa}{2(1+\kappa)}}\right)\\
\leq&N\cdot\exp\left(k\log\left(\frac{eN}{k}\right)
- c^2N^{\frac{\kappa}{2(1+\kappa)}}k^{\frac{2+\kappa}{2(1+\kappa)}}\right)\\
\leq&\exp\left(4\cdot 2^l\beta- c^2N^{\frac{\kappa}{2(1+\kappa)}}k^{\frac{2+\kappa}{2(1+\kappa)}}\right),
\end{align*}
where the second to last inequality follows since by the definition \eqref{def:ckappa} of $c_\kappa$, $c_\kappa\geq\sqrt{4(1+\kappa)/\kappa}$, the function 
$v(i)=i\log\left(\frac{eN}{i}\right) 
-c_\kappa^2N^{\frac{\kappa}{2(1+\kappa)}}\cdot i^{\frac{2+\kappa}{2(1+\kappa)}}$ 
is monotonically decreasing with respect to $i$ (recall that $i\leq N$), and thus is dominated by $v(k)$.
The final inequality follows from Lemma \ref{support-1} as well as the fact that 
$\log N \le \log(eN)\leq 2^l\beta$. 
Furthermore, by Lemma \ref{support-2} in the Appendix \ref{app-A} and \eqref{def:ckappa} implying $c_\kappa\geq\sqrt{5}\left(2+\frac{4}{\kappa}\right)^{\frac{2+\kappa}{4(1+\kappa)}}/e^{1/2(1+\kappa)}$, we have
\[
c_\kappa^2N^{\frac{\kappa}{2(1+\kappa)}}k^{\frac{2+\kappa}{2(1+\kappa)}}
\geq 5\cdot2^l\beta.
\]
Overall, we have the following bound:
\begin{align*}
\mb P\left[\exists i>k: w_{i}^*\geq C'\left(\frac{N}{i}\right)^{\frac{\kappa}{4(1+\kappa)}}\|\pi_l(t)-\pi_{l-1}(t)\|_2\right]
\leq\exp\left(4\cdot2^l\beta-5\cdot2^l\beta\right)
\leq\exp(-2^l\beta).
\end{align*}
Thus, with probability at least $1-\exp(-2^{l}\beta)$,
\[
w_{i}^*\leq C'\left(\frac{N}{i}\right)^{\frac{\kappa}{4(1+\kappa)}}\|\pi_l(t)-\pi_{l-1}(t)\|_2,~\forall i>k,
\]
hence with the same probability
\begin{align*}
\left(\sum_{i=k+1}^N(w_i^*)^{\frac{2(1+\kappa)}{\kappa}}\right)^{\frac{\kappa}{2(1+\kappa)}}
\leq&C'\|\pi_l(t)-\pi_{l-1}(t)\|_2\left(\sum_{i=k+1}\left(\frac{N}{i}\right)^{1/2}\right)^{\frac{\kappa}{2(1+\kappa)}}\\
\leq&C'\|\pi_l(t)-\pi_{l-1}(t)\|_2m^{\frac{\kappa}{4(1+\kappa)}}\left(\int_1^m\frac{dx}{x^{1/2}}\right)^{\frac{\kappa}{2(1+\kappa)}}\\
\leq&2^{\frac{\kappa}{2(1+\kappa)}}C'\|\pi_l(t)-\pi_{l-1}(t)\|_2N^{\frac{\kappa}{2(1+\kappa)}},
\end{align*}
and the desired result follows.
\end{proof}

\begin{lemma}\label{cut-point-3}
The following inequalities hold for any $\beta\geq 8$:
\begin{align*}
&\mb P\left(
\left(\sum_{i=1}^N\widetilde{q}_i^2\right)^{1/2}\geq C'\sqrt{\beta N}
\right)
\leq 2e^{-\beta}, 
\\
& \mb P\left(
\left(\sum_{i=1}^N\widetilde{q}_i^{2(1+\kappa)}\right)^{\frac{1}{2(1+\kappa)}}\geq C''(\beta N)^{\frac{1}{2(1+\kappa)}}
\right)
\leq2e^{-\beta},
\end{align*}
for some positive constants $C'=C'(\phi,\kappa), \ C''=C''(\phi,\kappa)$.
\end{lemma}
\begin{proof}
Recall that $\widetilde{q}_i=\textrm{sign}(q_i)(|q_i|\wedge\tau)$, 
$\tau=N^{1/2(1+\kappa)}$, and $\phi=\expect{q_i^{2(1+\kappa)}}$. 
Thus,
$\expect{\widetilde{q}_i^2}\leq\expect{q_i^2}\leq\phi^{1/1+\kappa}$,
and for any integer $p\geq2$, we have
\[
\expect{\widetilde{q}_i^{2p}} = \expect{\widetilde{q}_i^{2p-2(1+\kappa)}\widetilde{q}_i^{2(1+\kappa)}}
\leq m^{\frac{p-1-\kappa}{1+\kappa}}\expect{q_i^{2(1+\kappa)}}
\leq m^{\frac{p-1-\kappa}{1+\kappa}}\phi.
\]
Thus, for any $p\geq2$,
\[
\expect{|\widetilde{q}_i^2-\expect{\widetilde{q}_i^2}|^{p}}
\leq
\expect{\widetilde{q}_i^{2p}}+\left(\expect{q_i^2} \right)^p
\leq 
m^{\frac{p-1-\kappa}{1+\kappa}}\phi+\phi^{\frac{p}{1+\kappa}}
\leq 
(m+\phi)^{\frac{1-\kappa}{1+\kappa}}\phi(m+\phi)^{\frac{p-2}{1+\kappa}}.
\]
By Bernstein's inequality (Lemma \ref{Bernstein}), with probability at least $1-2e^{-\beta}$,
\begin{align*}
\left|\frac1N\sum_{i=1}^N\widetilde{q}_i^2-\expect{\widetilde{q}_i^2}\right|
&\leq\left(\frac{\sqrt{2\beta} (N+\phi)^{\frac{1-\kappa}{2(1+\kappa)}}\phi^{1/2}}{N^{1/2}}
+\frac{\beta (N+\phi)^{\frac{1}{1+\kappa}}}{m}\right)\\
&\leq\frac{\sqrt{2\beta} (1+\phi)^{\frac{1-\kappa}{2(1+\kappa)}}\phi^{1/2}+\beta (1+\phi)^{\frac{1}{1+\kappa}}}{N^{\frac{\kappa}{1+\kappa}}},
\end{align*}
which implies the first claim. 
To establish the second claim, note that for any $p \geq 2$,
\begin{align*}
\mb E\left|\widetilde{q}_i^{2(1+\kappa)}-\expect{\widetilde{q}_i^{2(1+\kappa)}}\right|^p
\leq &
C(p)\l( 
\mb E \l| \widetilde{q}_i^{2(1+\kappa)p} \r| + \l(\mb E \l| q_i^{2(1+\kappa)} \r|\r)^p \r) \\
\leq
&C(p) \l( \mb E\l| \widetilde{q}_i^{2(1+\kappa)(p-1)}q_i^{2(1+\kappa)} \r| + \phi^p\r) \\
\leq& C(p) (N^{p-1}\phi+\phi^p) \leq C(p)(N+\phi)^{p-2}(N+\phi)\phi,
\end{align*}
where we used the fact that $|\widetilde{q_i} |\leq N^{1/2(1+\kappa)}$ to obtain the third inequality. 
Bernstein's inequality implies that with probability at least $1-2e^{-\beta}$,
\[
\left|\frac1N\sum_{i=1}^N\widetilde{q}_i^{2(1+\kappa)}-\expect{\widetilde{q}_i^{2(1+\kappa)}}\right|
\leq\sqrt{2\beta}(1+\phi)\phi^{1/2}+\beta(1+\phi),
\]
which yields the second part of the claim.
\end{proof}

\begin{proof}[Proof of inequality \eqref{chaining-goal} for the index set $I_2$]
Combining Lemmas 
\ref{cut-point-1} and \ref{cut-point-2} with the inequality \eqref{inter-3}, and setting $u=2^{l/2}\sqrt{\beta}$, 
we get that with probability at least $1-4\exp(-2^l\beta)$, for all $l \in I_2$, 
\begin{align*}
|Z(\pi_l(t))- &Z(\pi_{l-1}(t))|\leq 
\\
&
C\|\pi_l(t)-\pi_{l-1}(t)\|_2\frac{2^{l/2}\sqrt{\beta}}{N}\left(\left(\sum_{i=1}^N\widetilde{q}_i^2\right)^{1/2}
+N^{\frac{\kappa}{2(1+\kappa)}}\left(\sum_{i=1}^N\widetilde{q}_i^{2(1+\kappa)}\right)^{\frac{1}{2(1+\kappa)}}
\right), 
\end{align*}
for some constant $C=C(\kappa,\phi)>0$; note that the factor $1/m$ appears due to equality \eqref{inter-average}. 
Next, we apply a chaining argument similar to the one used in Section \ref{first-chaining}, we obtain that with probability at least $1-ce^{-\beta/2}$,
\begin{align}\label{inter-4}
\sup_{t\in T}\left|\sum_{l\in I_2}\left(Z(\pi_l(t))-Z(\pi_{l-1}(t))\right)\right|
\leq C\frac{\gamma_2(T)\sqrt{\beta}}{N}\left(\left(\sum_{i=1}^N\widetilde{q}_i^2\right)^{1/2}
+N^{\frac{\kappa}{2(1+\kappa)}}\left(\sum_{i=1}^N\widetilde{q}_i^{2(1+\kappa)}\right)^{\frac{1}{2(1+\kappa)}}\right),
\end{align}
 for a positive constant $C=C(\kappa,\phi)$ and an absolute constant $c>0$. 
In order to handle the remaining terms involving $\widetilde{q}_i$ in \eqref{inter-4}, we apply 
 Lemma \ref{cut-point-3}, which gives
\begin{align*}
\sup_{t\in T}\left|\sum_{l\in I_2}\left(Z(\pi_l(t))-Z(\pi_{l-1}(t))\right)\right|
\leq C\frac{\gamma_2(T)\beta}{\sqrt{N}},
\end{align*}
with probability at least $1-ce^{-\beta/2}$, where $C=C(\kappa,\phi)$ and $c>0$ are positive constants and $\beta\geq8$. 
This completes the second part of the chaining argument.
\end{proof}

%
%

\subsubsection{The case $l\in I_3$.}
\begin{proof}[Proof of inequality \eqref{chaining-goal} for the index set $I_3$]
Direct application of Cauchy-Schwartz on \eqref{inter-average} yields, for all $t\in T$,
\[
|Z(\pi_l(t))-Z(\pi_{l-1}(t))|\leq\left(\frac1N\sum_{i=1}^Nw_i^2\right)^{1/2}\left(\frac1N\sum_{i=1}^N\widetilde{q}_i^2\right)^{1/2},
\]
where $w_i=\dotp{\widetilde{U}_i}{\pi_l(t)-\pi_{l-1}(t)}$ are sub-Gaussian random variables. 
Thus, by Lemma \ref{prop-1}, $\omega_i^2$ are sub-exponential with norm bounded as in \eqref{inter-subexp}. 
Using Bernstein's inequality again, we deduce that  
\begin{align*}
\mb P \left(
\left|\frac1N\sum_{i=1}^N\left(w_i^2-\expect{w_i^2}\right)\right|\geq2\|\widetilde{U}_i\|_{\psi_2}^2\|\pi_l(t)-\pi_{l-1}(t)\|_2^2
\left(\sqrt{\frac{2u}{N}}+\frac{u}{N}\right)
\right)
\leq2\exp(-u).
\end{align*}
Let $u=2^{l}\beta$. 
Using the fact that $2^l\beta/N\geq 1$ as well as $\expect{w_i^2}=\|\pi_l(t)-\pi_{l-1}(t)\|_2^2$, we see that the term $u/m$ dominates the right hand side and
\begin{align*}
\mb P \left(
\left(\frac1N\sum_{i=1}^Nw_i^2\right)^{1/2}\geq C\|\pi_l(t)-\pi_{l-1}(t)\|_2\frac{2^{l/2}\sqrt{\beta}}{\sqrt{N}}
\right)
\leq2\exp(-2^l\beta),
\end{align*}
for some absolute constant $C>0$. 
Thus, repeating a chaining argument of section \ref{first-chaining} (namely, the argument following \eqref{inter-2}), we obtain
\begin{align*}
\sup_{t\in T}\left|\sum_{l\in I_3}\left(Z(\pi_l(t))-Z(\pi_{l-1}(t))\right)\right|
\leq C\frac{\gamma_2(T)\sqrt{\beta}}{\sqrt{N}}\left(\frac1N\sum_{i=1}^N\widetilde{q}_i^2\right)^{1/2}
\end{align*}
with probability at least $1-ce^{-\beta/2}$ for some absolute constants $C,c>0$. 
Combining this inequality with the first claim of Lemma \ref{cut-point-3} gives
\begin{align*}
\sup_{t\in T}\left|\sum_{l\in I_3}\left(Z(\pi_l(t))-Z(\pi_{l-1}(t))\right)\right|
\leq C\frac{\gamma_2(T)\beta}{\sqrt{N}},
\end{align*}
with probability at least $1-ce^{-\beta/2}$ for absolute constants $C,c>0$ and any $\beta\geq 8$. 
This finishes the bound for the third (and final) segment of the ``chain''.
\end{proof}

\subsubsection{Finishing the proof of Lemma \ref{concentration-multiplier}}

\begin{proof}
So far, we have shown that
\begin{align}
\sup_{t\in T}\left|Z(t)-Z(t_0)\right|
=&\sup_{t\in T}\left|\sum_{l\geq1}\left(Z(\pi_l(t))-Z(\pi_{l-1}(t))\right)\right|\nonumber\\
\leq& \sum_{j\in\{1,2,3\}}\sup_{t\in T}\left|\sum_{l\in I_j}\left(Z(\pi_l(t))-Z(\pi_{l-1}(t))\right)\right|\nonumber\\
\leq&C\frac{\gamma_2(T)\beta}{\sqrt{N}}, \label{final-inter}
\end{align}
with probability at least $1-ce^{-\beta/2}$ for some positive constants $C=C(\kappa,\phi)$ and $c$, and any $\beta\geq8$. 
To finish the proof, it remains to bound 
$|Z(t_0)|=\left|\frac1N\sum_{i=1}^N\varepsilon_i\widetilde{q}_i
\dotp{\widetilde{U}_i}{t_0} \right|$. 
With $\Delta_d(T)$ defined in \eqref{def:Deltad}, and since $t_0$ is an arbitrary point in $T$, we trivially have $\|t_0\|_2\leq\Delta_d(T)$. 
Applying Bernstein's inequality in a way similar to Section \ref{first-chaining} yields
\begin{align*}
\mb P \left(
\left|\frac1N\sum_{i=1}^N\varepsilon_i\widetilde{q}_i\dotp{\widetilde{U}_i}{t_0} \right|\geq\left(\frac{C'\sqrt{2u}}{\sqrt{N}}+\frac{C'' u}{N^{1-\frac{1}{2(1+\kappa)}}}\right)\Delta_d(T)
\right)
\leq 2e^{-u},
\end{align*}
for some constants $C'=C'(\kappa,\phi), \ C''=C''(\kappa,\phi)>0$ and any $u>0$. 
Choosing $u=\beta$ gives 
\begin{align*}
\mb P \left(
\left|\frac1N\sum_{i=1}^N\varepsilon_i\widetilde{q}_i\dotp{\widetilde{U}_i}{t_0} \right|\geq 
\frac{C\Delta_d(T)\beta}{\sqrt{N}}
\right)
\leq 2e^{-\beta},
\end{align*}
for a constant $C=C(\kappa,\phi)>0$ and any $\beta\geq 0$.
Combining this bound with \eqref{final-inter} shows that with probability at least $1-ce^{-\beta/2}$, 
\[
\sup_{t\in T}\left|\frac1N\sum_{i=1}^N\varepsilon_i\langle\widetilde{U}_i,t\rangle \widetilde{q}_i\right|
\leq C\frac{(\gamma_2(T)+\Delta_d(T))\beta}{\sqrt{N}}
\leq C\frac{(L\omega(T)+\Delta_d(T))\beta}{\sqrt{N}},
\]
for $C=C(\kappa,\phi)$, an absolute constant $L>0$ and all $\beta\geq8$; note that the last inequality follows from Lemma \ref{mmt}. 
We have established \eqref{major-criterion}, thus completing the proof.
\end{proof}


\section{Technical Results.}
\label{app-A}

\begin{lemma}
\label{basic-inequality}
For any nonnegative random variable $X$, if 
$\mb P \l( X>K\beta \r) \leq ce^{-\beta/2}$ for some constants $K,c>0$ and all $\beta\geq\beta_0\geq 0$, then,
\[
\expect{X}\leq K\left(\beta_0+2ce^{-\beta_0/2}\right).\]
\end{lemma}
\begin{proof}
Using a well known identity for the expectation of non-negative random variables,
\begin{align*}
\expect{X}=&\int_0^{\infty}\mb P\l(X>u\r)du
=K\int_0^{\infty}\mb P\l( X>K\beta \r)d\beta\\
\leq&K \left( \beta_0+\int_{\beta_0}^{\infty}\mb P\l( X>K\beta \r) d\beta \right)
\leq K\left( \beta_0+\int_{\beta_0}^{\infty}ce^{-\beta/2}d\beta\right)\\
=&K\left(\beta_0+2ce^{-\beta_0/2}\right).
\end{align*}
\end{proof}


\begin{lemma} 
\label{prop-1}
If $X$ and $Y$ are sub-Gaussian random variables, then the product $XY$ is a subexponential random variable, and
\[
\|XY\|_{\psi_1}\leq \|X\|_{\psi_2}\|Y\|_{\psi_2}.
\]
\end{lemma}
\begin{proof}
See \cite{wellner1}.

\end{proof}

\begin{lemma}\label{support-1}
Let $k=\lfloor2^l\beta/\log(eN/2^l\beta)\rfloor$ and $l\in I_2$, then,
$\left(\frac{eN}{k}\right)^k\leq\exp(3\cdot2^l\beta).$
\end{lemma}
\begin{proof}
If $k\geq2$, then, $2^l\beta/\log(eN/2^l\beta)\geq2$, which implies $2^l\beta\geq2\log(eN/2^l\beta)$. Thus,
\begin{align*}
\left(\frac{eN}{k}\right)^k
\leq&2\exp\left(\frac{2^l\beta}{\log\frac{eN}{2^l\beta}}\log\left(\frac{eN}{\frac{2^l\beta}{\log\frac{eN}{2^l\beta}}-1}\right)\right)\\
\leq&2\exp\left(\frac{2^l\beta}{\log\frac{eN}{2^l\beta}}\log\left(\frac{eN}{2^l\beta-\log\frac{eN}{2^l\beta}}\log\frac{eN}{2^l\beta}\right)\right)\\
\leq&2\exp\left(\frac{2^l\beta}{\log\frac{eN}{2^l\beta}}\log\left(\frac{2eN}{2^l\beta}\log\frac{eN}{2^l\beta}\right)\right)\leq\exp(3\cdot2^l\beta),
\end{align*}
where the second from last inequality follows from $\left(\frac{eN}{k}\right)^k\leq\exp(3\cdot2^l\beta)$, and the last inequality follows from $N\geq2^l\beta$, thus, $\log(2eN/2^l\beta)/\log(eN/2^l\beta)\leq2$.

On the other hand, if $k=1$, then, since $\log eN\leq2^l\beta$,
$
\left(\frac{eN}{k}\right)^k=eN=\exp(\log eN)\leq\exp(2^l\beta),
$
finishing the proof. 
\end{proof}

\begin{lemma}
\label{support-2}
With $N \ge 1, \beta \ge 1, \kappa \in (1,0)$ and 
$l\in I_2=\{l\geq1:\log eN\leq 2^{l}\beta< N\}$, the integer
$k=\lfloor2^l\beta/\log(eN/2^l\beta)\rfloor$ satisfies $k \ge 1$, and
\[
\frac{\left(2+\frac{4}{\kappa}\right)^{\frac{2+\kappa}{2(1+\kappa)}}}{e^{1/(1+\kappa)}}N^{\frac{\kappa}{2(1+\kappa)}}k^{\frac{2+\kappa}{2(1+\kappa)}}\geq 2^l\beta.
\] \end{lemma}

\begin{proof}
Since $2^l \beta\geq\log(eN) \ge 1$, 
it follows that $k\geq1$, and thus $k\geq2^l\beta/2\log(eN/2^l\beta)$. It is then enough to show that
\[
\frac{\left(1 + \frac{2}{\kappa}\right)^{\frac{2+\kappa}{2(1+\kappa)}}}{e^{1/(1+\kappa)}}
\left(\frac{N}{2^l\beta}\right)^{\frac{\kappa}{2(1+\kappa)}}
\geq
\left(\log\frac{eN}{2^l\beta}\right)^{\frac{2+\kappa}{2(1+\kappa)}}.
\]
Raising both sides to the power of $2(1+\kappa)/\kappa$, equivalently
\[
\left.\left(
1+\frac2\kappa\right)^{\frac{2+\kappa}{\kappa}}\right/e^{\frac2\kappa}
\geq
\left.\left(
\log\frac{eN}{2^l\beta}\right)^{\frac{2+\kappa}{\kappa}}\right/\frac{N}{2^l\beta}.
\] 
Consider the function $g(x)=\left(\log ex\right)^{\frac{2+\kappa}{\kappa}}/x$. Note that as $m>2^l\beta$, to prove the inequality above it suffices to show that the $\sup_{x \ge 1}g(x)$ is upper bounded by the left hand side. Taking the derivative of $g(x)$ yields
\[g'(x)=\frac{\frac{2+\kappa}{\kappa}(1+\log x)^{2/\kappa}-(1+\log x)^{(2+\kappa)/\kappa}}{x^2}.\]
Since $x\geq1$, the only critical point at which the global maximum occurs is given by $x=e^{2/\kappa}$. 
As $g\left(e^{2/\kappa}\right)$ is exactly equal to the left hand side the proof is complete.
\end{proof}

\section{Decomposable Norms and Restricted Compatibility.}
\label{app-B}

In this section, we recall some facts about decomposable norms that have been introduced in \cite{general-m-estimator}. 
\begin{definition}
\label{def:decomposable}
Suppose that $\mathcal{L}\subseteq \m L_1$ are two subspace of $\mathbb{R}^d$, and let $\mathcal{L}_1^{\perp}$ be the orthogonal complement of $\mathcal{L}_1$. 
Norm $\|\cdot\|_{\mathcal{K}}$ is said to be decomposable with respect to $(\mathcal{L},~\mathcal{L}_1^\perp)$ if for any $\theta\in\mathbb{R}^d$,
\begin{align*}
\|\theta_1+\theta_2\|_{\mathcal{K}}=\| \Pi_{\mathcal{L}}\theta_1\|_{\mathcal{K}}+\| \Pi_{\mathcal{L}_1^{\perp}}\theta\|_{\mathcal{K}},
\end{align*}
where $\Pi_{\mathcal{L}}$ and $\Pi_{\mathcal{L}_1^{\perp}}$ stand for the orthogonal projectors onto $\mathcal{L}$ and $\mathcal{L}_1^\perp$ respectively.
\end{definition}

It is well known that many frequently used norms, including the $\ell_1$ norm of a vector and the nuclear norm of a matrix, are decomposable with respect to the appropriately chosen pair of subspaces. 
For instance, the $\ell_1$ norm is decomposable with respect to the pair of subspaces $(\m L(J), \m L(J)^\perp)$, where 
\begin{align}
\label{eq:L(J)}
&
\m L(J):=\l\{ v\in \mb R^d: \ v_j=0 \text{ for all } j\notin J\r\}
\end{align} 
consists of sparse vectors with non-zero coordinates indexed by a set $J\subseteq \l\{ 1,\ldots, d\r\}$. 

Let $W_1\subseteq \mb R^{d_1}, \ W_2\subseteq \mb R^{d_2}$ be two linear subspaces. 
Then we define the subspace $\m L(W_1,W_2)\subseteq \mb R^{d_1\times d_2}$ via 
\begin{align*}
&
\m L(W_1,W_2):= \l\{ M\in \mb R^{d_1\times d_2}: \ \mathrm{row}(M)\subseteq W_1, \ \mathrm{col}(M)\subseteq W_2  \r\},
\end{align*}
where $\mathrm{row}(M)$ and $\mathrm{col}(M)$ are the linear subspaces spanned by the rows and columns of $M$ respectively, and 
\begin{align}
\label{eq:L_12}
&
\m L_1^\perp(W_1,W_2):= \l\{ M\in \mb R^{d_1\times d_2}: \ \mathrm{row}(M)\subseteq W_1^\perp, \ \mathrm{col}(M)\subseteq W_2^\perp  \r\}.
\end{align}
Then the nuclear norm $\|\cdot\|_\ast$ is decomposable with respect to $\l( \m L(W_1,W_2),\m L_1^\perp(W_1,W_2 )\r)$ (see \cite{general-m-estimator} for details). 

Assume that the norm $\|\cdot\|_\m K$ is decomposable with respect to $(\m L,\m L_1^\perp)$, and let $\theta\in \m L$. 
It is clear that for any $\mathbf{v}\in\mb{S}_{c_0}(\theta)$
\begin{equation}
\label{app-1}
\|\theta+\mathbf{v}\|_{\mathcal{K}}=
\| \Pi_{\mathcal{L}}\theta+\Pi_{\mathcal{L}_1}\mathbf{v}+\Pi_{\mathcal{L}_1^{\perp}}\mathbf{v}\|_{\mathcal{K}}
\leq \| \Pi_{\mathcal{L}}\theta\|_{\mathcal{K}}+\frac{1}{c_0}\| \Pi_{\mathcal{L}_1}\mathbf{v} \|_{\m K}+ \| \Pi_{\mathcal{L}_1^{\perp}}\mathbf{v}\|_{\mathcal{K}}.
\end{equation}
Since $\theta\in\mathcal{L}$, decomposability and the triangle inequality imply that 
\begin{align*}
\| \Pi_{\mathcal{L}}\theta+\Pi_{\mathcal{L}_1}\mathbf{v}+\Pi_{\mathcal{L}_1^{\perp}}\mathbf{v}\|_{\mathcal{K}}
&
=\|\Pi_{\mathcal{L}}\theta+\Pi_{\mathcal{L}_1}\mathbf{v}\|_{\mathcal{K}} 
+\|\Pi_{\mathcal{L}_1^{\perp}}\mathbf{v}\|_{\mathcal{K}} \\
&
\geq\|\Pi_{\mathcal{L}}\theta\|_{\mathcal{K}}-\|\Pi_{\mathcal{L}_1}\mathbf{v}\|_{\mathcal{K}}
+\|\Pi_{\mathcal{L}_1^{\perp}}\mathbf{v}\|_{\mathcal{K}}.
\end{align*}
Substituting this bound into \eqref{app-1} gives
\begin{align*}
-\| \Pi_{\mathcal{L}}\mathbf{v}\|_{\mathcal{K}}+\| \Pi_{\mathcal{L}_1^{\perp}}\mathbf{v}\|_{\mathcal{K}}
\leq\frac{1}{c_0}\| \Pi_{\mathcal{L}_1}\mathbf{v}\|_{\mathcal{K}}+\frac{1}{c_0}\| \Pi_{\mathcal{L}_1^{\perp}}\mathbf{v}\|_{\mathcal{K}},
\end{align*}
which implies that for any $\mathbf{v}\in\mb{S}_{c_0}(\theta)$
\[
\| \Pi_{\mathcal{L}_1^{\perp}}\mathbf{v}\|_{\mathcal{K}} \leq 
\frac{c_0+1}{c_0-1}\| \Pi_{\mathcal{L}_1}\mathbf{v}\|_{\mathcal{K}}.
\]
It is easy to see that the set of all $\mf{v}$ satisfying the inequality above is a convex cone, which we will denote by $C_{c_0}=C_{c_0}(\m K)$. 
Since ${\mathbb{S}}_{c_0}(\theta)\subseteq C_{c_0}$,  
\[
\Psi\left(\mb{S}_{c_0}(\theta)\right)\leq\Psi\left(C_{c_0}\right)
\]
by definition of the restricted compatibility constant. 
This inequality is useful due to the fact that it is often easier to estimate $\Psi\left(C_{c_0}\right)$. 

Finally, we make a remark that is useful when dealing with non-isotropic measurements. 
Let $\mf\Sigma\succ 0$ be a $d\times d$ matrix, and consider the norm corresponding to the convex set $\mf \Sigma^{1/2}\m K$, so that 
$\|\mf v\|_{\mf \Sigma^{1/2}\m K}=\|\mf\Sigma^{-1/2}\mf v\|_\m K$. 
It is easy to see that $C_{c_0}(\mf \Sigma^{1/2}\m K) = \mf\Sigma^{1/2}C_{c_0}(\m K)$, hence 
\begin{align*}
\Psi \l( C_{c_0}(\mf \Sigma^{1/2}\m K); \mf\Sigma^{1/2} \m K \r) & = 
\sup_{\mf v\in \mf\Sigma^{1/2}\m K \setminus \{0\}}\frac{\| \mf v\|_{\mf\Sigma^{1/2}\m K}}{\|\mf v\|_2} = 
\sup_{\mf u\in \m K\setminus \{0\}}\frac{\|\mf u\|_\m K}{\| \mf \Sigma^{1/2}\mf u\|_2} \\
&
\leq \| \mf\Sigma^{-1/2}\| \, \Psi \l( C_{c_0}(\m K); \m K \r). 
\end{align*}
\textbf{Example 1: $\ell_1$ norm. }
Let $\m L(J)$ be as in \eqref{eq:L(J)} with $|J|=s\leq d$. 
If $v\in \mb R^d$ belongs to the corresponding cone $C(c_0)$, then clearly $\|v\|_1\leq \frac{2c_0}{c_0-1}\|v_J\|_1$, where $v_J:=\Pi_{\m L(J)} v$. 
Hence 
\[
\|v\|_1 \leq \frac{2c_0}{c_0-1}\|v_J\|_1\leq \frac{2c_0}{c_0-1}\sqrt{|J|}\|v\|_2,
\] 
and $\Psi(C_{c_0})\leq \frac{2c_0}{c_0-1}\sqrt{s}.$
\\
\textbf{Example 2: nuclear norm.}
Let $\m L_1^\perp(W_1,W_2)$ be as in \eqref{eq:L_12}. 
Note that for any $v\in \mb R^{d_1\times d_2}$, 
$\Pi_{\m L_1^\perp(W_1,W_2)}v = \Pi_{W_2^\perp} v \Pi_{W_1^\perp}$, where $\Pi_{W_1^\perp}$ and $\Pi_{W_2^\perp}$ are the orthogonal projectors onto subspaces $W_1\subseteq \mb R^{d_1}$ and $W_2\subseteq \mb R^{d_2}$ respectively. 
Then for any $v\in C_{c_0}$, we have that 
\begin{align}
\label{eq:rank}
\|v\|_\ast \leq \| \Pi_{\m L_1^\perp(W_1,W_2)}v \|_\ast + \| \Pi_{\m L_1(W_1,W_2)}v \|_\ast \leq \frac{2c_0}{c_0-1}\| \Pi_{\m L_1(W_1,W_2)}v \|_\ast.
\end{align}
Note that 
\begin{align*}
& 
\Pi_{\m L_1(W_1,W_2)}v = v - \Pi_{W_2^\perp} v \Pi_{W_1^\perp} =  \Pi_{W_2^\perp} v \Pi_{W_1} + \Pi_{W_2}v, 
\end{align*}
hence $\mathrm{rank}\l(  \Pi_{\m L_1(W_1,W_2)}v \r)\leq 2\max\l(\dim(W_1),\dim(W_2)\r)$, which yields together with \eqref{eq:rank} that
\[
\|v\|_\ast \leq \frac{2c_0}{c_0-1}\| \Pi_{\m L_1(W_1,W_2)}v \|_\ast 
\leq \frac{2c_0}{c_0-1}\sqrt{2\max\l(\dim(W_1),\dim(W_2)\r)}\|v\|_2, 
\]
and 
$\Psi(C_{c_0})\leq \frac{2\sqrt{2}c_0}{c_0-1}\sqrt{\max\l(\dim(W_1),\dim(W_2)\r)}.$



\chapter{Estimation of the Covariance Structure of Heavy-tailed Distributions}
In this chapter, we propose and analyze a new estimator of the covariance matrix that admits strong theoretical guarantees under weak assumptions on the underlying distribution, such as existence of moments of only low order. 
While estimation of covariance matrices corresponding to sub-Gaussian distributions is well-understood, much less in known in the case of heavy-tailed data. 
As K. Balasubramanian and M. Yuan write \cite{balasubramanian2016discussion}, ``data from real-world experiments oftentimes tend to be corrupted with outliers and/or exhibit heavy tails. In such cases, it is not clear that those covariance matrix estimators .. remain optimal'' and ``..what are the other possible strategies to deal with heavy tailed distributions warrant further studies.'' 
We make a step towards answering this question and prove tight deviation inequalities for the proposed estimator that depend only on the parameters controlling the ``intrinsic dimension'' associated to the covariance matrix (as opposed to the dimension of the ambient space); in particular, our results are applicable in the case of high-dimensional observations.

\section{Introduction}

Estimation of the covariance matrix is one of the fundamental problems in data analysis: many important statistical tools, such as Principal Component Analysis(PCA) \cite{hotelling1933analysis} and regression analysis, involve covariance estimation as a crucial step. 
For instance, PCA has immediate applications to nonlinear dimension reduction and manifold learning techniques \cite{allard2012multi}, genetics \cite{novembre2008genes}, computational biology \cite{alter2000singular}, among many others. 

However, assumptions underlying the theoretical analysis of most existing estimators, such as various modifications of the sample covariance matrix, are often restrictive and do not hold for real-world scenarios. 
Usually, such estimators rely on heuristic (and often bias-producing) data preprocessing, such as outlier removal.  
To eliminate such preprocessing step from the equation, one has to develop a class of new statistical estimators that admit strong performance guarantees, such as exponentially tight concentration around the unknown parameter of interest, under weak assumptions on the underlying distribution, such as existence of moments of only low order. 
In particular, such heavy-tailed distributions serve as a viable model for data corrupted with outliers -- an almost inevitable scenario for applications. 

We make a step towards solving this problem: using tools from the random matrix theory, we will develop a class of \textit{robust} estimators that are numerically tractable and are supported by strong theoretical evidence under much weaker conditions than currently available analogues. The term ``robustness'' refers to the fact that our estimators admit provably good performance even when the underlying distribution is heavy-tailed.

\subsection{Notation}

Given $A\in \mb R^{d_1\times d_2}$, let $A^T\in \mb R^{d_2\times d_1}$ be transpose of $A$. 
If $A$ is symmetric, we will write $\lambda_{\mx}(A)$ and $\lambda_{\mn}(A)$ for the largest and smallest eigenvalues of $A$. 
Next, we will introduce the matrix norms used in the chapter. 
Everywhere below, $\|\cdot\|$ stands for the operator norm $\|A\|:=\sqrt{\lambda_{\mx}(A^T A)}$. 
If $d_1=d_2=d$, we denote by $\tr A$ the trace of $A$.
 For $A\in \mb R^{d_1\times d_2}$, the nuclear norm $\|\cdot\|_1$ is defined as 
$\|A\|_1=\tr(\sqrt{A^T A})$, where $\sqrt{A^T A}$ is a nonnegative definite matrix such that $(\sqrt{A^T A})^2=A^T A$. 
The Frobenius (or Hilbert-Schmidt) norm is $\|A\|_{\mathrm{F}}=\sqrt{\tr(A^T A)}$, and the associated inner product is 
$\dotp{A_1}{A_2}=\tr(A_1^\ast A_2)$. 
For $z\in \mb R^d$, $\l\| z \r\|_2$ stands for the usual Euclidean norm of $z$. 
Let $A$, $B$ be two self-adjoint matrices. We will write $A\succeq B \ (\text{or }A\succ B)$ iff $A-B$ is nonnegative (or positive) definite.
For $a,b\in \mb R$, we set $a\vee b:=\max(a,b)$ and $a\wedge b:=\min(a,b)$. 
We will also use the standard Big-O and little-o notation when necessary.  

Finally, we give a definition of a matrix function. 
Let $f$ be a real-valued function defined on an interval $\mb T\subseteq \mb R$, and let $A\in \mb R^{d\times d}$ be a symmetric matrix with the eigenvalue decomposition 
$A=U\Lambda U^\ast$ such that $\lambda_j(A)\in \mb T,\ j=1,\ldots,d$. 
We define $f(A)$ as 
$f(A)=Uf(\Lambda) U^\ast$, where 
\[
f(\Lambda)=f\l( \begin{pmatrix}
\lambda_1 & \,  & \,\\
\, & \ddots & \, \\
\, & \, & \lambda_d
\end{pmatrix} \r)
:=\begin{pmatrix}
f(\lambda_1) & \,  & \,\\
\, & \ddots & \, \\
\, & \, & f(\lambda_d)
\end{pmatrix}.
\] 
Few comments about organization of the material in the rest of the chapter: section \ref{sec:literature} provides an overview of the related work. Section \ref{sec:main} contains the mains results of the chapter. 
The proofs are outlined in section \ref{sec:proofs}; longer technical arguments can be found in the supplementary material. 

\subsection{Problem formulation and overview of the existing work}
\label{sec:literature}

Let $X\in \mb R^d$ be a random vector with 
mean $\mb EX = \mu_0$, covariance matrix $\Sigma_0 = \mb E\l[ (X - \mu_0)(X - \mu_0)^T \r]$, and assume $\mb E \|X - \mu_0\|_2^4<\infty$. 
Let $X_1,\ldots, X_{m}$ be i.i.d. copies of $X$. 
Our goal is to estimate the covariance matrix $\Sigma$ from $X_j, \ j\leq m$. 
This problem and its variations have previously received significant attention by the research community: excellent expository chapters by \cite{cai2016} and \cite{fan2016overview} discuss the topic in detail. 
However, strong guarantees for the best known estimators hold (with few exceptions mentioned below) under the restrictive assumption that $X$ is either bounded with probability 1 or has sub-Gaussian distribution, meaning that there exists $\sigma>0$ such that for any $v\in \mb R^d$ of unit Euclidean norm, 
\[
\Pr\l(\l|\dotp{v}{X-\mu_0}\r|\geq t \r)\leq 2 e^{-\frac{t^2 \sigma^2}{2}}.
\]
In the discussion accompanying the chapter by \cite{cai2016}, \cite{balasubramanian2016discussion} write that ``data from real-world experiments oftentimes tend to be corrupted with outliers and/or exhibit heavy tails.
In such cases, it is not clear that those covariance matrix estimators described in this article remain optimal'' and ``..what are the other possible strategies to deal with heavy tailed distributions warrant further studies.'' 
This motivates our main goal: develop new estimators of the covariance matrix that (i) are computationally tractable and perform well when applied to heavy-tailed data and (ii) admit strong theoretical guarantees (such as exponentially tight concentration around the unknown covariance matrix) under weak assumptions on the underlying distribution. 
Note that, unlike the majority of existing literature, we do not impose \textit{any further conditions} on the moments of $X$, or on the ``shape'' of its distribution, such as elliptical symmetry. 

Robust estimators of covariance and scatter have been studied extensively during the past few decades. 
However, majority of rigorous theoretical results were obtained for the class of elliptically symmetric distributions which is a natural generalization of the Gaussian distribution; we mention just a small subsample among the thousands of published works.  
Notable examples include the Minimum Covariance Determinant estimator and the Minimum Volume Ellipsoid estimator which are discussed in \cite{hubert2008high}, as well Tyler's \cite{tyler1987distribution} M-estimator of scatter. 
Works by \cite{fan2016overview,wegkamp2016adaptive, han2016eca} exploit the connection between Kendall's tau and Pearson's correlation coefficient \cite{fang1990symmetric} in the context of elliptical distributions to obtain robust estimators of correlation matrices. 
Interesting results for shrinkage-type estimators have been obtained by \cite{ledoit2004well,ledoit2012nonlinear}. 
In a recent work, \cite{chen2015robust} study Huber's $\eps$-contamination model which assumes that the data is generated from the distribution of the form $(1-\eps)F + \eps Q$, where $Q$ is an arbitrary distribution of ``outliers'' and $F$ is an elliptical distribution of ``inliers'', and propose novel estimator based on the notion of ``matrix depth'' which is related to Tukey's depth function \cite{tukey1975mathematics}; a related class of problems has been studies by \cite{diakonikolas2016robust}.  
The main difference of the approach investigated in this chapter is the ability to handle a much wider class of distributions that are not elliptically symmetric and only satisfy weak moment assumptions. 
Recent papers by \cite{catoni2016pac}, \cite{giulini2015pac}, \cite{fan2016eigenvector,fan2017estimation,fan2017robust} and \cite{minsker2016sub} are closest in spirit to this direction. 
For instance, \cite{catoni2016pac} constructs a robust estimator of the Gram matrix of a random vector $Z\in \mb R^d$ (as well as its covariance matrix) via estimating the quadratic form $\mb E \dotp{Z}{u}^2$ uniformly over all $\|u\|_2=1$. 
However, the bounds are obtained under conditions more stringent than those required by our framework, and resulting estimators are difficult to evaluate in applications even for data of moderate dimension. 
\cite{fan2016eigenvector} obtain bounds in norms other than the operator norm which the focus of the present chapter. 
\cite{minsker2016sub} and \cite{fan2016robust} use adaptive truncation arguments to construct robust estimators of the covariance matrix. 
However, their results are only applicable to the situation when the data is centered (that is, $\mu_0=0$). 
In the robust estimation framework, rigorous extension of the arguments to the case of non-centered high-dimensional observations is non-trivial and requires new tools, especially if one wants to avoid statistically inefficient procedures such as sample splitting. 
We formulate and prove such extensions in this chapter.



\section{Main Results}
\label{sec:main}

Definition of our estimator has its roots in the technique proposed by \cite{catoni2012challenging}. 
Let
\begin{align}
\label{eq:psi2}
\psi(x) = \l( |x|\wedge 1 \r)\sign(x)
\end{align}
be the usual truncation function. 
As before, let $X_1,\ldots,X_m$ be i.i.d. copies of $X$, and assume that $\widehat \mu$ is a suitable estimator of the mean $\mu_0$ from these samples, to be specified later. 
We define $\widehat\Sigma$ as 
\begin{align}
\label{eq:rob-cov}
\widehat\Sigma := \frac{1}{m\theta}\sum_{i=1}^m \psi\l( \theta(X_i - \widehat\mu)(X_i - \widehat\mu)^T \r),  
\end{align}
where $\theta\simeq m^{-1/2}$ is small (the exact value will be given later). 
It easily follows from the definition of the matrix function that 
\[
\widehat\Sigma = \frac{1}{m\theta}\sum_{i=1}^m \frac{(X_i - \widehat\mu)(X_i - \widehat\mu)^T}{\l\| X_i - \widehat\mu\r\|_2^2} \psi\l( \theta \l\| X_i - \widehat\mu \r\|_2^2 \r),
\]
hence it is easily computable. 
Note that $\psi(x)= x$ in the neighborhood of $0$; it implies that whenever all random variables $\theta \l\| X_i - \widehat\mu \r\|_2^2, \ 1\leq i\leq m$ are ``small'' (say, bounded above by $1$) and $\hat\mu$ is the sample mean, $\widehat \Sigma$ is close to the usual sample covariance estimator. 
On the other hand, $\psi$ ``truncates'' $\l\| X_i - \widehat\mu \r\|_2^2$ on level $\simeq \sqrt{m}$, thus limiting the effect of outliers. 
Our results (formally stated below, see Theorem \ref{th:lepski}) imply that for an appropriate choice of $\theta=\theta(t,m,\sigma)$,
\[
\l\| \widehat\Sigma - \Sigma_0 \r\| \leq C_0\sigma_0\sqrt{\frac{\beta}{m}} 
\]
with probability $\geq 1 - d e^{-\beta}$ for some positive constant $C_0$, where 
\[
\sigma_0^2 := \l\| \mb E \l\| X - \mu_0\r\|_2^2 (X - \mu_0)(X - \mu_0)^T \r\|
\] 
is the "matrix variance".

\subsection{Robust mean estimation}
\label{ssec:mean}

There are several ways to construct a suitable estimator of the mean $\mu_0$. 
We present the one obtained via the ``median-of-means'' approach. 
Let $x_1,\ldots,x_k\in \mb R^d$. Recall that the \textit{geometric median} of $x_1,\ldots,x_k$ is defined as
\[
Med{x_1,\ldots,x_k}:=\argmin\limits_{z\in \mb R^d}\sum_{j=1}^k \l\|z- x_j \r\|_2.
\]
Let $1<\beta<\infty$ be the confidence parameter, and set $k=\Big\lfloor 3.5 \beta\Big\rfloor+1$; we will assume that $k\leq \frac{m}{2}$.
Divide the sample $X_1,\ldots, X_m$ into $k$ disjoint groups $G_1,\ldots, G_k$ of size $\Big\lfloor \frac{m}{k}\Big\rfloor$ each, and define 
\begin{align}
\label{eq:median_mean}
\nonumber
\hat\mu_j&:=\frac{1}{|G_j|}\sum_{i\in G_j}X_i, \ j=1\ldots k,\\
\hat\mu&:=Med{\hat\mu_1,\ldots,\hat\mu_k}.
\end{align}
It then follows from Corollary 4.1 in \cite{minsker2013geometric} that
\begin{align}
\label{eq:deviation1}
&
\Pr\Big(\l\| \hat\mu-\mu \r\|_2 \geq 11\sqrt{\frac{\tr(\Sigma_0)(\beta+1)}{m}}\Big)\leq e^{-\beta}.
\end{align}

\subsection{Robust covariance estimation}

Let $\widehat\Sigma$ be the estimator defined in \eqref{eq:rob-cov} with $\widehat\mu$ being the ``median-of-means'' estimator \eqref{eq:median_mean}. 
Then $\widehat \Sigma$ admits the following performance guarantees:
\begin{lemma}
\label{lemma:main}
Assume that $\sigma \geq \sigma_0$, and set $\theta=\frac{1}{\sigma}\sqrt{\frac{\beta}{m}}$.
Moreover, let $\overline{d}:=\sigma_0^2/\|\Sigma_0\|^2$, and suppose that $m\geq C\overline{d}\beta$, where $C>0$ is an absolute constant. Then
\begin{equation}
\label{simple-bound}
\left\| \widehat\Sigma - \Sigma_0\right\| \leq 3\sigma\sqrt{\frac{\beta}{m}}
\end{equation}
with probability at least $1-5d e^{-\beta}$.
\end{lemma}
\begin{remark}
The quantity $\bar d$ is a measure of ``intrinsic dimension'' akin to the ``effective rank'' $r=\frac{\tr\l(\Sigma_0\r)}{\|\Sigma_0\|}$; see Lemma \ref{effective-rank-bound} below for more details. 
Moreover, note that the claim of Lemma \ref{lemma:main} holds for any $\sigma\geq\sigma_0$, rather than just for $\sigma=\sigma_0$; this ``degree of freedom'' allows construction of adaptive estimators, as it is shown below.  
\end{remark}


The statement above suggests that one has to know the value of (or a tight upper bound on) the ``matrix variance'' 
$\sigma_0^2$ in order to obtain a good estimator $\widehat\Sigma$. 
More often than not, such information is unavailable. 
To make the estimator completely data-dependent, we will use Lepski's method \cite{lepskii1992asymptotically}. 
To this end, assume that $\sigma_{\mn}, \ \sigma_{\mx}$ are ``crude'' preliminary bounds such that 
\[
\sigma_{\mn}\leq \sigma_0 \leq \sigma_{\mx}.
\]
Usually, $\sigma_{\mn}$ and $\sigma_{\mx}$ do not need to be precise, and can potentially differ from $\sigma_0$ by several orders of magnitude. Set 
\[
\sigma_j := \sigma_{\mn} 2^j \text{ and }
\m J=\l\{ j\in \mb Z: \  \sigma_{\mn} \leq \sigma_j  < 2\sigma_{\mx} \r\}.
\]
Note that the cardinality of $J$ satisfies $\card(\m J)\leq 1+\log_2(\sigma_{\mx}/\sigma_{\mn})$. 
For each $j\in \m J$, define $\theta_j:=\theta(j,\beta) = \frac{1}{\sigma_j} \sqrt{\frac{\beta}{m}}$. 
Define
\[
\widehat\Sigma_{m,j}=\frac{1}{m\theta_j}\sum_{i=1}^m \psi\l( \theta_j (X_i-\widehat\mu)(X_i-\widehat\mu)^T \r).
\]
Finally, set 
\begin{align}
\label{eq:lepski}
j_\ast:=\min\l\{ j\in \m J: \forall k>j \text{ s.t. } k\in \m J,\ \l\|  \widehat\Sigma_{m,k} - \widehat\Sigma_{m,j} \r\|\leq 6 \sigma_{k} \sqrt{\frac{\beta}{m}}  \r\}
\end{align}
and $\widehat\Sigma_\ast:=\widehat\Sigma_{m,j_\ast}$. 
Note that the estimator $\widehat\Sigma_\ast$ depends only on $X_1,\ldots,X_m$, as well as $\sigma_{\mn}, \ \sigma_{\mx}$. 
Our main result is the following statement regarding the performance of the data-dependent estimator $\widehat\Sigma_\ast$:
\begin{theorem}
\label{th:lepski}
Suppose $m\geq C\overline{d}\beta$, then,
the following inequality holds with probability at least $1 - 5d \log_2\l(\frac{2\sigma_{\mx}}{\sigma_{\mn}}\r) e^{-\beta}$:
\[
\l\| \widehat\Sigma_\ast - \Sigma_0 \r\| \leq 18\sigma_0 \sqrt{\frac{\beta}{m}}.
\]
\end{theorem}

An immediate corollary of Theorem \ref{th:lepski} is the quantitative result for the performance of PCA based on the estimator 
$\widehat\Sigma_\ast$. 
Let $\proj_k$ be the orthogonal projector on a subspace corresponding to the $k$ largest positive eigenvalues $\lambda_1,\ldots,\lambda_k$ of $\Sigma_0$ (here, we assume for simplicity that all the eigenvalues are distinct), and $\widehat{\proj_k}$ -- the orthogonal projector of the same rank as $\proj_k$ corresponding to the $k$ largest eigenvalues of $\widehat\Sigma_\ast$. 
The following bound follows from the Davis-Kahan perturbation theorem \cite{davis1970rotation}, more specifically, its version due to \cite[][Theorem 3 ]{Zwald2006On-the-Converge00}. 
\begin{corollary}
\label{cor:PCA}
Let $\Delta_k=\lambda_k - \lambda_{k+1}$, and assume that $\Delta_k\geq 72\sigma_0 \sqrt{\frac{\beta}{m}}$. 
Then  
\[
\big\|\widehat{\proj_k}-\proj_k\big\| \leq \frac{36}{\Delta_k}\sigma_0 \sqrt{\frac{\beta}{m}} 
\]
with probability $\geq 1 - 5d \log_2\l(\frac{2\sigma_{\mx}}{\sigma_{\mn}}\r) e^{-\beta}$. 
\end{corollary}

It is worth comparing the bound of Lemma \ref{lemma:main} and Theorem \ref{th:lepski} above to results of the paper by \cite{fan2016robust}, which constructs a covariance estimator $\widehat{\Sigma}_m'$ under the assumption that the random vector $X$ is centered, and $\sup_{\mathbf v\in \mb R^d: \|\mathbf{v}\|_2\leq1}\expect{|\langle\mathbf{v},X\rangle|^4}=B<\infty$. 
More specifically, $\widehat{\Sigma}_m'$ satisfies the inequality 
\begin{align}
\label{eq:fan}
\pr{\l\| \widehat{\Sigma}_m'-\Sigma_0 \r\| \geq \sqrt{\frac{C_1\beta Bd}{m}} } \leq de^{-\beta},
\end{align}
where $C_1>0$ is an absolute constant. 
The main difference between \eqref{eq:fan} and the bounds of Lemma \ref{lemma:main} and Theorem \ref{th:lepski} is that the latter are expressed in terms of $\sigma_0^2$, while the former is in terms of $B$. 
The following lemma demonstrates that our bounds are at least as good:
\begin{lemma}
\label{bound-on-sigma}
Suppose that $\mb E X = 0$ and $\sup_{\mathbf v\in \mb R^d:\|\mathbf{v}\|_2\leq 1}\expect{|\langle\mathbf{v},X\rangle|^4}=B<\infty$. 
Then $Bd\geq\sigma_0^2$. 
\end{lemma}
It follows from the above lemma that $\overline{d}=\sigma_0^2/\|\Sigma_0\|^2\lesssim d$.
Hence, By Theorem \ref{th:lepski}, the error rate of estimator $\widehat \Sigma_\ast$ is bounded above by $\mathcal{O}(\sqrt{d/m})$ if $m\gtrsim d$. 
It has been shown (for example, see \cite{lounici2014high}) that the minimax lower bound of covariance estimation is of order 
$\Omega(\sqrt{d/m})$. 
Hence, the bounds of \cite{fan2016robust} as well as our results imply correct order of the error. 
That being said, the ``intrinsic dimension'' $\bar d$ reflects the structure of the covariance matrix and can potentially be much smaller than $d$, as it is shown in the next section.

\subsection{Bounds in terms of intrinsic dimension}

In this section, we show that under a slightly stronger assumption on the fourth moment of the random vector $X$, the bound $\mathcal{O}(\sqrt{d/m})$ is suboptimal, while our estimator can achieve a much better rate in terms of the ``intrinsic dimension'' associated to the covariance matrix. 
This makes our estimator useful in applications involving high-dimensional covariance estimation, such as PCA.
Assume the following uniform bound on the \textit{kurtosis} of linear forms $\langle Z,v\rangle$:
\begin{equation}
\label{kurtosis}
\sup_{\|\mathbf{v}\|_2\leq1}\frac{\sqrt{\mb E \dotp{Z}{\mathbf{v}}^4}}{\mb E \dotp{Z}{\mathbf{v}}^2}=R<\infty.
\end{equation}
The intrinsic dimension of the covariance matrix $\Sigma_0$ can be measured by the \textit{effective rank} defined as 
\[
\mathbf{r}(\Sigma_0)=\frac{\tr(\Sigma_0)}{\|\Sigma_0\|}.
\]
Note that we always have $\mathbf{r}(\Sigma_0)\leq \text{rank}(\Sigma_0)\leq d$, and it some situations 
$\mathbf{r}(\Sigma_0)\ll \text{rank}(\Sigma_0)$, for instance if the covariance matrix is ``approximately low-rank'', meaning that it has many small eigenvalues.
The constant $\sigma_0^2$ is closely related to the effective rank as is shown in the following lemma (the proof of which is included in the supplementary material):
\begin{lemma}
\label{effective-rank-bound}
Suppose that \eqref{kurtosis} holds. 
Then,
\[
\mathbf{r}(\Sigma_0)\|\Sigma_0\|^2\leq \sigma_0^2\leq R^2\mathbf{r}(\Sigma_0)\|\Sigma_0\|^2.
\]
\end{lemma}
As a result, we have $\mathbf{r}(\Sigma_0)\leq\overline{d}\leq R^2\mathbf{r}(\Sigma_0)$.
The following corollary immediately follows from Theorem \ref{th:lepski} and Lemma \ref{effective-rank-bound}:
\begin{corollary}
Suppose that $m\geq C\beta \mathbf{r}(\Sigma_0)$ for an absolute constant $C>0$ and that \eqref{kurtosis} holds. 
Then
\[
\l\| \widehat\Sigma_\ast - \Sigma_0 \r\| \leq 18R\|\Sigma_0\| \sqrt{\frac{\mathbf{r}(\Sigma_0)\beta}{m}}
\]
with probability at least 
$1 - 5d \log_2\l(\frac{2\sigma_{\mx}}{\sigma_{\mn}}\r) e^{-\beta}$.
\end{corollary}

\section{Applications: Low-rank Covariance Estimation}

In many data sets encountered in modern applications (for instance, gene expression profiles \cite{saal2007poor}), dimension of the observations, hence the corresponding covariance matrix, is larger than the available sample size. 
However, it is often possible, and natural, to assume that the unknown matrix possesses special structure, such as low rank, thus reducing the ``effective dimension'' of the problem. 
The goal of this section is to present an estimator of the covariance matrix that is ``adaptive'' to the possible low-rank structure; such estimators are well-known and have been previously studied for the bounded and sub-Gaussian observations \cite{lounici2014high}. 
We extend these results to the case of heavy-tailed observations; in particular, we show that the estimator obtained via soft-thresholding applied to the eigenvalues of $\widehat\Sigma_\ast$ admits optimal guarantees in the Frobenius (as well as operator) norm.

Let $\widehat\Sigma_\ast$ be the estimator defined in the previous section, see equation \eqref{eq:lepski}, and set
\begin{align}
&
\widehat \Sigma_\ast^{\tau}=\argmin_{A\in \mb R^{d\times d}}
\l[  \l\| A - \widehat \Sigma_\ast \r\|^2_{\mathrm{F}} +\tau \l\| A \r\|_1\r],
\end{align}
where $\tau>0$ controls the amount of penalty. 
It is well-known (e.g., see the proof of Theorem 1 in \cite{lounici2014high}) that 
$\widehat \Sigma_{2n}^\tau$ can be written explicitly as 
\[
\widehat \Sigma_{\ast}^\tau = \sum_{i=1}^d \max\l(\lambda_i\l(\widehat \Sigma_{\ast}\r) -\tau/2, 0\r) v_i(\widehat \Sigma_{\ast}) v_i(\widehat \Sigma_{\ast})^T,
\]
where $\lambda_i(\widehat \Sigma_{\ast})$ and $v_i(\widehat \Sigma_{\ast})$ are the eigenvalues and corresponding eigenvectors of $\widehat \Sigma_{\ast}$. 
We are ready to state the main result of this section. 
\begin{theorem}
\label{th:covariance}
For any 
$
\tau \geq 36 \sigma_0 \sqrt{\frac{\beta}{m}},
$
\begin{align}
&\label{eq:ex70}
\l\| \widehat \Sigma_\ast^\tau - \Sigma_0 \r\|_{\mathrm{F}}^2\leq \inf_{A\in \mb R^{d\times d}} \l[  \l\| A - \Sigma_0 \r\|_{\mathrm{F}}^2 + \frac{(1+\sqrt{2})^2}{8}\tau^2\mathrm{rank}(A)  \r].
\end{align}
with probability $\geq 1 - 5d \log_2\l(\frac{2\sigma_{\mx}}{\sigma_{\mn}}\r) e^{-\beta}$.
\end{theorem}
In particular, if $\rank(\Sigma_0) = r$ and $\tau = 36 \sigma_0 \sqrt{\frac{\beta}{m}}$, we obtain that 
\[
\l\| \widehat \Sigma_\ast^\tau - \Sigma_0 \r\|_{\mathrm{F}}^2 \leq 162\,\sigma_0^2 \l(1+\sqrt{2}\r)^2 \frac{\beta r}{m}
\]
with probability $\geq 1 - 5d \log_2\l(\frac{2\sigma_{\mx}}{\sigma_{\mn}}\r) e^{-\beta}$.

\section{Proofs}
\label{sec:proofs}


\subsection{Proof of Lemma \ref{lemma:main}}
\label{ssec:mainproof}
The result is a simple corollary of the following statement.
\begin{lemma}
\label{main:lemma-2}
Set $\theta=\frac{1}{\sigma}\sqrt{\frac{\beta}{m}}$, where $\sigma \geq \sigma_0$ and $m\geq\beta$.
Let $\overline{d}:=\sigma_0^2/\|\Sigma_0\|^2$. 
Then, with probability at least $1-5de^{-\beta}$,
\begin{multline*}
\left\| \widehat\Sigma - \Sigma_0\right\|
\leq 2\sigma\sqrt{\frac{\beta}{m}} \\
+C'\|\Sigma_0\|  \l( \sqrt{\frac{\overline{d}\sigma}{\|\Sigma_0\|}}\l(\frac{\beta}{m}\r)^{\frac34} + \frac{\sqrt{\overline{d}}\sigma}{\|\Sigma_0\|}\frac{\beta}{m}
+ \sqrt{\frac{\overline{d}\sigma}{\|\Sigma_0\|}}\l(\frac{\beta}{m}\r)^{\frac54}  
+\overline{d}
\l(\frac{\beta}{m}\r)^{\frac32} + \frac{\overline{d}\beta^2}{m^2} + \overline{d}^{\frac54}\l(\frac{\beta}{m}\r)^{\frac94} \right),
\end{multline*}
where $C'>1$ is an absolute constant.
\end{lemma}
Now, by Corollary \ref{FKG-bound} in the supplement, it follows that 
$\overline{d} = \sigma_0^2/\|\Sigma_0\|^2\geq\tr(\Sigma_0)/\|\Sigma_0\|\geq1$. Thus, 
assuming that the sample size satisfies $m\geq(6C')^4\overline{d}\beta$, then, 
$\overline{d}\beta/m\leq1/(6C')^4<1$, and by some algebraic manipulations
we have that
\begin{equation}\label{need-steps}
\left\| \widehat\Sigma - \Sigma_0\right\|
\leq 2\sigma\sqrt{\frac{\beta}{m}} +  \sigma\sqrt{\frac{\beta}{m}}=3\sigma\sqrt{\frac{\beta}{m}}.
\end{equation}
For completeness, a detailed computation is given in the supplement. This
finishes the proof.

\subsection{Proof of Lemma \ref{main:lemma-2}}

Let $B_\beta = 11\sqrt{2\tr(\Sigma_0)\beta/m}$ be the error bound of the robust mean estimator $\widehat\mu$ defined in \eqref{eq:median_mean}.
Let $Z_i = X_i - \mu_0$,
$\Sigma_\mu = \expect{(Z_i-\mu)(Z_i-\mu)^T}$, $\forall i=1,2,\cdots,d$, and
\[
\hat{\Sigma}_\mu = \frac{1}{m\theta}\sum_{i=1}^m \frac{(X_i - \mu)(X_i - \mu)^T}{\l\| X_i - \mu\r\|_2^2} \psi\l( \theta \l\| X_i - \mu \r\|_2^2 \r),
\]
for any $\|\mu\|_2\leq B_\beta$. 
We begin by noting that the error can be bounded by the supremum of an empirical process indexed by $\mu$, i.e.
\begin{equation}\label{triangle-inequality}
\left\| \hat{\Sigma} - \Sigma_0 \right\| 
\leq \sup_{\|\mu\|_2\leq B_\beta}\left\| \hat{\Sigma}_\mu - \Sigma_0 \right\|
\leq \sup_{\|\mu\|_2\leq B_\beta}\left\| \hat{\Sigma}_\mu - \Sigma_\mu \right\| 
+ \left\| \Sigma_\mu - \Sigma_0 \right\|
\end{equation}
with probability at least $1-e^{-\beta}$.
We first estimate the second term $\left\| \Sigma_\mu - \Sigma_0 \right\|$. 
For any $\|\mu\|_2\leq B_\beta$,
\begin{multline*}
\left\| \Sigma_\mu - \Sigma_0 \right\| 
= \left\| \expect{(Z_i-\mu)(Z_i-\mu)^T - Z_iZ_i^T} \right\|
= \sup_{\mathbf{v}\in \mb R^d:\|\mathbf{v}\|_2 \leq 1} \left| \expect{\dotp{Z_i-\mu}{\mathbf{v}}^2 - \dotp{Z_i}{\mathbf{v}}^2 } \right|     \\
= (\mu^T\mathbf{v})^2 \leq \|\mu\|_2^2 \leq B_\beta^2 =242 \frac{\tr(\Sigma_0)\beta}{m},
\end{multline*}
with probability at least $1-e^{-\beta}$.
It follows from Corollary \ref{FKG-bound} in the supplement that with the same probability
\begin{equation}
\label{mean-bound}
\left\| \Sigma_\mu - \Sigma_0 \right\| \leq 242\frac{\sigma_0^2\beta}{\|\Sigma_0\|m}
\leq 242\frac{\sigma^2\beta}{\|\Sigma_0\|m} = 242\|\Sigma_0\|\frac{\overline{d}\beta}{m}.
\end{equation}
Our main task is then to bound the first term in \eqref{triangle-inequality}. 
To this end, we rewrite it as a double supremum of an empirical process:
\[
\sup_{\|\mu\|_2\leq B_\beta}\left\| \hat{\Sigma}_\mu - \Sigma_\mu \right\|
= \sup_{\|\mu\|_2\leq B_\beta,\|\mathbf{v}\|_2\leq1} \l|\mathbf{v}^T\l(\hat{\Sigma}_\mu - \Sigma_\mu\r)\mathbf{v}\r|
\]
It remains to estimate the supremum above. 
\begin{lemma}
\label{key-lemma}
Set $\theta=\frac{1}{\sigma}\sqrt{\frac{\beta}{m}}$, where $\sigma \geq \sigma_0$ and $m\geq\beta$.
Let $\overline{d}:=\sigma_0^2/\|\Sigma_0\|^2$. 
Then, with probability at least $1-4de^{-\beta}$,
\begin{multline*}
\sup_{\|\mu\|_2\leq B_\beta,\|\mathbf{v}\|_2\leq1} \l|\mathbf{v}^T\l(\hat{\Sigma}_\mu - \Sigma_\mu\r)\mathbf{v}\r|
\leq 2\sigma\sqrt{\frac{\beta}{m}} \\
+C''\|\Sigma_0\|  \l( \sqrt{\frac{\overline{d}\sigma}{\|\Sigma_0\|}}\l(\frac{\beta}{m}\r)^{\frac34} + \frac{\sqrt{\overline{d}}\sigma}{\|\Sigma_0\|}\frac{\beta}{m}
+ \sqrt{\frac{\overline{d}\sigma}{\|\Sigma_0\|}}\l(\frac{\beta}{m}\r)^{\frac54}  
+\overline{d}
\l(\frac{\beta}{m}\r)^{\frac32} + \frac{\overline{d}\beta^2}{m^2} + \overline{d}^{\frac54}\l(\frac{\beta}{m}\r)^{\frac94} \right),
\end{multline*}
where $C''>1$ is an absolute constant.
\end{lemma}
Note that $\sigma\geq\sigma_0$ by defnition, thus, $\overline{d}\leq\sigma^2/\|\Sigma_0\|^2$.
Combining the above lemma with \eqref{triangle-inequality} and \eqref{mean-bound} finishes the proof.


\subsection{Proof of Theorem \ref{th:lepski}}
\label{ssec:lepskiproof}

Define $\bar j:=\min\l\{  j\in \m J: \ \sigma_j \geq \sigma_0\r\}$, and note that $\sigma_{\bar j}\leq 2\sigma_0$. 
We will demonstrate that $j_\ast \leq \bar j$ with high probability. 
Observe that
\begin{align*}
\Pr\l( j_\ast > \bar j\r)&\leq \Pr\l( \bigcup_{k\in \m J: k>\bar j} \l\{  \l\|  \widehat \Sigma_{m,k} - \Sigma_{m,\bar j} \r\| > 6\sigma_k \sqrt{\frac{\beta}{n}} \r\} \r)\\
& 
\leq \Pr\l(  \l\|  \widehat\Sigma_{m,\bar j} - \Sigma_0 \r\| > 3\sigma_{\bar j} \sqrt{\frac{\beta}{m}} \r) + 
\sum_{k\in \m J: \ k>\bar j}\Pr\l( \l\| \widehat\Sigma_{m,k} - \Sigma_0 \r\| > 3\sigma_k \sqrt{\frac{\beta}{m}}  \r)   \\
&
\leq 5de^{-\beta} + 5d \log_2\l(\frac{\sigma_{\mx}}{\sigma_{\mn}}\r) e^{-\beta},
\end{align*}
where we applied \eqref{simple-bound} to estimate each of the probabilities in the sum under the assumption that the number of samples $m\geq C\overline{d}\beta$ and $\sigma_k\geq\sigma_{\bar j}\geq\sigma_0$. 
It is now easy to see that the event  
\[
\m B = \bigcap_{k\in \m J: k\geq \bar j} 
\l\{ \l\|  \widehat\Sigma_{m,k} - \Sigma_0 \r\|\leq 3\sigma_k\sqrt{\frac{\beta}{m}}  \r\} 
\]
of probability $\geq 1 - 5d \log_2\l(\frac{2\sigma_{\mx}}{\sigma_{\mn}}\r) e^{-\beta}$ is contained in 
$\m E=\l\{  j_\ast\leq \bar j \r\}$. 
Hence, on $\m B$  
\begin{align*}
\l\| \widehat\Sigma_\ast - \Sigma_0 \r\|&
\leq \| \widehat\Sigma_\ast - \widehat\Sigma_{m,\bar j} \| + \| \widehat\Sigma_{m,\bar j} - \Sigma_0 \| \leq 
6 \sigma_{\bar j}\sqrt{\frac{\beta}{m}} + 3\sigma_{\bar j}\sqrt{\frac{\beta}{m}} \\
&\leq 12\sigma_0\sqrt{\frac{\beta}{m}} + 6\sigma_0\sqrt{\frac{\beta}{m}} = 18\sigma_0 \sqrt{\frac{\beta}{m}},
\end{align*}
and the claim follows. 

\subsection{Proof of Theorem \ref{th:covariance}}

The proof is based on the following lemma:
\begin{lemma}
Inequality (\ref{eq:ex70}) holds on the event $\m E=\l\{ \tau\geq 2\l\| \widehat \Sigma_{\ast} - \Sigma_0 \r\| \r\}$. 
\end{lemma}
To verify this statement, it is enough to repeat the steps of the proof of Theorem 1 in \cite{lounici2014high}, replacing each occurrence of the sample covariance matrix by its ``robust analogue'' $\widehat \Sigma_\ast$. \\
It then follows from Theorem \ref{th:lepski} that $\Pr(\m E)\geq 1 - 5d \log_2\l(\frac{2\sigma_{\mx}}{\sigma_{\mn}}\r) e^{-\beta}$ whenever $\tau \geq 36 \sigma_0 \sqrt{\frac{\beta}{m}}$.




\section{Proof of Additional Technical Lemmas}
\subsection{Preliminaries}
\begin{lemma}\label{log-bounded-function}
Consider any function $\phi:\mathbb{R}\rightarrow\mathbb{R}$ and $\theta>0$. Suppose the following holds
\begin{equation}\label{assumption-1}
-\frac1\theta\log\left(1-\theta x+\theta^2x^2\right)\leq \phi(x)
\leq \frac1\theta\log\left(1+\theta x+\theta^2x^2\right), ~\forall x\in\mathbb{R}
\end{equation}
then, we have for any matrix $A\in\mathbb{H}^{d\times d}$,
\[
-\frac1\theta\log\left(1-\theta A+\theta^2A^2\right)\leq \phi(A)
\leq \frac1\theta\log\left(I+\theta A+\theta^2A^2\right).
\]
\end{lemma}
\begin{proof}
Note that for any $x\in\mathbb{R}$, 
$-\frac1\theta\log\left(1 - x\theta + x^2\theta^2\right)\leq\frac1\theta\log\left(1 + x\theta+ x^2\theta^2\right)$,
then, the claim follows immediately from the definition of the matrix function.
\end{proof}

The above lemma is useful in our context mainly due to the following lemma,
\begin{lemma}\label{truncation-function}
The truncation function $\frac1\theta\psi(\theta x) = \textrm{sign}(x)\cdot\l(|x|\wedge\frac1\theta\r)$ satisfies the assumption \eqref{assumption-1} in Lemma \ref{log-bounded-function}.
\end{lemma}
\begin{proof}
Denote $f_1(x) = -\frac1\theta\log\left(1-\theta x+\theta^2x^2\right)$, $f_2(x) =  \frac1\theta\log\left(1+\theta x+\theta^2x^2\right)$ and $g(x) = \textrm{sign}(x)\cdot\l(|x|\wedge\frac1\theta\r)$. Note first that 
\begin{align*}
&f_1(0) = g(0) = f_2(0) = 0,\\
&f_1(1/\theta)\leq g(1/\theta) \leq f_2(1/\theta),\\
&f_1(-1/\theta)\leq g(-1/\theta) \leq f_2(-1/\theta),\\
\end{align*}
and the subgradient
\[
\partial g(x) = 
\begin{cases}
1, &~~x\in(-1/\theta,1/\theta),\\
0, &~~x\in(-\infty,-1/\theta)\cup(1/\theta,+\infty),\\
[0,1], &~~ x = -1/\theta,1/\theta.
\end{cases}
\]
Next, we take the derivative of $f_2(x)$ and compare it to the derivative of $g(x)$.
\[
f_2'(x) = \frac1\theta\cdot\frac{\theta+2x\theta^2}{1+x\theta + x^2\theta^2} 
=\frac{1+2x\theta}{1+ x\theta+x^2\theta^2}.
\]
Note that $f_2'(x)\geq1,x\in(0,1/\theta)$, $f_2'(x)\geq0,x\geq1/\theta$, $f_2'(x)\leq1,x\in(-1/\theta,0]$ and 
$f_2'(x)\leq0,x\leq-1/\theta$. Thus, we have $g(x)\leq f_2(x),~\forall x\in\mathbb{R}$. Similarly, we can take the derivative of $f_1(x)$ and compare it to $g(x)$, which results in $f_1'(x)\leq1,x\in(0,1/\theta)$, $f_1'(x)\leq0,x\geq1/\theta$, $f_1'(x)\geq1,x\in(-1/\theta,0]$ and 
$f_2'(x)\geq0,x\leq-1/\theta$. This implies $f_1(x)\leq g(x)$ and the Lemma is proved.
\end{proof}

The following lemma demonstrates the importance of matrix logarithm function in matrix analysis, whose proof can be found in \cite{bhatia2013matrix} and \cite{tropp2015introduction},
\begin{lemma}
(a) The matrix logarithm is operator monotone, that is, 
if $A\succ B\succ0$ are two matrices in $\mathbb{H}^{d\times d}$, then, $\log(A)\succ\log(B)$.\\
(b) Given a fixed matrix $H\in \mathbb{H}^{d\times d}$, the function
\[
A\rightarrow tr\exp(H+\log(A))
\]
is concave on the cone of positive semi-definite matrices.
\end{lemma}

The following lemma is a generalization of Chebyshev's association inequality. See Theorem 2.15 of \cite{boucheron2013concentration} for proof.
\begin{lemma}[FKG inequality]
Suppose $f,g:\mathbb{R}^d\rightarrow\mathbb{R}$ are two functions non-decreasing on each coordinate. Let $Y=[Y_1,~Y_2,~\cdots,~Y_d]$ be a random vector taking values in $\mathbb{R}^d$, then, 
\[
\expect{f(X)g(X)}\geq\expect{f(X)}\expect{g(X)}.
\]
\end{lemma}
The following corollary follows immediately from the FKG inequality.
\begin{corollary}\label{FKG-bound}
Let $Z=X-\mu_0$, then, we have 
$\sigma_0^2 = \|\expect{ZZ^T\|Z\|_2^2}\|\geq tr\left(\expect{ZZ^T}\right)\left\|\expect{ZZ^T}\right\|
=tr(\Sigma_0)\|\Sigma_0\|$.
\end{corollary}
\begin{proof}
Consider any unit vector $\mathbf{v}\in\mathbb{R}^d$. It is enough to show $\expect{(\mathbf{v}^TZ)^2\|Z\|_2^2}\geq\expect{(\mathbf{v}^TZ)^2}\expect{\|Z\|_2^2}$. We change the coordinate by considering an orthonormal basis $\{\mathbf{v}_1,\cdots,\mathbf{v}_d\}$ with 
$\mathbf{v}_1=\mathbf{v}$. Let $Y_i = \mathbf{v}_i^TZ$, $i=1,2,\cdots,d$, then we obtain,
\[
\expect{(\mathbf{v}^TZ)^2\|Z\|_2^2} = \expect{Y_1^2\|Y\|_2^2}\geq\expect{Y_1^2}\expect{\|Y\|_2^2},
\]
where the last inequality follows from FKG inequality by taking $f\left(Y_1^2,~\cdots,~Y_d^2\right) = Y_1^2$ and $g\left(Y_1^2,~\cdots,~Y_d^2\right) = \|Y\|_2^2$.
\end{proof}

\subsection{Additional computation in the proof of Lemma \ref{lemma:main}}
In order to show \eqref{need-steps}, it is enough to show that
\begin{align*}
C'\|\Sigma_0\|  \l( \sqrt{\frac{\overline{d}\sigma}{\|\Sigma_0\|}}\l(\frac{\beta}{m}\r)^{\frac34} + \frac{\sqrt{\overline{d}}\sigma}{\|\Sigma_0\|}\frac{\beta}{m}
+ \sqrt{\frac{\overline{d}\sigma}{\|\Sigma_0\|}}\l(\frac{\beta}{m}\r)^{\frac54}  
+\overline{d}
\l(\frac{\beta}{m}\r)^{\frac32} + \frac{\overline{d}\beta^2}{m^2} 
+ \overline{d}^{\frac54}\l(\frac{\beta}{m}\r)^{\frac94} \right)
\leq\sigma\sqrt{\frac\beta m}.
\end{align*}
Note that $\overline{d} = \sigma_0^2/\|\Sigma_0\|^2\geq\tr(\Sigma_0)/\|\Sigma_0\|\geq1$, and
assuming that the sample size satisfies $m\geq(6C')^4\overline{d}\beta$, we have 
$\overline{d}\beta/m\leq1/(6C')^4<1$.
We then bound each of the 6 terms on the left side. 
\begin{align*}
C'\|\Sigma_0\|\sqrt{\frac{\overline{d}\sigma}{\|\Sigma_0\|}}\l(\frac{\beta}{m}\r)^{\frac34}
=&C'\sqrt{\sigma}\l(\frac\beta m\r)^{\frac14}\cdot\l(\frac{\|\Sigma_0\|\overline{d}\beta}{m}\r)^{1/4}
\cdot\l(\frac{\|\Sigma_0\|\overline{d}\beta}{m}\r)^{1/4}\\
\leq&C'\sqrt{\sigma}\l(\frac\beta m\r)^{\frac14}\cdot\l(\frac{\|\Sigma_0\|\overline{d}\beta}{m}\r)^{1/4}
\cdot\frac{1}{6C'}\\
=&\frac16\sqrt{\sigma\sigma_0}\sqrt{\frac{\beta}{m}} \leq \frac16\sigma\sqrt{\frac{\beta}{m}},\\
C'\|\Sigma_0\| \cdot\sqrt{\overline{d}}\frac{\sigma}{\|\Sigma_0\|}\frac{\beta}{m}~~~~
=&C'\sigma\sqrt{\frac\beta m}\cdot\sqrt{\frac{\overline{d}\beta}{m}}
\leq C'\sigma\sqrt{\frac\beta m}\frac{1}{(6C')^2}\leq\frac16\sigma\sqrt{\frac{\beta}{m}},\\
C'\|\Sigma_0\|\sqrt{\frac{\overline{d}\sigma}{\|\Sigma_0\|}}\l(\frac{\beta}{m}\r)^{\frac54}
\leq&C'\|\Sigma_0\|\sqrt{\frac{\overline{d}\sigma}{\|\Sigma_0\|}}\l(\frac{\beta}{m}\r)^{\frac34}
\leq\frac16\sigma\sqrt{\frac{\beta}{m}}.
\end{align*}
Note that we have the following
\begin{multline*}
C'\|\Sigma_0\|\overline{d}\frac\beta m 
= C'\|\Sigma_0\|\l(\frac{\overline{d}\beta}{m} \r)^{\frac12}\l(\frac{\overline{d}\beta}{m} \r)^{\frac12}
\leq C'\|\Sigma_0\|\l(\frac{\overline{d}\beta}{m} \r)^{\frac12}\frac{1}{(6C')^2}
\leq\frac16\sigma_0\sqrt{\frac\beta m}\leq\frac16\sigma\sqrt{\frac\beta m},
\end{multline*}
thus, the rest three terms can be bounded as follows,
\begin{align*}
C'\|\Sigma_0\|\overline{d}\l(\frac{\beta}{m}\r)^{\frac32}
\leq& C'\|\Sigma_0\|\overline{d}
\frac{\beta}{m}\leq \frac16\sigma\sqrt{\frac\beta m}\\
C'\|\Sigma_0\|\overline{d}\frac{\beta^2}{m^2} ~~~~
\leq& C'\|\Sigma_0\|\overline{d}
\frac{\beta}{m}\leq \frac16\sigma\sqrt{\frac\beta m}\\
C'\|\Sigma_0\|\overline{d}^{\frac54}\l(\frac{\beta}{m}\r)^{\frac94}
\leq&C'\|\Sigma_0\|\overline{d}^{\frac54}\l(\frac{\beta}{m}\r)^{\frac54}
\leq C'\|\Sigma_0\|\overline{d}
\frac{\beta}{m}\leq \frac16\sigma\sqrt{\frac\beta m}.
\end{align*}
Overall, we have \eqref{need-steps} holds.

\subsection{Proof of Lemma \ref{key-lemma}}
First of all, by definition of $\widehat\Sigma_\mu$, we have
\begin{align*}
\sup_{\|\mu\|_2\leq B_\beta, \|\mathbf{v}\|_2\leq 1}\left| \mathbf{v}^T(\hat{\Sigma}_\mu - \Sigma_\mu)\mathbf{v} \right|
= \sup_{\|\mu\|_2\leq B_\beta, \|\mathbf{v}\|_2\leq 1}\left| \frac{1}{m\theta}\sum_{i=1}^m\dotp{Z_i-\mu}{\mathbf{v}}^2
\frac{\psi\l(\theta\|Z_i-\mu\|_2^2\r)}{\|Z_i-\mu\|_2^2} - 
\expect{\dotp{Z_i-\mu}{\mathbf{v}}^2} \right|.
\end{align*}
Expanding the squares on the right hand side gives
\begin{align*}
\sup_{\|\mu\|_2\leq B_\beta}\left\| \hat{\Sigma}_\mu - \Sigma_\mu \right\|
\leq& 
\sup_{\|\mu\|_2\leq B_\beta, \|\mathbf{v}\|_2\leq 1}
\left| \frac1m\sum_{i=1}^m\dotp{Z_i}{\mathbf{v}}^2
\frac{\psi\l(\theta\|Z_i-\mu\|_2^2\r)}{\theta\|Z_i-\mu\|_2^2} - \expect{\dotp{Z_i}{\mathbf{v}}^2}\right| ~~\text{(I)}\\
&+ 2 \sup_{\|\mu\|_2\leq B_\beta, \|\mathbf{v}\|_2\leq 1}
\left| \frac1m\sum_{i=1}^m\dotp{Z_i}{\mathbf{v}}\dotp{\mu}{\mathbf{v}}
\frac{\psi\l(\theta\|Z_i-\mu\|_2^2\r)}{\theta\|Z_i-\mu\|_2^2} - \expect{\dotp{Z_i}{\mathbf{v}}\dotp{\mu}{\mathbf{v}}}\right| ~~\text{(II)}\\
& + \sup_{\|\mu\|_2\leq B_\beta, \|\mathbf{v}\|_2\leq 1}
\left| \frac1m\sum_{i=1}^m\dotp{\mu}{\mathbf{v}}^2
\frac{\psi\l(\theta\|Z_i-\mu\|_2^2\r)}{\theta\|Z_i-\mu\|_2^2} - \dotp{\mu}{\mathbf{v}}^2\right|.~~\text{(III)}
\end{align*}
We will then bound these three terms separately. Note that given $\|\widehat\mu-\mu_0\|_2\leq B_\beta$, the term (III) can be readily bounded as follows using the fact that $0\leq\psi(x)\leq x,~\forall x\geq0$,
\begin{multline}\label{final-bound-III}
\text{(III)} = \sup_{\|\mu\|_2\leq B_\beta, \|\mathbf{v}\|_2\leq 1}
\left| \dotp{\mu}{\mathbf{v}}^2\left(\frac1m\sum_{i=1}^m\frac{\psi\l(\theta\|Z_i-\mu\|_2^2\r)}{\theta\|Z_i-\mu\|_2^2}  -  1\right) \right|
\leq \sup_{\|\mu\|_2\leq B_\beta, \|\mathbf{v}\|_2\leq 1}\dotp{\mu}{\mathbf{v}}^2
\leq B_\beta^2 \\
= 242\frac{tr(\Sigma_0)}{m}\beta \leq 242\frac{\sigma_0^2\beta}{\|\Sigma_0\|m}
\leq 242\|\Sigma_0\|\frac{\overline{d}\beta}{m},
\end{multline}
where the second from the last inequality follows from Corollary \ref{FKG-bound} and the last inequality follows from $\overline{d}=\sigma_0^2/\|\Sigma_0\|^2$.

The rest two terms are bounded through the following lemma whose proof is delayed to the next section:
\begin{lemma}\label{key-sublemma}
Given $\|\widehat\mu-\mu_0\|_2\leq B_\beta$, with probability at least $1-4de^{-\beta}$, we have the following two bounds hold,
\begin{align*}
\text{(I)}\leq 2\sigma\sqrt{\frac{\beta}{m}} 
+22\|\Sigma_0\|  \l( \sqrt{2}\overline{d}^{\frac14}\l(\frac{\beta}{m}\r)^{\frac34}
+ 2\sqrt{2}\sqrt{\frac{\overline{d}\sigma}{\|\Sigma_0\|}}\l(\frac{\beta}{m}\r)^{\frac54}  
+11\overline{d}^{\frac12}
\l(\frac{\beta}{m}\r)^{\frac32} + 22\frac{\overline{d}\beta^2}{m^2}\r),
\end{align*}
\begin{multline*}
\text{(II)} \leq 11\|\Sigma_0\|  \l( \sqrt2\sqrt{\frac{\overline{d}\sigma}{\|\Sigma_0\|}}\l(\frac{\beta}{m}\r)^{\frac34} + 3\sqrt2\sqrt{\overline{d}}\frac{\sigma}{\|\Sigma_0\|}\frac{\beta}{m} 
+ 44\overline{d}^{\frac34}\l(\frac{\beta}{m}\r)^{\frac54}  \r.\\
\l.
+44\sqrt2\overline{d}
\l(\frac{\beta}{m}\r)^{\frac32} + 242\sqrt2\frac{\overline{d}\beta^2}{m^2} + 484\overline{d}^{\frac54}\l(\frac{\beta}{m}\r)^{\frac94} \r).
\end{multline*}
\end{lemma}
Note that since $\sigma\geq\sigma_0$, we have
$\sigma/\|\Sigma_0\|\geq\sigma_0/\|\Sigma_0\|=\sqrt{\overline{d}}$.
Combining the above lemma with \eqref{final-bound-III}
finishes the proof of Lemma \ref{key-lemma}.

\subsection{Proof of Lemma \ref{key-sublemma}}
Before proving the Lemma, we introduce the following abbreviations:
\begin{align*}
&g_\mathbf{v}(Z_i) = \dotp{Z_i}{\mathbf{v}}^2\frac{\psi\l(\theta\|Z_i\|_2^2\r)}{\theta\|Z_i\|_2^2},
~~h_\mu(Z_i) = \frac{\|Z_i\|_2^2}{\psi\l(\theta\|Z_i\|_2^2\r)}\frac{\psi\l(\theta\|Z_i-\mu\|_2^2\r)}{\|Z_i-\mu\|_2^2},\\
&\tilde{g}_\mathbf{v}(Z_i) = \dotp{Z_i}{\mathbf{v}}\frac{\psi\l(\theta\|Z_i\|_2^2\r)}{\theta\|Z_i\|_2^2}.
\end{align*}
Our analysis relies on the following simply yet important fact which gives deterministic upper and lower bound of $h_\mu(Z_i)$ around 1. Its proof is delayed to the next section.

\begin{lemma}\label{ratio-bound}
For any $\mu$ such that $\|\mu\|_2\leq B_\beta$, the following holds:
\[
1 - 2B_\beta\sqrt{\theta} - B_\beta^2\theta\leq h_\mu(Z_i) \leq 1+ 2B_\beta\sqrt{\theta} + B_\beta^2\theta.
\]
\end{lemma}

The following Lemma gives a general concentration bound for heavy tailed random matrices under a mapping $\phi(\cdot)$. 
\begin{lemma}\label{concentration-bound}
Let $A_1,~A_2,\cdots,~A_m$ be a sequence of i.i.d. random matrices in $\mathbb{H}^{d\times d}$ with zero mean and finite second moment $\sigma_A = \|\expect{A_i^2}\|$. Let $\phi(\cdot)$ be any function satisfying the assumption \eqref{assumption-1} of Lemma \ref{log-bounded-function}. Then, for any $t>0$,
\[
Pr\left(\sum_{i=1}^m\left(\phi(A_i) - \expect{A_i}\right)\geq t \sqrt{m}\right)
\leq 2d\exp\left( -t\theta\sqrt{m} + m\theta^2\sigma_A^2 \right).
\]
Specifically, if the assumption \eqref{assumption-1} holds for $\theta = \frac{t}{2\sqrt{m}\sigma_A^2}$, then we obtain the subgaussian tail 
$2d\exp(-t^2/4\sigma_A^2)$.
\end{lemma}
The intuition behind this lemma is that the $\log(1+x)$ tends to ``robustify'' a random variable by implicitly trading the bias for a tight concentration. A scalar version of such lemma with a similar idea is first introduced in the seminal work \cite{catoni2012challenging}. 
The proof of the current matrix version is similar to Lemma 3.1 and Theorem 3.1 of \cite{minsker2016sub} by modifying only the constants. We omitted the details here for brevity. 
Note that this lemma is useful in our context by choosing $\phi(x) = \frac1\theta\psi(\theta x)$. Next, we prove two parts of Lemma \ref{key-sublemma} separately.

\begin{proof}[Proof of (I) in Lemma \ref{key-sublemma}]
 Using the abbreviation introduced at the beginning of this section, we have
\[
(I) = \sup_{\|\mu\|_2\leq B_\beta, \|\mathbf{v}\|_2\leq 1}
\left| \frac1m\sum_{i=1}^m g_\mathbf{v}(Z_i)h_\mu(Z_i) - \expect{\dotp{Z_i}{\mathbf{v}}^2} \right|
\]

We further split it into two terms as follows:
\begin{equation}\label{bound-of-I}
(I) \leq \sup_{\|\mu\|_2\leq B_\beta, \|\mathbf{v}\|_2\leq 1}
\left| \frac1m\sum_{i=1}^m g_\mathbf{v}(Z_i)\left(h_\mu(Z_i) - 1\right) \right|
+ \sup_{\|\mathbf{v}\|\leq}\left| \frac1m\sum_{i=1}^m g_\mathbf{v}(Z_i) -  \expect{\dotp{Z_i}{\mathbf{v}}^2} \right|
\end{equation}
The two terms in \eqref{bound-of-I} are bounded as follows:
\begin{enumerate}
\item For the second term in \eqref{bound-of-I}, note that we can write it back into the matrix form as 
\[
\left\| \frac{1}{m\theta}\sum_{i=1}^m Z_iZ_i^T\frac{\psi\l(\theta\|Z_i\|_2^2\r)}{\|Z_i\|_2^2}-\expect{Z_iZ_i^T}\right\|.
\]
Note that the matrix $Z_iZ_i^T$ is a rank one matrix with the eigenvalue equal to $\|Z_i\|_2^2$, so it follows from the definition of matrix function,
\[
Z_iZ_i^T\frac{\psi\l(\theta\|Z_i\|_2^2\r)}{\|Z_i\|_2^2}=\frac1\theta\psi\l(\theta Z_iZ_i^T\r).
\]
Now, applying Lemma \ref{truncation-function} setting $\theta = \frac{t}{2\sigma^2\sqrt{m}}$ together with Lemma \ref{concentration-bound} gives
\begin{equation*}
Pr\left(\left\| \frac{1}{m\theta}\sum_{i=1}^m Z_iZ_i^T\frac{\psi\l(\theta\|Z_i\|_2^2\r)}{\|Z_i\|_2^2}-\expect{Z_iZ_i^T}\right\|\geq t/\sqrt{m} \right)\leq 2d\exp(-t^2/4\sigma^2).
\end{equation*}
Setting $t = 2\sigma\sqrt{\beta}$ (which results in $\theta = \frac1\sigma\sqrt{\frac{\beta}{m}}$) gives
\begin{equation}\label{inter-bound-I}
\left\| \frac{1}{m\theta}\sum_{i=1}^m Z_iZ_i^T\frac{\psi\l(\theta\|Z_i\|_2^2\r)}{\|Z_i\|_2^2}-\expect{Z_iZ_i^T}\right\| 
\leq 2\sigma\sqrt{\frac{\beta}{m}}
\end{equation}
with probability at least $1-2de^{-\beta}$.

\item For the first term in \eqref{bound-of-I}, by the fact that $g_{\mathbf{v}}(Z_i)\geq0$ and Lemma \ref{ratio-bound}, 
\begin{align*}
&\sup_{\|\mu\|_2\leq B_\beta, \|\mathbf{v}\|_2\leq 1}
\left| \frac1m\sum_{i=1}^m g_\mathbf{v}(Z_i)\left(h_\mu(Z_i)-1\right) \right|\\
&\leq \sup_{\|\mu\|_2\leq B_\beta, \|\mathbf{v}\|_2\leq 1}
 \frac1m\sum_{i=1}^m g_\mathbf{v}(Z_i) \l|h_\mu(Z_i) - 1\r|\\
 &\leq \sup_{\|\mathbf{v}\|_2\leq 1}
 \frac1m\sum_{i=1}^m g_\mathbf{v}(Z_i)\l( 2B_\beta\sqrt{\theta} + B_\beta^2\theta \r)\\
 &\leq \left(\l\|\expect{Z_iZ_i^T}\r\| + 2\sigma\sqrt{\frac{\beta}{m}}\right)\l(2B_\beta\sqrt{\theta} + B_\beta^2\theta \r)  ,
\end{align*}
with probability at least $1 - 2de^{-\beta}$,
where the last inequality follows from the same argument leading to \eqref{inter-bound-I}. Note that $\expect{Z_iZ_i^T} = \Sigma_0$.
\end{enumerate}
Overall, we get
\begin{equation*}
\text{(I)} \leq 2\sigma\sqrt{\frac{\beta}{m}} + \left(\l\|\Sigma_0\r\| + 2\sigma\sqrt{\frac{\beta}{m}}\right)\l( 
2B_\beta\sqrt{\theta} + B_\beta^2\theta\r),
\end{equation*}
with probability at least $1-2de^{-\beta}$. Now we substitute $B_\beta = 11\sqrt{2\tr(\Sigma_0)\beta/m}$ and $\theta = \frac1\sigma\sqrt{\frac{\beta}{m}}$ into the above bound gives
\begin{multline*}
\text{(I)} \leq 2\sigma\sqrt{\frac{\beta}{m}} + 22\sqrt{2}\|\Sigma_0\|\sqrt{\frac{\tr(\Sigma_0)}{\sigma}}\l(\frac{\beta}{m}\r)^{\frac34} + 242\|\Sigma_0\|\frac{\tr\Sigma_0}{\sigma}\l(\frac\beta m\r)^{\frac32}\\
+44\sqrt{2}\sqrt{\sigma\tr(\Sigma_0)}\l(\frac\beta m\r)^{\frac54} + 484\tr(\Sigma_0)\l(\frac\beta m\r)^2
\end{multline*}
Using Corollary \ref{FKG-bound}, we have 
\begin{equation}\label{support-1}
\frac{\tr(\Sigma_0)}{\sigma}\leq \frac{\tr(\Sigma_0)}{\sigma_0}\leq
\frac{\tr(\Sigma_0)}{\sqrt{\tr(\Sigma_0)\|\Sigma_0\|}}\leq \frac{\sigma_0}{\|\Sigma_0\|}\leq\overline{d},
\end{equation}
and also, 
\begin{equation}\label{support-2}
\tr(\Sigma_0)\leq\|\Sigma_0\|\sigma_0^2/\|\Sigma_0\|^2\leq\|\Sigma_0\|\overline{d}.
\end{equation}
Substitute these two bounds into the bound of (I) gives the final bound for (I) stated in Lemma 
\ref{key-sublemma} with probability at least $1-2de^{-\beta}$.
\end{proof}

\begin{proof}[Proof of (II) in Lemma \ref{key-sublemma}]
First of all, using the definition of $\tilde{g}_\mathbf{v}(Z_i)$ and $h_\mu(Z_i)$, we can rewrite (II) as follows:
\begin{align*}
\text{(II)} =& \sup_{\|\mu\|_2\leq B_\beta, \|\mathbf{v}\|_2\leq 1}
\left| \frac1m\sum_{i=1}^m   \tilde{g}_\mathbf{v}(Z_i)h_\mu(Z_i)\dotp{\mu}{\mathbf{v}}
 - \expect{\dotp{Z_i}{\mathbf{v}}}\dotp{\mu}{\mathbf{v}}\right|  \\
 \leq& B_\beta  \cdot  \sup_{\|\mu\|_2\leq B_\beta, \|\mathbf{v}\|_2\leq 1}
\left| \frac1m\sum_{i=1}^m   \tilde{g}_\mathbf{v}(Z_i)h_\mu(Z_i)
 - \expect{\dotp{Z_i}{\mathbf{v}}}\right|.
\end{align*}
Similar to the analysis of (I), we further split the above term into two terms and get
\begin{align}
(II) \leq \underbrace{B_\beta  \sup_{\|\mu\|_2\leq B_\beta, \|\mathbf{v}\|_2\leq 1}
\l|  \frac1m\sum_{i=1}^m   \tilde{g}_\mathbf{v}(Z_i)\l(h_\mu(Z_i)-1\r) \r|}_{(IV)}
+ \underbrace{B_\beta  \sup_{\|\mathbf{v}\|_2\leq 1}
\l| \frac1m\sum_{i=1}^m \tilde{g}_\mathbf{v}(Z_i) - \expect{\dotp{Z_i}{\mathbf{v}}} \r|}_{(V)}.
\end{align}
For the first term, by Cauchy-Schwarz inequality and then Lemma \ref{ratio-bound}, we get
\begin{align*}
\text{(IV)}
\leq&
B_\beta  \sup_{\|\mu\|_2\leq B_\beta, \|\mathbf{v}\|_2\leq 1}
\frac1m\sum_{i=1}^m \l|\tilde{g}_\mathbf{v}(Z_i)\l(h_\mu(Z_i)-1\r) \r|\\
\leq&
B_\beta  \sup_{\|\mu\|_2\leq B_\beta, \|\mathbf{v}\|_2\leq 1}
\l(\frac1m\sum_{i=1}^m  \tilde{g}_\mathbf{v}(Z_i)^2\r)^{1/2} 
\l(\frac1m\sum_{i=1}^m \l|h_\mu(Z_i)-1 \r|^2\r)^{1/2}\\
\leq& 
B_\beta  \sup_{\|\mathbf{v}\|_2\leq 1} \l(\frac1m\sum_{i=1}^m  \tilde{g}_\mathbf{v}(Z_i)^2\r)^{1/2}
\l( 2B_\beta\sqrt{\theta} + B_\beta^2\theta \r). 
\end{align*}
Note that $\frac1\theta\psi\l(\theta\|Z_i\|_2^2\r)/\|Z_i\|_2^2\leq1$, then, it follows,
\[
\tilde{g}_\mathbf{v}(Z_i)^2=\dotp{Z_i}{\mathbf{v}}^2\l(\frac{\frac1\theta\psi\l(\theta\|Z_i\|_2^2\r)}{\|Z_i\|_2^2}\r)^2\leq\dotp{Z_i}{\mathbf{v}}^2\frac{\frac1\theta\psi\l(\theta\|Z_i\|_2^2\r)}{\|Z_i\|_2^2}. 
\]
Thus, by the same analysis leading to \eqref{inter-bound-I}, we get 
\begin{equation}\label{first-term-bound}
\text{(IV)}\leq
B_\beta\left(\l\|\expect{Z_iZ_i^T}\r\| + 2\sigma\sqrt{\frac{\beta}{m}}\right)^{1/2}\l( 2B_\beta\sqrt{\theta} + B_\beta^2\theta \r),
\end{equation}
with probability at least $1-2de^{-\beta}$. For the second term (V), notice that $\expect{Z_i} = 0$, thus we have 
\begin{multline}\label{inter-bound-V}
\text{(V)}\leq
B_\beta  \sup_{\|\mathbf{v}\|_2\leq 1}
\l| \dotp{\frac1m\sum_{i=1}^m \frac{Z_i}{\|Z_i\|_2^2}\frac1\theta\psi(\theta\|Z_i\|_2^2)}{\mathbf{v}}  \r|  
\leq B_\beta\l\|\frac1m\sum_{i=1}^m \frac{Z_i}{\|Z_i\|_2^2}\|Z_i\|_2^2\wedge\frac1\theta\r\|_2\\
\leq B_\beta\l\|\frac1m\sum_{i=1}^m \frac{Z_i}{\|Z_i\|_2^2}\|Z_i\|_2^2\wedge\frac1\theta 
- \expect{\frac{Z_i}{\|Z_i\|_2^2}\|Z_i\|_2^2\wedge\frac1\theta}\r\|_2
+ B_\beta \l\|\expect{\frac{Z_i}{\|Z_i\|_2^2}\|Z_i\|_2^2\wedge\frac1\theta}\r\|_2. 
\end{multline}

For the second term, which measures the bias, we have by the fact $\expect{Z_i} = 0$,
\begin{multline*}
\l\|\expect{\frac{Z_i}{\|Z_i\|_2^2}\|Z_i\|_2^2\wedge\frac1\theta}\r\|_2
=\l\| \expect{Z_i\l( \frac{\|Z_i\|_2^2\wedge\frac1\theta}{\| Z_i \|_2^2} - 1 \r)} \r\|_2
=\sup_{\|\mathbf{v}\|_2\leq1}\expect{\dotp{Z_i}{\mathbf{v}}\l( \frac{\|Z_i\|_2^2\wedge\frac1\theta}{\| Z_i \|_2^2} - 1 \r)}\\
\leq\sup_{\|\mathbf{v}\|_2\leq1}\expect{\dotp{Z_i}{\mathbf{v}}1_{\{\|Z_i\|_2\geq1/\sqrt{\theta}\}}}.
\end{multline*}
Now by Cauchy-Schwarz inequality and then Markov inequality, we obtain,
\begin{multline*}
\sup_{\|\mathbf{v}\|_2\leq1}\expect{\dotp{Z_i}{\mathbf{v}}1_{\{\|Z_i\|_2\geq1/\sqrt\theta\}}}
\leq\sqrt{\sup_{\|\mathbf{v}\|_2\leq1}\expect{\dotp{Z_i}{\mathbf{v}}^2}} Pr(\|Z_i\|_2\geq1/\sqrt{\theta})^{1/2}
\leq\sqrt{\|\Sigma_0\|}\expect{\|Z_i\|_2^2}^{1/2}\sqrt{\theta}\\
= \sqrt{\|\Sigma_0\|}\frac{\tr(\Sigma_0)^{1/2}\beta^{1/4}}{m^{1/4}\sigma^{1/2}}
\leq\frac{(\|\Sigma_0\|tr(\Sigma_0))^{1/4}\beta^{1/4}}{m^{1/4}}
\leq\l( \frac{\sigma^2}{m}\beta \r)^{1/4},
\end{multline*}
where the last two inequalities both follow from Lemma \ref{FKG-bound}. This gives the second term in \eqref{inter-bound-V} is given by $B_\beta\l( \frac{\sigma^2}{m}\beta \r)^{1/4}$.

For the first term in \eqref{inter-bound-V}, note that for any vector $\mathbf{x}\in\mathbb{R}^d$, 
\[
\|\mathbf{x}\|_2 = \left\| 
\l[
\begin{matrix}
0 & \mathbf{x}^T\\
\mathbf{x} & 0
\end{matrix}
\r]
 \right\|,
\]
and furthermore, the matrix 
$\l[
\begin{matrix}
0 & \mathbf{x}^T\\
\mathbf{x} & 0
\end{matrix}
\r]$
has two same eigenvalues equal to $\|\mathbf{x}\|_2$, which follows from
\[
\l[
\begin{matrix}
0 & \mathbf{x}^T\\
\mathbf{x} & 0
\end{matrix}
\r]^2
=
\l[
\begin{matrix}
\|\mathbf{x}\|_2^2 & 0\\
0 & \mathbf{x}\mathbf{x}^T
\end{matrix}
\r].
\]
Thus, if we take

\[
A_i = 
\l[
\begin{matrix}
0 & Z_i^T\\
Z_i & 0
\end{matrix}
\r]
\frac{\|Z_i\|_2^2\wedge\frac1\theta}{\|Z_i\|_2^2},
\]
Then, the first term of \eqref{inter-bound-V} is equal to
$
\l\| \frac1m\sum_{i=1}^mA_i - \expect{ A_i} \r\|
$. For this $A_i$, we have
\[
\|\expect{A_i^2}\|\leq \expect{\|Z_i\|_2^2} = tr(\Sigma_0),
~~\|A_i\|\leq\frac{1}{\sqrt{\theta}} = \frac{m^{1/4}\sigma^{1/2}}{\beta^{1/4}}.
\]
By matrix Bernstein's inequality (\cite{tropp1}), we obtain the bound
\begin{align*}
Pr\l( \l\|\frac1m\sum_{i=1}^mA_i -\expect{A_i}\r\| \geq t \r)
\leq d \exp\l( -\frac38\l( \frac{mt^2}{\sigma^2}\wedge m\sqrt{\theta}t \r) \r)
= d \exp\l( -\frac38\l( \frac{mt^2}{\sigma^2}\wedge\frac{m^{3/4}\beta^{1/4}t}{\sigma^{1/2}} \r) \r),
\end{align*}
where $c$ is a fixed positive constant.
Taking $t = 3\sqrt{\frac{\sigma^2\beta}{\|\Sigma_0\|m}}$ gives 
\begin{multline*}
Pr\l( \l\|\frac1m\sum_{i=1}^mA_i -\expect{A_i}\r\| \geq 3\sqrt{\frac{\sigma^2}{m}\beta} \r)
\leq d \exp\l(-3\beta   \wedge  \l(m^{1/4}\beta^{3/4}\overline{d}^{1/4}\r) \r)
\leq d \exp(-\beta),
\end{multline*}
where $\overline{d} = \sigma^2/\|\Sigma_0\|^2\geq\sigma_0^2/\|\Sigma_0\|^2\geq\tr(\Sigma_0)/\|\Sigma_0\|\geq1$ and
 the last inequality follows from the assumption that $m\geq\beta$.
Overall, term (V) is bounded as follows
\[\text{(V)}\leq B_\beta\l( \frac{\sigma^2}{m}\beta \r)^{1/4} + 3B_\beta \sqrt{\frac{\sigma^2\beta}{\|\Sigma_0\|m}},\]
with probability at least $1-de^{-\beta}$. Note that $\expect{Z_iZ_i^T} = \Sigma_0$, then, combining with \eqref{first-term-bound},
the term (II) is bounded as

\begin{equation*}
\text{(II)}\leq B_\beta\left(\l\|\Sigma_0\r\|^{\frac12} + \sqrt2\sigma^{\frac12}\l(\frac{\beta}{m}\r)^{\frac14}\right)\l( 2B_\beta\sqrt\theta + B_\beta^2\theta \r)   +  B_\beta\l( \frac{\sigma^2}{m}\beta \r)^{1/4} + 3B_\beta \sqrt{\frac{\sigma^2\beta}{\|\Sigma_0\|m}},
\end{equation*}
with probability at least $1-2de^{-\beta}$. Substituting $B_\beta=11\sqrt{\frac{2\tr(\Sigma_0)\beta}{m}}$ and $\theta=\frac1\sigma\sqrt{\frac{\beta}{m}}$ gives
\begin{multline*}
\text{(II)}\leq 11\sqrt2\sqrt{\tr(\Sigma_0)\sigma}\l(\frac\beta m\r)^{\frac34}
+33\sqrt2\frac{\sqrt{\tr(\Sigma_0)}\sigma}{\|\Sigma_0\|^{1/2}}\frac\beta m
+484\|\Sigma_0\|^{1/2}\frac{\tr(\Sigma_0)}{\sigma^{1/2}}\l(\frac\beta m\r)^{\frac54}\\
+484\sqrt2\tr(\Sigma_0)\l(\frac\beta m\r)^{\frac32} + 2\sqrt2\cdot11^3\|\Sigma_0\|^{\frac12}
\frac{\tr(\Sigma_0)^{3/2}}{\sigma}\l(\frac\beta m\r)^2
+4\cdot11^3\frac{\tr(\Sigma_0)^{3/2}}{\sigma^{1/2}}\l(\frac\beta m \r)^{9/4}.
\end{multline*}
Using the bounds \eqref{support-1} and \eqref{support-2} with some algebraic manipulations, we have the second bound in Lemma \ref{key-sublemma} holds with probability at least $1-2de^{-\beta}$. 
\end{proof}

\subsection{Proof of Lemma \ref{ratio-bound}}
We divide our analysis into the following four cases:
\begin{enumerate}
\item If $\|Z_i\|_2^2\leq1/\theta$ and $\|Z_i-\mu\|_2^2\leq1/\theta$, then, we have $h_\mu(Z_i) = 1$.

\item If $\|Z_i\|_2^2\leq1/\theta$ and $\|Z_i-\mu\|_2^2>1/\theta$. Since $\|\mu\|\leq B_\beta$, it follows
$\|Z_i-\mu\|_2\leq\sqrt{1/\theta}+B_\beta$, and we have
\begin{align*}
h_\mu(Z_i) &= \frac{1/\theta}{\|Z_i-\mu\|_2^2}\leq1,\\
h_\mu(Z_i) &\geq \frac{1/\theta}{\l(\sqrt{1/\theta}+B_\beta\r)^2}
=\frac{1}{1+2B_\beta\sqrt{\theta} + B_\beta^2\theta}\\
&\geq 1 - 2B_\beta\sqrt{\theta} - B_\beta^2\theta,
\end{align*}
where the last inequality follows from the fact $\frac{1}{1+x}\geq1-x,~\forall x\geq0$.

\item If $\|Z_i\|_2^2>1/\theta$ and $\|Z_i-\mu\|_2^2 \leq 1/\theta$. Since $\|\mu\|_2\leq B_\beta$, it follows $\|Z_i\|_2\leq\sqrt{1/\theta}+B_\beta$, and we have
\begin{align*}
h_\mu(Z_i) &= \frac{\|Z_i\|_2^2}{1/\theta}\geq1,\\
h_\mu(Z_i) &\leq \frac{\l(\sqrt{1/\theta}+B_\beta\r)^2}{1/\theta}
=1+2B_\beta\sqrt{\theta} + B_\beta^2\theta.
\end{align*}

\item If $\|Z_i\|_2^2>1/\theta$ and $\|Z_i-\mu\|_2^2 > 1/\theta$. Then, we have
\begin{align*}
h_\mu(Z_i) &= \frac{\|Z_i\|_2^2}{\|Z_i - \mu\|_2^2}\leq \frac{(\|Z_i - \mu\|_2+B_\beta)^2}{\|Z_i - \mu\|_2^2}\\
&\leq\l(\frac{1/\sqrt{\theta}+B_\beta}{1/\sqrt{\theta}}\r)^2\leq  1+2B_\beta\sqrt{\theta} + B_\beta^2\theta,\\
h_\mu(Z_i) &\geq \frac{\|Z_i\|_2^2}{(\|Z_i\|_2 + B_\beta)^2}
\geq\l(\frac{1/\sqrt{\theta}}{1/\sqrt{\theta}+B_\beta}\r)^2\\
&=\frac{1}{1+2B_\beta\sqrt{\theta} + B_\beta^2\theta}
\geq 1 - 2B_\beta\sqrt{\theta} - B_\beta^2\theta,
\end{align*}
\end{enumerate}
Overall, we proved the lemma.

\subsection{Proof of Lemma \ref{bound-on-sigma}}
By definition,
\[
B = \sup_{\|\mathbf{v}\|_2\leq1}\expect{|\langle\mathbf{v},X\rangle|^4}
\geq \expect{\left|X^j\right|^4},~\forall j=1,2,\cdots,d,
\]
where $X^j$ denotes the $j$-th entry of the random vector $X$. Also, for any fixed vector $\mathbf{v}\in\mathbb{R}^d$, we have

\begin{align*}
&0\leq \expect{\left(|\langle\mathbf{v},X\rangle|^2-\left|X^j\right|^2\right)^2}
= \expect{|\langle\mathbf{v},X\rangle|^4} + \expect{\left|X^j\right|^2} 
- 2 \expect{|\langle\mathbf{v},X\rangle|^2\left|X^j\right|^2}\\
&\Rightarrow
\expect{|\langle\mathbf{v},X\rangle|^4} + \expect{\left|X^j\right|^2}\geq
2 \expect{|\langle\mathbf{v},X\rangle|^2\left|X^j\right|^2},~~
\forall j=1,2,\cdots,d.
\end{align*}
Taking the supremum from both sides of the above inequality and use the previous bound on $B$, we get
\[
\sup_{\|\mathbf{v}\|_2\leq1}\expect{|\langle\mathbf{v},X\rangle|^4}
\geq \sup_{\|\mathbf{v}\|_2\leq1}\expect{|\langle\mathbf{v},X\rangle|^2\left|X^j\right|^2},~~\forall j=1,2,\cdots,d.
\]
Summing over $i=1,2,\cdots,d$ gives
\begin{multline*}
Bd = \sup_{\|\mathbf{v}\|_2\leq1}\expect{|\langle\mathbf{v},X\rangle|^4}d
\geq \sum_{j=1}^d\sup_{\|\mathbf{v}\|_2\leq1}\expect{|\langle\mathbf{v},X\rangle|^2\left|X^j\right|^2}
\geq \sup_{\|\mathbf{v}\|_2\leq1}\expect{|\langle\mathbf{v},X\rangle|^2\left\|X\right\|^2}\\
=\left\| XX^T\|X\|_2^2 \right\| = \sigma_0^2.
\end{multline*}

\subsection{Proof of Lemma \ref{effective-rank-bound}}
First of all, let $Z = X-\mu_0$, then, we have $\expect{Z} = 0$. The lower bound of $\sigma_0^2$ follows directly from Corollary \ref{FKG-bound}. It remains to show the upper bound.
Note that by Cauchy-Schwarz inequality, 
\begin{align*}
\sigma_0^2=\left\| ZZ^T\|Z\|_2^2 \right\| 
=&   \sup_{\|\mathbf{v}\|_2\leq1}\expect{\langle Z,\mathbf{v}\rangle^2\|Z\|_2^2}\\
\leq& \sup_{\|\mathbf{v}\|_2\leq1}\expect{\langle Z,\mathbf{v}\rangle^4}^{1/2}
\expect{\| Z \|_2^4}^{1/2}.
\end{align*}
We then bound the two terms separately. For any vector $\mathbf{x}\in\mathbb{R}^d$, let $x^j$ be the $j$-th entry. Note that for any $\mathbf{v}\in \mathbb{R}^d$ such that $\|\mathbf{v}\|_2\leq1$, we have
\begin{align*}
\expect{\langle Z,\mathbf{v}\rangle^4}^{1/2}
\leq R\cdot \expect{\langle Z,\mathbf{v}\rangle^2} \leq R \sup_{\|\mathbf{v}\|_2\leq1}\expect{\langle Z,\mathbf{v}\rangle^2}\leq R  \|\Sigma_0\|,
\end{align*}
where the first inequality uses the fact that the kurtosis is bounded.

Also, we have
\begin{align*}
\expect{\|Z\|_2^4}^{1/2} =&\l( \sum_{j=1}^d\expect{(Z^j)^4} + \sum_{j,k=1,~j\neq k}^d\expect{(Z^j)^2(Z^k)^2}\r)^{1/2}\\
\leq& \l( \sum_{j=1}^d\expect{(Z^j)^4} +  \sum_{j,k=1,~j\neq k}^d \expect{(Z^j)^4}^{1/2}\expect{(Z^k)^4}^{1/2} \r)^{1/2}\\
\leq& \sum_{j=1}^d\sqrt{\expect{(Z^j)^4}}\leq  R\cdot\sum_{j=1}^d\expect{(Z^j)^2}
=  R\cdot\tr(\Sigma_0)
\end{align*}
Combining the above two bounds gives 
\[
\sigma_0^2\leq  R^2\|\Sigma_0\|\tr(\Sigma_0),\]
which implies the result.


\bibliographystyle{alphaabbr}
\bibliography{bibliography,bibliography2,bibliography3,asyn-theory,bibliography_cov}

\newcommand{\etalchar}[1]{$^{#1}$}
\begin{thebibliography}{VDVW96a}

\bibitem[ABB00]{alter2000singular}
O.~Alter, P.~O. Brown, and D.~Botstein.
\newblock Singular value decomposition for genome-wide expression data
  processing and modeling.
\newblock {\em Proceedings of the National Academy of Sciences},
  97(18):10101--10106, 2000.

\bibitem[ACM12]{allard2012multi}
W.~K. Allard, G.~Chen, and M.~Maggioni.
\newblock Multi-scale geometric methods for data sets ii: Geometric
  multi-resolution analysis.
\newblock {\em Applied and Computational Harmonic Analysis}, 32(3):435--462,
  2012.

\bibitem[ALPV14]{ai2014one}
A.~Ai, A.~Lapanowski, Y.~Plan, and R.~Vershynin.
\newblock One-bit compressed sensing with non-{G}aussian measurements.
\newblock {\em Linear Algebra and its Applications}, 441:222--239, 2014.

\bibitem[Bal97]{convex-geometry-Ball}
K.~Ball.
\newblock {\em An elementary introduction to modern convex geometry}.
\newblock Cambridge University Press, New York,, 1997.

\bibitem[BBM{\etalchar{+}}05]{bartlett2005local}
P.~L. Bartlett, O.~Bousquet, S.~Mendelson, et~al.
\newblock Local rademacher complexities.
\newblock {\em The Annals of Statistics}, 33(4):1497--1537, 2005.

\bibitem[BCFS14]{m-estimator-2}
A.~Banerjee, S.~Chen, F.~Fazayeli, and V.~Sivakumar.
\newblock Estimation with norm regularization.
\newblock {\em Advances Neural Information Processing Systems (NIPS) 27}, 2014.

\bibitem[Bha13]{bhatia2013matrix}
R.~Bhatia.
\newblock {\em Matrix analysis}, volume 169.
\newblock Springer Science \& Business Media, 2013.

\bibitem[BLM13]{boucheron2013concentration}
S.~Boucheron, G.~Lugosi, and P.~Massart.
\newblock {\em Concentration inequalities: A nonasymptotic theory of
  independence}.
\newblock Oxford university press, 2013.

\bibitem[BM02]{bartlett2002rademacher}
P.~L. Bartlett and S.~Mendelson.
\newblock Rademacher and gaussian complexities: Risk bounds and structural
  results.
\newblock {\em Journal of Machine Learning Research}, 3(Nov):463--482, 2002.

\bibitem[Bri83]{Brillinger1983A-generalized-l00}
D.~R. Brillinger.
\newblock A generalized linear model with ``{G}aussian'' regressor variables.
\newblock In {\em A {F}estschrift for {E}rich {L}. {L}ehmann}, Wadsworth
  Statist./Probab. Ser., pages 97--114. Wadsworth, Belmont, CA, 1983.

\bibitem[Bro86]{brown1986fundamentals}
L.~D. Brown.
\newblock Fundamentals of statistical exponential families: with applications
  in statistical decision theory.
\newblock Ims, 1986.

\bibitem[BRT09]{bickel2009simultaneous}
P.~Bickel, Y.~Ritov, and A.~Tsybakov.
\newblock Simultaneous analysis of {L}asso and {D}antzig selector.
\newblock {\em The Annals of Statistics}, 37(4):1705--1732, 2009.

\bibitem[BY16]{balasubramanian2016discussion}
K.~Balasubramanian and M.~Yuan.
\newblock Discussion of {``Estimating structured high-dimensional covariance
  and precision matrices: optimal rates and adaptive estimation''}.
\newblock {\em Electronic Journal of Statistics}, 10(1):71--73, 2016.

\bibitem[Can08]{candes2008restricted}
E.~Cand{\`e}s.
\newblock The restricted isometry property and its implications for compressed
  sensing.
\newblock {\em Comptes Rendus Mathematique}, 346(9):589--592, 2008.

\bibitem[Cat12]{catoni2012challenging}
O.~Catoni.
\newblock Challenging the empirical mean and empirical variance: a deviation
  study.
\newblock In {\em Annales de l'Institut Henri Poincar{\'e}, Probabilit{\'e}s et
  Statistiques}, volume~48, pages 1148--1185, 2012.

\bibitem[Cat16]{catoni2016pac}
O.~Catoni.
\newblock {PAC}-{B}ayesian bounds for the {G}ram matrix and least squares
  regression with a random design.
\newblock {\em arXiv preprint arXiv:1603.05229}, 2016.

\bibitem[CGR15]{chen2015robust}
M.~Chen, C.~Gao, and Z.~Ren.
\newblock Robust covariance matrix estimation via matrix depth.
\newblock {\em arXiv preprint arXiv:1506.00691}, 2015.

\bibitem[CHS81]{elliptical-paper}
S.~Cambanis, S.~Huang, and G.~Simons.
\newblock On the theory of elliptically contoured distributions.
\newblock {\em Journal of Multivariate Analysis}, 11:368--385, 1981.

\bibitem[CRPW12]{chandrasekaran2012convex}
V.~Chandrasekaran, B.~Recht, P.~A. Parrilo, and A.~S. Willsky.
\newblock The convex geometry of linear inverse problems.
\newblock {\em Foundations of Computational mathematics}, 12(6):805--849, 2012.

\bibitem[CRT04]{candes2004robust}
E.~Candes, J.~Romberg, and T.~Tao.
\newblock Robust uncertainty principles: Exact signal reconstruction from
  highly incomplete frequency information.
\newblock {\em arXiv preprint math/0409186}, 2004.

\bibitem[CRT06]{candes2006stable}
E.~Cand{\`e}s, J.~Romberg, and T.~Tao.
\newblock Stable signal recovery from incomplete and inaccurate measurements.
\newblock {\em Communications on pure and applied mathematics},
  59(8):1207--1223, 2006.

\bibitem[CRZ16]{cai2016}
T.~T. Cai, Z.~Ren, and H.~H. Zhou.
\newblock Estimating structured high-dimensional covariance and precision
  matrices: optimal rates and adaptive estimation.
\newblock {\em Electron. J. Statist.}, 10(1):1--59, 2016.

\bibitem[DDS{\etalchar{+}}09]{deng2009imagenet}
J.~Deng, W.~Dong, R.~Socher, L.-J. Li, K.~Li, and L.~Fei-Fei.
\newblock Imagenet: A large-scale hierarchical image database.
\newblock In {\em 2009 IEEE conference on computer vision and pattern
  recognition}, pages 248--255. Ieee, 2009.

\bibitem[Dir13]{tail-bound-chaining}
S.~Dirksen.
\newblock Tail bounds via generic chaining.
\newblock {\em arXiv preprint arXiv:1309.3522}, 2013.

\bibitem[DK70]{davis1970rotation}
C.~Davis and W.~M. Kahan.
\newblock The rotation of eigenvectors by a perturbation. iii.
\newblock {\em SIAM Journal on Numerical Analysis}, 7(1):1--46, 1970.

\bibitem[DKK{\etalchar{+}}16]{diakonikolas2016robust}
I.~Diakonikolas, G.~Kamath, D.~M. Kane, J.~Li, A.~Moitra, and A.~Stewart.
\newblock Robust estimators in high dimensions without the computational
  intractability.
\newblock In {\em Foundations of Computer Science (FOCS), 2016 IEEE 57th Annual
  Symposium on}, pages 655--664. IEEE, 2016.

\bibitem[Dur19]{durrett2019probability}
R.~Durrett.
\newblock {\em Probability: theory and examples}, volume~49.
\newblock Cambridge university press, 2019.

\bibitem[FK17]{fan2017robust}
J.~Fan and D.~Kim.
\newblock Robust high-dimensional volatility matrix estimation for
  high-frequency factor model.
\newblock {\em Journal of the American Statistical Association}, 2017.

\bibitem[FKN90]{fang1990symmetric}
K.-T. Fang, S.~Kotz, and K.~W. Ng.
\newblock {\em Symmetric multivariate and related distributions}.
\newblock Chapman and Hall, 1990.

\bibitem[FLL16]{fan2016overview}
J.~Fan, Y.~Liao, and H.~Liu.
\newblock An overview of the estimation of large covariance and precision
  matrices.
\newblock {\em The Econometrics Journal}, 19(1):C1--C32, 2016.

\bibitem[FLW17]{fan2017estimation}
J.~Fan, Q.~Li, and Y.~Wang.
\newblock Estimation of high dimensional mean regression in the absence of
  symmetry and light tail assumptions.
\newblock {\em Journal of the Royal Statistical Society: Series B (Statistical
  Methodology)}, 79(1):247--265, 2017.

\bibitem[FWZ16a]{fan2016eigenvector}
J.~Fan, W.~Wang, and Y.~Zhong.
\newblock An $\ell_\infty$ eigenvector perturbation bound and its application
  to robust covariance estimation.
\newblock {\em arXiv preprint arXiv:1603.03516}, 2016.

\bibitem[FWZ16b]{truncation-paper}
J.~Fan, W.~Wang, and Z.~Zhu.
\newblock Robust low-rank matrix recovery.
\newblock {\em arXiv:1603.08315}, 2016.

\bibitem[FWZ16c]{fan2016robust}
J.~Fan, W.~Wang, and Z.~Zhu.
\newblock Robust low-rank matrix recovery.
\newblock {\em arXiv preprint arXiv:1603.08315}, 2016.

\bibitem[FWZ17]{Fan-robust-estimation-2017}
J.~Fan, W.~Wang, and Z.~Zhu.
\newblock {A Shrinkage Principle for Heavy-Tailed Data: High-Dimensional Robust
  Low-Rank Matrix Recovery}.
\newblock {\em arXiv preprint arXiv:1603.08315v2}, 2017.

\bibitem[Gen16]{genzel2016high}
M.~Genzel.
\newblock High-dimensional estimation of structured signals from non-linear
  observations with general convex loss functions.
\newblock {\em arXiv preprint arXiv:1602.03436}, 2016.

\bibitem[Giu15]{giulini2015pac}
I.~Giulini.
\newblock {PAC-Bayesian} bounds for {Principal Component Analysis} in {Hilbert}
  spaces.
\newblock {\em arXiv preprint arXiv:1511.06263}, 2015.

\bibitem[GMW16]{goldstein2016structured}
L.~Goldstein, S.~Minsker, and X.~Wei.
\newblock Structured signal recovery from non-linear and heavy-tailed
  measurements.
\newblock {\em arXiv preprint arXiv:1609.01025}, 2016.

\bibitem[Gor88]{gordon1988milman}
Y.~Gordon.
\newblock On milman's inequality and random subspaces which escape through a
  mesh in $\mathbb{R}^n$.
\newblock In {\em Geometric Aspects of Functional Analysis}, pages 84--106.
  Springer, 1988.

\bibitem[GW18]{goldstein2018non}
L.~Goldstein and X.~Wei.
\newblock Non-gaussian observations in nonlinear compressed sensing via stein
  discrepancies.
\newblock {\em Information and Inference: A Journal of the IMA}, 8(1):125--159,
  2018.

\bibitem[GW19]{goldstein2019non}
L.~Goldstein and X.~Wei.
\newblock Non-gaussian observations in nonlinear compressed sensing via stein
  discrepancies.
\newblock {\em Information and Inference: A Journal of the IMA}, 8(1):125--159,
  2019.

\bibitem[HHI{\etalchar{+}}93]{hardle1993optimal}
W.~Hardle, P.~Hall, H.~Ichimura, et~al.
\newblock Optimal smoothing in single-index models.
\newblock {\em The annals of Statistics}, 21(1):157--178, 1993.

\bibitem[HJS01]{hristache2001direct}
M.~Hristache, A.~Juditsky, and V.~Spokoiny.
\newblock Direct estimation of the index coefficient in a single-index model.
\newblock {\em Annals of Statistics}, pages 595--623, 2001.

\bibitem[HL16]{han2016eca}
F.~Han and H.~Liu.
\newblock Eca: High dimensional elliptical component analysis in non-gaussian
  distributions.
\newblock {\em Journal of the American Statistical Association},
  (just-accepted), 2016.

\bibitem[Hot33]{hotelling1933analysis}
H.~Hotelling.
\newblock Analysis of a complex of statistical variables into principal
  components.
\newblock {\em Journal of educational psychology}, 24(6):417, 1933.

\bibitem[HRVA08]{hubert2008high}
M.~Hubert, P.~J. Rousseeuw, and S.~Van~Aelst.
\newblock High-breakdown robust multivariate methods.
\newblock {\em Statistical Science}, pages 92--119, 2008.

\bibitem[HTW15]{hastie2015statistical}
T.~Hastie, R.~Tibshirani, and M.~Wainwright.
\newblock {\em Statistical learning with sparsity: the lasso and
  generalizations}.
\newblock CRC press, 2015.

\bibitem[KLT11]{koltchinskii2011nuclear}
V.~Koltchinskii, K.~Lounici, and A.~B. Tsybakov.
\newblock Nuclear-norm penalization and optimal rates for noisy low-rank matrix
  completion.
\newblock {\em The Annals of Statistics}, 39(5):2302--2329, 2011.

\bibitem[KM15]{koltchinskii2015bounding}
V.~Koltchinskii and S.~Mendelson.
\newblock Bounding the smallest singular value of a random matrix without
  concentration.
\newblock {\em International Mathematics Research Notices},
  2015(23):12991--13008, 2015.

\bibitem[LD89]{li1989regression}
K.-C. Li and N.~Duan.
\newblock Regression analysis under link violation.
\newblock {\em The Annals of Statistics}, pages 1009--1052, 1989.

\bibitem[Lep92]{lepskii1992asymptotically}
O.~Lepski.
\newblock Asymptotically minimax adaptive estimation. {I}: Upper bounds.
  optimally adaptive estimates.
\newblock {\em Theory of Probability \& Its Applications}, 36(4):682--697,
  1992.

\bibitem[LL17]{lecue2017robust}
G.~Lecu{\'e} and M.~Lerasle.
\newblock Robust machine learning by median-of-means: theory and practice.
\newblock {\em arXiv preprint arXiv:1711.10306}, 2017.

\bibitem[LM16]{lugosi2016risk}
G.~Lugosi and S.~Mendelson.
\newblock Risk minimization by median-of-means tournaments.
\newblock {\em arXiv preprint arXiv:1608.00757}, 2016.

\bibitem[LM17a]{lecue2017regularization}
G.~Lecu{\'e} and S.~Mendelson.
\newblock Regularization and the small-ball method ii: complexity dependent
  error rates.
\newblock {\em The Journal of Machine Learning Research}, 18(1):5356--5403,
  2017.

\bibitem[LM17b]{lecue2017sparse}
G.~Lecu{\'e} and S.~Mendelson.
\newblock Sparse recovery under weak moment assumptions.
\newblock {\em Journal of the European Mathematical Society}, 19(3):881--904,
  2017.

\bibitem[LMPV17]{liaw2017simple}
C.~Liaw, A.~Mehrabian, Y.~Plan, and R.~Vershynin.
\newblock A simple tool for bounding the deviation of random matrices on
  geometric sets.
\newblock In {\em Geometric aspects of functional analysis}, pages 277--299.
  Springer, 2017.

\bibitem[Lou14]{lounici2014high}
K.~Lounici.
\newblock High-dimensional covariance matrix estimation with missing
  observations.
\newblock {\em Bernoulli}, 20(3):1029--1058, 2014.

\bibitem[LT91]{Talagrand-book}
M.~Ledoux and M.~Talagrand.
\newblock {\em Probability in Banach Spaces: isoperimetry and processes}.
\newblock Springer-Verlag, Berlin, 1991.

\bibitem[LT13]{ledoux2013probability}
M.~Ledoux and M.~Talagrand.
\newblock {\em Probability in Banach Spaces: isoperimetry and processes}.
\newblock Springer Science \& Business Media, 2013.

\bibitem[LW04]{ledoit2004well}
O.~Ledoit and M.~Wolf.
\newblock A well-conditioned estimator for large-dimensional covariance
  matrices.
\newblock {\em Journal of multivariate analysis}, 88(2):365--411, 2004.

\bibitem[LW{\etalchar{+}}12]{ledoit2012nonlinear}
O.~Ledoit, M.~Wolf, et~al.
\newblock Nonlinear shrinkage estimation of large-dimensional covariance
  matrices.
\newblock {\em The Annals of Statistics}, 40(2):1024--1060, 2012.

\bibitem[M{\etalchar{+}}00]{massart2000constants}
P.~Massart et~al.
\newblock About the constants in talagrand's concentration inequalities for
  empirical processes.
\newblock {\em The Annals of Probability}, 28(2):863--884, 2000.

\bibitem[Men14a]{mendelson2014learning}
S.~Mendelson.
\newblock Learning without concentration.
\newblock In {\em Conference on Learning Theory}, pages 25--39, 2014.

\bibitem[Men14b]{Mendelson-2}
S.~Mendelson.
\newblock Upper bounds on product and multiplier empirical processes.
\newblock {\em arXiv preprint arXiv:1410.8003}, 2014.

\bibitem[Men16]{mendelson2016upper}
S.~Mendelson.
\newblock Upper bounds on product and multiplier empirical processes.
\newblock {\em Stochastic Processes and their Applications},
  126(12):3652--3680, 2016.

\bibitem[Min15]{minsker2013geometric}
S.~Minsker.
\newblock Geometric median and robust estimation in {B}anach spaces.
\newblock {\em Bernoulli}, 21(4):2308--2335, 2015.

\bibitem[Min16]{minsker2016sub}
S.~Minsker.
\newblock {Sub-Gaussian} estimators of the mean of a random matrix with
  heavy-tailed entries.
\newblock {\em arXiv preprint arXiv:1605.07129}, 2016.

\bibitem[MP12]{mendelson2012generic}
S.~Mendelson and G.~Paouris.
\newblock On generic chaining and the smallest singular value of random
  matrices with heavy tails.
\newblock {\em Journal of Functional Analysis}, 262(9):3775--3811, 2012.

\bibitem[MPTJ07]{Mendelson-1}
S.~Mendelson, A.~Pajor, and N.~Tomczak-Jaegermann.
\newblock Reconstruction and subgaussian operators in asymptotic geometric
  analysis.
\newblock {\em Geometric and Functional Analysis}, 17(4):1248--1282, 2007.

\bibitem[MS90a]{montgomery1990distribution}
S.~J. Montgomery-Smith.
\newblock The distribution of rademacher sums.
\newblock {\em Proceedings of the American Mathematical Society},
  109(2):517--522, 1990.

\bibitem[MS90b]{Rademancher-sums}
S.~J. Montgomery-Smith.
\newblock The distribution of rademacher sums.
\newblock In {\em Proceedings of the AMS}, pages 517--522, 109(2), 1990.

\bibitem[MW{\etalchar{+}}20]{minsker2020robust}
S.~Minsker, X.~Wei, et~al.
\newblock Robust modifications of u-statistics and applications to covariance
  estimation problems.
\newblock {\em Bernoulli}, 26(1):694--727, 2020.

\bibitem[NJB{\etalchar{+}}08]{novembre2008genes}
J.~Novembre, T.~Johnson, K.~Bryc, Z.~Kutalik, A.~R. Boyko, A.~Auton, A.~Indap,
  K.~S. King, S.~Bergmann, M.~R. Nelson, et~al.
\newblock Genes mirror geography within {E}urope.
\newblock {\em Nature}, 456(7218):98--101, 2008.

\bibitem[NRW{\etalchar{+}}12]{negahban2012unified}
S.~N. Negahban, P.~Ravikumar, M.~J. Wainwright, B.~Yu, et~al.
\newblock A unified framework for high-dimensional analysis of $ m $-estimators
  with decomposable regularizers.
\newblock {\em Statistical Science}, 27(4):538--557, 2012.

\bibitem[NRWY12]{general-m-estimator}
S.~N. Negahban, P.~Ravikumar, M.~J. Wainwright, and B.~Yu.
\newblock A unified framework for high-dimensional analysis of m-estimators
  with decomposable regularizers.
\newblock {\em Statistical Science}, 27(4):538--557, 2012.

\bibitem[Oli13]{oliveira2013lower}
R.~I. Oliveira.
\newblock The lower tail of random quadratic forms, with applications to
  ordinary least squares and restricted eigenvalue properties.
\newblock {\em arXiv preprint arXiv:1312.2903}, 2013.

\bibitem[PV16]{plan2016generalized}
Y.~Plan and R.~Vershynin.
\newblock The generalized {L}asso with non-linear observations.
\newblock {\em IEEE Transactions on Information Theory}, 62(3):1528--1537,
  2016.

\bibitem[PVY14]{plan2014high}
Y.~Plan, R.~Vershynin, and E.~Yudovina.
\newblock High-dimensional estimation with geometric constraints.
\newblock {\em arXiv preprint arXiv:1404.3749}, 2014.

\bibitem[PVY16]{plan2016high}
Y.~Plan, R.~Vershynin, and E.~Yudovina.
\newblock High-dimensional estimation with geometric constraints.
\newblock {\em Information and Inference: A Journal of the IMA}, 6(1):1--40,
  2016.

\bibitem[PZ30]{paley1930some}
R.~Paley and A.~Zygmund.
\newblock On some series of functions,(1).
\newblock In {\em Mathematical Proceedings of the Cambridge Philosophical
  Society}, volume~26, pages 337--357. Cambridge University Press, 1930.

\bibitem[RV08]{rudelson2008sparse}
M.~Rudelson and R.~Vershynin.
\newblock On sparse reconstruction from fourier and gaussian measurements.
\newblock {\em Communications on Pure and Applied Mathematics: A Journal Issued
  by the Courant Institute of Mathematical Sciences}, 61(8):1025--1045, 2008.

\bibitem[RWY11]{raskutti2011minimax}
G.~Raskutti, M.~J. Wainwright, and B.~Yu.
\newblock Minimax rates of estimation for high-dimensional linear regression
  over $\ell_q$-balls.
\newblock {\em IEEE transactions on information theory}, 57(10):6976--6994,
  2011.

\bibitem[SJH{\etalchar{+}}07]{saal2007poor}
L.~H. Saal, P.~Johansson, K.~Holm, S.~K. Gruvberger-Saal, Q.-B. She, M.~Maurer,
  S.~Koujak, A.~A. Ferrando, P.~Malmstr{\"o}m, L.~Memeo, et~al.
\newblock Poor prognosis in carcinoma is associated with a gene expression
  signature of aberrant {PTEN} tumor suppressor pathway activity.
\newblock {\em Proceedings of the National Academy of Sciences},
  104(18):7564--7569, 2007.

\bibitem[Sto86]{stoker1986consistent}
T.~M. Stoker.
\newblock Consistent estimation of scaled coefficients.
\newblock {\em Econometrica: Journal of the Econometric Society}, pages
  1461--1481, 1986.

\bibitem[SZF17]{Fan-huber-regression-2017}
Q.~Sun, W.~Zhou, and J.~Fan.
\newblock {Adaptive Huber Regression: Optimality and Phase Transition}.
\newblock {\em arXiv preprint arXiv:1706.06991}, 2017.

\bibitem[TAH15]{thrampoulidis2015lasso}
C.~Thrampoulidis, E.~Abbasi, and B.~Hassibi.
\newblock {Lasso} with non-linear measurements is equivalent to one with linear
  measurements.
\newblock In {\em Advances in Neural Information Processing Systems}, pages
  3420--3428, 2015.

\bibitem[Tal95]{talagrand1995concentration}
M.~Talagrand.
\newblock Concentration of measure and isoperimetric inequalities in product
  spaces.
\newblock {\em Publications Math{\'e}matiques de l'Institut des Hautes Etudes
  Scientifiques}, 81(1):73--205, 1995.

\bibitem[Tal14a]{talagrand2014upper}
M.~Talagrand.
\newblock {\em Upper and lower bounds for stochastic processes: modern methods
  and classical problems}, volume~60.
\newblock Springer Science \& Business Media, 2014.

\bibitem[Tal14b]{Talagrand-book-2}
M.~Talagrand.
\newblock {\em Upper and lower bounds for stochastic processes: modern methods
  and classical problems}.
\newblock Ergebnisse der Mathematik und ihrer Grenzgebiete, Springer, 2014.

\bibitem[Tib96]{tibshirani1996regression}
R.~Tibshirani.
\newblock Regression shrinkage and selection via the {L}asso.
\newblock {\em Journal of the Royal Statistical Society. Series B
  (Methodological)}, pages 267--288, 1996.

\bibitem[Tro12]{tropp1}
J.~A. Tropp.
\newblock User-friendly tail bounds for sums of random matrices.
\newblock {\em Found. Comput. Math.}, 12(4):389--434, 2012.

\bibitem[Tro15a]{tropp2015convex}
J.~A. Tropp.
\newblock Convex recovery of a structured signal from independent random linear
  measurements.
\newblock In {\em Sampling Theory, a Renaissance}, pages 67--101. Springer,
  2015.

\bibitem[Tro15b]{tropp2015introduction}
J.~A. Tropp.
\newblock An introduction to matrix concentration inequalities.
\newblock {\em arXiv preprint arXiv:1501.01571}, 2015.

\bibitem[Tuk75]{tukey1975mathematics}
J.~W. Tukey.
\newblock Mathematics and the picturing of data.
\newblock In {\em Proceedings of the international congress of mathematicians},
  volume~2, pages 523--531, 1975.

\bibitem[Tyl87]{tyler1987distribution}
D.~E. Tyler.
\newblock A distribution-free {M}-estimator of multivariate scatter.
\newblock {\em The Annals of Statistics}, pages 234--251, 1987.

\bibitem[VDVW96a]{van1996weak}
A.~W. Van Der~Vaart and J.~A. Wellner.
\newblock Weak convergence.
\newblock In {\em Weak convergence and empirical processes}, pages 16--28.
  Springer, 1996.

\bibitem[vdVW96b]{wellner1}
A.~W. van~der Vaart and J.~A. Wellner.
\newblock {\em Weak convergence and empirical processes}.
\newblock Springer Series in Statistics. Springer-Verlag, New York, 1996.

\bibitem[Ver10a]{introduction-to-random-matrix}
R.~Vershynin.
\newblock Introduction to the non-asymptotic analysis of random matrices.
\newblock In Y.~C. Eldar and G.~Kutyniok, editors, {\em Compressed Sensing:
  Theory and Applications}. Cambridge University Press, 2010.

\bibitem[Ver10b]{vershynin2010lectures}
R.~Vershynin.
\newblock Lectures in functional analysis.
\newblock {\em Department of Mathematics, University of Michigan}, 2010.

\bibitem[Ver15]{vershynin2015estimation}
R.~Vershynin.
\newblock Estimation in high dimensions: a geometric perspective.
\newblock In {\em Sampling Theory, a Renaissance}, pages 3--66. Springer, 2015.

\bibitem[W{\etalchar{+}}13]{wellner2013weak}
J.~Wellner et~al.
\newblock {\em Weak convergence and empirical processes: with applications to
  statistics}.
\newblock Springer Science \& Business Media, 2013.

\bibitem[Wei18]{wei2018structured}
X.~Wei.
\newblock Structured recovery with heavy-tailed measurements: A thresholding
  procedure and optimal rates.
\newblock {\em arXiv preprint arXiv:1804.05959}, 2018.

\bibitem[WM17]{wei2017estimation}
X.~Wei and S.~Minsker.
\newblock Estimation of the covariance structure of heavy-tailed distributions.
\newblock In {\em Advances in Neural Information Processing Systems}, pages
  2859--2868, 2017.

\bibitem[WYG{\etalchar{+}}09]{face-recognition}
J.~Wright, A.~Yang, A.~Ganesh, S.~Sastry, and Y.~Ma.
\newblock {Robust face recognition via sparse representation.}
\newblock {\em IEEE Trans. PAMI}, 31(2):210--227, 2009.

\bibitem[WZ{\etalchar{+}}16]{wegkamp2016adaptive}
M.~Wegkamp, Y.~Zhao, et~al.
\newblock Adaptive estimation of the copula correlation matrix for
  semiparametric elliptical copulas.
\newblock {\em Bernoulli}, 22(2):1184--1226, 2016.

\bibitem[YBL17]{yang2017stein}
Z.~Yang, K.~Balasubramanian, and H.~Liu.
\newblock On stein's identity and near-optimal estimation in high-dimensional
  index models.
\newblock {\em arXiv preprint arXiv:1709.08795}, 2017.

\bibitem[YBWL17]{yang2017learning}
Z.~Yang, K.~Balasubramanian, Z.~Wang, and H.~Liu.
\newblock Learning non-gaussian multi-index model via second-order stein’s
  method.
\newblock {\em Advances in Neural Information Processing Systems},
  30:6097--6106, 2017.

\bibitem[YWCL15]{yi2015optimal}
X.~Yi, Z.~Wang, C.~Caramanis, and H.~Liu.
\newblock Optimal linear estimation under unknown nonlinear transform.
\newblock In {\em Advances in Neural Information Processing Systems}, pages
  1549--1557, 2015.

\end{thebibliography}

\end{document}